\documentclass[12pt]{amsart}

\usepackage[colorlinks=true,urlcolor=blue,
bookmarks=true,bookmarksopen=true,citecolor=blue
]{hyperref}

\usepackage{color,soul}
\setstcolor{red}
\usepackage{pdfsync}
\usepackage[all, cmtip]{xy}
\usepackage{graphicx}

\usepackage{epsfig}
\usepackage{amscd}
\usepackage[mathscr]{eucal}
\usepackage{amssymb}
\usepackage{amsxtra}
\usepackage{amsmath}
\usepackage{latexsym} 
\usepackage[all]{xy}
\usepackage{enumerate}
\usepackage{mathrsfs}

\usepackage{enumitem}

\usepackage{color}
\usepackage{xcolor, soul}
\sethlcolor{green}

\theoremstyle{plain}

\newtheorem{mainthm}{Theorem}
\newtheorem{thm}{Theorem}[subsection]

\newtheorem{cor}[thm]{Corollary}

\newtheorem{lem}[thm]{Lemma}
\newtheorem{prop}[thm]{Proposition}

\theoremstyle{definition}
\newtheorem{dfn}[thm]{Definition}

\theoremstyle{remark}
\newtheorem{rem}[thm]{Remark}
\newtheorem*{remnonum}{Remark}

\newtheorem*{remsnonum}{Remarks}

\theoremstyle{plain}

\newcommand{\Qed}{\hfill \qedsymbol \medskip}

\newcommand{\hooklongrightarrow}{\lhook\joinrel\longrightarrow}

\newcommand{\cobto}{\leadsto}
\newcommand{\id}{\textnormal{id}}

\newcommand{\fk}{\mathcal{F}uk}

\newcommand{\mntlf}{strongly monotone}
\newcommand{\Mntlf}{Strongly monotone}

\newcommand{\R}{\mathbb{R}}
\newcommand{\Z}{\mathbb{Z}}
\newcommand{\Q}{\mathbb{Q}}
\newcommand{\N}{\mathbb{N}}
\newcommand{\C}{\mathbb{C}}

\newcommand{\fuk}{\mathcal{F}uk}

\newcommand{\mor}{{\textnormal{Mor\/}}}

\newcommand{\form}{\Theta}

\newcommand{\trl}{trail}

\newcommand{\thmb}{\mathscr{T}}
\newcommand{\barthmb}{\bar{\mathscr{T}}}

\newtheorem*{assumption-T}{Assumption $T_{\infty}$ (Triviality at infinity)}

\newcommand{\pbred}[1]{#1}

\newcommand{\pbgreen}[1]{#1}

\newcommand{\pbhl}[1]{#1}

\newcommand{\pbaddress}{biran@math.ethz.ch}
\newcommand{\ocaddress}{cornea@dms.umontreal.ca}

\begin{document}

\title[Lagrangian Cobordism in Lefschetz fibrations]{Lagrangian
  Cobordism in Lefschetz Fibrations.}

\date{\today}

\thanks{The second author was supported by an NSERC Discovery grant, a FQRNT Group Research grant,
a Simons Fellowship and an Institute for Advanced Study fellowship grant.}

\author{Paul Biran and Octav Cornea}

\address{Paul Biran, Department of Mathematics, ETH-Z\"{u}rich,
  R\"{a}mistrasse 101, 8092 Z\"{u}rich, Switzerland}\email{\pbaddress} 
 
 \address{Octav Cornea, Department of Mathematics
  and Statistics, University of Montreal, C.P. 6128 Succ.  Centre-Ville
  Montreal, QC H3C 3J7, Canada and Institute for Advanced Study, Einstein Drive, Princeton NJ 08540, USA} \email{\ocaddress}

\bibliographystyle{alphanum}

% ----------------------------------------------------------------------
%

% ----------------------------------------------------------------------
%
% Abstract

\begin{abstract}
   Given a symplectic manifold $(M^{2n},\omega)$ we study Lagrangian
   cobordisms $V\subset E$ where $E$ is the total space of a Lefschetz
   fibration having $M$ as generic fiber. We prove a generation result
   for these cobordisms in the appropriate derived Fukaya category.
   As a corollary, we analyze the relations among the Lagrangian
   submanifolds $L\subset M$ that are induced by these cobordisms.
   This leads to a unified treatment - and a generalization - of the
   two types of relations among Lagrangian submanifolds of $M$ that
   were previously identified in the literature: those associated to
   Dehn twists that were discovered by Seidel~\cite{Se:long-exact} and
   the relations induced by cobordisms in trivial symplectic
   fibrations described in our previous work \cite{Bi-Co:cob1,
     Bi-Co:lcob-fuk}.
\end{abstract}

\maketitle

% ----------------------------------------------------------------------
%
% Beginning of text
%

\tableofcontents 

% !TEX root = lefcob.tex

\section{Introduction}
\subsection{Motivation}

The derived Fukaya category $D\fuk (N)$ of a symplectic manifold
$(N,\omega)$ is a triangulated category whose objects are obtained as
the completion of a certain class - here denoted by $\mathcal{L}(N)$ -
of Lagrangian submanifolds of $N$. The completion can be summarized as
follows.  As a set, each Lagrangian $L$ can be described as a
collection of sets each consisting of intersection points $L'\cap L$
where $L'$ is a variable Lagrangian transverse to $L$.  This family of
intersection points can be assembled in a family of vector spaces
$\Z_{2} \langle L'\cap L \rangle$ again with $L'$ viewed as a
variable.  In the absence of some coherence relations among all these
vector spaces this is obviously not a useful description of $L$.
However, given some almost complex structure $J$, compatible with
$\omega$, there are natural relations among the vector spaces
$\Z_{2}\langle -\cap L\rangle$ that reflect the existence of
$J$-holomorphic curves with Lagrangian boundary conditions along
families $L_{1},\ldots, L_{k}\in \mathcal{L}(N)$ and $L$.  The formal
way to express this is to construct first an $A_{\infty}$-category
$\fuk(N)$ called the Fukaya category of $N$ with objects
$\mathcal{L}(N)$, with morphisms the vector spaces
$\hom (L',L'')=\Z_{2}\langle L'\cap L'' \rangle$ and so that the
higher multiplications $\mu_{k}$ are given by counts of
$J$-holomorphic polygons with boundary components along
$L_{1}, L_{2},\ldots L_{k+1}$.  In this formalism the family
$\Z_{2} \langle -\cap L \rangle$ becomes a module over $\fuk(N)$,
called the Yoneda module associated to $L$, $\mathcal{Y}(L)$. The
modules over an $A_{\infty}$-category are algebraic objects
that behave in ways very similar to chain complexes. In particular,
given a morphism between two modules $f:\mathcal{M}\to\mathcal{M'}$,
one can take the cone over it $\mathcal{M}''=\textnormal{cone}(f)$,
which is a module given by a formula similar to the cone over a chain
map. The category $D\fuk (N)$ has as objects all the modules that can
be obtained by iterated cones from the Yoneda modules. The morphisms
in this category are the homology classes of the module
morphisms. The exact triangles are the homology images of the
  chain-level triangles of morphisms that are quasi-isomorphic to the
  module-level cone attachments. We refer to~\cite[Section
3e]{Se:book-fukaya-categ} for the detailed construction. We remark
that our variant of the derived Fukaya category is not completed with
respect to idempotents, by contrast to other versions of this notion
that are present in the literature. Note also that in this
  paper we work with ungraded $A_{\infty}$-categories, in particular
  there are no shift operations.

Two closely inter-related types of results are key from this
perspective. The first is decomposition results, that show that all
objects in some class can be decomposed in $D\fuk (-)$ in terms of
basic objects, similarly to the way a $CW$-complex can be
decomposed into cells. The second one is constructive results
producing exact triangles in $D\fuk(-)$ out of geometric structures or
operations.

\subsection{Main result}
The main aim of this paper it to prove a decomposition result for a
class of Lagrangian submanifolds with cylindrical ends - called
cobordisms - that are embedded in the total space of a Lefschetz
fibration $\pi :E\to \C$.  We consider here such cobordisms $V$ with
``negative'' ends only: outside of a compact subset, the
  projection of $V$ to $\C$ is a union of rays of the type
  $\ell_{i}=(-\infty, a_{i}]\times \{i\}$, $i \in \mathbb{N}$. Such
cobordisms will be called negatively-ended.

We work with uniformly monotone Lagrangians and with a class of
Lefschetz fibrations that satisfy a strong variant of the monotonicity
condition - see~\S\ref{subsec:Fuk-fibr},~\S\ref{sb:monlef} for the
definitions. Let $\mathcal{L}^{\ast}(E)$ be the class of these
cobordisms in $E$. The superscript ${-}^{\ast}$ will denote at all
times below the monotonicity constraint imposed on the Lagrangians
involved. We denote by $\mathcal{A}$ the universal Novikov ring over
the base field $\Z_{2}$. The Fukaya categories in this paper
  will generally be over the field $\mathcal{A}$. Finally, recall that
  we work at all times in an ungraded context.

We state here the main decomposition result and refer to
\S\ref{subsec:main-decomp} where the result is restated after making
the various ingredients more precise. Our conventions and notation
regarding iterated cone decompositions are explained
in~\S\ref{sbsb:iter-cone}.  \pbhl{Henceforth we make the following
  standing assumption: all our Lefschetz fibrations $E$ are assumed to
  have a positive dimensional fiber (hence
  $\dim_{\mathbb{R}}E \geq 4$).}

\begin{mainthm}\label{thm:main-dec-gen0} 
   There exists a Fukaya category with objects the cobordisms in
   $\mathcal{L}^{\ast}(E)$.  Let $D\fuk^{\ast}(E)$ be the associated
   derived Fukaya category. Consider one object,
   $V\in\mathcal{L}^{\ast}(E)$, fix points $z_{i}\in \ell_{i}$ along
   the rays associated to $V$ and let $L_{i}=V\cap \pi^{-1}(z_{i})$.
   Let $T_{i}$ be the thimbles associated to the curves $t_{i}$ as in
   Figure \ref{fig:gamma-thimbles0}, and let $\gamma_i L_i \subset E$
   be obtained by the (union of) parallel transports of $L_{i}$ along
   the curve $\gamma_{i}$, in the same figure.

   There exist finite rank $\mathcal{A}$-modules $E_{k}$, $1\leq k\leq
   m$, and an iterated cone decomposition taking place in
   $D\fuk^{\ast}(E)$:
   $$V \cong (T_{1}\otimes E_{1}\to T_{2}\otimes E_{2}\to
   \ldots \to T_{m}\otimes E_{m}\to\gamma_{s} L_s \to
   \gamma_{s-1}L_{s-1} \to \ldots \to \gamma_2 L_2 )~.~$$
\end{mainthm}
The precise meaning of the notation in the last formula will be be
explained in~\S\ref{sbsb:iter-cone}. The $\mathcal{A}$-modules
  $E_{i}$ are made explicit in the proof - see~\eqref{eq:E-i-s}. For
  the time being, let us only mention that they are obtained as Floer
  homologies between $V$ and certain Lagrangian spheres constructed in
  an auxiliary Lefschetz fibration associated to $E$.

\begin{figure}[htbp]
   \includegraphics[scale=0.5]{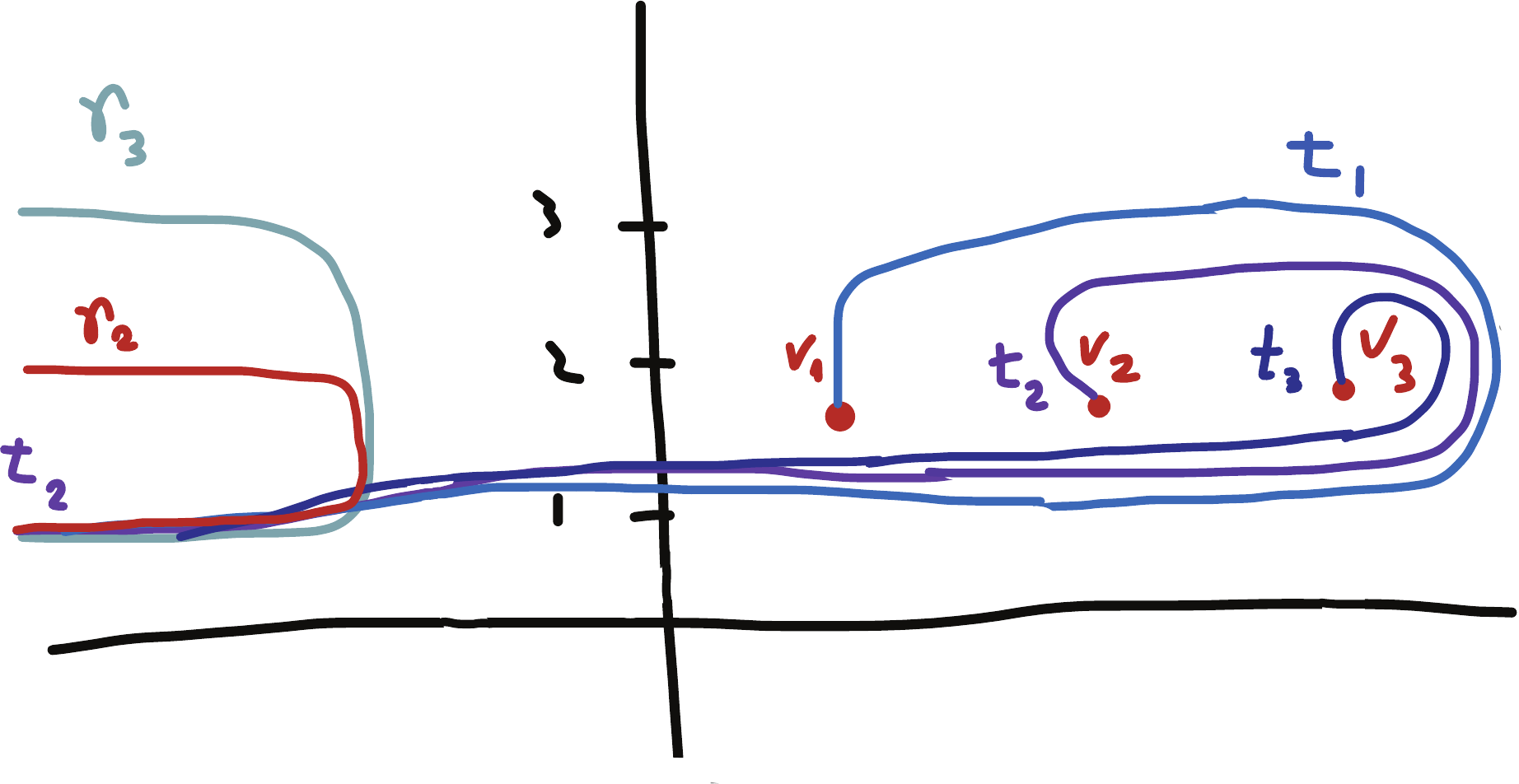}
   \caption{The curves $\gamma_{i}$, and the curves $t_{j}$
       emanating from the critical values $v_j$ of the Lefschetz
       fibration.
     \label{fig:gamma-thimbles0}}
\end{figure}

\subsection{Some consequences} Cobordisms are of interest not only for
their own sake but also because they can be viewed as relators among
their ends, in the sense of the usual cobordism relation. In this
direction, one of the main consequences of
Theorem~\ref{thm:main-dec-gen0} is that each such cobordism $V$
produces an iterated cone decomposition inside $D\fuk^{\ast}(M)$,
where $M=\pi^{-1}(z_{1})$ is the general fiber of $E$. This cone
decomposition expresses the end $L_{1}$ of $V$ as an iterated cone
involving the ends $L_{i}$, $i\geq 2$ and the vanishing cycles of the
singularities of $\pi$ - see~\S\ref{subsec:fibr-tot}.  Thus,
cobordisms in $E$ and the triangular decompositions in the
  (derived) Fukaya category of the fiber are intimately related - see
Corollary \ref{cor:dec-M}.

To discuss a further consequence, recall that to any triangulated
category $\mathcal{C}$ one can associate a Grothendieck group
$K_{0}\mathcal{C}$ defined as the quotient of the free abelian group
generated by the objects of $\mathcal{C}$ modulo the relations $B=A+C$
associated to each exact triangle $A\to B\to C$. We remark that in
this paper we work with ungraded categories, hence our Grothendieck
groups will always be $2$-torsion (i.e. $2 A=0$ for every
$A \in K_0 \mathcal{C}$).

Another application of Theorem \ref{thm:main-dec-gen0} - see
\S\ref{subsec:groth} - is to give a description of the Grothendieck
group $K_{0}D\fuk^{\ast}(M)$ as an ``algebraic'' cobordism group.  To
explain this result we focus here on the case of the trivial fibration
$E=\C\times M$ even if we establish the relevant results in more
generality in the paper.  Recall from \cite{Bi-Co:lcob-fuk} the
definition of the cobordism group $\Omega_{Lag}^{\ast}(M)$. It is the
quotient of the free abelian group generated by the objects in
$\mathcal{L}^{\ast}(M)$ modulo the relations
$L_{1}+ L_{2}+\ldots + L_{s}=0$ for each negatively-ended cobordism
$V\subset \C\times M$ whose ends are $L_{1},\ldots, L_{s}$. For
  every $i \in \mathbb{N}$ there is a natural restriction operation
  that associates to a cobordism $V$ its $i$-th end. These operations
  admit extensions to all objects of $D\fuk^{\ast}(\C\times M)$.  The
$i$-th end of an object $\mathcal{M}$ in $D\fuk^{\ast}(\C\times M)$ is
denoted by $[\mathcal{M}]_{i}\in\mathcal{O}b(D\fuk^{\ast}(M))$.  It is
natural to define an algebraic cobordism group
$\Omega_{Alg}^{\ast}(M)$ as the free abelian group generated by the
(isomorphism classes of) objects of $D\fuk^{\ast}(M)$ modulo the
relations $\sum_{i}[\mathcal{M}]_{i}=0$ for each object $\mathcal{M}$
of $D\fuk^{\ast}(\C\times M)$. Equivalently,
  $\Omega_{Alg}^{\ast}(M)$ is defined in a similar way to
  $\Omega^{\ast}_{Lag}(M)$ only that the generators and relations now
  come also from the non-geometric objects in $D\fuk^{\ast}(M)$ and
  $D\fuk^{\ast}(\C\times M)$. There is an obvious map
  $q:\Omega_{Lag}^{\ast}(M)\to \Omega_{Alg}^{\ast}(M)$. A consequence
of Theorem \ref{thm:main-dec-gen0}, Corollary \ref{cor:alg-cob}, is
that there exists a group isomorphism
$$\Theta_{Alg}: \Omega_{Alg}^{\ast}(M)\to K_{0} D\fuk^{\ast}(M)$$
so that the composition $\Theta_{Alg}\circ q$ coincides with the
Lagrangian Thom morphism
\begin{equation}\label{eq:theta}
  \Theta : \Omega_{Lag}^{\ast}(M)\to K_{0}D\fuk^{\ast}(M)
\end{equation} 
previously introduced in~\cite{Bi-Co:lcob-fuk}.  One of the reasons
why this is of interest is that this result should shed some light on
the kernel of $\Theta$ which is at present somewhat mysterious.
Another implication of the fact that $\Theta_{Alg}$ is an isomorphism
appears in Corollary \ref{cor:quantum-incl} which asserts that
the obvious map $\Omega^{\ast}_{Lag}(M)\to QH_{\ast}(M)$ admits an
extension to $\Omega^{\ast}_{Alg}(M)$. Here $QH_*(M)$ stands
  for the quantum homology of the ambient manifold $M$.

\

Finally, we also obtain a periodicity result for $K_{0}$ -
Corollary~\ref{cor:periodicity}:
\begin{equation}\label{eq:K0}
  K_{0}(D\fuk^{\ast}(\C \times M)) \cong \Z_{2}[t] 
  \otimes  K_{0} (D\fuk^{\ast}(M))~.~
\end{equation}
Here $t$ is a formal variable whose role will become clear in
  the proof (roughly speaking, different powers of $t$ are used to
  label the $K_0$-classes associated to different ends of a cobordism,
  or more generally, ``ends'' of an object of
  $D\fuk^{\ast}(\C \times M)$).

\

\subsection{Relation to previous work}
Theorem \ref{thm:main-dec-gen0} can be viewed as a simultaneous
generalization of the two previously known methods to produce exact
triangles in the derived Fukaya category.  

The first such method is due to
Seidel~\cite{Se:long-exact},~\cite[Chapter III,
Section~17]{Se:book-fukaya-categ} and, in its basic form, it
associates an exact triangle of the form:
\begin{equation}\label{eq:dehn-ex}
   \tau_{S} L\to L \to S\otimes HF(S,L)
\end{equation}
to the Dehn twist $\tau_{S}:M\to M$ corresponding to a Lagrangian
sphere $S$ and any $L\in \mathcal{L}^{\ast}(M)$ (Seidel works in an
exact setting, but as we will see below, this triangle remains
valid in the monotone context too. Other cases have been
  treated in the literature too, e.g. see~\cite{Oh:Seidel-triangle}
  for the case of Lagrangians with vanishing Maslov class in
  Calabi-Yau manifolds).  Seidel also considers a Fukaya category
$\fuk(\pi)$ associated to a Lefschetz fibration $\pi : E\to \C$,
\cite{Se:book-fukaya-categ,Se:Lefschetz-Fukaya}. In our setting, this
category corresponds to the full and faithful subcategory of
$\fuk^{\ast}(E)$ generated by the thimbles $T_{i}$. He also proves a
decomposition result for this category that, in our context,
essentially implies the statement of
Theorem~\ref{thm:main-dec-gen0} in the special case when $V$ has a
single end.  This category is related to mirror symmetry
questions~\cite{Se:Lefschetz-Fukaya-II} and, indeed,
cobordisms with a single end appear in relation to mirror symmetry,
see for instance~\cite{Ho-Iq-Va:ms}. Cobordisms with multiple ends as
well as a category somewhat similar to $\fuk^{\ast}(E)$ appear in the
recent paper \cite{Ab-Sm:Kho}.

The second method appears in our previous paper~\cite{Bi-Co:lcob-fuk}.
It is shown there that if $V\subset \C\times M$ is a cobordism, then
the ends of $V$ are related by a cone-decomposition in
$D\fuk^{\ast}(M)$. This decomposition coincides with the one in
  Corollary~\ref{cor:dec-M} below when $E$ is the trivial fibration
  $\C\times M$. Nevertheless, we remark that the statement of
Theorem~\ref{thm:main-dec-gen0} - which concerns decompositions of
cobordisms - is new even for the trivial fibration.

The exact triangle associated to a Dehn twist and the exact triangle
obtained through the cobordism machinery coincide when there is a
single and transverse intersection between $S$ and $L$. This
can be shown by methods already in the literature. For example, this
follows from a combination of the results from~\cite{Se:knotted}
and~\cite{Bi-Co:lcob-fuk} (see also~\cite{FO3:book-chap-10,
  Oh:Seidel-triangle} for an earlier approach). In this case, Seidel's
exact triangle coincides with the surgery exact sequence which is
associated to a specific cobordism (in $\mathbb{C} \times M$) whose
ends are $\tau_{S}L, L, S$. This cobordism is constructed as the trace
of the Lagrangian surgery at the intersection point $S\cap L$.
Theorem~\ref{thm:main-dec-gen0} and its proof go beyond this case and
further clarify the interplay between these two constructions.

From a technical standpoint, we rely heavily on Seidel's
work~\cite{Se:book-fukaya-categ} - in particular, the detailed
constructions of $D\fuk(-)$, which we adapt to the monotone
  setting. We also build on Seidel's set-up of Lefschetz fibrations in
  the symplectic framework in~\cite{Se:book-fukaya-categ,
    Se:long-exact}.  There is also a variety of other
specific points where our work is related to his and these are
mentioned along the text. We also make heavy use of the constructions
in our previous papers \cite{Bi-Co:cob1, Bi-Co:lcob-fuk}. At the same
time, in attempt to keep this text readable we will recall several
ingredients from~\cite{Bi-Co:cob1, Bi-Co:lcob-fuk} that are crucial
for the present paper.

\subsection{Outline of the paper}

Most of the paper is aimed towards the proof of
Theorem~\ref{thm:main-dec-gen0}.  This proof requires two
preliminaries.  The first is contained in~\S\ref{subsec:lef-fibr}.
That section contains the general set-up and terminology concerning
Lefschetz fibrations. We introduce a special type of such fibrations
called {\em tame} which are basically Lefschetz fibrations over $\C$
that are symplectically trivial  {\em outside} a $U$-like
  region in the plane. (See Definition~\ref{df:tame-lef-fib}. See also
  Figure~\ref{f:fib-NW} on page~\pageref{f:fib-NW}, where the
  complement of the $U$-like region is denoted by $\mathcal{W}$.)
Tame fibrations are much easier to handle in the technical parts of
the proof.  One of the reasons is that cylindrical ends can be
easily moved around in the trivial region since
  parallel transport is trivial over there. Additionally, the Fukaya
$A_{\infty}$ category with objects cobordisms in such fibrations can
be defined following closely the constructions in
\cite{Bi-Co:lcob-fuk}.  In \S\ref{sb:tame-vs-gnrl} we show that any
Lefschetz fibration with a finite number of (simple) singularities can
be transformed into a tame one. As a consequence, Theorem
\ref{thm:main-dec-gen0} follows from the corresponding result - stated
as Theorem \ref{thm:main-dec} - for tame fibrations.

The second preliminary is the construction of the Fukaya category
$\fuk^{\ast}(E)$. This is described in \S\ref{subsec:fuk-cat}. We
first give the main elements of the construction when the Lefschetz
fibration $\pi:E\to \C$ is tame. In this case, the construction that
appears in \cite{Bi-Co:lcob-fuk} applies essentially without change
and we review the main steps. We then indicate the modifications
needed to define such a category in the general case. In the
  discussion below we will mainly assume that all critical values of
  the Lefschetz fibration $E \to \mathbb{C}$ lie in the upper
  half-plane. Moreover, the objects in our categories will be
  cobordisms in $E$ whose projection to $\C$ is contained in the upper
  half-plane and that are cylindrical outside some fixed strip
  $[-a,a]\times
  \R$. (See~\S\ref{subsubsec:Fuk-cob},~\S\ref{subsec:main-decomp} for
  the precise setting.)

With this preparation, the actual proof of Theorem
\ref{thm:main-dec-gen0} is contained in \S\ref{sec:main} and it
consists of three main ingredients.  The first one deals with
decompositions of cobordisms $V'$ - called remote with respect to $E$
- that are included in the total space $E'$ of a Lefschetz fibration
that coincides with $E$ over the upper half-plane. The defining
property of such a $V'$ is that it can be moved inside $E'$ away from
the critical points of $E \longrightarrow \mathbb{C}$, so that
its only intersection with an object $X$ of $\fuk^{\ast}(E)$ occurs in
the region where both $V'$ and $X$ are cylindrical. We show in
\S\ref{subsec:dec-Yo} that such a remote cobordism viewed as a module
over $\fuk^{\ast}(E)$ admits a decomposition just as the one in the
statement of Theorem \ref{thm:main-dec-gen0} but without any of terms
$T_{i}\otimes E_{i}$.  The second step, in \S\ref{subsec:Dehn-disj},
shows how to transform a general cobordism $V$ into a remote one.
This is a geometric step, potentially of independent interest.  It is
done, roughly speaking, by placing $V$ inside a new Lefschetz
fibration $E'$ obtained from $E$ by adding singularities over
the lower half-plane and showing that the cobordism $V' \subset E'$
obtained as an iterated Dehn twist of $V$,
$V'=(\tau_{S_{m}}\circ\ldots \circ \tau_{S_{i}}\circ \ldots \circ
\tau_{S_{1}})(V)$, where $S_{i}$ are certain matching cycles in $E'$,
is remote with respect to $E$.  The third ingredient -
in~\S\ref{s:cob-vpt} - is Seidel's exact triangle for which we provide
a new proof reflecting our cobordism perspective.  These ingredients
are put together in~\S\ref{subsec:prof-main-t}. In short, the
cobordism $V'=(\tau_{S_{m}}\circ \ldots \circ \tau_{S_{1}})(V)$ is
remote with respect to $E$ and thus, by the first step, it admits a
certain decomposition involving the ends of $V$, but as it is obtained
by an iterated Dehn twist from $V$, it can be related to $V$ by
another decomposition, involving the matching cycles $S_{i}$, by using
the relevant Seidel exact triangles.  The two decompositions combine
as in the statement of Theorem~\ref{thm:main-dec-gen0}.

The Corollaries of Theorem~\ref{thm:main-dec-gen0} described above are
proven in~\S\ref{sec:conseq}.

The paper ends with \S \ref{S:examples} that consists of examples and
related discussion. The main part of the section - \S\ref{sb:real-lef}
- is focused on a class of Lagrangian cobordisms in real Lefschetz
fibrations.

\subsection*{Acknowledgments} The first author thanks Jean-Yves Welschinger for
useful discussions concerning the examples in real algebraic geometry.
Part of this work was accomplished
during a stay at the Simons Center for Geometry and Physics. We thank
the SCGP and its staff for their gracious hospitality.  We thank the referee for carefully reading an earlier version of the 
paper and for remarks that were helpful to improve the exposition.

\section{Lefschetz fibrations} \label{subsec:lef-fibr}

\subsection{Basic definitions} \label{sb:defs-lef-fibr} Lefschetz
fibrations will play a central role in this paper.  From the
symplectic viewpoint there are several versions of this notion in the
literature. Our setup is similar to~\cite{Se:book-fukaya-categ,
  Se:long-exact} but with some modifications.

We begin with Lefschetz fibrations having a compact fiber.
\begin{dfn} \label{df:lef-fib} A Lefschetz fibration with compact
   fiber consists of the following data:
   \begin{enumerate}
     \item[i.] A symplectic manifold $(E, \Omega_E)$ without boundary,
      endowed with a compatible almost complex structure $J_E$.
     \item[ii.] A Riemann surface $(S, j)$ (which is generally not
      assumed to be compact; typically we will have $S =
        \mathbb{C}$).
       \item[iii.] A proper $(J_E, j)$-holomorphic map $\pi: E
        \longrightarrow S$. (In particular all fibers of $\pi$ are
        closed manifolds.)
      \item[iv.] We assume that $\pi$ has a finite number of critical
        points. Moreover, we assume that every critical value of $\pi$
        corresponds to precisely one critical point of $\pi$. We
        denote the set critical points of $\pi$ by
        $\textnormal{Crit}(\pi)$ and by
        $\textnormal{Critv}(\pi) \subset S$ the set of critical values
        of $\pi$. {\em Below we will use the words ``critical
            points of $\pi$'' and ``singularities of $E$''
            interchangeably.}
       \item[v.] All the critical point of $\pi$ are ordinary double
        points in the following sense. For every $p \in
        \textnormal{Crit}(\pi)$ there exist a local $J_E$-holomorphic
        chart around $p$ and a $j$-holomorphic chart around $\pi(p)$
        with respect to which $\pi$ is a holomorphic Morse function.
   \end{enumerate}

   For $z \in S$ we denote by $E_z = \pi^{-1}(z)$ the fiber over $z$.
   We will sometimes fix a base-point $z_0 \in S \setminus
   \textnormal{Critv}(\pi)$ and refer to the symplectic manifold $(M
   := \pi^{-1}(z_0), \omega_M := \Omega_E|_{M})$ as ``the'' fiber of
   the Lefschetz fibration. We will also use the following notation:
   for a subset $\mathcal{S} \subset S$ we denote $V|_{\mathcal{S}} =
   \pi^{-1}(\mathcal{S}) \cap V$.
\end{dfn}

Our constructions work for the most part also when the fiber is not
compact.  To this end we will need some adjustments to the preceding
definition as follows. Let $(M, \omega_M)$ be a (non-compact)
symplectic manifold which is convex at infinity. We define {\em a
  Lefschetz fibration $\pi: E \longrightarrow S$ with fiber $(M,
  \omega_M)$} to be as in Definition~\ref{df:lef-fib} with the
following modifications.  Firstly, properness in condition~iii is
removed (thus allowing, in particular, for the fibers to be
non-compact).  Secondly, the map $\pi: E \setminus
\pi^{-1}(\textnormal{Critv}(\pi)) \longrightarrow S \setminus
\textnormal{Critv}(\pi)$ is now explicitly assumed to be a smooth
locally trivial fibration. Finally, $E$ is assumed to satisfy the
following additional condition.
\begin{assumption-T} 
   Let $\pi: E \longrightarrow S$ be as above. We say that $E$ is
   trivial at infinity if there exists a subset $E^0 \subset E$ with
   the following properties: \label{pg:assumption-T-infty}
   \begin{enumerate}
     \item For every compact subset $K \subset S$, $E^0 \cap
      \pi^{-1}(K)$ is also compact. (In other words, $\pi|_{E^0}
      \longrightarrow S$ is a proper map.)
     \item Set $E^{\infty} = E \setminus E^0$ and $E^{\infty}_{z_0} =
      E^{\infty} \cap \pi^{-1}(z_0)$, where $z_0 \in S \setminus
      \textnormal{Critv}(\pi)$ is a fixed base-point. Then there
      exists a trivialization $\phi: S \times E^{\infty}_{z_0}
      \longrightarrow E^{\infty}$ of $\pi|_{E^{\infty}}: E^{\infty}
      \longrightarrow S$ such that
      $$\phi^* \Omega_{E} = \omega_S \oplus 
      \omega_{M}|_{E^{\infty}_{z_0}},\ \mathrm{and} \ \phi^{\ast}
      J_{E}= j\oplus J_{0}$$ where $\omega_S$ is a positive (with
      respect to $j$) symplectic form on $S$ and $J_{0}$ is a fixed
      almost complex structure on $M=\pi^{-1}(z_{0})$, compatible with
      $\omega_{M}$ .
   \end{enumerate}
\end{assumption-T}
This extended definition in fact generalizes the preceding one: if $M$
is compact we take $E^0 = E$ and $E^{\infty} = \emptyset$. From now
on, unless otherwise stated, by a Lefschetz fibration we mean one with
compact fiber that satisfies Definition \ref{df:lef-fib} or, more
generally, with a non-compact fiber that is convex at infinity and
satisfies the conditions above, including $T_{\infty}$.

\pbgreen{Before we go on, we recall again that {\em in this paper all
    Lefschetz fibrations are assume to have positive dimensional
    fibers.}}

\begin{rem}\label{rem:triv-infty}
  \begin{enumerate}
  \item[a.] The assumption that the fiber of $E$ is either closed or
    symplectically convex was made in order to assure that the fiber
    is amenable to techniques of symplectic topology such as
    pseudo-holomorphic curves and Floer theory. (Specifically,
      these conditions assure that holomorphic curves and Floer
      trajectories cannot ``escape to infinity'', hence standard
      compactness results hold for them.) Nevertheless in one
    instance later on in the paper we will drop this assumption and
    assume instead that $M$ is itself the total space of another
    Lefschetz fibration.
  \item[b.] Assumption $T_{\infty}$ is a variant of boundary
    horizontality that appears in \cite{Se:long-exact} and
    \cite{Se:book-fukaya-categ}.
  \end{enumerate}
\end{rem}

\subsubsection{Connections, parallel transport and trails of
  Lagrangians} \label{sbsb:con-ptrans-trail} To a Lefschetz fibration
as above we can associate a connection $\Gamma = \Gamma(\Omega_E)$ on
$E \setminus \textnormal{Crit}(\pi)$ as follows. The connection
$\Gamma$ is defined by setting its horizontal distribution
$\mathcal{H} \subset T(E)$ to be the $\Omega_E$-orthogonal complement
of the tangent spaces to the fibers.  More specifically, for every $x
\in E \setminus \textnormal{Crit}(\pi)$ we set
$$\mathcal{H}_x = \bigl\{ u \in T_x(E) 
\mid \Omega_E(\xi, u) = 0 \;\; \forall \; \xi \in T_x^v(E) \bigr\},$$
where $T_x^v(E)$ stands for the vertical tangent space at $x$.

The connection $\Gamma$ induces parallel transport maps. Let
$\lambda: [a,b] \longrightarrow \mathbb{C} \setminus
\textnormal{Critv}(\pi)$ be a smooth path. We denote by
$\Pi_{\lambda}: E_{\lambda(a)} \longrightarrow E_{\lambda(b)}$ the
parallel transport along $\lambda$ with respect to the connection
$\Gamma$. Notice that even when the fiber of $E$ is not compact,
parallel transport is still well defined. Indeed, thanks to assumption
$T_{\infty}$, the connection $\Gamma$ is trivial at infinity with
respect to the trivialization $\phi$. In particular, relative to the
trivialization $\phi$, parallel transport becomes the identity at
infinity in the sense that
$\phi^{-1} \circ \Pi_{\lambda} \circ \phi (\lambda(a), x) =
(\lambda(b), x)$ for every $x \in E^{\infty}_{z_0}$.

It is well known that $\Pi_{\lambda}$ is a symplectomorphism, where we
endow the fibers of $\pi$ with the symplectic structure induced by
$\Omega_E$) (See
e.g.~\cite[Chapter~8]{McD-Sa:jhol},~\cite[Chapter~6]{McD-Sa:Intro}.)
If $\lambda$ is a loop starting and ending at
$z \in \mathbb{C} \setminus \textnormal{Critv}(\pi)$ then the
symplectomorphism $\Pi_{\lambda}: E_z \longrightarrow E_z$ is also
called the holonomy of $\Gamma$ along $\lambda$. If the loop $\lambda$
is contractible (within
$\mathbb{C} \setminus \textnormal{Critv}(\pi)$) then the holonomy
$\Pi_{\lambda}$ is in fact a Hamiltonian diffeomorphism of $E_z$
(see~\cite[Section~6.4]{McD-Sa:Intro}).

Let $\lambda: [a,b] \longrightarrow \mathbb{C} \setminus
\textnormal{Critv}(\pi)$ be a smooth embedding and $L \subset
E_{\lambda(a)}$ a Lagrangian submanifold. Consider the images of $L$
under the parallel transport along $\lambda$, namely $L_t :=
\Pi_{\lambda|_{[a,t]}}(L) \subset E_{\lambda(t)}$, $t \in [a,b]$ and
set
$$\lambda L : = \cup_{t \in [a,b]} L_t.$$ 
Then $\lambda L$ is a Lagrangian submanifold of $(E, \Omega_E)$.
We call $\lambda L$ the {\em \trl} of $L$ along $\lambda$.

We refer the reader to~\cite{Se:book-fukaya-categ} for the foundations
of the symplectic theory of Lefschetz fibrations and
to~\cite[Chapter~6]{McD-Sa:Intro} and~\cite[Chapter~8]{McD-Sa:jhol}
for symplectic fibrations.

%\subsection{Lagrangians with cylindrical ends in Lefschetz fibrations.}

\subsection{Lagrangians with cylindrical ends} \label{subsec:cyl-lag}
% \label{sbsb:lag-cyl}

Let $\pi: E \longrightarrow \mathbb{C}$ be a Lefschetz fibration and
$\mathcal{U} \subset \mathbb{C}$ an open subset containing
$\textnormal{Critv}(\pi)$.  The following terminology is useful.  A
horizontal ray $\ell \subset \mathbb{C}$ is a half-line of the type
$(-\infty, -a_{\ell}]\times \{b_{\ell}\}$ or $[a_{\ell}, \infty)\times
\{b_{\ell}\}$ with $a_{\ell}>0$, $b_{\ell}\in \R$.  The imaginary
coordinate $b_{\ell}$ is also referred to as the ``height'' of $\ell$.

\begin{dfn}\label{def:cyl-ge}
  A Lagrangian submanifold (without boundary)
  $V \subset (E, \Omega_E)$ is said to have {\sl cylindrical ends
    outside of $\mathcal{U}$} if the following conditions are satisfied:
  \begin{enumerate}
  \item[i.] For every $R>0$, the subset $V \cap \pi^{-1}([-R,R] \times
    \mathbb{R})$ is compact.
  \item[ii.] \pbred{$\pi(V) \cap \mathcal{U}$ is bounded.}
  \item[iii.] $\pi(V) \setminus \mathcal{U}$ consists of a finite union
    of horizontal rays, $\ell_i \subset \mathbb{C}$, $i=1, \ldots, r$.
    Moreover, for every $i$ we have
    \pbred{$V|_{\ell_i} = \ell_i L_{i}$ for some Lagrangian
      $L_i \subset E_{\sigma_i}$,} where $\sigma_i \in \mathbb{C}$
    stands for the starting point of the ray $\ell_i$, and
      $\ell_i L_{i}$ is the trail of $L_i$ along $\ell_i$ as defined
      above. (Note that we do allow $r=0$, i.e. that $V$ has no ends
    at all.)
  \end{enumerate}
  In case all the heights of the rays $\ell_{i}$ are positive integers
  $b_{l_{i}}\in \N^{\ast}$ the Lagrangian $V$ is called a {\em
    cobordism} in $E$.
\end{dfn}
In short, over each of the rays appearing in
$\pi(V) \setminus \mathcal{U}$ the Lagrangian submanifold $V$ is the
{\trl} under parallel transport of $L_{i}$ along $\ell_{i}$ - see
Figure~\ref{fig:rays}.

The role of the condition~ii above is to exclude the
  possibility that $\pi^{-1}(\mathcal{U})$ entirely covers some of the
  ends of $V$. For most of the time we will work with subsets
  $\mathcal{U}$ that are $U$-shaped (see Figure~\ref{fig:null-cob} on
  page~\pageref{fig:null-cob}), and then condition~ii is automatically
  satisfied (in view of condition~i). However, occasionally we will
  have to consider $\mathcal{U}$'s that are not compact in the
  horizontal direction (see e.g.~\S\ref{subsec:Dehn-disj} and
  Figure~\ref{fig:complex-thimble}), and then condition~ii is
  necessary.

\begin{figure}[htbp]
   \includegraphics[scale=0.5]{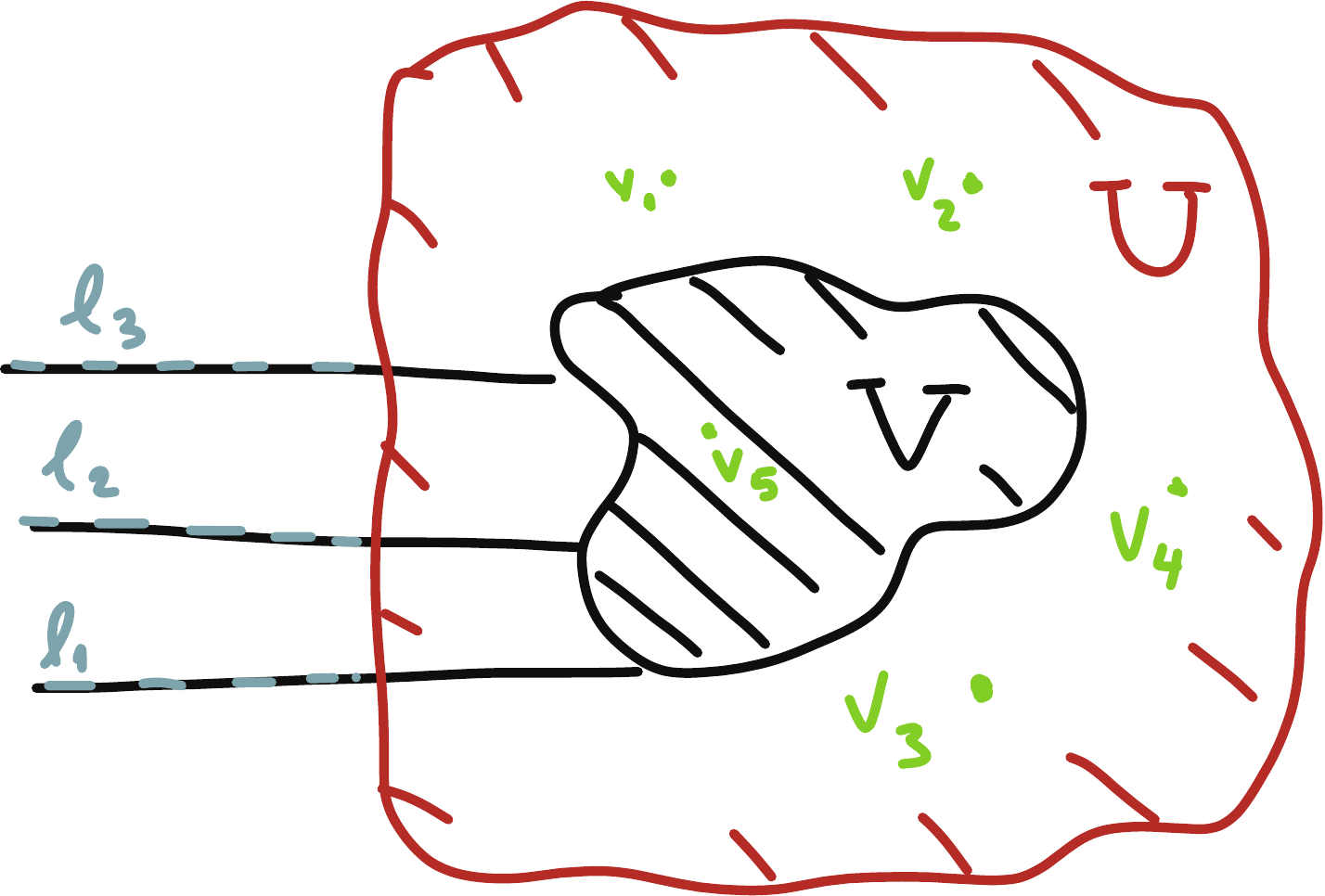}
   \caption{A Lagrangian $V$ with cylindrical ends outside $\mathcal{U}$ in a
     Lefschetz fibration $\pi:E\to \C$ with critical values $v_{i}$.
     \label{fig:rays}}
\end{figure}

The above notion of cobordism extends the definition of Lagrangian
cobordism as given for the trivial fibration $E$ in \cite{Bi-Co:cob1}.
Note however that this terminology is slightly imprecise
  because we have not specified a (topological) trivialization of the
  fibration $E \longrightarrow \mathbb{C}$ at infinity (and in general
  there is no canonical trivialization). Moreover, even when one fixes
  such a trivialization the parallel transport along a ray $\ell_{i}$
  might not be trivial (even not at infinity), hence the actual ends
  of $V$ at infinity are not well defined. In view of that, we will
  often work with a restricted type of Lefschetz fibrations, called
  tame, where this imprecision is not present and that have a number
  of additional technical advantages. We will see later on that this
  does not restrict the generality of our theory.

%\subsection{Tame Lefschetz fibrations} \label{sb:tame-fib} 

\begin{dfn} \label{df:tame-lef-fib} Let $\pi: E \longrightarrow
  \mathbb{C}$. Let $U \subset \mathbb{C}$ be a closed subset, let $z_0
  \in \mathbb{C} \setminus U$ be a base point and $(M, \omega_M)$ be
  the fiber over $z_0$. We say that this Lefschetz fibration is tame
  outside of $U$ if there exists a trivialization
  $$\psi_{E,\mathbb{C} \setminus B}: 
  (\mathbb{C} \setminus U) \times M \longrightarrow E|_{\mathbb{C}
    \setminus U}$$ such that $\psi_{E,\mathbb{C} \setminus
    U}^*(\Omega_E) = c \omega_{\mathbb{C}} \oplus \omega_M$, where
  $\omega_{\mathbb{C}}$ is the standard symplectic structure on
  $\mathbb{C} \cong \mathbb{R}^2$ and $c > 0$ is a constant. The
  manifold $(M,\omega_{M})$ is called the generic fiber of $\pi$.
\end{dfn}
It follows from the definition that all the critical values of $\pi$
must be contained inside $U$. Sometimes it will be more natural to fix
the complement of $U$, say $\mathcal{W} = \mathbb{C} \setminus U$, and
say that the fibration is tame over $\mathcal{W}$. Given a tame
Lefschetz fibration, the set $U=U_{E}$, the point $z_{0}$ and the
symplectic trivialization $\psi_{E,\mathbb{C} \setminus B}$, are all
viewed as part of the fixed data associated to the fibration.
     
Moreover, we will assume that the set $U=U_{E}$ is so that there
exists $a_{U}>0$ sufficiently large with the property that $U$ is
disjoint from both quadrants:
\begin{equation}\label{eq:quadr}
  Q_{U}^{-}=(-\infty, -a_{U}]\times [0,+\infty)\ , 
  \ Q_{U}^{+}=[a_{U},\infty)\times [0,+\infty)
\end{equation} 

The cobordism relation, as defined in \cite{Bi-Co:cob1}, admits an
obvious extension in a tame Lefschetz fibration.

\begin{dfn}\label{def:Lcobordism} Fix a Lefschetz fibration that is
   tame outside $U\subset \C$ with fiber $(M,\omega)$ over $z_{0}\in
   \C\setminus U$.  Let $(L_{i})_{1\leq i\leq k_{-}}$ and
   $(L'_{j})_{1\leq j\leq k_{+}}$ be two families of closed Lagrangian
   submanifolds of $M$. We say that that these two families are
   Lagrangian cobordant in $E$, if there exists a Lagrangian
   submanifold $V \subset E$ with the following properties:
  \begin{itemize}
    \item[i.] There is a compact set $K\subset E$ so that $V\cap
     U\subset V\cap K$ and $V \setminus K\subset
     \pi^{-1}(Q^{+}_{U}\cup Q^{-}_{U})$.
  \item[ii.] $V\cap \pi^{-1}(Q^{+}_{U})= \coprod_{j} 
    ([a_{U},+\infty) \times \{j\})  \times L'_j $     
  \item[iii.] $V\cap \pi^{-1}(Q^{-}_{U})= \coprod_{i} 
    ((-\infty,- a_{U}] \times \{i\})  \times L_i $
  \end{itemize}
  The formulas at ii and iii are written with respect to the
  trivialization of the fibration over the complement of $U$.
\end{dfn}

The manifold $V$ is obviously a Lagrangian cobordism in the sense of
Definition \ref{def:cyl-ge} and - because of tameness - its ends at
$\infty$ are well defined so that we can say that $V$ is a cobordism
from the Lagrangian family $(L'_{j})$ to the family $(L_{i})$. We
write $V:(L'_{j}) \cobto (L_{i})$ or $(V; (L_{i}), (L'_{j}))$.

\subsection{From general Lefschetz fibrations to tame ones}
\label{sb:tame-vs-gnrl}

We will now see that it is always possible to pass from a general
Lefschetz fibration $\pi: E \longrightarrow \mathbb{C}$, as in
\S \ref{sb:defs-lef-fibr}, to a tame one.

\begin{prop} \label{p:from-gnrl-to-tame} Let $\pi: E \longrightarrow
  \mathbb{C}$ be a Lefschetz fibration and let $\mathcal{N} \subset
  \mathbb{C}$ be an open subset that contains all the critical values
  of $\pi$ and has the shape depicted in Figure~\ref{f:fib-NW}. Let
  $\mathcal{W} \subset \mathbb{C}$ be another open subset of the shape
  depicted in Figure~\ref{f:fib-NW} with $\overline{\mathcal{W}} \cap
  \overline{\mathcal{N}} = \emptyset$ and
  $\textnormal{dist}(\overline{\mathcal{W}}, \overline{\mathcal{N}}) >
  0$. Then there exists a symplectic structure $\Omega' = \Omega'_{E,
    \mathcal{N}, \mathcal{W}}$ on $E$ and a trivialization $\varphi:
  \mathcal{W} \times M \longrightarrow E|_{\mathcal{W}}$ with the
  following properties:
   \begin{enumerate}
     \item On $\mathcal{W} \times M$ we have $\varphi^* \Omega' = c
      \omega_{\mathbb{C}} \oplus \omega_M$ for some $c> 0$.
     \item $\Omega'$ coincides with $\Omega_{E}$ on all the fibers of
      $\pi$.
     \item $\Omega' = \Omega_{E}$ on $\pi^{-1}(\mathcal{N})$.
     \item There exists an $\Omega'$-compatible almost complex
      structure $J'_E$ on $E$ which coincides with $J_E$ on
      $\pi^{-1}(\mathcal{N})$ and such that the projection $\pi: E
      \longrightarrow \mathbb{C}$ is $(J'_E, i)$-holomorphic.
   \end{enumerate}
   In particular, when endowed with the symplectic structure
   $\Omega'$, the Lefschetz fibration $\pi: E \longrightarrow
   \mathbb{C}$ is tame over $\mathcal{W}$.
\end{prop}

\begin{figure}[htbp]
   \begin{center}
      \includegraphics[scale=0.5]{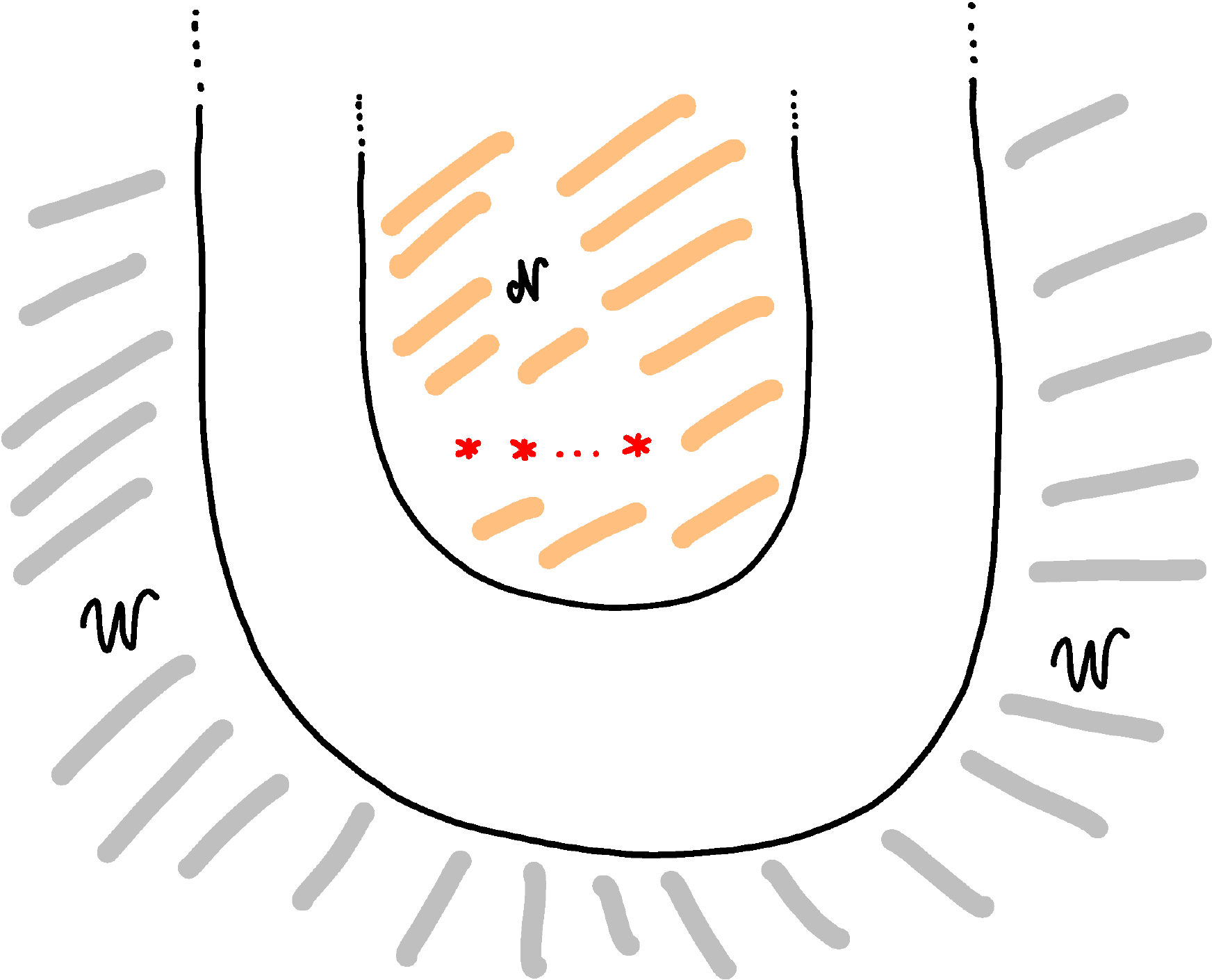}
   \end{center}
   \caption{A Lefschetz fibration $\pi:E\to \C$; the domains 
     $\mathcal{N}$ and $\mathcal{W}$ and, 
     in red, the critical values of $\pi$.
     \label{f:fib-NW}}
\end{figure}

\begin{rem}\label{rem:gen-tame-cob}
It is easy to pass from a cobordism in a general
Lefschetz fibration to a cobordism in a tame fibration.

Indeed, let $\pi: E \longrightarrow \mathbb{C}$ be a Lefschetz
fibration and $V \subset E$ a
Lagrangian submanifold with cylindrical ends. Let $\mathcal{N} \subset
\mathbb{C}$ be a subset as in Proposition~\ref{p:from-gnrl-to-tame}
and assume that $V$ has cylindrical ends outside of $\mathcal{N}'$,
where $\mathcal{N}' \subset \mathcal{N}$ is a slightly smaller subset
than $\mathcal{N}$ which contains $\textnormal{Critv}(\pi)$ and is of
the same shape as $\mathcal{N}$. Denote the horizontal rays
corresponding to the ends of $V$ by $\ell_i \subset \mathbb{C}$, $i=1,
\ldots, r$ and by $L_i \subset E_{\sigma_i}$ the corresponding
Lagrangians over the starting points of these rays. Let $\mathcal{W}
\subset \mathbb{C}$ be a subset as in
Proposition~\ref{p:from-gnrl-to-tame} and consider the new symplectic
structure $\Omega'$ on $E$ provided by that proposition.  By
performing parallel transport of the $L_i$'s along the horizontal rays
$\ell_i$, but this time with respect to the connection corresponding
to $(E, \Omega')$ we obtain a new Lagrangian submanifold $V' \subset
(E, \Omega')$ with the following properties:
\begin{enumerate}
  \item[i.] $V'$ coincides with $V$ over $\mathcal{N}$.
  \item[ii.] $V'$ has cylindrical ends outside of $\mathcal{N}$.
  \item[iii.] Over $\mathcal{W}$, $V'$ looks like
   $$V'|_{\mathcal{W}} = \cup_{i=1}^r \ell'_i \times L'_i,$$ where
   $\ell'_i = \ell_i \cap \mathcal{W}$ and $L'_i$ is the image of the
   parallel transport of $L_i$ (with respect to the connection
   $\Gamma(\Omega')$) along the portion of $\ell_i$ that connects
   $\mathcal{N}'$ with $\mathcal{W}$.
\end{enumerate}
\end{rem}

\subsubsection{Preparation for the proof of
  Proposition~\ref{p:from-gnrl-to-tame}} \label{sbsb:prep-prf-gnrl-tame}

Let $(M, \omega)$ be a symplectic manifold, $Q \subset \mathbb{C}$ an
open subset and $f: Q \times M \longrightarrow \mathbb{R}$ a smooth
function. We denote by $z = y_1 + iy_2$ the standard complex
coordinate in $\mathbb{C}$. Let $\alpha = \{\alpha\}_{z \in Q}, \beta
= \{\beta_z\}_{z \in Q}$ be two families of $1$-forms on $M$,
parametrized by $z \in Q$ (alternatively we can view $\alpha, \beta$
as differential forms on $Q \times M $ with
$\alpha(\tfrac{\partial}{\partial y_j}) =
\beta(\tfrac{\partial}{\partial y_j}) = 0$). For $z \in Q$, $p \in M$
we write $\alpha_{z,p}$ for the restriction of $\alpha_z$ to $T_p(M)$
and similarly for $\beta$. We denote by $d^v$ the exterior derivative
of differential forms on $Q \times M$ in the $M$-direction (i.e. $(d^v
\alpha)_z = d^M(\alpha_z)$, where $d^M$ is the exterior derivative in
$M$.) Below we will abbreviate the partial derivatives
$\tfrac{\partial}{\partial y_1}$, $\tfrac{\partial}{\partial y_2}$ by
$\partial_{y_1}$, $\partial_{y_2}$.

Consider now the following $2$-form on $Q \times M$
$$\Omega^{f,\alpha,\beta} : = \omega + f dy_1 \wedge dy_2 + \alpha
\wedge dy_1 + \beta \wedge dy_2.$$

A simple calculation shows that:
\begin{lem} \label{l:Om-closed} $\Omega^{f,\alpha,\beta}$ is closed
   iff $d^v \alpha = d^v \beta = 0$ and $d^v f =
   \partial_{y_2}\alpha - \partial_{y_1}\beta$.
\end{lem}

Define now two families of vector fields $u_0, v_0$ on $M$
(parametrized by the points of $Q$) as follows.  For every $z \in Q$,
$p \in M$, define $u_0(z,p), v_0(z,p) \in T_p(M)$ by requiring that
for every $\xi \in T_p(M)$ we have:
\begin{equation} \label{eq:u0-v0} \omega_p(\xi, u_0(z,p)) +
   \alpha_{z,p}(\xi) = 0, \quad \omega_p(\xi, v_0(z,p)) +
   \beta_{z,p}(\xi) = 0.
\end{equation}
Denote by $\mathcal{H} \subset T(Q \times M)$ the following
$2$-dimensional distribution:
\begin{equation} \label{eq:hdist}
   \mathcal{H}_{z,p} := \mathbb{R}
   \bigl(\tfrac{\partial}{\partial y_1} + u_0(z,p) \bigr) +
   \mathbb{R}\bigl(\tfrac{\partial}{\partial y_2} + v_0(z,p) \bigr).
\end{equation}
Note that $\mathcal{H}$ depends on $\omega$, $\alpha, \beta$ but not
on $f$.

The following two lemmas can be proved by direct calculation.

\begin{lem} \label{l:H-alpha-beta} For every $(z,p) \in Q \times M$,
  $\xi \in T_p(M)$ and $w \in \mathcal{H}_{z,p}$ we have
  $\Omega^{f,\alpha,\beta}(\xi, w) = 0$. 
  In particular, if $\Omega^{f,\alpha,\beta}$ is non-degenerate then
  $\mathcal{H}$ is the horizontal distribution of the connection
  induced by $\Omega^{f,\alpha,\beta}$.
\end{lem}

\begin{lem} \label{l:Om-non-deg} Assume that $f(z,p) \neq
   \omega_p(u_0(z,p), v_0(z,p))$ for some $(z,p) \in Q \times M$. Then
   $\Omega^{f,\alpha,\beta}$ is non-degenerate at $(z,p)$. Moreover,
   there exists an $\Omega_{z,p}^{f,\alpha,\beta}$-compatible complex
   structure $J_{z,p}$ on $T_{z,p}(Q \times M)$ such that the
   projection $Q \times M \longrightarrow Q$ is $(J_{z,p},
   i)$-holomorphic at $(z,p)$ if and only if $f(z,p) >
   \omega_p(u_0(z,p), v_0(z,p))$.
\end{lem}

\subsubsection{Proof of Proposition~\ref{p:from-gnrl-to-tame}}
\label{sbsb:prf-gnrl-tame} 

\

To fix ideas, we first provide the proof in the case of compact fibre.

\noindent \textbf{Step 1.} Using parallel transport with respect to
the connection $\Gamma_{\Omega_{E}}$ along a system of curves in
$\mathbb{C} \setminus \overline{\mathcal{N}}$ emanating from a fixed
point $z_0\in\mathcal{W}$, and using the fact that $\mathbb{C}
\setminus \overline{\mathcal{N}}$ is contractible we obtain a
trivialization
$$\varphi : (\mathbb{C} \setminus \overline{\mathcal{N}}) \times M 
\longrightarrow E|_{\mathbb{C} \setminus \overline{\mathcal{N}}}$$
with $M=\pi^{-1}(z_{0})$ and with the property that the form $\Omega_1
:= \varphi^* \Omega_E$ admits the following form
\begin{equation} \label{eq:Omega_1}
   \Omega_1 = f dy_1\wedge dy_2 + \alpha
   \wedge dy_1 + \beta \wedge dy_2 + \omega,
\end{equation}
where $\omega  = \Omega|_M$ and $f:
(\mathbb{C} \setminus \overline{\mathcal{N}}) \times M \longrightarrow
\mathbb{R}$ is a smooth function, and $\alpha, \beta$ are vertical
$1$-forms on $(\mathbb{C} \setminus \overline{\mathcal{N}}) \times M$
with the property that for every $z \in \mathbb{C} \setminus
\overline{\mathcal{N}}$ the $1$-forms $\alpha_z = \alpha|_{z \times
  M}$, $\beta_z = \beta|_{z \times M}$ are exact (see~\S~8.2
of~\cite{McD-Sa:jhol} and~\S~6.4 of~\cite{McD-Sa:Intro} for a proof of
that). Fix two functions $F, G : (\mathbb{C} \setminus
\overline{\mathcal{N}}) \times M \longrightarrow \mathbb{R}$ such that
$\alpha = d^v F$, $\beta = d^v G$.

By Lemma~\ref{l:Om-closed} we have:
\begin{equation} \label{eq:dvf} d^v f = \partial_{y_2} \alpha -
   \partial_{y_1} \beta.
\end{equation}

\ \\

Apart from $\mathcal{W}$ and $\mathcal{N}$ we will fix three
additional open subsets
$\mathcal{W}_{\epsilon}, \mathcal{N}_{\epsilon},
\mathcal{N}_{2\epsilon}$ with
$$\overline{\mathcal{W}} \subset \mathcal{W}_{\epsilon}, \quad
\overline{\mathcal{N}} \subset \mathcal{N}_{\epsilon}, \quad
\overline{\mathcal{N}}_{\epsilon} \subset \mathcal{N}_{2\epsilon},$$
and with shapes as described in Figure~\ref{f:fib-NW-eps}.  To be more
precise, consider the curves $\gamma_1, \gamma_2, \gamma_3 \subset
\mathbb{C}$ depicted in Figure~\ref{f:fib-NW-eps}. The domain
$\mathcal{N}_{\epsilon}$ is defined to be the connected component of
$\mathbb{C} \setminus \gamma_1$ in which all the points have bounded
real coordinate. The domain $\mathcal{N}_{2\epsilon}$ is defined
similarly but with the curve $\gamma_1$ replaced by $\gamma_2$. The
domain $\mathcal{W}_{\epsilon}$ is defined as the connected component
of $\mathbb{C} \setminus \gamma_3$ in which the real coordinate of the
points is unbounded. We also require that
$\textnormal{dist}(\overline{\mathcal{W}}_{\epsilon},
\overline{\mathcal{N}}_{2 \epsilon}) > 0$.

\begin{figure}[htbp]
   \begin{center}
      \includegraphics[scale=0.5]{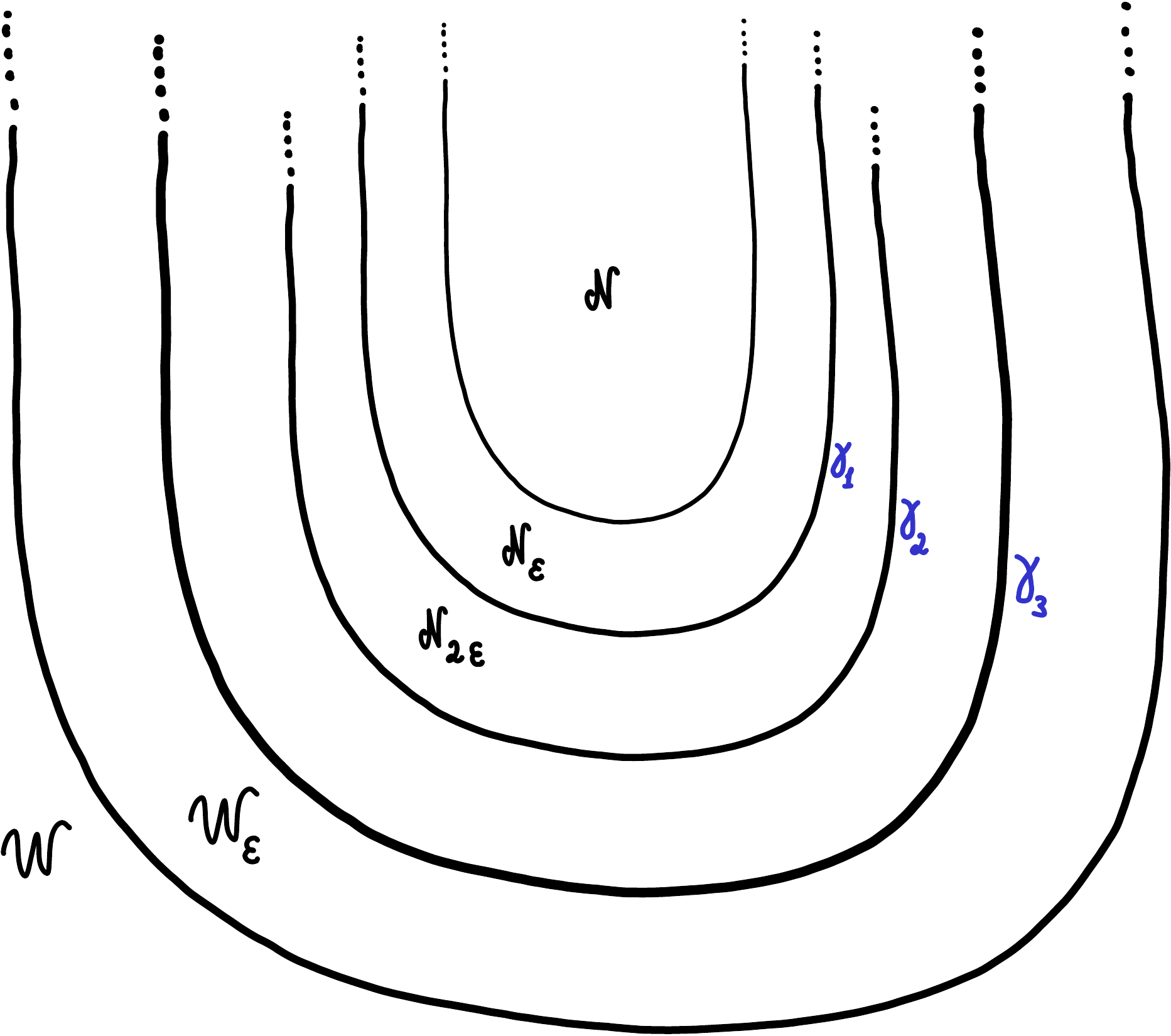}
   \end{center}
   \caption{The domains $\mathcal{N}_{\epsilon}$,
     $\mathcal{N}_{2\epsilon}$, and $\mathcal{W}_{\epsilon}$.
     \label{f:fib-NW-eps}}
\end{figure}

\ \\
\noindent \textbf{Step 2.} We will modify now the form $\Omega_1$ in
the following way. Fix a smooth function
$\sigma: \mathbb{C} \longrightarrow [0,1]$ such that:
\begin{equation} \label{eq:sig-funct}
   \sigma(z) = 
   \begin{cases}
      1 & z \in \mathcal{N}_{2\epsilon}, \\
      0 & z \in \mathcal{W}_{\epsilon}.
   \end{cases}
\end{equation}
Define $g: \mathbb{C} \times M \longrightarrow \mathbb{R}$ by
\begin{equation} \label{eq:g-funct-1} g(z,p) = \partial_{y_2}(\sigma)
   F(z,p) - \partial_{y_1}(\sigma) G(z,p).
\end{equation}
Then we have:
\begin{equation} \label{eq:g-funct-2} g(z,p) = 0 \quad \forall \, z \in
   \mathcal{N}_{2\epsilon} \cup \mathcal{W}_{\epsilon}.
\end{equation}

%
% $g$ satisfies:
%\begin{equation} \label{eq:g-funct-2}
%   g(z,p) = 
%   \begin{cases}
%      0 & z \in \mathcal{N}_{2\epsilon},\\
%      0 & z \in \mathcal{W}_{\epsilon}.
%   \end{cases}
%\end{equation}

Next, choose a function $A: \mathbb{C} \longrightarrow \mathbb{R}$
with the following properties:
%\begin{enumerate}[(\text{A}.1)]
\begin{enumerate}[label={(A.\arabic*)}]
  \item $A(z) \geq 0$ for every $z \in \mathbb{C}$.
  \item $A(z) = 0$ for every $z \in \mathcal{N}_{\epsilon}$.
  \item $A(z) \geq |g(z,p)|$ for every $z \in \mathbb{C}$, $p\in M$.
  \item \label{i:cond-A-4} Let $u_0, v_0$ be the vector fields
   associated to the form $\Omega_1 = \Omega^{f, \alpha, \beta}$
   from~\eqref{eq:Omega_1} using the recipe from~\eqref{eq:u0-v0}. We
   require that
   $$A(z) > \sigma(z) \bigl | f(z,p) - \sigma(z) \omega_p \bigl(u_0(z,p),
   v_0(z,p) \bigr) \bigr | + |g(z,p)|$$ for every $z \in \mathbb{C}
   \setminus \mathcal{N}_{2\epsilon}$, $p \in M$.
  \item $A(z) = C$ for every $z \in \mathcal{W}$, for some constant
   $C>0$.
\end{enumerate}
\pbred{The role of the function $A$ is to flatten the form $\Omega_1$
  on $\mathcal{W}$, so it is split there, while ensuring
  non-degeneracy.} Such a function $A$ can be constructed as
follows. We start by defining a function
$A': \mathbb{C} \longrightarrow \mathbb{R}$ which is positive
\pbred{and satisfies condition~(A.4) (with $A'(z)$ on the left-hand
  side of the inequality).} Such a function obviously exists because
$M$ is compact.  We then cut $A'$ off to make it $0$ on
$\mathcal{N}_{\epsilon}$ and constant on $\mathcal{W}$, where the
cutting off takes place within
$\mathcal{N}_{2 \epsilon} - \mathcal{N}_{\epsilon}$ and within
$\mathcal{W}_{\epsilon} - \mathcal{W}$, where the function $g$ is $0$
anyway. \pbred{It is easy to see that the cutting off can be done in
  such that the inequality in~(A.4) continues to hold and similarly
  for~(A.3).} The function resulting from $A'$ after this procedure
can be taken to be the desired function $A$. See
Figure~\ref{f:fib-A-sigma-g}.

\begin{figure}[htbp]
   \begin{center}
      \includegraphics[scale=0.5]{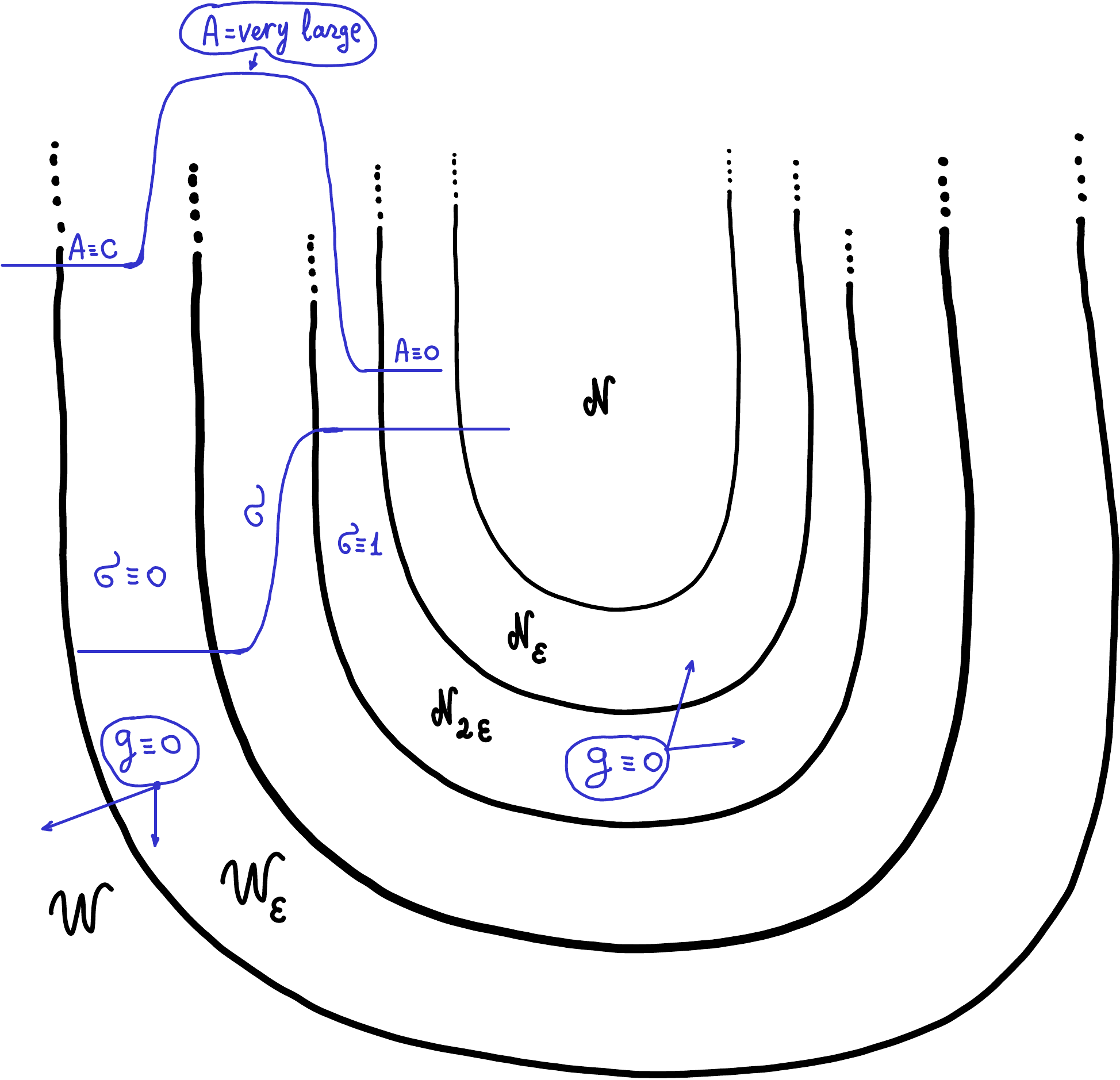}
   \end{center}
   \caption{The functions $\sigma$, $A$ and $g$.
     \label{f:fib-A-sigma-g}}
\end{figure}

Finally, define:
\begin{equation}
   \begin{aligned}
      & f'(z,p) := \sigma(z) f(z,p) + g(z,p)+ A(z), \\
      & \alpha'_{z,p} := \sigma(z) \alpha_{z,p} = d^v (\sigma(z) F)_{z,p}, \\
      & \beta'_{z,p} := \sigma(z) \beta_{z,p} = d^v (\sigma(z)
      G)_{z,p}.
   \end{aligned}
\end{equation}

Consider now the form
\begin{equation} \label{eq:omega-2}
  \Omega_2 : = \Omega^{f',\alpha', \beta'} = f' dy_1 \wedge dy_2 + 
  \alpha' \wedge dy_1 + \beta' \wedge dy_2 + \omega.
\end{equation}

Note that $\Omega_2$ coincides with $\Omega_1$ over a small
neighborhood of $\overline{\mathcal{N}}$ and therefore $\Omega_2$
gives rise via the trivialization $\varphi$ to a well defined $2$-form
$\Omega'$ over the \pbred{whole} of $E$. Moreover $\Omega'$ coincides
with $\Omega$ on $\pi^{-1}(\mathcal{N})$.

We claim that $\Omega'$ is a symplectic form on $E$ and that it
satisfies all the properties claimed by
Proposition~\ref{p:from-gnrl-to-tame}. 

We first show that $\Omega_2$ is closed using Lemma~\ref{l:Om-closed}.
Indeed
\begin{align*}
   d^v f' & = \sigma d^vf + d^v g = \sigma \partial_{y_2}(\alpha) -
   \sigma \partial_{y_1}(\beta) + d^v g \\
   & = \partial_{y_2} (\sigma \alpha) - \partial_{y_1}(\sigma \beta) +
   \bigl(d^v g -
   \partial_{y_2}(\sigma) \alpha + \partial_{y_1}(\sigma) \beta \bigr) \\
   & = \partial_{y_2}(\alpha') - \partial_{y_1}(\beta').
\end{align*}
Here the last term (between the brackets) on the second line vanishes
by~\eqref{eq:g-funct-1}.

We now prove that $\Omega_2$ is non-degenerate and moreover admits a
compatible almost complex structure $J'$ for which the projection
$\mathbb{C} \times M \longrightarrow \mathbb{C}$ is
$(J',i)$-holomorphic. Note that with the notation
from~\eqref{eq:u0-v0} and~\eqref{eq:hdist} the effect of replacing
$\alpha$ and $\beta$ by $\alpha' = \sigma \alpha$ and $\beta' = \sigma
\beta$ results in changing the vector fields $u_0, v_0$ to $u'_0 =
\sigma u_0$, $v'_0 = \sigma v_0$. Thus by Lemma~\ref{l:Om-non-deg} we
only need to check that:
\begin{equation} \label{eq:f'-omega-1} f'(z,p) > \omega_p(u'_0(z,p),
   v'_0(z,p)) \quad \forall \,  p \in M, \; z \in \mathbb{C} \setminus
   \overline{\mathcal{N}}.
\end{equation}

We have:
\begin{equation} \label{eq:f'-omega-2}
   \begin{aligned}
      f'(z,p) - \omega_p(u'_0(z,p), v'_0(z,p)) & = \sigma(z)f(z,p) +
      g(z,p) + A(z) - \sigma^2(z) \omega_p(u_0, v_0) \\
      & =\sigma(z) \bigl(f(z,p) - \sigma(z) \omega_p(u_0, v_0) \bigr)
      + \bigl( g(z,p) + A(z) \bigr).
   \end{aligned}
\end{equation}
We denote by $T_1 = \sigma(z) \bigl(f(z,p) - \sigma(z) \omega_p(u_0,
v_0) \bigr)$ the first term on the last line of~\eqref{eq:f'-omega-2}
and by $T_2 = g(z,p) + A(z)$ the second one.

We first verify~\eqref{eq:f'-omega-1} over
$\pi^{-1}(\mathcal{W}_{\epsilon})$. Indeed, when $z \in
\mathcal{W}_{\epsilon}$ we have $\sigma(z)=0$ hence $T_1 = 0$. By the
construction of the function $A$ we have $T_2 > 0$, hence $T_1 + T_2 >
0$.

Next we check~\eqref{eq:f'-omega-1} over
$\pi^{-1}(\mathcal{N}_{2\epsilon} \setminus \overline{\mathcal{N}})$.
Let $z \in \mathcal{N}_{2\epsilon} \setminus \overline{\mathcal{N}}$
and $p \in M$. Note that $\sigma(z) = 1$ hence
$T_1 = f(z,p) - \omega_p(u_0(z,p), v_0(z,p)) > 0$ by
Lemma~\ref{l:Om-non-deg}. Since $T_2 \geq 0$ we have $T_1 + T_2 > 0$.

Finally, the inequality~\eqref{eq:f'-omega-1} for $z \in \mathbb{C}
\setminus (\mathcal{N}_{2\epsilon} \cup \mathcal{W}_{\epsilon})$
follows easily from requirement~\ref{i:cond-A-4} in the construction
of the function $A$.

\

To finish the proof, we turn to the case of a non-compact fibre. Thus
we assume the conditions in~\S\ref{sb:defs-lef-fibr} and, in
particular, assumption $T_{\infty}$.  The proof above applies in this
case too, and we will preserve all the notation above, but there are a
number of adjustments that we describe below. Recall the set
$E^{\infty}$ that appears in the assumption $T_{\infty}$ and put
$M^{\infty}= M\cap E^{\infty}$. Recall also that, as before,
$M=\pi^{-1}(z_{0})$. Let
$$\phi: \C\times M^{\infty}\to E^{\infty}$$ be the 
trivialization provided by $T_{\infty}$.  Consider also the
restriction of this trivialization to $\C\setminus
\overline{\mathcal{N}}$:
\begin{equation}
   \phi : (\C \setminus \overline{\mathcal{N}})
   \times M^{\infty}\to E^{\infty}|_{\C\setminus \overline{\mathcal{N}}}
\end{equation} and put $\phi_{0}:M^{\infty}\to M^{\infty}$,
$\phi_{0}(p)=\phi (z_{0},p)$.

Consider also the map $\varphi$ constructed at the {\bf Step 1} above
and its restriction:
$$\varphi: (\C\setminus \overline{\mathcal{N}})\times M^{\infty}
\to E^{\infty}|_{\C\setminus \overline{\mathcal{N}}}$$ which is well
defined due to Assumption $T_{\infty}$.

\pbred{For brevity, write $\Omega = \Omega_{E}$.}  Given that the
connection associated to $\phi^{\ast}\Omega$ is trivial on
$(\C\setminus \overline{\mathcal{N}})\times M^{\infty}$, we deduce
that $\varphi(z,p)=\phi(z,\phi_{0}^{-1}(p))$ for all
$z\in \C\setminus \overline{\mathcal{N}}$, $p\in
M^{\infty}$. Therefore
$\phi^{\ast}\Omega|_{(\C\setminus \overline{\mathcal{N}})\times
  M^{\infty}}=\omega_{\C}\oplus \omega$.

Recall that over $(\C\setminus\overline{\mathcal{N}})\times M$ the
form \pbred{$\Omega_1 = \varphi^*\Omega$} can be written as
$$\Omega_1 = 
\omega + \alpha\wedge dy_{1}+\beta\wedge dy_{2} + fdy_{1} \wedge
dy_{2}~.~$$ This means that $\alpha,\beta$ vanish over
$(\C\setminus \overline{\mathcal{N}})\times M^{\infty}$ and $f$ is
constant there.  Therefore, we can choose the functions $F$, $G$ so
that they both vanish on
$(\C\setminus\overline{\mathcal{N}})\times M^{\infty}$.  Starting from
this point the remainder of the proof continues as in the compact
fibre case by using the fact that $g(z,p)$, as well as $\alpha'$,
$\beta'$, $u_{0}(z,p)$, $v_{0}(z,p)$ all vanish over
$(\C\setminus \overline{\mathcal{N}})\times M^{\infty}$.

\pbred{Recall now the forms $\Omega_2$ and $\Omega'$ (defined by
  formula~\eqref{eq:omega-2} and the paragraph following it). Summing
  up the preceding discussion,} the form $\Omega_{2}$ hence also
$\Omega'$ satisfies $\phi^{\ast}\Omega'=B(z)\omega_{\C}\oplus \omega$
over $\C\times M^{\infty}$, where $B(z)$ is positive and bounded. By
adding to $\Omega'$ another term of the form
$D(z)\pi^{\ast}\omega_{\C}$ we obtain a form that satisfies all the
properties claimed in Proposition~\ref{p:from-gnrl-to-tame} as well as
the assumption $T_{\infty}$. \pbred{(The role of adding the last term
  is to ensure that property~$(1)$ in
  Proposition~\ref{p:from-gnrl-to-tame} is satisfied.)}  \Qed

% !TEX root = lefcob.tex

\section{Fukaya categories} \label{subsec:fuk-cat}

The purpose of this section is to introduce the various Fukaya
categories that play a role in the paper.  We start with a brief
sketch of the construction of the Fukaya category $\fuk^{\ast}(M)$ of
uniformly monotone, closed Lagrangian submanifolds of a symplectic
manifold $(M,\omega)$ which is assumed to be either closed or
\pbred{convex at infinity}. The full construction in the exact case
can be found \pbred{in~\cite[Sections~8-12]{Se:book-fukaya-categ}}
(the minor adjustments required in the monotone case are described,
for instance, in \cite{Bi-Co:lcob-fuk}).  In
\S\ref{subsubsec:Fuk-cob}, we pursue with the construction of the
Fukaya category $\fuk^{\ast}(E)$ of uniformly monotone cobordisms in a
tame Lefschetz fibration $\pi:E\to \C$ of generic fiber $(M,\omega)$.
This follows closely \S 3 of \cite{Bi-Co:lcob-fuk} where this
construction is implemented for the trivial fibration $E=\C\times M$.
The passage from a trivial fibration to a tame one is quite
straightforward but we provide enough details on this construction as
required for further arguments later in the paper and also to ensure
that the notions involved are accessible to a reader without prior
detailed knowledge of the techniques in~\cite{Bi-Co:lcob-fuk}.
In~\S\ref{subsubsec:fuk-cob-gen} we use the construction in the tame
setting together with the results in~\S\ref{sb:tame-vs-gnrl} to define
a Fukaya category associated to a general Lefschetz fibration.

In the definition of the various algebraic objects used in the paper
there are two coefficient rings of interest, $\Z_{2}$ and the
universal Novikov ring $\mathcal{A}$ over $\Z_{2}$:
$$\mathcal{A}=\{ \ \sum_{k=0}^{\infty} 
a_{k}T^{\lambda_{k}} \ : \ a_{k}\in \Z_{2}, \ \lambda_{k}\in \R, \
\lim_{k_{\to \infty}}\lambda_{k}\to \infty \ \}~.~$$ We work over
$\mathcal{A}$ at all times except if otherwise indicated.

\subsection{The Fukaya category of $M$}\label{subsec:Fuk-fibr}
The main structures in use in the paper are the Fukaya category,
$\fk^{\ast}(-)$, and the derived Fukaya category, $D\fuk^{\ast}(-)$.
Here $*$ encodes a uniform monotonicity constraint imposed to the
objects of $\fuk^{\ast}(M)$.  This constraint is necessary to define
the $A_{\infty}$-operations.

The book \cite{Se:book-fukaya-categ} is a comprehensive reference for
the basic definitions of the $A_{\infty}$ machinery as well as the
construction of the Fukaya category and its derived
version. \pbred{Our notation - which is
  homological}\footnote{\pbred{Since we work in an ungraded setting,
    the difference between homological and cohomological might seem
    invisible. However, our Floer homologies correspond to Morse
    homology rather than cohomology. In particular the unity in
    $HF(L,L)$ corresponds to the fundamental class of $L$ etc. Apart
    from that, the ordering of the terms in the higher operations
    $\mu_k$ is opposite to Seidel's and our conventions for the Yoneda
    embedding differs from Seidel's. This is all described in detail
    in the Appendix to~\cite{Bi-Co:lcob-fuk}.}}, \pbred{in contrast to
  Seidel's which is cohomological - is the same as
  in~\cite{Bi-Co:lcob-fuk}, see in particular the Appendix to that
  paper}. There is a single difference with respect
to~\cite{Bi-Co:lcob-fuk} which is that we use here the universal
Novikov ring $\mathcal{A}$ in the place of $\Z_{2}$.  As we shall see,
this is not a matter of choice, rather a requirement for a certain
part of our results to hold.  We emphasize that in the construction of
$D\fuk^{\ast}(-)$ we do not complete with respect to idempotents.
Moreover, as in \cite{Bi-Co:lcob-fuk} we work in an ungraded context.

Fix a symplectic manifold $(M,\omega)$, compact or \pbred{convex} at
infinity.  Given a \pbred{closed} Lagrangian submanifold $L\subset M$
there are two morphisms
$$\mu :\pi_{2}(M,L)\to \Z\ , \ \omega :\pi_{2}(M,L)\to \R$$ 
given, the first, by the Maslov index and, the second, by integration
of $\omega$.  We say that $L$ is monotone if
$\omega (\alpha) = \rho \mu(\alpha)$ for some constant $\rho \geq 0$
and if the number
$$N_{L}=\min\{\mu(\alpha) : \alpha \in \pi_{2}(M,L) \ , \
\omega(\alpha)>0 \}$$ is at least $2$.

\pbhl{Note that we do allow $\rho=0$ in the definition of
  monotonicity. This means that $\omega$ vanishes on $\pi_2(M,L)$
  (such Lagrangians are sometime called weakly exact). In this case we
  set $N_L= \infty$.}

\label{pg:dL}
For a connected monotone Lagrangian $L$ and for a generic almost
complex structure $J$ compatible with $\omega$, the number (mod $2$)
of $J$-holomorphic disks of Maslov number $2$ that pass through a
generic point of $L$ is an invariant (in the sense that it does not
depend either on the point or on the choice of $J$). It is denoted by
$d_{L}$ (and is defined in detail, for instance, in
\cite{Bi-Co:rigidity}). Note that in case $\rho=0$ we set $d_L=0$ by
definition.

In order to define the Fukaya category of $M$ we first need to specify
its underlying class of Lagrangian submanifolds. In what follows we
will mainly consider two classes of Lagrangians $\mathcal{L}^{(0)}(M)$
and $\mathcal{L}^{(\rho,1)}$, which are defined as
follows: \label{pg:mon-class-M}
\begin{enumerate}
\item[a.] The class $\mathcal{L}^{(0)}(M)$: this class consists of all
  closed monotone Lagrangians $L \subset M$ with $d_L = 0$. This
  includes in particular \pbred{all} Lagrangians with $N_L \geq 3$ as
  well as the case $\rho=0$.
\item[b.] Class $\mathcal{L}^{(\rho, 1)}(M)$: consists of all the
  closed monotone Lagrangians $L \subset M$ with $d_L = 1$ and with
  monotonicity constant $\rho$, where $\rho>0$ is a prescribed
  positive real number.
\end{enumerate}
Of course one could restrict also to some subclasses of the above. For
example, when $M$ is exact it makes sense to restrict to the subclass
$\mathcal{L}^{(\textnormal{ex})}(M) \subset \mathcal{L}^{(0)}(M)$ of
exact Lagrangian submanifolds.
  
To simplify the notation will denote any of these two choices by
$\mathcal{L}^*(M)$, where the symbol $*$ stands for either $(0)$ in
the first case, or for $(\rho, 1)$ in the second case.  Lagrangians in
the class $\mathcal{L}^*(M)$ will be called {\em uniformly monotone}
of class $*$.

In what follows we will work also with uniformly monotone {\em
  negatively-ended} Lagrangian cobordisms in the total space of a
Lefschetz fibration $E \longrightarrow \mathbb{C}$. Similarly to the
Lagrangians in $M$ we will denote the various classes of uniformly
monotone Lagrangian cobordisms in $E$ by $\mathcal{L}^*(E)$, where the
definition of these classes is the same as above except that the
Lagrangians in $E$ are not assumed to be compact.

Floer homology will be taken in this paper with coefficients in the
Novikov ring $\mathcal{A}$ and its definition will be shortly reviewed
below. It was introduced by Floer in \cite{Fl:Morse-theory} and, in
this monotone setting, by Oh \cite{Oh:HF1, Oh:HF1-add}.

\begin{remsnonum} 
   \begin{enumerate}
     \item[a.] In contrast to~\cite{Bi-Co:lcob-fuk} there is no injectivity
      condition on the inclusions $\pi_{1}(L)\to \pi_{1}(M)$ (this is
      because the coefficient ring is $\mathcal{A}$ and not $\Z_{2}$).
    \item[b.] In case there exists a spherical class $A \in \pi_2(M)$
      with $\omega(A)>0$, the monotonicity constant $\rho$ is
      determined by the proportionality constant between $[\omega]$
      and the first Chern class of the ambient symplectic
      manifold. Thus in this case there is only one class of the type
      $\mathcal{L}^{(\rho,1)}$.
   \end{enumerate}
\end{remsnonum}

The Fukaya $A_{\infty}$-category $\fuk^{\ast}(M)$ has as objects the
Lagrangians in $\mathcal{L}^{\ast}(M)$,
$$\mathcal{O}b(\fuk^{\ast}(M))=\mathcal{L}^{\ast}(M)~.~$$

Let $L,L'\in \mathcal{L}^{\ast}(M)$ and assume for the moment that $L$
and $L'$ intersect transversely.  In this case, the Floer complex,
$(CF(L,L'; J), d)$, associated to $L$ and $L'$ is defined by choosing
a regular almost complex structure $J$ compatible with $\omega$ and is
a free $\mathcal{A}$-module with generators the intersection points of
$L$ and $L'$. In this paper $CF(L,L')$ is a complex without grading.
 
The differential $d$ is defined in terms of $J$-holomorphic strips
$u:\R\times [0,1]\to M$ \pbred{with} $u(\R\times \{0\})\subset L$,
$u(\R\times \{1\})\subset L'$ \pbred{and}
$\lim_{s\to-\infty}u(s,t)=x\in L\cap L'$,
$\lim_{s\to +\infty}u(s,t)=y\in L\cap L'$.  We have:
$$d(x)=\sum_{y}\sum_{u\in\mathcal{M}_{0}(x,y)} T^{\omega(u)}y$$ where 
the sum is over all the intersection points $y\in L\cap L'$ and
$\mathcal{M}_{0}(x,y)$ is the $0$-dimensional subspace of the moduli
space of $J$-strips $u$ joining $x$ to $y$.  Uniform monotonicity is
used to show that $d^{2}=0$.

The homology of this complex, $HF(L,L')$, is the Floer homology of $L$
and $L'$. It is independent of $J$ as well as of Hamiltonian
perturbation of $L$ and of $L'$.

\

The morphisms in $\fuk^{\ast}(M)$ are
$\mor_{\fuk^{\ast}(M)}(L,L')=CF(L,L')$. The $A_{\infty}$ structural
maps are, by the definition of an $A_{\infty}$-category, multilinear
maps
$$\mu_{k}: CF(L_{1},L_{2})\otimes CF(L_{2}, L_{3})\otimes \ldots
\otimes CF(L_{k},L_{k+_1})\to CF(L_{1},L_{k+1})$$ that satisfy the
relation $\mu\circ \mu=\sum \mu(-,-,\ldots, \mu,\ldots, -,-)=0$.
\pbred{In our case, these maps are such that $\mu_{1}=d$ = the Floer
  differential and, for $k>1$, $\mu_{k}$ is defined by:
\begin{equation} \label{eq:mu-k-no-pert}
  \mu_{k}(x_{1}, \ldots, x_{k})=
  \sum_{y} \sum_{u\in \mathcal{M}_{0}(x_{1}, \ldots, x_{k}; y)
  }T^{\omega(u)} y.
\end{equation}
Here, at least when the $L_{i}$'s and $L$ are in general position,
$x_{i}\in L_{i}\cap L_{i+1}$, $y \in L_{1}\cap L_{k+1}$ and
$\mathcal{M}_{0}(x_{1},\ldots, x_{k}; y)$ is the $0$-dimensional
moduli space of (perturbed) $J$-holomorphic polygons with $k+1$ sides
that have $k$ ``inputs'' asymptotic - in order - to the intersection
points $x_{i}$ and one ``exit'' asymptotic to $y$. Monotonicity is
used to show that the sums in~\eqref{eq:mu-k-no-pert} are well defined
over $\mathcal{A}$.  The relation $\mu\circ\mu=0$ extends the relation
$d^{2}=0$.}

This is just a rough summary of the construction as, in particular,
the operations $\mu_{k}$ have to be defined for all families
$L_{1},\ldots, L_{k+1}$ and not only when $L_{i}, L_{i+1}$, etc., are
transverse. \pbred{In reality one has to add perturbation terms to the
  Cauchy-Riemann equation that come from Hamiltonian functions
  associated to each vertex of the polygon and the asymptotic
  conditions $x_i$, $y$, are replaced by trajectories $\gamma_i$,
  $\gamma$ of the flows of these Hamiltonian functions that start on
  $L_i$ and end on $L_{i+1}$, respectively start on $L_1$ and end on
  $L_k$.} Moreover, the regularity of these moduli spaces depends on a
number of choices of auxiliary data, basically a coherent system of
{\em strip-like ends} and coherent {\em perturbation data}. We refer
to \cite{Se:book-fukaya-categ} for the actual implementation of the
construction which is considerably more involved.  Additionally, these
notions are made more precise in \S\ref{subsubsec:Fuk-cob} where we
discuss in more detail some of the ingredients used in the
construction of a Fukaya category $\fuk^{\ast}(E)$ with objects
certain cobordisms in $E$.

Consider next the category of $A_{\infty}$-modules over the Fukaya
category
$$mod (\fuk^{\ast}(M)):=fun(\fuk^{\ast}(M), Ch^{opp})$$ where 
$Ch^{opp}$ is the opposite of the dg-category of chain complexes over
$\mathcal{A}$.  The category of $A_{\infty}$-modules is an
$A_{\infty}$-category in itself \pbred{(in fact a dg-category)} and is
triangulated in the $A_{\infty}$-sense with the triangles being
inherited from the triangles in $Ch$ (where they correspond to the
usual cone-construction for chain complexes). There is a Yoneda
embedding $\mathcal{Y}:\fuk^{\ast}(M)\to mod (\fuk^{\ast}(M))$, the
functor associated to an object $L\in \mathcal{L}^{\ast}(M)$ being
$CF(-,L)$.  The derived Fukaya category $D\fuk^{\ast}(M)$ is the
homology category associated to the triangulated completion of the
image of the Yoneda embedding inside $mod(\fuk^{\ast}(M))$.

\subsubsection{Iterated cone decompositions} \label{sbsb:iter-cone}

We now briefly fix the notation for writing iterated
cone-decompositions in a triangulated category $\mathcal{C}$.  Suppose
that there are exact triangles:
$$C_{i+1}\to Z_{i}\to Z_{i+1}$$
with $1\leq i\leq n$ and with $X=Z_{n+1}$, $Z_{0}=C_{0}$. We write
such an iterated cone-decomposition as $$ X=(C_{n+1}\to (C_{n}\to (
C_{n-1}\to \ldots \to C_{0}))\ldots )~.~$$ With this notation
$$Z_{k}=(C_{k}\to (C_{k-1}\to \ldots \to C_{0}))\ldots )~.~$$ We also
notice that we can in fact omit the parentheses in this notation
without ambiguity.  This follows from the following equality of the
two iterated cones:
$$((A\to B)\to C) = (A\to (B\to C))~.~$$
In turn, this follows immediately from the axioms of a triangulated
category together with the fact that we work here in an ungraded
setting (the formula can also be easily adjusted to the graded case).
In short, we will write:
$$X= (C_{n+1}\to C_{n}\to  C_{n-1}\to \ldots \to C_{0})~.~$$
There is a slight abuse of notation in the above formula in
  that, in the absence of the relevant parentheses, the arrows in the formula do not
  independently correspond to 
  morphisms in the category $\mathcal{C}$. The formula should
  be interpreted as saying that $X$ can be expressed as an iterated
  cone attachment with the objects $C_0, \ldots, C_{n+1}$ as described above.

\
\subsubsection{The Grothendieck group} \label{sbsb:Groth-grp} The
Grothendieck group of a triangulated category $\mathcal{C}$ is the
abelian group generated by the objects of $\mathcal{C}$ modulo the
relations generated by $B=A+C$ as soon as
$$A\to B\to C$$ is an exact triangle. 
We denote the Grothendieck group of $\mathcal{C}$ by
$K_0(\mathcal{C})$. Notice that, with our terminology, if
$$L_{1}= (L_{n}\to L_{n-1}\to  L_{n-2}\to \ldots \to L_{2}),$$ then,
because we work in an
  ungraded setting, in $K_{0}(\mathcal{C})$ we have the relation
$L_{n}+L_{n-1}+\ldots +L_{1}=0$.  Notice also that, due to the same reason, our
  version of $K_0(\mathcal{C})$ is always $2$-torsion, i.e. $2A=0$ for
  every $A \in K_0(\mathcal{C})$.

The main Grothendieck groups of interest in this paper will be those
of derived Fukaya categories $K_0D\fuk^{\ast}(-)$.

\subsection{{\Mntlf} Lefschetz fibrations} \label{sb:monlef}

In order to define a Fukaya category of cobordisms in a Lefschetz
fibration that is suitable for our needs we need to impose additional
conditions on the Lefschetz fibration. These will ensure that all the
thimbles and vanishing spheres are monotone Lagrangian submanifolds
(with the right monotonicity paramters) in their respective ambient
manifolds and so can be included as objects in the same Fukaya
categories.

Let $\pi: E \longrightarrow \mathbb{C}$ be a Lefschetz fibration as in
Definition~\ref{df:lef-fib}. Fix a base point $z_0 \in \mathbb{C}$ and
let $M = \pi^{-1}(z_0)$ be the fiber over $z_0$, endowed with the
symplectic structure $\omega = \Omega_E|_M$ induced from $E$. Denote
by $x_1, \ldots, x_m \in E$ the critical points of $\pi$ and by
$v_1, \ldots, v_m \in \mathbb{C}$ the corresponding critical values of
$\pi$.  Fix $m$ smooth paths
$\lambda_1, \ldots, \lambda_m \subset \mathbb{C}$ such that for every
$k$ $\lambda_k$ starts at $v_k$ and ends at $z_0$ and such that except
of their end points none of the paths $\lambda_k$ passes through the
critical values of $\pi$. Denote by $S_1, \ldots, S_m \subset M$ the
Lagrangian vanishing spheres associated to the paths
$\lambda_1, \ldots, \lambda_m$.

\begin{dfn}[{\Mntlf} Lefschetz fibrations] \label{df:monlef} We say
  that $\pi:E \longrightarrow \mathbb{C}$ is a {\em \mntlf} Lefschetz
  fibration if the following conditions holds:
  \begin{enumerate}
  \item In case $\dim_{\mathbb{R}} M \geq 4$ we require that $M$ is a
    monotone symplectic manifold, that is $\omega = 2\rho c_1$ on
    $\pi_2(M)$ \pbhl{for some $\rho \geq 0$.} \label{i:monot-ge4}
  \item In case $\dim_{\mathbb{R}} M =2$ we require that
    $(E,\Omega_E)$ is a monotone symplectic manifold. Note that this
    implies that $M$ is monotone too and we define $\rho$ as in
    point~(\ref{i:monot-ge4}) above.
  \end{enumerate}
  In addition to the above we also make the following
  assumptions. Denote by $c_1^{\textnormal{min}} \in \mathbb{Z}_{>0}$
  the minimal Chern number of $M$. Then:
    \begin{enumerate}
    \item[(i)] \pbhl{If $\rho=0$ set $d_E=0$ and $*=(0)$.}
    \item[(ii)] If $\rho>0$ and $c_1^{\min}=1$ then we require that
      $d_{S_1} = \cdots = d_{S_m}$ (see Page~\pageref{pg:dL} for what
      $d_{S_k}$ is). Denote the latter number by
      $d_E \in \mathbb{Z}_2$. In case $d_E=0$ set $*=(0)$ and if
      $d_E=1$ set $*=(\rho,1)$.
    \item[(iii)] If $c_1^{\min}>1$ set $d_E=0$ and $*=(0)$.
    \end{enumerate}
\end{dfn}

\pbhl{We will refer to $*$ from} Definition~\ref{df:monlef} as the
monotonicity class of the Lefschetz fibration $E$. By
Proposition~\ref{p:monot-E} below it depends only on the fibration
$E$.  In~\S\ref{subsubsec:Fuk-cob} \pbhl{below we will set up the
  Fukaya category of (negative ended) cobordisms in $E$ and the
  monotonicity class $*$ will be used in order to constrain the class
  of Lagrangian cobordisms that are objects of this category.}

\label{pg:mon-exception}
\pbhl{We will make one exception to the definition above, namely when
  $E$ has no critical values at all, i.e.
  $E \approx \mathbb{C} \times M$ is the trivial fibration. In this
  case we only assume that $M$ is a monotone symplectic manifold and
  will choose the monotonicity class $*$ to be arbitrary subject to
  the restrictions made on} page~\pageref{pg:mon-class-M}
in~\S\ref{subsec:Fuk-fibr} above. See also
Remark~\ref{r:remote-monotone} below.

\begin{rem} \label{r:S_k-monot} It is easy to see that when
  $\dim_{\mathbb{R}}M \geq 4$, $(M, \omega)$ is monotone iff
  $(E, \Omega_E)$ is monotone and in that case
  $c_1^{\min}(E) = c_1^{\min}(M)$. This is so because under this
  dimension assumption, the map induced by inclusion
  $\pi_2(M) \to \pi_2(E)$ is surjective. Apart from that we also have
  $c_1(E)|_{H^2(M)} = c_1(M)$. Moreover, as will be seen in the proof
  of Proposition~\ref{p:monot-E} below, the monotonicity of the
  symplectic manifold $(E, \Omega_E)$ implies that the spheres
  $S_1, \ldots, S_k \subset M$ are all monotone (even when
  $\dim_{\mathbb{R}}M=2$).
\end{rem}

\begin{prop} \label{p:monot-E} The Definition~\ref{df:monlef} is
  independent of the choice of paths $\lambda_1, \ldots, \lambda_m$.

  Let $E$ be a {\mntlf} Lefschetz fibration and $T$ a thimble over any
  path $\gamma$ (that starts at a critical value of $\pi$). Then $T$
  is monotone with minimal Maslov number $2c_1^{\min}(E)$ and
  monotonicity ratio $\rho$. If moreover, $\gamma$ is horizontal at
  $-\infty$ (or $+\infty$) and $S$ is the Lagrangian sphere associated
  to the end of $T$ then we also have $d_T = d_E = d_S$. In
  particular, both $T$ and $S$ are monotone of class $*$ in their
  respective ambient manifolds.
\end{prop}

\begin{proof}
  That all thimbles are monotone follows easily from the fact that $T$
  is simply connected and that $(E,\Omega_E)$ is a monotone symplectic
  manifold.

  Denote now by $T_{\lambda_k}$ the thimble over the path
  $\lambda_k$. Since $T_{\lambda_k}$ is monotone then so is $S_k$
  because $c_1(E)|_{H^2(M)} = c_1(M)$.

  We now turn to the first statement in the proposition. This follows
  from the fact that if we change the given set of paths
  $\lambda_1, \ldots, \lambda_m$ by another set
  $\lambda'_1, \ldots, \lambda'_m$ then each of the new vanishing
  spheres $S'_k$ is the image of $S_k$ under some symplectic
  diffeomorphism of $M$ (which is in fact, up to symplectic isotopy, a
  certain composition of Dehn twists and their inverses along the
  spheres $S_1, \ldots, S_m$). Therefore, the monotonicity of $S'_k$
  is preserved and so is the value of $d_{S'_k}$.

  Finally, let $T$ be a thimble over a path $\gamma$ which is
  horizontal at $\pm \infty$. By the results of~\cite{Chek:cob} (see
  also~\cite[Remark 2.2.4]{Bi-Co:cob1}) with obvious adaption to
  Lefschetz fibrations it follows that $d_T = d_{S}$, where $S$ is the
  Lagrangian sphere associated to the end of $T$. Since $S$ is a
  vanishing sphere we have $d_S=d_E$.
\end{proof}

\begin{rem} \label{r:tvsg-mntlf} The procedure from
  Proposition~\ref{p:from-gnrl-to-tame}, that modifies the symplectic
  structure on a Lefschetz fibration to render it tame, does not
  affect the property of being {\mntlf}. This is so because, in the
  notation of Proposition~\ref{p:from-gnrl-to-tame}, the map induced
  by the inclusion $\pi_2(\pi^{-1}(\mathcal{N})) \to \pi_2(E)$ is an
  isomorphism.
\end{rem}

{\em From now on, we will generally assume that our Lefschetz
  fibrations are {\mntlf}.}

\subsection{The Fukaya category of negative ended cobordisms in tame
  Lefschetz fibrations} \label{subsubsec:Fuk-cob} We consider a
\pbgreen{{\mntlf}} Lefschetz fibration $\pi:E\to\C$ that is tame
outside $U\subset \C$ and has as generic fibre the symplectic manifold
$(M,\omega)$.  We will also assume that \pbred{$U$ is $U$-shaped, as
  in Figure~\ref{fig:null-cob}, and that} \pbred{
  \begin{equation} \label{eq:u-tame}
  \overline{U}\subset \R\times [0,+\infty).
\end{equation}
} The main object of study in this paper is the Fukaya category
$\fk^{\ast}(E)$, \pbgreen{where $*$ is the monotonicity class of $E$
  and has been set in Defintion~\ref{df:monlef}.}  It has as objects
the cobordisms $V$ as in Definition~\ref{def:Lcobordism} such that the
following additional conditions are satisfied:
\begin{itemize}
  \item[i.] $V$ is monotone in the class $\ast$.
  \item[ii.] $V\subset \pi^{-1}(\R\times [\frac{1}{2},+\infty))$
  \item[iii.] $V$ has only negative ends that all belong to
   $\mathcal{L}^{\ast}(M)$. In particular, with the notation from
   Definition \ref{def:Lcobordism}, $k_{+}=1$ and $L'_{1}=\emptyset$.
\end{itemize}

This family of Lagrangians of $E$ with the properties above will be
denoted by $\mathcal{L}^{\ast}(E)$. In other words,
$\mathcal{O}b(\fk^{\ast}(E))=\mathcal{L}^{\ast}(E)$. Such an object is
represented schematically in Figure \ref{fig:null-cob}.

\begin{figure}[htbp]
   \begin{center}
      \includegraphics[scale=0.4]{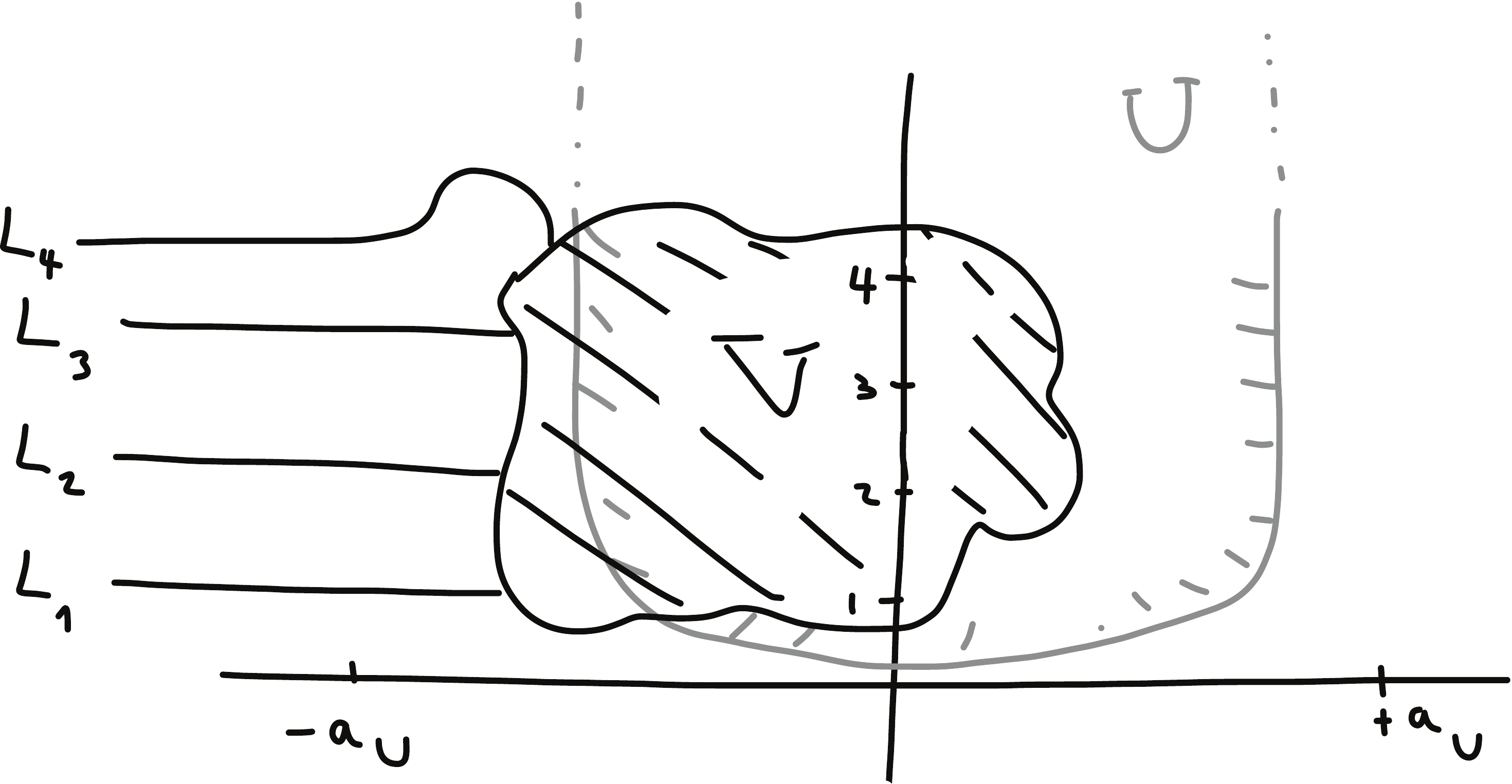}
   \end{center}
   \caption{\pbred{The projection on $\C$ of an object
       $V\in \mathcal{O}b(\fk^{\ast}(E))$ together with the set $U$
       outside which $E$ is tame.} \label{fig:null-cob}}
\end{figure}

We call the objects $V\in \mathcal{L}^{\ast}(E)$ {\em negatively-ended
  cobordisms}: they are cobordisms from the void set to a family
$(L_{1},\ldots, L_{s})$.

\begin{rem} \label{rem:null-cob}
   \begin{enumerate}
   \item[a.] In this paper we restrict \pbred{ourselves} to
     negatively-ended cobordisms but this is more a matter of
     convenience than of necessity. Some of the arguments in the paper
     are simpler in this setting but the same type of constructions
     allow the definition of a Fukaya category with both negative and
     positive ends.  Similarly, our decomposition results can also be
     adapted to this more general setting. We do not require $V$ to be
     connected.  Notice also that every Lagrangian cobordism
     $V \subset E$ that contains positive ends can be transformed to a
     negatively-ended cobordism by e.g. bending its positive ends
     along curves that turn to the left, then go above the
     singularities of $E$ and continue horizontally to $-\infty$.
   \item[b.] We remark that our notation $\mathcal{L}^*(E)$ and
     $\fk^*(E)$ somewhat differs from the one used
     in~\cite{Bi-Co:lcob-fuk}. In that paper we studied Lagrangian
     cobordisms in trivial fibrations $E = \mathbb{C} \times M$ and
     denoted by $\mathcal{C}\mathcal{L}_d(\mathbb{C} \times M)$ the
     collection of monotone Lagrangian cobordisms in
     $\mathbb{C} \times M$ (with possibly negative and positive
     ends). The corresponding Fukaya category was denoted by
     $\fk^d_{cob}(\mathbb{C} \times M)$.  Thus, in the present paper,
     we could have denoted our $\mathcal{L}^*(E)$ by
     $\mathcal{C}\mathcal{L}_{\ast}^{null}(E)$ and $\fk^*(E)$ by
     $\fk^{*, null}_{cob}(E)$, but we have decided to drop the
     additional decorations in order to keep the notation simpler.
   \end{enumerate}
\end{rem}

The operations $\mu_{k}$ of the Fukaya category $\fuk^{\ast}(E)$ are
defined following closely the construction in~\cite{Bi-Co:lcob-fuk}
which is basically a variant of the set-up in \pbred{Seidel's
  book~\cite[Sections~8-12]{Se:book-fukaya-categ}.} We review here the
technical points that will be needed later in the paper.  We will
first focus on the case when $M$ is compact and we will discuss the
additional modifications required when $M$ is convex at infinity at
the end of the construction.  There are two structures that need to be
added compared to the construction of the category $\fuk^{\ast}(M)$:
{\em transition functions} associated to a system of strip-like ends
and {\em profile functions}. As always, the operations $\mu_{k}$ are
defined in terms of counting (with coefficients in $\mathcal{A}$)
perturbed $J$-holomorphic polygons $u$. The role of the transition
functions is to allow such $u$ to be transformed by a change of
variables into curves $v$ that project holomorphically onto certain
regions of $\C$. The role of the profile functions - and particularly
that of their {\em bottlenecks} - is to ensure compactness at infinity
for the Floer complexes $CF(V,V')$ and to further restrict the
behavior of the $J$-polygons $u$.  We explain this point, which is
crucial for the arguments used later in the paper, at the end
of~\S\ref{subsubsec:Fuk-cob}.

\subsubsection{Transition functions} \label{subsubsec:transition} We
first recall the notion of a consistent choice of strip-like ends
\pbred{from~\cite[Sections~8d,~9g]{Se:book-fukaya-categ}.} Fix
$k \geq 2$. Let $\mathrm{Conf}_{k+1}(\partial D)$ be the space of
configurations of $(k+1)$ distinct points $(z_1, \ldots, z_{k+1})$ on
$\partial D$ that are ordered clockwise. Denote by
$Aut(D) \cong PLS(2,\mathbb{R})$ the group of holomorphic
automorphisms of the disk $D$. Let
$$\mathcal{R}^{k+1}=\mathrm{Conf}_{k+1}
(\partial D)/Aut(D)\ ,\
\widehat{\mathcal{S}}^{k+1}=\bigl(\mathrm{Conf}_{k+1}(\partial
D)\times D\bigr)/Aut(D)~.~$$ The projection
$\widehat{\mathcal{S}}^{k+1} \to \mathcal{R}^{k+1}$ has sections
$\zeta_i[z_1, \ldots, z_{k+1}] = [(z_1, \ldots, z_{k+1}), z_i]$, $i=1,
\ldots, k+1$ and let $\mathcal{S}^{k+1} = \widehat{\mathcal{S}}^{k+1}
\setminus \bigcup_{i=1}^{k+1} \zeta_i(\mathcal{R}^{k+1}).$ The fiber
bundle $\mathcal{S}^{k+1} \to \mathcal{R}^{k+1}$ is called a universal
family of $(k+1)$-pointed disks. Its fibers $S_r$, $r \in
\mathcal{R}^{k+1}$, are called $(k+1)$-pointed (or punctured) disks.

Let $Z^{+}= [0,\infty)\times [0,1]$, $Z^{-}=(-\infty, 0]\times [0,1]$
be the two infinite semi-strips and let $S$ be a $(k+1)$ pointed disk
with punctures at $(z_1, \ldots, z_{k+1})$. A choice of strip-like
ends for $S$ is a collection of embeddings: $\epsilon^{S}_{i}:Z^{-}\to
S$, $ 1\leq i \leq k$, $\epsilon^{S}_{k+1}:Z^{+}\to S$ that are proper
and holomorphic and
\begin{equation*} \label{eq:strip-like-ends}
   \begin{aligned}
     & (\epsilon^{S}_{i})^{-1}(\partial S)  = (-\infty, 0] \times \{0,
     1\}, \quad && \lim_{s \to -\infty}
     \epsilon^{S}_{i}(s, t) = z_i, \quad \forall \, 1 \leq i \leq k, \\
     & (\epsilon^{S}_{k+1})^{-1}(\partial S)  = [0, \infty) \times \{0,
     1\}, \quad && \lim_{s \to \infty} \epsilon^{S}_{k+1}(s, t) =
     z_{k+1}.
   \end{aligned}
\end{equation*}
such that the $\epsilon^{S}_i$'s have pairwise disjoint images. A
universal choice of strip-like ends for
$\mathcal{S}^{k+1}\to \mathcal{R}^{k+1}$ is a choice of $k+1$ proper
embeddings
$\epsilon^{\mathcal{S}}_i: \mathcal{R}^{k+1} \times Z^- \to
\mathcal{S}^{k+1}$, $i=1, \ldots, k$,\;
$\epsilon^{\mathcal{S}}_{k+1}: \mathcal{R}^{k+1} \times Z^+ \to
\mathcal{S}^{k+1}$ such that for every $r \in \mathcal{R}^{k+1}$ the
restrictions ${\epsilon^{\mathcal{S}}_i}|_{r \times Z^{\pm}}$ consists
of a choice of strip-like ends for $S_r$. \pbred{See~\cite[Section
  9c]{Se:book-fukaya-categ}} for more details.  In the case $k=1$, we
put $\mathcal{R}^2 = \textnormal{pt}$ and
$\mathcal{S}^{2} = D \setminus \{-1,1\}$. We endow
$D \setminus \{-1, 1\}$ with strip-like ends by identifying it
holomorphically with the strip $\mathbb{R} \times [0,1]$, \pbred{where
  the latter is endowed with its standard complex structure. The
  identification is done such that $-1 \in D$ corresponds to
  $-\infty \times [0,1]$ and $+1 \in D$ to $+\infty \times [0,1]$.}

Pointed disks with strip-like ends can be glued in a natural way.
Further, the space $\mathcal{R}^{k+1}$ has a natural compactification
$\overline{\mathcal{R}}^{k+1}$ described by parametrizing the elements
of $\overline{\mathcal{R}}^{k+1} \setminus \mathcal{R}^{k+1}$ by
trees~\cite{Se:book-fukaya-categ}. The family
$\mathcal{S}^{k+1} \to \mathcal{R}^{k+1}$ admits a partial
compactification
$\overline{\mathcal{S}}^{k+1} \to \overline{\mathcal{R}}^{k+1}$ which
can be endowed with a smooth structure. Moreover, the fixed choice of
universal strip-like ends for $\mathcal{S}^{k+1}\to \mathcal{R}^{k+1}$
admits an extension to
$\overline{\mathcal{S}}^{k+1}\to
\overline{\mathcal{R}}^{k+1}$. Further, these choices of universal
strip-like ends for the spaces $\mathcal{R}^{k+1}$ for different $k$'s
can be made in a way consistent with these compactifications
\pbred{(see~\cite[Sections~9d,~9e]{Se:book-fukaya-categ} and Lemma~9.3
  in that book).}

Our construction requires the additional auxiliary structure of {\em
  transition functions}. This structure can be defined once a choice
of universal strip-like ends is fixed.  It consists of a smooth
function $\mathbf{a}^{k+1}: \mathcal{S}^{k+1} \to [0,1]$ with the
following properties.  First let $k=1$. In this case
$\mathcal{S}^2 = D \setminus \{-1, 1\} \cong \mathbb{R} \times [0,1]$
\pbred{and we define $\mathbf{a}^2(s,t) = t$,} where
$(s,t)\in \R\times [0,1]$. To describe $\mathbf{a}^{k+1}$ for
$k \geq 2$ write $a_r := \mathbf{a}^{k+1}|_{S_r}$,
$r \in \mathcal{R}^{k+1}$. We require the functions $a_r$ to satisfy
the following for every $r \in \mathcal{R}^{k+1}$ - see
Figure~\ref{fig:transition}:
\begin{figure}[htbp]
   \begin{center}
      \includegraphics[scale=0.5]{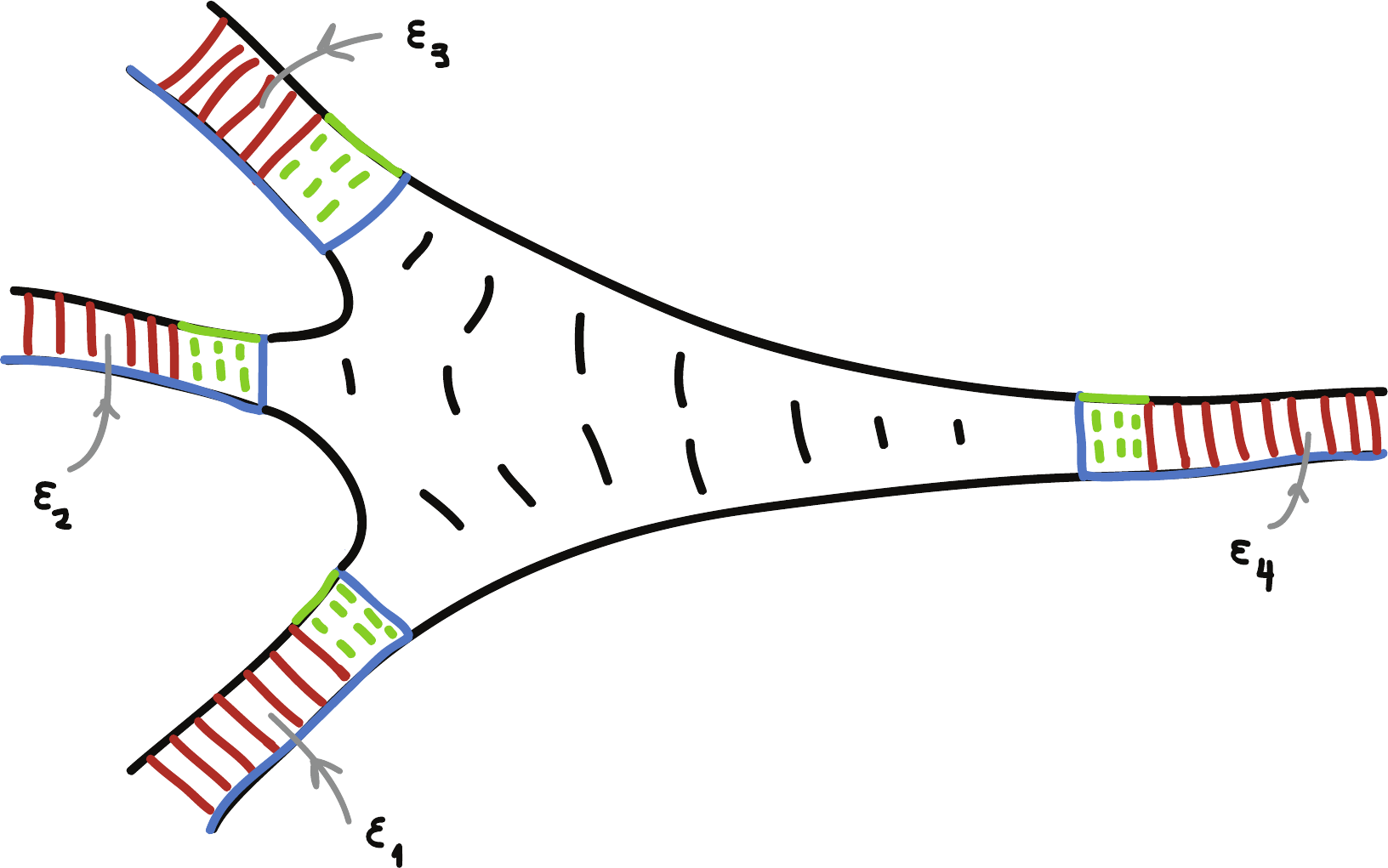}
   \end{center}
   \caption{\label{fig:transition} The constraints imposed on a
     transition function for a domain with three entries and one exit:
     in the red region the function $a$ equals $(s,t)\to t$; along the
     blue arcs the function $a$ vanishes; the green region is a
     transition region. There  are no additional constraints in the
     black region.}
\end{figure}

\begin{itemize}
  \item[i.] For each entry strip-like end $\epsilon_{i}:Z^{-}\to
   S_{r}$, $1\leq i\leq k$, we have:
   \begin{itemize}
     \item[a.]$a_{r} \circ \epsilon_{i}(s,t)=t$, $\forall \ (s,t)\in
      (-\infty, -1]\times [0,1]$.
     \item[b.] $\frac{\partial }{\partial
        s}(a_{r}\circ\epsilon_{i})(s,1)\leq 0$ for $s\in [-1,0]$.
     \item[c.] $a_{r}\circ \epsilon_{i}(s,t)=0$ for $(s,t)\in
      ((-\infty, 0]\times \{0\})\cup (\{0\}\times [0,1])$.
   \end{itemize}
  \item[ii.] For the exit strip-like end $\epsilon_{k+1}:Z^{+}\to
   S_{r}$ we have:
   \begin{itemize}
     \item[a'.]$a_{r} \circ \epsilon_{k+1}(s,t)=t$, $\forall \
      (s,t)\in [1, \infty)\times [0,1]$.
     \item[b'.] $\frac{\partial }{\partial
        s}(a_{r}\circ\epsilon_{k+1})(s,1)\geq 0$ for $s\in [0,1]$.
     \item[c'.] $a_{r}\circ \epsilon_{k+1}(s,t)=0$ for $(s,t)\in
      ([0,+\infty)\times \{0\}) \cup (\{0\}\times [0,1])$.
   \end{itemize}
 % \item[iii.] \CPB{There exists a number $\delta_{k+1} \in
  %   \mathbb{N}$, that depends only on $k$, such that for every $r \in
   %  \mathcal{R}^{k+1}$ the function $a_r|_{\partial S_r}: \partial
  %   S_r \to [0,1]$ changes sign at most $\delta_{k+1}$ times.}
\end{itemize} 
The total function $\mathbf{a}^{k+1}: \mathcal{S}^{k+1} \to [0,1]$
will be called a global transition function.  The functions
$\mathbf{a}^{k+1}$ can be picked consistently for different values of
$k$ in the sense that $\mathbf{a}$ extends smoothly to
$\overline{\mathcal{S}}^{k+1}$ and along the boundary $\partial
\overline{\mathcal{S}}^{k+1}$ it coincides with the corresponding
pairs of functions $\mathbf{a}^{k'+1}: \mathcal{S}^{k'+1} \to [0,1]$,
$\mathbf{a}^{k''+1}:\mathcal{S}^{k'+1} \to [0,1]$ with $k'+k'' = k+1$,
associated to trees of split pointed disks.

\subsubsection{Profile function} \label{subsubsec:profile} We now
discuss the second special ingredient in our construction: profile
functions.

\pbred{To fix ideas we suppose from now on in this construction that
\begin{equation}\label{eq:U-restriction}
  U\subset [-\frac{1}{2},\frac{1}{2}]\times [0,\infty).
\end{equation}
According to the notation in~\eqref{eq:quadr} and together
with~\eqref{eq:u-tame} this means that $a_{U}\leq \frac{1}{2}$. (The
real number $a_{U}$ from~\eqref{eq:quadr} should not be confused with
the functions $a_r$ from the preceding section.)}

We will use a {\em profile function}: $h:\R^{2}\to \R$ which, by
definition, has the following properties (see Figure~\ref{fig:kinks}):
\begin{itemize}
  \item[i.] The support of $h$ is contained in the union of the sets
   $$W_{i}^{+}= [2,\infty)\times [i-\epsilon, i+\epsilon] 
   \quad \textnormal{and } \; W_{i}^{-}= (-\infty, -1]\times
   [i-\epsilon, i+\epsilon], \; i \in \mathbb{Z}~,~$$ where $0<
   \epsilon<1/4$.
 \item[ii.] \pbred{The restriction of $h$ to each set
     $F_{i}^{+}=[2,\infty)\times [i-\epsilon/2, i+\epsilon/2]$ and
     $F_{i}^{-}=(-\infty, -1] \times [i-\epsilon/2, i+\epsilon/2]$ is}
   respectively of the form $h(x,y)=h_{\pm}(x)$, where the smooth
   functions $h_{\pm}$ satisfy:
   \begin{itemize}
     \item[a.] $h_{-}:(-\infty, -1]\to \R$ has a single critical point
      in $(-\infty,-1]$ at $-\frac{3}{2}$ and this point is a
      non-degenerate local maximum. Moreover, for all $x\in (-\infty,
      -2)$, we have $h_{-}(x)=\alpha^{-}x+\beta^{-}$ for some
      constants $\alpha^{-},\beta^{-}\in\R$ with $\alpha^->0$.
     \item[b.] $h_{+}:[2,\infty)\to \R$ has a single critical point in
      $[2,\infty)$ at $\frac{5}{2}$ and this point is also a
      non-degenerate maximum.  Moreover, for all $x\in (3,\infty)$ we
      have $h_{+}(x)=\alpha^{+}x+\beta^{+}$ for some constants
      $\alpha^{+},\beta^{+}\in\R$ with $\alpha^+ < 0$.
   \end{itemize}
 \item[iii.] The Hamiltonian isotopy
   $\phi_t^h: \mathbb{R}^2 \to \mathbb{R}^2$ associated to $h$ exists
   for all $t \in \mathbb{R}$; the derivatives of the functions
   $h_{\pm}$ are sufficiently small such that the Hamiltonian isotopy
   $\phi^{h}_{t}$ keeps the sets $[2,\infty)\times \{i\}$ and
   $(-\infty, -1]\times \{{i}\}$ inside \pbred{the respective
     $F_{i}^{\pm}$} for $-1\leq t\leq 1$.
  \item[iv.] The Hamiltonian isotopy $\phi_t^h$ preserves the strip
   $[-\tfrac{3}{2}, \tfrac{5}{2}] \times \mathbb{R}$ for all $t$, in
   other words $\phi_t^h \bigl([-\tfrac{3}{2}, \tfrac{5}{2}] \times
   \mathbb{R} \bigr) = [-\tfrac{3}{2}, \tfrac{5}{2}] \times
   \mathbb{R}$ for every $t$.
\end{itemize}
\begin{figure}[htbp]\vspace{0.0in}
   \begin{center}
     \includegraphics[scale=0.75]{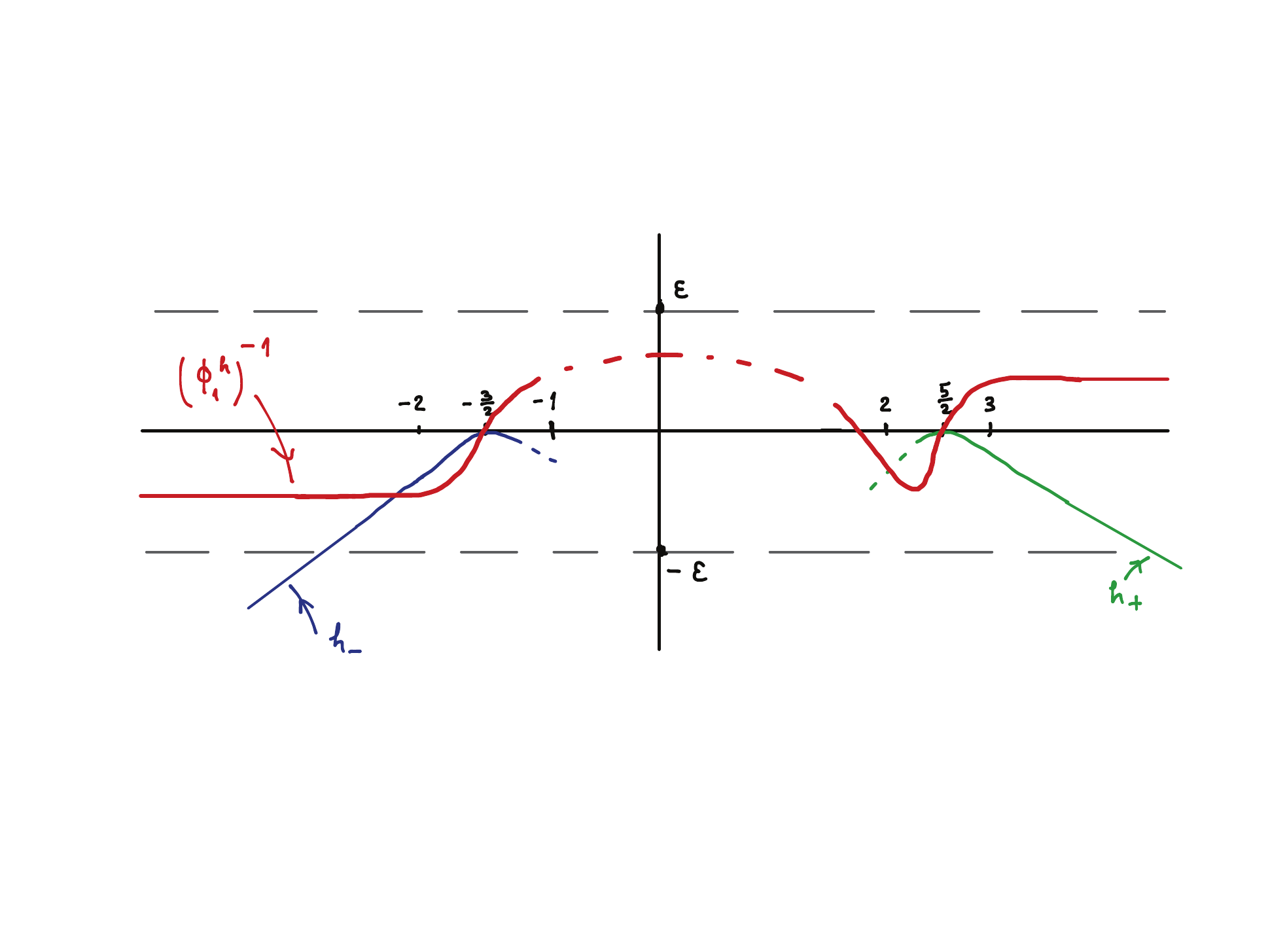}
   \end{center}
   \vspace{-1.5in}\caption{ \label{fig:kinks} The graphs of $h_{-}$
     and $h_{+}$ and the image of $\R$ by the Hamiltonian
     diffeomorphism $(\phi^{h}_{1})^{-1}$. The profile of the
     functions $h_{-}$ at $-3/2$ and $h_{+}$ at $5/2$ are the
     ``bottlenecks''.}
\end{figure}
Such functions $h$ are easy to construct. Their main role is to
disjoin the ends corresponding to two (or more) cobordisms at
$\pm\infty$. The critical points $(-3/2, i)$ and $(5/2,i)$ are called
{\em bottlenecks}.

\subsubsection{Perturbation data, $J$-holomorphic polygons and
  $\mu_{k}$} \label{subsubsec:perturb}
At this step we describe the (perturbed) $J$-holomorphic polygons that
define the $\mu_{k}$'s.

The construction of $\mu_{k}$ starts with $\mu_{1}$ and the so-called
Floer datum. For each pair of cobordisms $V,V'\subset E$ the Floer
datum $\mathscr{D}_{V,V'} = (\bar{H}_{V, V'}, J_{V, V'})$ consists of
a Hamiltonian $\bar{H}_{V,V'}:[0,1]\times E \to \R$ and a (possibly
time dependent) almost complex structure $J_{V,V'}$ on $E$ which is
compatible with $\Omega_{E}$.  We will also assume that each Floer
datum $(\bar{H}_{V,V'},J_{V,V'})$ satisfies the following conditions:
\begin{itemize}
  \item[i.] $\phi^{\bar{H}_{V,V'}}_{1}(V)$ is transverse to $V'$.
  \item[ii.] Write points of $E \setminus \pi^{-1}(U)$ as $(x,y,p)$
    with $x+iy \in \mathbb{C}$, $p \in M$. We require that there
    exists a compact set
    $K_{V,V'}\subset (-\frac{5}{4},\frac{9}{4})\times \R \subset \C$
    such that $\bar{H}_{V,V'}(t,(x,y,p))= h(x,y)+H_{V,V'}(t,p)$ for
    $(x+iy, p)$ outside of $\pi^{-1}(K_{V,V'})$, for some
    $H_{V,V'}: [0,1]\times M\to \R$.
  \item[iii.] The projection $\pi: E \to \mathbb{C}$ is $(J_{V,V'}(t),
   (\phi_t^h)_* i)$-holomorphic outside of $\pi^{-1}(K_{V,V'})$ for
   every $t \in [0,1]$. \label{pg:pi-hol-1}
\end{itemize}

\begin{rem}\label{rem:J-perturb}
  The almost complex structure $J_{V,V'}$ can be viewed \pbred{in some
    sense} as a perturbation of the almost complex structure $J_{E}$
  that is part of the Lefschetz fibration structure as in
  Definition~\ref{df:lef-fib}.  \pbred{Indeed, if the profile function
    $h$ is taken to be arbitrarily small then $J_{V,V'}$ can be chosen
    to be arbitrarily close to $J_E$. In practice we will not take
    this viewpoint and will not insist that $J_{V,V'}$ is a good
    approximation of $J_E$.}  
%\CCPB{I suggest to remove the whole of
%    Remark 3.2.2. It's not used anywhere anyway.}
\end{rem}

The time-$1$ Hamiltonian chords $\mathcal{P}_{\bar{H}_{V,V'}}$ of
$\bar{H}_{V,V'}$ that start on $V$ and end on $V'$, form a finite set.

For a $(k+1)$-pointed disk $S_{r}$, let $C_{i}\subset
\partial S_{r}$ be the connected components of $\partial S_{r}$
indexed so that $C_{1}$ goes from the exit to the first entry, $C_{i}$
goes from the $(i-1)$-th entry to the $i$, $1\leq i\leq k$, and
$C_{k+1}$ goes from the $k$-th entry to the exit.

Following Seidel's scheme from~\cite[Section~9]{Se:book-fukaya-categ},
we now need to choose additional perturbation data.

For every collection of cobordisms $V_{i}$, $1\leq i\leq k+1$ we
choose a perturbation datum
$\mathscr{D}_{V_1, \ldots, V_{k+1}} = (\form, \mathbf{J})$ consisting
of:
\begin{enumerate}
\item[I.] A family $\form = \{ \form^r \}_{r \in \mathcal{R}^{k+1}}$,
  where $\form^r \in \Omega^{1}(S_{r}, C^{\infty}(E))$ is a $1$-form
  on $S_r$ with values in smooth functions on $E$.  We write
  $\form^r(\xi): E \to \mathbb{R}$ for the value of $\form^r$ on
  $\xi \in T S_r$.
\item[II.] $\mathbf{J} = \{J_z \}_{z \in \mathcal{S}^{k+1}}$ is a
  family of $\Omega_{E}$-compatible almost complex structure on $E$,
  parametrized by $z \in S_r$, $r \in \mathcal{R}^{k+1}$.
\end{enumerate}
The forms $\form^r$ induce forms
$Y^r=Y^{\form^r} \in \Omega^{1}(S_{r}, C^{\infty}(TE))$ with values in
(Hamiltonian) vector fields on $E$ via the relation
$Y(\xi)=X^{\form(\xi)}$ for each $\xi\in T S_{r}$ (i.e. $Y(\xi)$ is
the Hamiltonian vector field on $E$ associated to the autonomous
Hamiltonian function $\form(\xi): E \to \mathbb{R}$).

The relevant Cauchy-Riemann equation associated to
$\mathscr{D}_{V_1, \ldots, V_{k+1}}$ is:
\begin{equation}\label{eq:jhol1} \ u:S_{r}\to E, \quad Du +
   J(z,u)\circ Du\circ j = Y + J(z,u)\circ Y\circ j, \quad 
   u(C_{i})\subset V_{i} ~.~
\end{equation}
Here $j$ stands for the complex structure on $S_r$. The $i$-th entry
of $S_r$ is labeled by a time$-1$ Hamiltonian orbit
$\gamma_{i}\in \mathcal{P}_{\bar{H}_{V_{i},V_{i+1}}}$ and the exit is
labeled by a time$-1$ Hamiltonian orbit
$\gamma_{k+1}\in \mathcal{P}_{\bar{H}_{V_{1}, V_{k+1}}}$. The map $u$
satisfies $u(C_{i})\subset V_{i}$ and $u$ is required to be asymptotic
- in the usual Floer sense - to the Hamiltonian orbits $\gamma_{i}$ on
each respective strip-like
end. \pbred{See~\cite[Section~8f]{Se:book-fukaya-categ} for more
  details on this equation, the boundary conditions and the
  asymptotics.}

The perturbation data $\mathscr{D}_{V_1, \ldots, V_{k+1}}$ are
constrained by a number of additional conditions that we now describe.
First, denote by $s_{V_{1}, \ldots, V_{k+1}}\in \mathbb{N}$ the
smallest $l \in \mathbb{N}$ such that
$\pi(V_1 \bigcup \cdots \bigcup V_{k+1})\subset \R\times (0,l)$. Write
$\bar{h} = h \circ \pi: E\to \mathbb{R}$, where
$h: \mathbb{R}^2 \to \mathbb{R}$ is the profile function fixed before.
We also write
\begin{equation*}
   \begin{aligned}
      & U^r_i = \epsilon_i^{S_r} \bigl( (-\infty, -1] \times [0,1]
      \bigr) \subset S_r, \quad i = 1, \ldots, k, \\
      & U^r_{k+1} = \epsilon_{k+1}^{S_r} \bigl( [1, \infty) \times [0,1]
      \bigr) \subset S_r, \\
      & \mathcal{W}^r = \bigcup_{i=1}^{k+1} U^r_i.
   \end{aligned}
\end{equation*}

The conditions on $\mathscr{D}_{V_1, \ldots, V_{k+1}}$ are the
following: \label{pg:conditions-abc}
\begin{itemize}
  \item[a.] {\em Asymptotic conditions.} For every $r \in
   \mathcal{R}^{k+1}$ we have $\form|_{U_i^r} = \bar{H}_{V_{i},
     V_{i+1}} dt$, $i=1, \ldots, k$ and $\form|_{U_{k+1}^r} =
   \bar{H}_{V_{1}, V_{k+1}} dt$. (Here $(s,t)$ are the coordinates
   parametrizing the strip-like ends.) Moreover, on each $U_i^r$,
   $i=1, \ldots, k$, $J_z$ coincides with $J_{V_{i},V_{i+1}}$ and on
   $U_{k+1}^r$ it coincides with $J_{V_{1}, V_{k+1}}$, i.e.
   $J_{\epsilon^{\mathcal{S}_r}_i(s,t)} = J_{V_i, V_{i+1}}(t)$ and
   similarly for the exit end. Thus, over the part of the strip-like
   ends $\mathcal{W}^r$ the perturbation datum $\mathscr{D}_{V_1,
     \ldots, V_{k+1}}$ is compatible with the Floer data
   $\mathscr{D}_{V_i, V_{i+1}}$, $i=1, \ldots, k$ and
   $\mathscr{D}_{V_1, V_{k+1}}$.
  \item[b.]{\em Special expression for $\form$.} The
   restriction of $\form$ to $S_r$ equals $$\form|_{S_r} = da_r \otimes
   \bar{h} +\form_0$$ for some $\form_0 \in \Omega^{1}(S_{r},
   C^{\infty}(E))$ which depends smoothly on $r \in
   \mathcal{R}^{k+1}$. Here $a_r:S_r \to \mathbb{R}$ are the
   transition functions  fixes at the point 1. The form
   $\form_0$ is required to satisfy the following two conditions:
   \begin{itemize}
     \item[1.] $\form_0(\xi)=0$ for all $\xi\in TC_{i}\subset
      T\partial S_{r}$.
    \item[2.]  There exists a compact set
      $K_{V_{1},\ldots , V_{k+1}}\subset
      (-\frac{3}{2},\frac{5}{2})\times \R$ which is independent of
      $r \in \mathcal{R}^{k+1}$ such that
      $\pi^{-1}(K_{V_{1},\ldots , V_{k+1}})$ contains all the sets
      $K_{V_{i}, V_{j}}$ involved in the Floer datum
      $\mathscr{D}_{V_i, V_j}$, and with
      $$K_{V_{1},\ldots, V_{k+1}}\supset\ ([-\frac{5}{4},
      \frac{9}{4}]\times [-s_{V_{1},\ldots, V_{k+1}},+s_{V_{1},\ldots,
        V_{k+1}}])$$ such that outside of
      $\pi^{-1}(K_{V_{1},\ldots , V_{k+1}})$ we have $D \pi (Y_0)=0$
      for every $r$, where $Y_{0}=X^{\form_0}$.
   \end{itemize}
  \item[c.] Outside of $\pi^{-1}(K_{V_{1},\ldots , V_{k+1}})$
   the almost complex structure $\mathbf{J}$ has the property that the
   projection $\pi$ is
   $(J_z,(\phi_{a_r(z)}^{h})_{\ast}(i))$-holomorphic for every $r \in
   \mathcal{R}^{k+1}$, $z \in S_r$.  \label{pg:pi-hol-2}
\end{itemize}

Using the above choices of data we construct the $A_{\infty}$-category
$\fuk^{\ast}(E)$ by the construction
\pbred{from~\cite[Section~9]{Se:book-fukaya-categ} with the
  modifications described in~\cite{Bi-Co:lcob-fuk} that are needed due
  to the fact that the Lagrangians are not compact.} As mentioned
before, the objects of this category are Lagrangians cobordisms
$V\subset E$ without positive ends that are uniformly monotone of
class $\ast$, the morphisms space between the objects $V$ and $V'$ are
$CF(V, V'; \mathscr{D}_{V,V'})$, the $\mathcal{A}$-vector space
generated by the Hamiltonian chords $\mathcal{P}_{\bar{H}_{V,V'}}$.
The $A_{\infty}$ structural maps
$$\mu_k :CF(V_{1},V_{2})\otimes CF(V_{2},V_{3})\otimes \ldots \otimes
CF(V_{k},V_{k+1})\to CF(V_{1}, V_{k+1})$$ are defined by summing -
with coefficients in $\mathcal{A}$ - pairs $(r,u)$ with
$r \in \mathcal{R}^{k+1}$ and $u$ a finite energy solution
of~\eqref{eq:jhol1} that belongs to a $0$-dimensional moduli space.
The coefficient in front of a perturbed $J$-holomorphic polygon $u$ is
$T^{\omega(u)}$.  The Gromov compactness and regularity arguments work
just as in~\cite{Bi-Co:lcob-fuk}. \pbred{(The fact that in that paper
  the the total space was $E = \mathbb{C} \times M$ whereas here $E$
  is a Lefschetz fibration plays no role in these arguments.)} In
fact, as we work here over the universal Novikov ring compactness is
easier to establish in this case (and we do not require the vanishing
of the inclusions $\pi_{1}(V)\to \pi_{1}(E)$ as
in~\cite{Bi-Co:lcob-fuk}).

The choice of strip-like ends, transition functions and profile
function (in particular, the placement of the bottlenecks) changes the
resulting $A_{\infty}$-category only up to quasi-equivalence.

Once the category $\fuk^{\ast}(E)$ is constructed the derived category
$D\fuk^{\ast}(E)$ is defined by again considering the
$A_{\infty}$-modules $mod(\fuk^{\ast}(M)):=fun(\fuk^{\ast}(E),
Ch^{opp})$ and by letting $D\fuk^{\ast}(E)$ be the homological
  category associated to the triangulated closure of the image of the
Yoneda functor $\mathcal{Y}:\fuk^{\ast}(E)\to mod(\fuk^{\ast}(E))$.

\subsubsection{The naturality transformation} \label{sbsb:nat-transf}
Assume that $u:S_{r}\to E$ is a solution of~\eqref{eq:jhol1},
\pbred{where the Floer and perturbation data satisfy the conditions
  discussed at the points~\emph{a, b, c} on
  page~\pageref{pg:conditions-abc}.} Define $v:S_r \to E$ by the
formula:
\begin{equation} \label{eq:nat-v} u(z) =
   \phi_{a_r(z)}^{\bar{h}}(v(z)),
\end{equation}
where $a_r:S_r \to [0,1]$ is the transition function. 

The Floer equation~\eqref{eq:jhol1} for $u$ transforms into the
following equation for $v$:
\begin{equation}\label{eq:jhol2}
   Dv + J'(z,v)\circ Dv\circ j = Y' + J'(z,v)\circ Y'\circ j.
\end{equation}
Here $Y'\in \Omega^{1}(S_{r}, C^{\infty}(TM))$ and $J'$ are defined
by:
\begin{equation} \label{eq:nat-Y'-J'} Y=D
   \phi_{a(z)}^{\bar{h}}(Y')+da_r \otimes X^{\bar{h}}, \quad
   J_z=(\phi_{a_r(z)}^{\bar{h}})_{\ast} J'_z.
\end{equation}
The map $v$ satisfies the following moving boundary conditions:
\begin{equation}\label{eq:mov-bdry} \forall \ z \in C_{i}, \quad 
   v(z)\in (\phi_{a(z)}^{\bar{h}})^{-1}(V_{i}).
\end{equation}
   
The asymptotic conditions for $v$ at the punctures of $S_r$ are as
follows. For $i=1, \ldots, k$, $v(\epsilon_i(s,t))$ tends as $s \to
-\infty$ to a time-$1$ chord of the flow $(\phi_t^{\bar{h}})^{-1}
\circ \phi_t^{\bar{H}_{V_i, V_{i+1}}}$ starting on $V_i$ and ending on
$(\phi_1^{\bar{h}})^{-1}(V_{i+1})$.  (Here $\epsilon_i(s,t)$ is the
parametrization of the strip-like end at the $i$'th puncture.)
Similarly, $v(\epsilon_{k+1}(s,t))$ tends as $s \to \infty$ to a chord
of $(\phi_t^{\bar{h}})^{-1} \circ \phi_t^{\bar{H}_{V_1, V_{k+1}}}$
starting on $V_1$ and ending on $(\phi_1^{\bar{h}})^{-1}(V_{k+1})$.

\pbred{It might be useful to spell out more geometrically the effect
  of the moving boundary conditions~\eqref{eq:mov-bdry} on the ends of
  the Lagrangians $V_i$. Identify a neighborhood of puncture number
  $i$, $1\leq i\leq k$, in $S_r$ with
  $Z^{-} = (-\infty, 0] \times [0,1]$ via the strip-like ends
  construction as in~\S\ref{subsubsec:transition}. Then for every
  $x \in (-\infty, 0]$, we have $v(x,0) \in V_i$ and
  $v(x,1) \in (\phi_{\alpha(x)}^{\bar{h}})^{-1}(V_{i+1})$, where
  $\alpha:(-\infty, 0] \to [0,1]$ is a function that equals $1$ on
  $(-\infty, -1]$ and on the interval $[-1,0]$ it decreases from $1$
  to $0$. Note that the part of
  $(\phi_{\alpha(x)}^{\bar{h}})^{-1}(V_{i+1})$ that lies over
  $(-\infty, -2] \times \mathbb{R}$ is just
  $(\phi_1^{\bar{h}})^{-1}(\textnormal{ends of }V_{i+1})$ hence
  coincides with the ends of $V_{i+1}$ after being pushed downwards
  (in the $y$-direction of the $\mathbb{C}$-factor) by a small
  amount. See the left-hand side of Figure~\ref{fig:kinks}.  Note also
  that for each $s \in \mathbb{N}$ such that both $V_i$ and $V_{i+1}$
  have an $s$-end, i.e. an end over $(-\infty, -a_U]\times \{s\}$, the
  following happens: the projections
  $\pi(s \textnormal{-end of } V_i)$ and
  $\pi\Bigl((\phi_{\alpha(x)}^{\bar{h}})^{-1}(s \textnormal{-end of }
  V_{i+1})\Bigr)$ intersect transversly at the points
  $(-\tfrac{3}{2},s)$. See again Figure~\ref{fig:kinks}. A similar
  description holds also for the exit strip-like end $Z^{+}$.}

Let now $v' = \pi \circ v: S_r \to \mathbb{C}$.  It is then easy to
see - \pbred{as in~\cite[Page~1766]{Bi-Co:lcob-fuk}} - that $v'$ is
holomorphic over
$\mathbb{C} \setminus ([-\tfrac{3}{2} + \delta', \tfrac{5}{2}-\delta']
\times \R)$ for small enough $\delta'>0$.

As discussed in~\cite{Bi-Co:lcob-fuk}, there are many useful
consequences of the holomorphicity of $v'$ around a bottleneck and we
will see some more later in this paper. To give a typical simple
example, assume that the bottleneck in question is
$a=(-\frac{3}{2}, 0)$ and that the regions $A$ and $B$ in
Figure~\ref{fig:bottleneck-quadr} are unbounded. In this case, the
image of $v'$ can not switch from region $D$ to region $C$ (or
vice-versa). More precisely, it is impossible to have that
$Image(v')\cap C\not=\emptyset$ and $Image(v')\cap D\not=\emptyset$
with the regions $C,D$ as in the picture.

\begin{figure}[htbp]\vspace{0in}
   \begin{center}
      \includegraphics[scale=0.5]{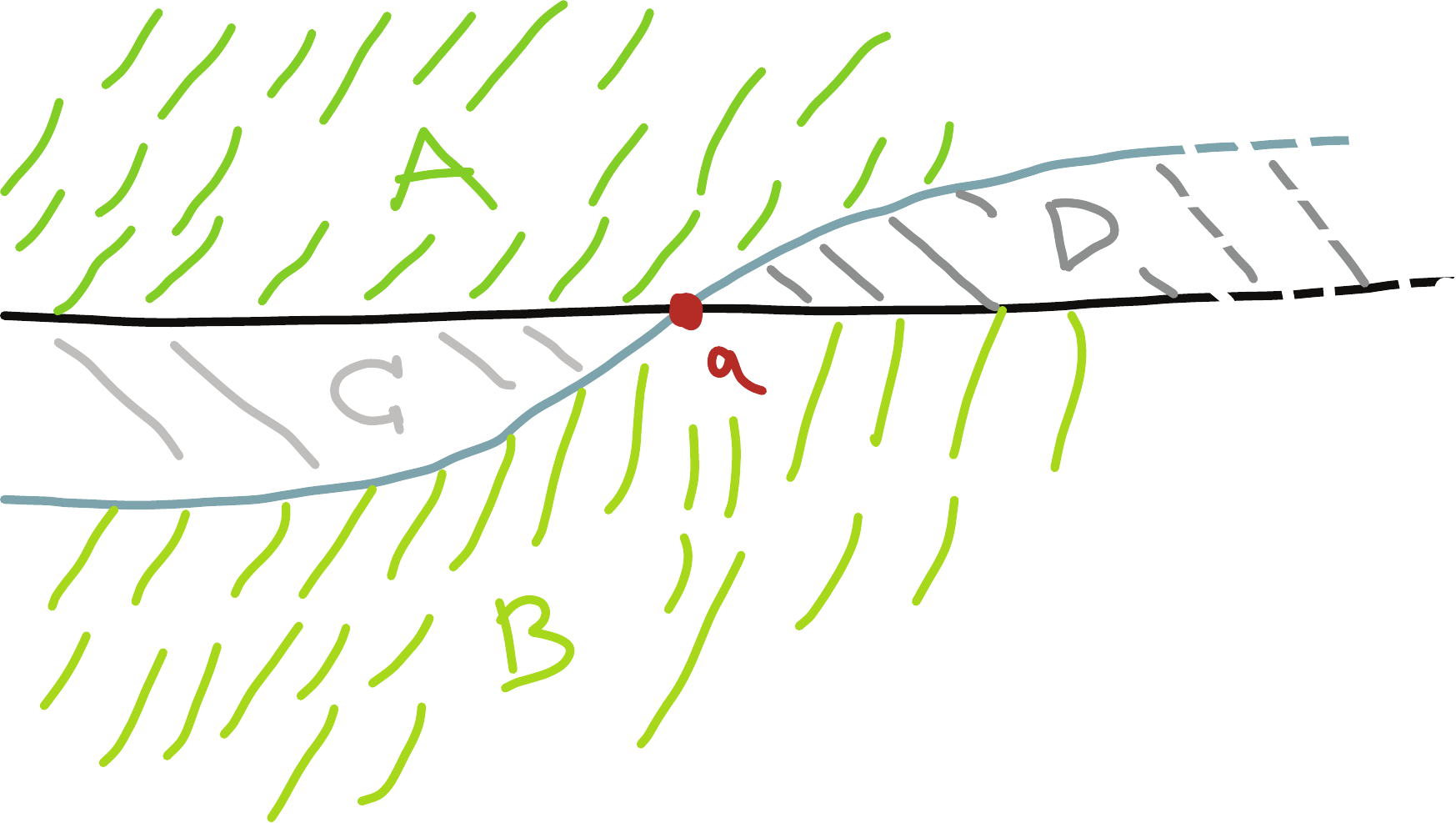}
   \end{center}
   \vspace{0in}\caption{ \label{fig:bottleneck-quadr} 
   The bottleneck $a$ and the regions $A$, $B$, $C$ and  $D$.}
\end{figure}

The argument is as follows: assume that $Image(v')$ intersects both
$C$ and $D$ and is disjoint from the interiors of both $A$ and $B$.
Let $x_{1}\in Image(v')\cap C$ and $x_{2}\in Image(v')\cap D$. Let $c$
be a curve inside the domain of $v'$ that connects $x_{1}$ to $x_{2}$.
It follows that $a\in v'(c)$. But as there are infinitely many
distinct curves $c$ joining $x_{1}$ to $x_{2}$ this means that there
are infinitely many interior points $z$ with $v'(z)=a$. But this
implies $Image(v')=a$. Thus $Image(v')$ has to intersect at least one
of $A$ and $B$ and, by the open mapping theorem, this contradicts the
fact that the closure of $Image(v')$ is compact.

This argument is used \pbred{in several instances
  in~\cite{Bi-Co:lcob-fuk}, for example to show the compactness of the
  moduli spaces required to define $\mu_{k}$ as well as those used to
  show $\mu\circ\mu=0$.}
 
Besides this compactness implication, the holomorphicity of $v'$ has
an important role in the proof of the main decomposition result in
\cite{Bi-Co:lcob-fuk} as well as in the main result of the current
paper.  Both these results are consequences of writing certain
$A_{\infty}$-module structures $\mu_{k}$ in an ``upper triangular''
form. In turn, this form is deduced from the fact that the planar
projections of the $J$-holomorphic polygons giving the module
multiplications are holomorphic (over an appropriate region in $\C$)
and a ``bottleneck-type'' argument is used repeatedly to show the
vanishing of the relevant components of the $\mu_{k}$'s. \pbred{See
  for example~\cite[Sections~4.2,~4.4]{Bi-Co:lcob-fuk}.}

\subsubsection{The case of a non-compact
  fibre}\label{subsubsec:non-comp-J}
We now assume that $(M,\omega)$ is non-compact and convex at infinity
{\pbred and that the Lefschetz fibration $E$ satisfies the conditions
  in~\S\ref{sb:defs-lef-fibr}} \pbred{as well as the
  Assumption~$T_{\infty}$ from
  page~\pageref{pg:assumption-T-infty}. Additionally, we continue to
  assume that $E$ is tame outside a $U$-shaped subset
  $U \subset \mathbb{C}$ as in~\S\ref{subsubsec:Fuk-cob}.} 

From Assumption $T_{\infty}$ we deduce that there is a trivialization
$\phi: \C\times M^{\infty}\to E^{\infty}$ with respect to which both
the symplectic form and the almost complex structure split so that, in
particular, $\phi^{\ast} J_{E}=j\oplus J_{0}$ where $J_{0}$ is a fixed
almost complex structure on $M$ compatible with $\omega$ and with the
symplectic convexity of $M$.  Recall also that
$E^{0}=E\setminus E^{\infty}$.

The objects of the category $\fuk^{\ast}(E)$ are the same as before.
Notice that, by Definition \ref{def:Lcobordism}, any cobordism $V$ has
the property that $V\cap \pi^{-1}(z)$ is compact for any $z\in \C$.
Furthermore, all the construction of the category $\fuk^{\ast}(E)$
proceeds exactly in the same fashion as in the compact case with an
additional requirement: all the almost complex structures involved are
required to coincide with $J_{E}$ outside a large enough neighborhood
of $E^{0}$.  More precisely, for any two objects $V,V'\in
\mathcal{O}b(\fuk^{\ast}(E))$ we require that $J_{V,V'}$ coincide with
$J_{E}$ outside a neighborhood of $E^{0}$ that contains both $V$ and
$V'$.  Similarly, each almost complex structure $J_{z}$ in the family
$\mathbf{J}$ that is part of the perturbation data associated to the
collection of cobordisms $V_{1},\ldots, V_{k+1}$ has to coincide with
$J_{E}$ outside of a neighborhood of $E^{0}$ that contains all of the
$V_{i}$'s.
 
Finally, notice that as explained in \S\ref{sbsb:nat-transf} the
actual curves $u$ that appear in the $\mu_{k}$'s are transformed into
curves $v$ \pbred{which} satisfy equations that are holomorphic with
respect to almost complex structures of the form
$J'_{z}=(\phi^{\overline{h}}_{a_{r}(z)})_{\ast}^{-1} J_{z}$. Due to
the splitting provided by the trivialization $\phi$ and because
$\overline{h}=h\circ \pi$ these structures are also split at $\infty$
(along the fibre) and, by using the trivialization $\phi$, it follows
that $J'_{z}$ restricted to the fiber direction coincides with $J_{0}$
(away from a compact \pbred{subset}). Therefore, over $E^{\infty}$ one
can again use $\phi$ to project such a curve $v$ on $M^{\infty}$ thus
getting a new curve $v'$ that way from a compact is
$J_{0}$-holomorphic.  The usual compactness arguments for manifolds
that are symplectically convex at infinity apply to this $v'$ and thus
compactness is achieved without issues.

\begin{rem}
  In~\cite{Se:book-fukaya-categ} (see also~\cite{Se:Lefschetz-Fukaya})
  Seidel introduced a Fukaya category associated to a Lefschetz
  fibration $\pi :E\to \C$.  By neglecting for a moment some technical
  points that will be revisited below, the relation between this
  category and the category $\fuk^{\ast}(E)$ introduced above is that
  Seidel's category is quasi-equivalent to the subcategory of
  $\fuk^{\ast}(E)$ with objects the thimbles $T_{i}$ covering the
  curves $t_{i}$ in Figure~\ref{fig:gamma-thimbles0}. The technical
  points are that, firstly, we work in a monotone and ungraded setting
  and Seidel's work is in the exact and graded case (and the grading
  plays an important role in his work).  Secondly, the type of
  perturbations at infinity that Seidel uses - see in particular
  \cite{Se:Lefschetz-Fukaya} - are different from ours. Despite these
  differences, \pbred{it is possible to show that Seidel's approach
    can also be implemented in the monotone case and the resulting
    category is quasi-equivalent to the subcategory of
    $\fuk^{\ast}(E)$ as mentioned above. One reason for not pursuing
    this direction in this paper is that in the construction of
    $\fuk^{\ast}(E)$ above we use the perturbations employing
    bottlenecks etc. These are very convenient if one uses the
    naturality transformation - as explained
    in~\S\ref{sbsb:nat-transf} - to reduce key steps of the proofs in
    this paper (as well as in~\cite{Bi-Co:lcob-fuk}) to properties of
    holomorphic planar curves.}  
%\CCPB{I think we should ommit this
%    remark, at least in the shortened version.}
\end{rem}

\subsection{Fukaya categories of negative ended cobordisms in general
  Lefschetz fibrations} \label{subsubsec:fuk-cob-gen} In this section
we use the construction in \S\ref{subsubsec:Fuk-cob} to associate a
Fukaya $A_{\infty}$-category to a general Lefschetz fibration.  Let
$\pi :E\to \C$ be a Lefschetz fibration as in
\S\ref{sb:defs-lef-fibr}.  The category we intend to construct will
depend on a tame Lefschetz fibration $\pi :E_{\tau}\to \C$ associated
to $E$ and will be denoted by $\fuk^{\ast}(E;\tau)$.  The parameter
$\tau$ indicates the choice of a tame symplectic structure on $E$ with
the properties described in the construction below.

\

We first fix an additional notation. For two constants $r<0<s$, put
$S_{r,s}=[r, s]\times \R\subset \C$.  Fix constants $x<0 <y$ such that all the
singularities of the fibration $E$ are contained in the interior of
$\pi^{-1}(S_{x,y})$. We also assume that the critical values of $\pi$ are included
in the upper half plane.

The construction is now the following.  The objects of the category
$\fuk^{\ast}(E;\tau)$ are cobordisms $V$ in $E$ - in the sense of
Definition \ref{def:cyl-ge} - that are cylindrical outside $S_{x-3,y+3}$
and satisfy the following additional constraints:
\begin{itemize}
\item[i.] $V$ is monotone of class $\ast$.
  \item[ii.] $V\subset \pi^{-1}(\R\times [\frac{1}{2},+\infty))$
  \item[iii.] $V$ has only negative ends belonging to
   $\mathcal{L}^{\ast}(M)$.
  
\end{itemize}
Condition {\em iii} means in this case that for some point $z$ along one of
the rays $\ell_{i}$ associated to the ends of $V$ we have that the
Lagrangian $V\cap \pi^{-1}(z)$ belongs to $\mathcal{L}^{\ast}(M)$. For
a fixed ray $\ell_{i}$ it is easy to see that this condition does not
depend on the choice of the point $z$.

To define the morphisms and the operations $\mu_{k}$ we proceed as
follows.  We fix a Lefschetz fibration $\pi : E_{\tau}\to \C$ that is
tame outside a set $U$ whose interior  contains $[x-4,y+4]\times (-1,\infty)$ and
coincides with $E$ over $[x-4,y+4]\times [-\frac{1}{2}, \infty)$. Such
a fibration exists due to the results from \S\ref{sb:tame-vs-gnrl}.
Recall from \S\ref{subsubsec:Fuk-cob} the construction of the category
$\fuk^{\ast}(E_{\tau})$. Each object $V\in
\mathcal{O}b(\fuk^{\ast}(E;\tau))$ corresponds to an object
$\overline{V}\in \mathcal{O}b(\fuk^{\ast}(E_{\tau}))$ that is
obtained, as in Remark \ref{rem:gen-tame-cob}, by cutting off the ends
of $V$ along the line $\{x-4\}\times \R\subset \C$ and extending them
horizontally by parallel transport in the fibration $E_{\tau}$. It is
easy to see that the subcategory of $\fuk^{\ast}(E_{\tau})$ that
consists of all the objects $\overline{V}$ obtained in this way is
quasi-equivalent to $\fuk^{\ast}(E_{\tau})$ itself because each object
of this larger category is quasi-isomorphic to one of the
$\overline{V}$'s.  Notice however that the category $\fuk^{\ast}(E_{\tau})$ contains
more objects than those of the form $\overline{V}$, an example is provided in
Figure \ref{fig:u-and-x}. We now put $\mor_{\fuk^{\ast}(E;\tau)}(V,V')=
\mor_{\fuk^{\ast}(E_{\tau})}(\overline{V},\overline{V}')$ and
similarly we define all operations in $\fuk^{\ast}(E;\tau)$ associated
to $V_{1},\ldots, V_{k+1}$ by means of the corresponding operations
associated to $\overline{V}_{1},\ldots, \overline{V}_{k+1}$ in
$\fuk^{\ast}(E_{\tau})$.

It is clear, by construction, that there is an inclusion:
$$ \fuk^{\ast}(E;\tau)\to \fuk^{\ast}(E_{\tau})$$
which is a quasi-equivalence.

The $A_{\infty}$-category in the statement of Theorem
\ref{thm:main-dec-gen0} can be taken to be any of the categories
$\fuk^{\ast}(E;\tau)$ described above. We will see later in the paper
that the derived category $D\fuk^{\ast}(E;\tau)$ is independent of
$\tau$ up to equivalence. Therefore, the omission of $\tau$ in the
statement of Theorem \ref{thm:main-dec-gen0} is justified.

\begin{rem}\label{rem:indep-tau}
  We believe that any two $A_{\infty}$-categories
  $\fuk^{\ast}(E;\tau)$ and $\fuk^{\ast}(E;\tau')$ are
  quasi-equivalent. Indeed, we expect that our construction of the
  Fukaya category of a tame fibration adapts to the case of a general
  Lefschetz fibration and the resulting fibration $\fuk^{\ast}(E)$ is
  expected to be quasi-equivalent to $\fuk^{\ast}(E;\tau)$ for all
  $\tau$. The technical ingredients required in the definition of
  $\fuk^{\ast}(E)$ go beyond the construction in the tame case so that
  we prefer not to further explore this issue here. In a different
  direction, we also expect that there is a derived Fukaya category of
  cobordisms with ends of arbitrary heights in $\R^{+}$ and not only
  with integral heights, as described in this paper.  First, given any
  infinite sequence of strictly increasing positive reals
  $S=\{a_{1}, \ldots, a_{n},\ldots\}$ there is a Fukaya category of
  cobordisms with ends in $S$ that is defined just as in the case of
  $S=\N^{\ast}$. The sets $S$ are ordered by inclusion in an obvious
  way and this order implies the existence of comparison maps among
  the corresponding categories.  The category in question is expected
  to be defined as an appropriate limit over $S$. Again, we do not
  pursue this construction here as it is not significant for the
  purpose of this paper.
\end{rem}

% !TEX root = lefcob.tex

\section{Decomposing cobordisms} \label{sec:main}

Fix a Lefschetz fibration $\pi : E\to \C$ and a Fukaya category
$\fuk^{\ast}(E; \tau)$ as defined in~\S\ref{subsubsec:fuk-cob-gen}. This
section contains the main result of the paper. It claims that each
object $V$ of $D\fk^{\ast}(E;\tau)$ admits an iterated cone decomposition
in terms of simpler objects. We will also see later in the paper that $D\fk^{\ast}(E;\tau)$
is independent of $\tau$.

\subsection{Statement of the main result} \label{subsec:main-decomp}

We will restate here Theorem \ref{thm:main-dec-gen0} after providing
the precise definitions of the objects involved.

To fix ideas, we assume that $\pi$ has $m$ critical points $x_{k}\in
E$, $k=1,\ldots, m$ of corresponding critical values $v_{k}=(k,
\frac{3}{2})\in \C$.  Consider a Fukaya category $\fuk^{\ast}(E;\tau)$ of
uniformly monotone negative ended cobordisms $V\subset E$ that are
cylindrical outside $\pi^{-1}(S_{x-3,y+3})$ with $x<0<y$ and
so that all the singularities of $\pi$ are contained in
$\pi^{-1}(S_{x,y})$. See \S\ref{subsubsec:fuk-cob-gen} for the
definition. In particular, $\tau$ indicates that the morphisms and operations in $\fuk^{\ast}(E;\tau)$
are defined by means of the Fukaya $A_{\infty}$-category $\fuk^{\ast}(E_{\tau})$ associated to a tame Lefschetz fibration $\pi:E_{\tau}\to \C$ that agrees with $E$ over
$[x-4,y+4]\times [-\frac{1}{2}, \infty)$.

The objects of $\fuk^{\ast}(E; \tau)$ are collected
in the set $\mathcal{L}^{\ast}(E)$. 

\subsubsection{The ``atoms'' of the decomposition }
\label{subsubsec:atoms}

Our first task is to describe the simpler objects that form the basic
pieces of our decomposition.

We will make use of two types of smooth curves in the plane.
\begin{enumerate}
  \item[(I)] These curves are denoted by $\gamma_{i}$, $i\geq 2$ and
   are so that $\gamma_{i}:\R\to \C$ is a smooth embedding with
   $$\gamma_{i}(\R)\subset (-\infty, x)\times [\frac{1}{2},+\infty) 
   \ , \ \gamma_{i}(-1,1)\subset [x-2,x-1]\times [1,i]$$ and:
   $$\gamma_{i}((-\infty, -1])=(-\infty, x-2]\times\{1\}\
   ,\ \gamma_{i}([+1,+\infty)) =(-\infty,x-2]\times \{i\}\ , \ ~.~$$
  \item[(II)] The second type of curve is denoted by $t_{k}$.  For
   $1\leq k\leq m$ the curve $t_{k}$ is given by a smooth embedding
   $t_{k}:(-\infty, 0]\to \C$ so that we have
   $$t_{k}(0)=v_{k} \ , \ t_{k}((-\infty, -2])=
   (-\infty,x-2]\times \{1\} \ , \ t_{k}((-\infty,0))\subset (-\infty,
   m+1)\times [1,3]$$ and $t_{k}$ turns once around all the points
   $v_{k+1}$, $v_{k+2}, \ldots, v_{m}$.
\end{enumerate}
 
Both types of curves are pictured in Figure~\ref{fig:spec-curves}.
 \begin{figure}[htbp]
   \begin{center}
  \includegraphics[scale=0.5]{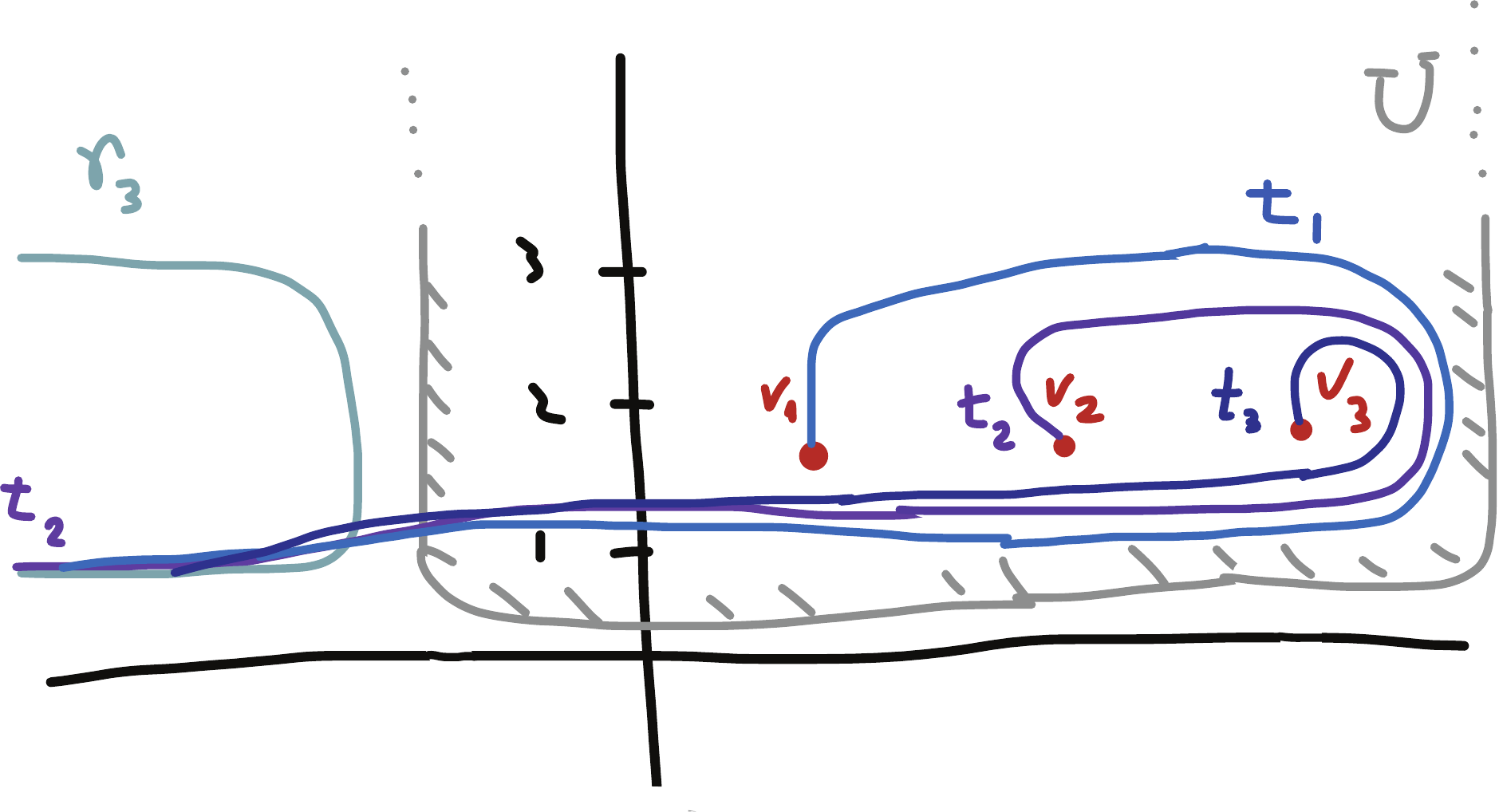}
   \end{center}
   \caption{The special curves $\gamma_{3}$ and $t_{1}, t_{2}, t_{3}$
     for a fibration $E$ with three critical
     points. \label{fig:spec-curves}}
\end{figure}
 
Let $x-3<a <x-2$ and fix the points $z_{i}=(a,i)\in \mathbb{R}^2 \approx
\C$, $i\in \N$.  Set also $z_* = (a, 1) \in \mathbb{R}^2$ (of course, $z_{1}=z_{\ast}$, we use this double notation because we want to view $z_{\ast}$ as a base-point). Let
$(M_{z_{i}},\omega_{z_{i}})$ be the fiber of $\pi$ over the point
$z_{i}$. There are two  families of Lagrangian cobordisms   in
$\mathcal{L}^{\ast}(E)$ that are associated to the geometric data given above.

\begin{enumerate}
  \item[(I')] For each Lagrangian in $L\in
   \mathcal{L}^{\ast}(M_{z_{i}})$ we consider the trail
     $\gamma_k L$ of $L$ along the curve $\gamma_{k}$.  This is a
     well-defined Lagrangian in $E$ and, further, $\gamma_k L \in
     \mathcal{L}^{\ast}(E)$.
   \item[(II')] \label{pg:Ti} Denote by $T_{i}$ the thimble associated
     to the singularity $x_{i}$ and the curve $t_{i}$. Denote by
     $S_{i}\subset M_{z_*}$ the vanishing sphere associated to the
     singularity $x_{i}$ such that $T_{i}$ is the trail of $S_{i}$
     along $t_{i}$. \pbgreen{Since $E$ is {\mntlf} it follows from
       Proposition~\ref{p:monot-E} that
       $T_{i}\in \mathcal{L}^{\ast}(E)$}.
\end{enumerate}

\subsubsection{The decomposition} We
now reformulate Theorem~\ref{thm:main-dec-gen0} in the setting and
notation above. Recall that we use the Novikov ring
  $\mathcal{A}$ as coefficients at all times. 

\begin{thm}[Theorem~\ref{thm:main-dec-gen0} reformulated]
   \label{thm:A-rfm} Let $V\in\mathcal{L}^{\ast}(E)$ be a Lagrangian with $s$ cylindrical ends
   $L_{i}=V|_{z_{i}}$, $1\leq i \leq s$ (as in Definition \ref{def:cyl-ge}).  There exist finite rank $\mathcal{A}$-modules
   $E_{k}$, $1\leq k\leq m$, and an iterated cone decomposition taking
   place in $D\fuk^{\ast}(E;\tau)$:
   $$V \cong (T_{1}\otimes E_{1}\to T_{2}\otimes E_{2}\to
   \ldots \to T_{m}\otimes E_{m}\to \gamma_s L_s \to \gamma_{s-1}
   L_{s-1} \to \ldots \to \gamma_2 L_2 )~.~$$
Moreover, the category $D\fuk^{\ast}(E;\tau)$ is independent of $\tau$ (up to equivalence).
\end{thm}

The proof of Theorem~\ref{thm:A-rfm} follows 
from an analogue result - Theorem \ref{thm:main-dec}, stated in the first subsection below - which
applies to tame Lefschetz fibrations.  The three subsequent
subsections \S\ref{subsec:dec-Yo} - \S\ref{s:cob-vpt} form the
technical heart of the paper.  They provide the arguments that are put
together in \S \ref{subsec:prof-main-t} to show Theorem \ref{thm:main-dec}. 
The decomposition in the statement of Theorem \ref{thm:A-rfm} follows directly
from that provided by Theorem \ref{thm:main-dec}. The modules $E_{i}$ are explicitly
identified along the proof - see equation (\ref{eq:E-i-s}). 
The independence of $D\fuk^{\ast}(E;\tau)$ from the choice of $\tau$
 is postponed to \S \ref{sec:conseq} as it is an immediate consequence of Corollary \ref{cor:cat-eq}
which is itself deduced from Theorem \ref{thm:main-dec}.

\subsection{Decomposition of cobordisms in tame fibrations}
\label{subsec:dec-tame}

Assume now that the Lefschetz fibration $\pi:E\to \C$ is tame outside
the set $U$ - as in Definition \ref{df:tame-lef-fib} - and is so that:
\begin{itemize}
  \item[i.] the set $U$ contains $[0,m+1]\times [\frac{1}{2},K]$ and,
   as in (\ref{eq:u-tame}), $U\subset \R \times [0,+\infty)$.
  \item[ii.] as before, $\pi$ has $m$ critical points $x_{k}\in E$ of
   corresponding critical values $v_{k}=(k, \frac{3}{2})$.
  \item[iii.]  we fix $a_{U}>0$
   sufficiently large so that the set $\{z\pm d\ |\ z\in U, \ d\in [0,4]\subset\R\}$ is disjoint from both quadrants
   $$Q_{U}^{-}=(-\infty, -a_{U}]\times [0,+\infty)\ , 
   \ Q_{U}^{+}=[a_{U},\infty)\times [0,+\infty).$$
\end{itemize}

In this setting we again first define the ``simple'' pieces that
appear in the relevant decomposition. They again involve two types of
curves, again denoted by $\gamma_{i}$ and $t_{j}$, and are defined as
at the points~(I) and~(II) in~\S\ref{subsubsec:atoms} but by using
instead of the constant $x$ the value $-a_{U}+3$.  As a consequence,
the position of these curves relative to the set $U$ is as in
Figure~\ref{fig:spec-curves}.  With this definition we then define the
two families of associated Lagrangians as at the points (I') and
(II'). Notice that the Lagrangian $\gamma_k L$ is a product
$\gamma_{k}L=\gamma_{k}\times L$. This is because the fibration is
trivial over the complement of $U$ and $\gamma_{k}$ is entirely
contained in this complement. At the same time, because of condition
{\em iii} above, $\gamma_{k}L$ as well as $T_{j}$ are cobordisms in
the sense of Definition \ref{def:Lcobordism} (relative to the constant
$a_{U}$). \pbgreen{Finally, assume that $L\in \mathcal{L}^{\ast}(M)$.
  Thus the $\gamma_{k}L$'s are objects of $\mathcal{L}^{\ast}(E)$, and
  by Proposition~\ref{p:monot-E} the same holds for the $T_{j}$'s.}

We reformulate again Theorem~\ref{thm:main-dec-gen0} in this context:

\begin{thm}\label{thm:main-dec} Let $V\in\mathcal{L}^{\ast}(E)$,
   $V:\emptyset \to (L_{1},\ldots, L_{s})$.
   There exist finite rank $\mathcal{A}$-modules $E_{k}$, $1\leq k\leq
   m$, and an iterated cone decomposition taking place in
   $D\fuk^{\ast}(E)$:
   $$V \cong (T_{1}\otimes E_{1}\to T_{2}\otimes E_{2}\to \ldots \to
   T_{m}\otimes E_{m}\to\gamma_{s}\times L_{s}\to 
   \gamma_{s-1}\times L_{s-1}\to\ldots \to\gamma_{2}\times L_{2})~.~$$
\end{thm} 

\subsection{Decomposition of remote Yoneda modules}
\label{subsec:dec-Yo}

In this subsection we assume the ``tame'' setting of
\S\ref{subsec:dec-tame} and we consider a particular class of
$A_{\infty}$-modules over $\fk^{\ast}(E)$ associated to certain
cobordisms $W$ included in Lefschetz fibrations that extend $E$.

Specifically, fix a large constant $K>0$ and consider a Lefschetz
fibration $\hat{\pi}:\hat{E}\to \C$ so that:
\begin{itemize}
  \item[i.] $\hat{\pi}$ is tame outside $\hat{U}$, with $U\subset
   \hat{U}$ and is so that condition (\ref{eq:quadr}) is satisfied for
   some constant $a_{\hat{U}}>a_{U}$.
  \item[ii.] $\hat{U}\subset \R\times [-K,+\infty)$.
  \item[iii.] $\hat{E}|_{\R\times [-\frac{1}{2},+\infty)}=E|_{\R\times
     [-\frac{1}{2},+\infty)}$ including their symplectic
     structures.
\end{itemize}

Similarly to the definition of the category $\fuk^{\ast}(E)$ in
\S\ref{subsubsec:Fuk-cob} we consider a Fukaya category
$\fuk^{\ast}(\hat{E})$ whose objects are cobordisms $W\subset \hat{E}$
as in Definition \ref{def:Lcobordism} so that $W$ is monotone of class
$\ast=(\rho,d)$, $W$ has only negative ends $L_{1},\ldots, L_{s}$ (all
in $\mathcal{L}^{\ast}(M)$) and, similarly to ii
in~\S\ref{subsubsec:Fuk-cob},
$$W\subset \hat{\pi}^{-1}(\R\times [-K+\frac{1}{2},\infty))~.~$$ 
Following Definition \ref{def:Lcobordism}, the cobordism
$W$ is cylindrical and the ends of $W$ project to rays of the form $(-\infty,
-a_{\hat{U}}]\times \{k\}$ with $k\in \N^{\ast}$. 

A cobordism $W$ as before is called {\em remote} relative to $E$ if,
in addition,
\begin{equation}\label{eq:remote-cob}
   W\subset  \hat{\pi}^{-1}(\R\times (-\infty,0]\cup Q^{-}_{U})~.~
\end{equation}
In this case, we deduce, in particular, that $W\cap
\pi^{-1}(U)=\emptyset$ (this explains the terminology, in the sense
that $W$ is remote from all the singularities of $\pi$).  See
Figure~\ref{fig:remote-cob}. It is important to note that because
$\hat{U}$ might contain an unbounded region disjoint from the upper
half plane (in the figure this region goes through the third
quadrant, it could as well also intersect the fourth quadrant but that is irrelevant for the argument), the conditions i,ii,iii allow for $\hat{E}$ to have more
singularities than $E$.

 \begin{figure}[htbp]
   \begin{center}
\includegraphics[scale=0.35]{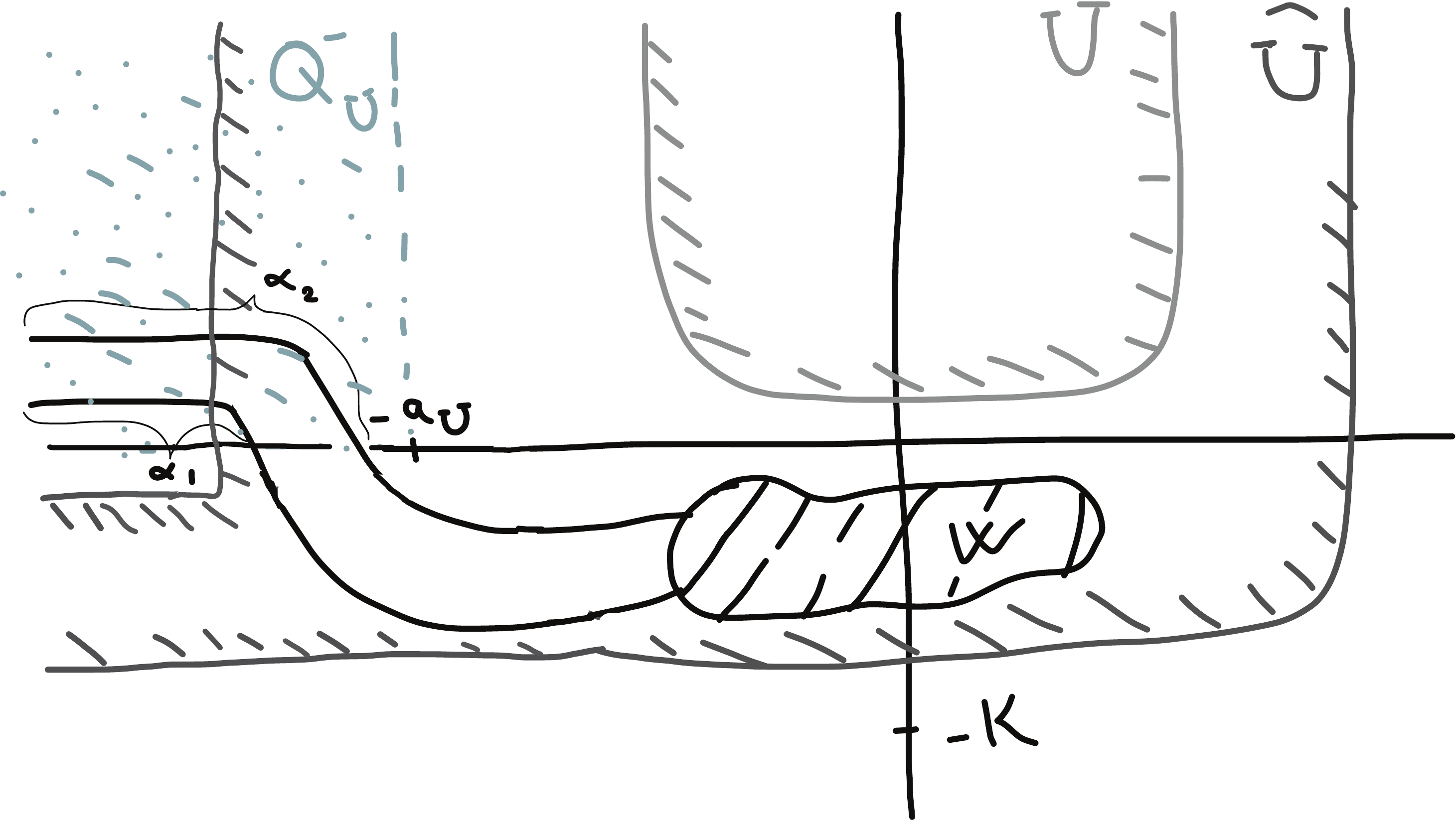}
   \end{center}
   \caption{The domains $\hat{U}$, $U$, the quadrant $Q_{U}^{-}$ and
     the cobordism $W$ that is remote relative to
     $E$.\label{fig:remote-cob}}
\end{figure}

Given property ii from \S\ref{subsubsec:Fuk-cob}, it is clear that
such remote cobordisms $W$ are not objects of $\fuk^{\ast}(E)$.  On
the other hand, each object of $\fk^{\ast}(E)$ is an object of
$\fk^{\ast}(\hat{E})$.  Moreover, by a simple application of the open
mapping theorem, we see
that there is an inclusion of $A_{\infty}$-categories
\begin{equation} \label{eq:Lef-inclusion}
   \mathrm{Incl}\/^{E,\hat{E}}:\fk^{\ast}(E)\to \fk^{\ast}(\hat{E})~.~
\end{equation} 
The relevant argument is as follows.  All objects of
$\fuk^{\ast}(E)$ project to the upper half plane so that the
$J$-polygons that compute the operations $\mu^{k}$ of
$\fuk^{\ast}(\hat{E})$ (for objects that are in $\fuk^{\ast}(E)$)
project to curves $v$ in $\C$ with boundary inside the upper half
plane. Our choice of almost complex structures imply that such a curve
$v$ can be assumed - after applying the change of coordinates as in \S\ref{sbsb:nat-transf} -
to be holomorphic outside (possibly a slightly
bigger set containing) $U$ and, by the open mapping theorem, we deduce
that $v$ can not extend outside of the region where $E$ and $\hat{E}$
coincide. Thus, for objects picked in $\fuk^{\ast}(E)$, the operations
$\mu_{k}$ are the same in $\fuk^{\ast}(\hat{E})$ and in
$\fuk^{\ast}(E)$. 

Let $\mathcal{Y}(W)$ be the Yoneda module associated to an object
$W\in \mathcal{O}b(\fk^{\ast}(\hat{E}))$. We denote by $W_E$ the
pull-back module:
\begin{equation} \label{eq:rem-module}
   W_E=(\mathrm{Incl}^{E,\hat{E}})^{\ast}(\mathcal{Y}(W))
\end{equation}

In case $W$ is remote with respect to $E$ we say that the module $W_E$
is a remote $\fk^{\ast}(E)$-module.

\begin{prop} \label{lem:decomp-remote} With the terminology above,
   assume that $W\in\mathcal{O}b(\fk^{\ast}(\hat{E}))$ is remote
   relative to $E$, $W:\emptyset \cobto (L_{1},\ldots, L_{s})$, then
   $W_E \in \mathcal{O}b(D\fk^{\ast}(E))$ and it admits a
   decomposition in $D\fk^{\ast}(E)$ of the following form:
   \begin{equation} \label{eq:cone-dec} W_E=(\gamma_{s}\times L_{s}\to
      \gamma_{s-1}\times L_{s-1}\to \ldots \to \gamma_{2}\times L_{2})
   \end{equation}
\end{prop}

To unwrap a bit the meaning of this Proposition consider a cobordism
$W$ in $E$.  If there is a horizontal hamiltonian isotopy
$\phi:\hat{E}\to \hat{E}$ that pushes $W$ away from the singularities
of $\pi$, in the sense that $\pi(\phi(W))\cap U=\emptyset$, then the
Proposition implies that $W$ admits a decomposition as claimed in
Theorem \ref{thm:main-dec} but with all the modules $E_{i}=0$. As a
particular case that is already of interest, if $\pi$ has no
singularities $E=\C\times M$ ($U=\emptyset$ and $m=0$), then
Proposition \ref{lem:decomp-remote} applies to any cobordism
$W\subset E=\C\times M$. Thus, for $E=\C\times M$,
Proposition~\ref{lem:decomp-remote} implies
Theorem~\ref{thm:main-dec}.

\begin{rem} \label{r:remote-monotone} \pbhl{In this paper we mostly
    assume that our Lefschetz fibrations are strongly monotone, which
    in turn determines a monotonicity class $*$ for the associated
    Fukaya categories.} However, Proposition~\ref{lem:decomp-remote}
  continues to hold for remote cobordisms of arbitrary monotonicity
  classes $*$ (subject to the restrictions on $*$ made on
  page~\pageref{pg:mon-class-M} in~\S\ref{subsec:Fuk-fibr}). \pbhl{The
    point is that we can analyze remote cobordisms as if they live in
    a trivial Lefschetz fibration, and so there is no need to take
    into account monotonicity properties of the thimbles and vanishing
    spheres. See the ``exception'' to} Definition~\ref{df:monlef} on
  page~\pageref{pg:mon-exception}.
\end{rem}

\begin{proof}[Proof of Proposition~\ref{lem:decomp-remote}]
  We start by repositioning $W$ by using a horizontal Hamiltonian
  isotopy in $\hat{E}$.  By definition, this is an isotopy possibly
  not with compact support, whose support contains a neighborhood of
  the singularities of $\hat{E}$, and which slides the ends of $W$
  along themselves just as in Definition 2.2.3 in
  \cite{Bi-Co:lcob-fuk}. It is immediate to see that such isotopies do
  not change the isomorphism type of objects in
  $\fuk^{\ast}(\hat{E})$.
   
   By applying such an isotopy to $W$ we may assume that
   not only $W\subset \hat{\pi}^{-1}(\R\times (-\infty,0]\cup
   Q^{-}_{U})$ as in the definition of remote cobordisms but that,
   moreover, the intersection
   $$W^{-}=W\cap Q^{-}_{U}$$
   coincides with a disjoint union of cylindrical ends of $W$. In
   other terms
   $$W^{-}=\cup_{i=1}^{s} \alpha_{i} \times L_{i}$$
   where $\alpha_{i}$ are curves in $\C$ as in Figure
   \ref{fig:remote-cob}.  In particular, for any object $X\in
   \mathcal{O}b(\fk^{\ast}(E))$, the intersection $W\cap X$ consists
   of a union of intersections of the ends of $W$ with the ends of $X$
   and is included in the quadrant $Q^{-}_{U}$.

   The main part of the proof makes essential use of constructions
   that appear in~\cite{Bi-Co:lcob-fuk}. It consists of three main
   steps.

   \smallskip

   \noindent \textbf{Step 1:} {\em Repositioning $W$}. Here we replace
   the module $W_E$ with a quasi-isomorphic module corresponding to a
   cylindrical Lagrangian that can be handled easier geometrically.
   For this purpose we include the two $A_{\infty}$-categories
   $\fuk^{\ast}(E)$ and $\fuk^{\ast}(\hat{E})$ in two other
   $A_{\infty}$-categories, respectively, $\fuk^{\ast}_{\frac{1}{2}}
   (E)$ and $\fuk^{\ast}_{\frac{1}{2}}(\hat{E})$.  These two
   categories have objects that are again cobordisms as before with
   the difference that their ends have heights $\in
     \frac{1}{2}\Z\subset \Q$. In other words, compared with
   Definition \ref{def:Lcobordism}, the difference is that $V\cap
   \pi^{-1}(Q_{U}^{-})= \cup_{i\in\N^{\ast}}((-\infty, -a_{U}]\times
   \{\frac{i}{2}\})\times L_{i}$.  The inclusion $\fuk^{\ast}(E)\to
   \fuk^{\ast}_{\frac{1}{2}}(E)$ is obvious and is clearly full and
   faithful and similarly for the two categories associated to
   $\hat{E}$.  We now perturb $W$ by a (non-horizontal) Hamiltonian
   isotopy so as to obtain an object $W'$ of
   $\fuk^{\ast}_{\frac{1}{2}}(\hat{E})$ that differs from $W$ only
   inside $(-\infty, -a_{U}-2]\times [\frac{1}{2},+\infty)$ and is so
   that the ends of $W'$ restricted to $(-\infty, -a_{U}-4 -s]\times
   [\frac{1}{2},+\infty)$ are of the form $(-\infty,-a_{U}-4 -s]\times
   \{i-\frac{1}{2}\}\times L_{i}$ (for all the definitions involved to
   be coherent we might need to enlarge here the set $\hat{U}$).  In
   other words, the ends of $W'$ are shifted down by $\frac{1}{2}$
   compared to the ends of $W$.  Let $W'_E$ be the
   $\fuk^{\ast}(E)$-module obtained as pull-back over the inclusions
   $$\fuk^{\ast}(E)\to \fuk^{\ast}(\hat{E})\to
   \fuk^{\ast}_{\frac{1}{2}}(\hat{E})$$ from the
   $\fuk^{\ast}_{\frac{1}{2}}(\hat{E})$-module $\mathcal{Y}(W')$.
   \begin{figure}[htbp]
      \begin{center}
         \includegraphics[scale=0.5]{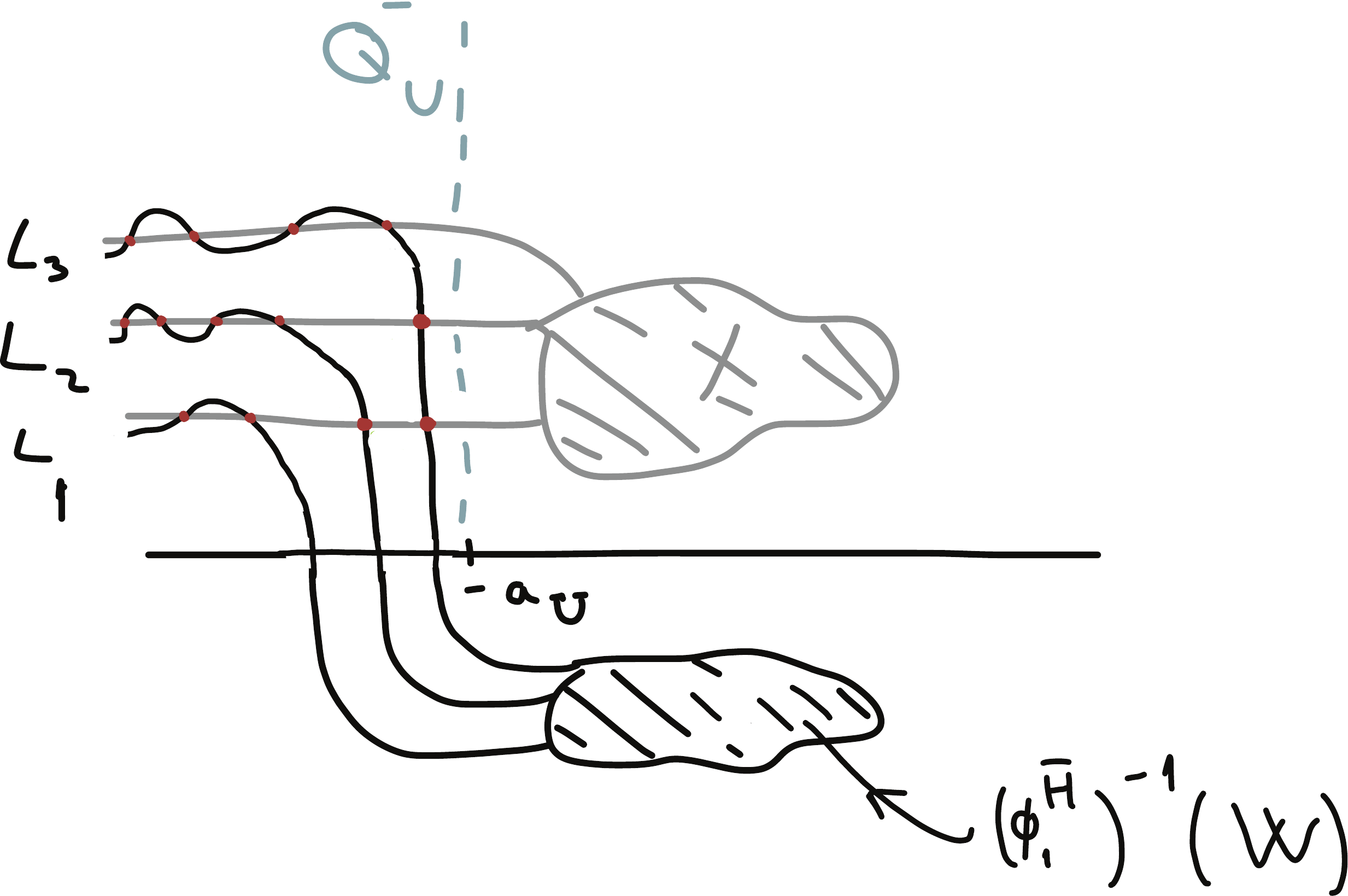}
      \end{center}
      \caption{The projections on $\C$ of
        $(\phi^{\bar{H}_{X,W}}_{1})^{-1}(W)$ and of $X$.The ends of $(\phi^{\bar{H}_{X,W}}_{1})^{-1}(W)$ are below those of $X$ at infinity.
        \label{fig:proj-W}}
   \end{figure}
   The two modules $W_E$ and $W'_E$ are quasi-isomorphic. This
   is a direct consequence of the definition of
   $\mor_{\fk^{\ast}(\hat{E})}(X,W) = CF(X,W)$. This uses a perturbation
   of $W$ in which its negative ends are ``moved'' down compared
   to those of $X$. More precisely, recall from \S 3 in \cite{Bi-Co:lcob-fuk} (see also Figure 8 there) that
   $CF(X,W)$ is defined by using a specific profile function $h$ and
   an associated Hamiltonian $\bar{H}_{X,W}$. With these choices
   $CF(X,W)$ is identified with
   $CF(X,(\phi^{\bar{H}_{X,W}}_{1})^{-1}(W))$ (under the assumption
   that $X$ and $(\phi^{\bar{H}_{X,W}}_{1})^{-1}(W)$ intersect
   transversely).  The projection of
   $(\phi^{\bar{H}_{X,W}}_{1})^{-1}(W)$ to $\mathbb{C}$ is as in
   Figure~\ref{fig:proj-W}. On the other hand the ends of $W'$ are, by construction, below
   the horizontal lines $\R\times \{i\}$ and therefore the complexes $CF(X,W)$ and $CF(X,W')$ are 
   quasi-isomorphic. Further, this quasi-isomorphism extends to a quasi-isomorphism of
   the modules $W_{E}$ and $W'_{E}$. 
 
   To summarize this first step, we have replaced in our argument the
   cobordism $W$ by the cobordism $W'$.  Moreover, by a further horizontal
   Hamiltonian isotopy, we may assume that $W'$ has a projection as in
   Figure \ref{fig:inter-pert}. More precisely, we assume that
   $(W')^{-}=W'\cap Q_{U}^{-}$ is a disjoint union of components
   $\alpha_{i}\times L_{i}$ so that $\alpha_{i}$ is obtained by
   rounding the corner of the union of two intervals $(-\infty,
   -a_{U}-4-s+i]\times \{i-\frac{1}{2}\}\cup \{-a_{U}-4-s+i\}\times
   [0,i-\frac{1}{2}]$.  In
   particular, the intersections of $X$ and $W'$ project onto $\C$ to
   the points $b_{ij}=\{-a_{U}-4-s+i\}\times \{j\}$ with $i> j$,
   $i,j\in \N^{\ast}$, $i=1,2,\ldots, s$; $b_{ij}$ is precisely the
   projection of the intersection of the $i$-th end of $W'$ with the
   $j$-th end of $X$.

   \begin{figure}[htbp]
      \begin{center}
      \includegraphics[scale=0.5]{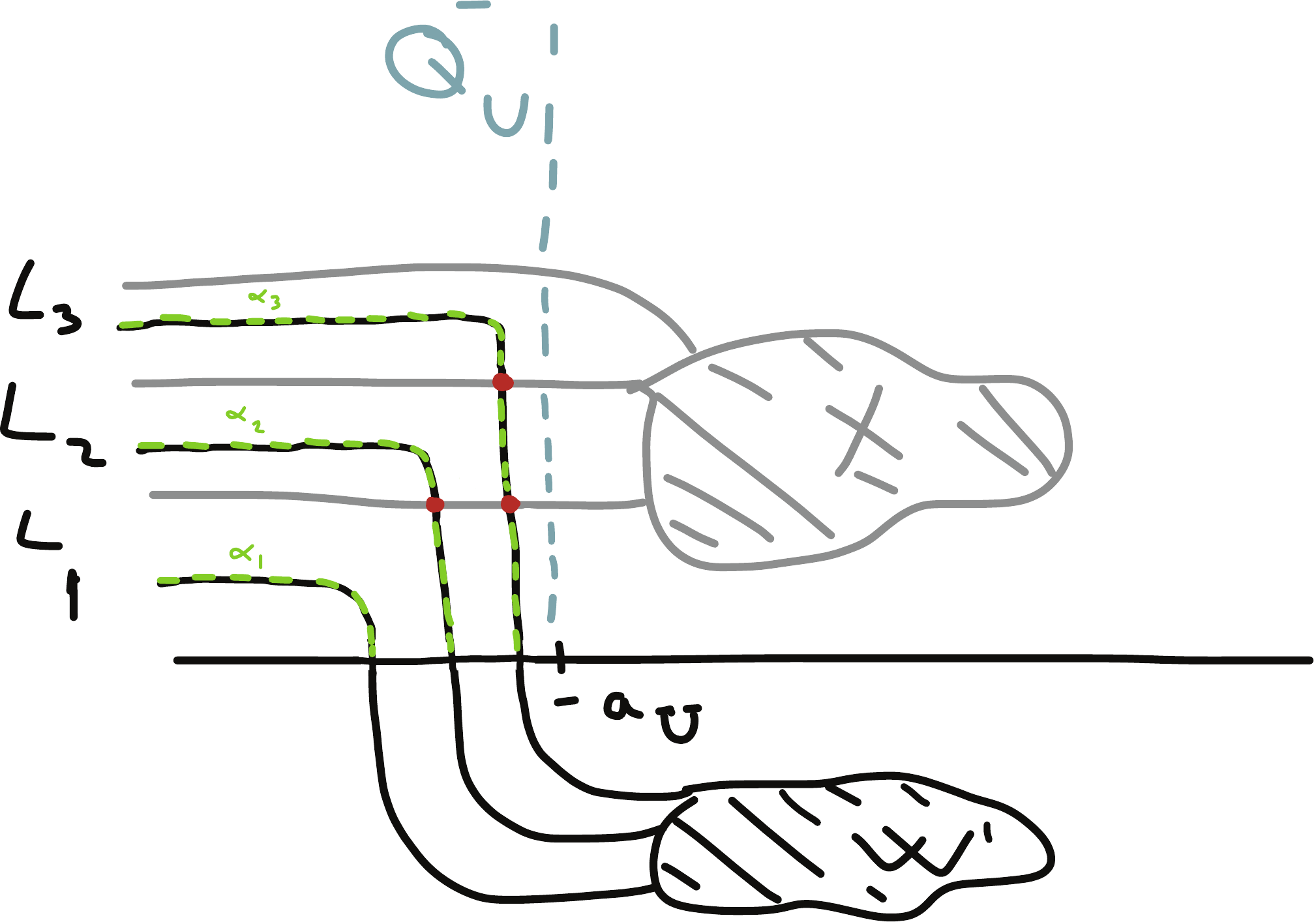}
      \end{center}
      \caption{The remote cobordism $W'\subset \hat{E}$, the object
        $X\in \mathcal{O}b(\fk^{\ast}(E))$ and the curves
        $\alpha_{i}$. The height of the $i$-th end of $W'$ is
        $i-\frac{1}{2}$ while the $i$-th end of $X$ has height $i$.
        \label{fig:inter-pert}}
   \end{figure}

   We may also assume, by a slight additional horizontal isotopy, that
   $W'\cap \pi^{-1}(\R\times [-\frac{1}{2},\infty))$ is a union of
   cylindrical ends.

   \

   \noindent \textbf{Step 2 :} {\em ``Snaky'' perturbation data.}
   This step of the proof consists in choosing the perturbation data
   used in the definition of $\fuk^{\ast}(E)$ and
   $\fuk^{\ast}(\hat{E})$ in a convenient way. Recall that $W'$ is
   already fixed as discussed at step 1. The perturbation data in
   question are chosen as described in \S\ref{subsubsec:Fuk-cob}
   except that the profile function $h$ as well as the almost complex
   structure $\mathbf{J}$ will be picked with some additional
   properties described below.

   We start with the choice of the profile function $h$. As can be
   seen from \S\ref{subsubsec:Fuk-cob} the fundamental ingredients in
   the definition of $h$ are the functions $h_{\pm}$.  We start with
   $h_{+}$: the only requirement in this case is that $h_{+}:
   [a_{U}+\frac{3}{2},\infty)\to \R$ has its single critical point
   (the bottleneck) at $a_{U}+2$. In other words the difference with
   respect to the construction at ~\S\ref{subsubsec:profile} is that
   the value $\frac{1}{2}$ is replaced with $a_{U}$.  In fact, as we
   only consider cobordisms without positive ends the choice of
   $h_{+}$ is not particularly important as long as the bottlenecks
   are away from $U$.  We now discuss the function $h_{-}$.  This is a
   smooth function $h_{-}: (-\infty, -a_{U}-1]\to \R$ with the
   following additional properties - see
   Figure~\ref{fig:snaky-perturb}:

   \begin{itemize}
     \item[$\textnormal{a}^{\prime}$.] The function $h_{-}$ has
      critical points $o_{i}=-a_{U}-3-i$, $i=0, 1, \ldots, s$
      that are non-degenerate local maxima.
     \item[$\textnormal{a}^{\prime \prime}$.] The function $h_{-}$ has
      critical points $o'_{i}=-a_{U}-\frac{7}{2}-i$,
        $i=0,1,\ldots, s-1$ that are non-degenerate local minima.
     \item[$\textnormal{a}^{\prime \prime \prime}$.] $h_{-}$ has no
      other critical points than those at $\textnormal{a}^{\prime}$,
      $\textnormal{a}^{\prime \prime}$ above and for all $x\in
      (-\infty, a_{U}-4-s]$ we have $h_{-}(x)=\alpha^{-}x+\beta^{-}$
      for some constants $\alpha^{-}$, $\beta^{-}$, $\alpha^{-}>0$.
   \end{itemize}

   \begin{figure}[htbp]
      \begin{center}
         \includegraphics[scale=0.8]{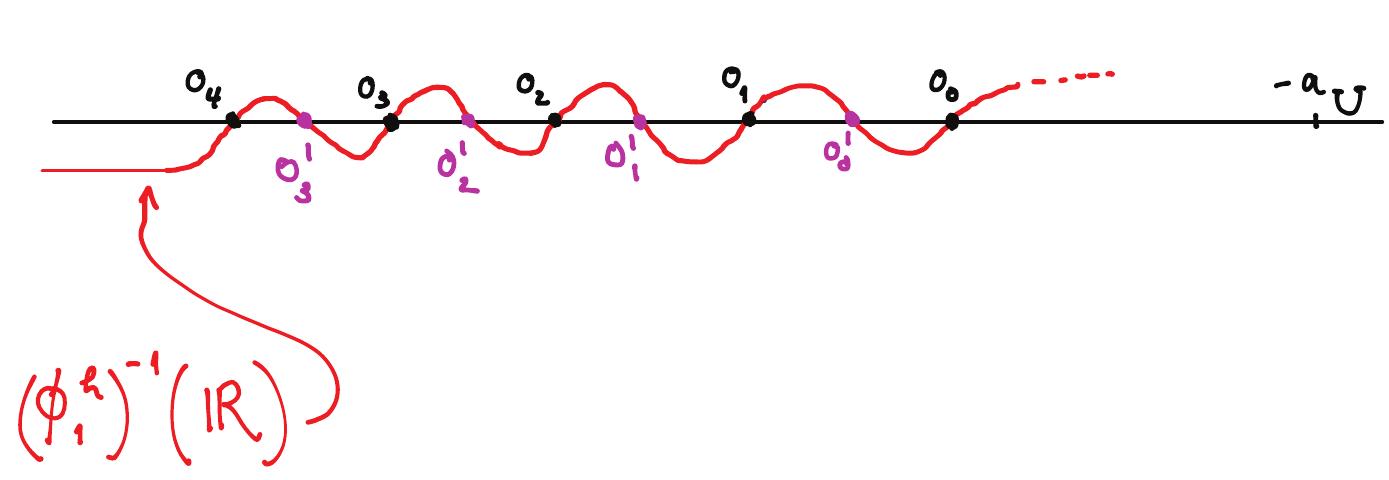}
      \end{center}
      \caption{The graph of $(\phi^{h}_{1})^{-1}(\R)$ for
        $s=4$.\label{fig:snaky-perturb}}
   \end{figure}

   Beyond this, the properties of the function $h$ are obtained by
   direct analogy with those given at the points i, ii, iii, iv
   in~\S\ref{subsubsec:profile} but with the point {\em
     a} replaced by the three conditions $\textnormal{a}^{\prime},
   \textnormal{a}^{\prime \prime}, \textnormal{a}^{\prime \prime
     \prime}$ above.  In particular, the set $W_{i}^{-}$ now
   becomes $W_{i}^{-}=(-\infty, -a_{U}-1]\times
   [i-\epsilon,i+\epsilon]$ and $T_{i}^{-}=(-\infty,-a_{U}-1]\times
   [i-\epsilon/2, i+\epsilon/2]$.  From this point on, the
   construction continues along the same approach as in
\S\ref{subsubsec:Fuk-cob}. In particular, the
   properties of the family $\Theta$ and those of $\mathbf{J}$ are
   just the same as properties \emph{a,b,c}
 in~\S\ref{subsubsec:perturb} but they are relative to sets
     $K_{V_{1},\ldots, V_{k+1}}$ that satisfy different requirements
     compared to those in~\S\ref{subsubsec:perturb}.

   We now discuss the two properties required of $K_{V_{1},\ldots,
     V_{k+1}}$. We start by underlining that, because we care here about a module structure,
     while $V_{1},\ldots,
   V_{k}$ are elements of $\mathcal{L}^{\ast}(E)$, $V_{k+1}$ is either
   an element of $\mathcal{L}^{\ast}(E)$ or $V_{k+1}=W'$.  Further, we
   fix small disks $D_{ij}\subset \C$ of radius smaller than
   $\frac{1}{8}$ that are respectively centered at the points
   $(o'_{i}, j)$, $i=0,\ldots, s-1$, $j\in \{1,\ldots,
   s_{V_{1},\ldots, V_{k+1}}\}$.  We denote by $D'_{ij}\subset D_{ij}$
   the disk with the same center but with radius half of that of
   $D_{ij}$.  Recall, that $s_{V_{1},\ldots, V_{k+1}}$ is the smallest
   $l\in \N$ so that $\pi (V_{1}\cup V_{2}\cup \ldots \cup
   V_{k+1})\subset [\frac{1}{2},l)$. We also pick a compact set
   $Z\subset \R\times (-\infty, -\frac{1}{4}]$ which contains in its
   interior $\pi (W') \cap \R\times (-\infty,-\frac{1}{2}]$ (recall
   that $W'$ is cylindrical outside $\pi^{-1}(\R\times
   (-\infty,-\frac{1}{2}]$) as well as a slightly bigger set
   $Z'\subset \R\times (-\infty, -\frac{1}{4}]$.  We require:
   \begin{equation} \label{eq:cpct1} K_{V_{1},\ldots, V_{k+1}}\supset
      \cup_{i,j} D'_{ij}\cup [-a_{U}-\frac{11}{4}, a_{U}+ \frac{7}{4}]
      \times [\frac{1}{4}, s_{V_{1},\ldots, V_{k+1}}+1] \cup Z~.~
   \end{equation}
   and
   \begin{equation}\label{eq:cpct2}
      K_{V_{1},\ldots, V_{k+1}}\subset \cup_{i,j} D_{ij}
      \cup [-a_{U}-\frac{13}{4}, a_{U}+2)\times 
      [\frac{1}{8},+\infty)\cup Z'~.~
   \end{equation}
   
   We now will see that this class of perturbation data is sufficient
   to insure the regularity and the compactness of the moduli spaces
   appearing in the definition of the category $\fuk^{\ast}(E)$ and of
   the $\fuk^{\ast}(E)$-module $W'_E$. In the next section we will use
   these specific perturbations to extract the exact triangles claimed
   in the statement.
 
   Let $u:S_{r}\to E$ be a solution of~\eqref{eq:jhol1} that satisfies
   the boundary and asymptotic conditions required to define the
   multiplications $\mu_{k}$ for $\fuk^{\ast}(E)$ or for the
   definition of the module $W_E$. In the first case the boundary
   conditions are along cobordisms $V_{1}, \ldots, V_{k+1}$ ($V_{i}\in
   \mathcal{L}^{\ast}(E)$, in particular, $V_{i}$ projects on the
   upper half plane). In the second case, the curve is defined on a
   punctured polygon so that the component $C_{i}$ of the boundary of
   the polygon is mapped to $V_{i}$ for $1\leq i\leq k$ and the
   $k+1$-th component $C_{k+1}$ is mapped to $W'$.

   By the change of variables in~\S\ref{sbsb:nat-transf}, (and by
   taking $h$ sufficiently small) we deduce that there exists some
   small $\delta >0$ so that if $u:S_{r}\to E$
   satisfies~\eqref{eq:jhol1} with the choice of perturbation data as
   just above and if $v:S_{r}\to E$ is defined by
   $u(z)=\phi^{\bar{h}}_{a_{r}(z)}(v(z))$, then $v'=\pi\circ v$ is
   holomorphic outside of the set
   \begin{equation} \label{eq:region-K-hat} \widehat{K}=\cup_{i,j}
      D''_{ij}\cup [-a_{U}-\frac{13}{4}-\delta, a_{U}+2 +\delta]\times
      [\frac{1}{8}-\delta,+\infty)\cup Z'',
   \end{equation}
   where $D''_{ij}$ is a disk with the same center as $D_{ij}$ but
   slightly bigger and, similarly, $Z''$ is a set slightly bigger than
   $Z'$ - see Figure \ref{fig:boxes}.
   \begin{figure}[htbp]
      \begin{center}
         \includegraphics[scale=0.8]{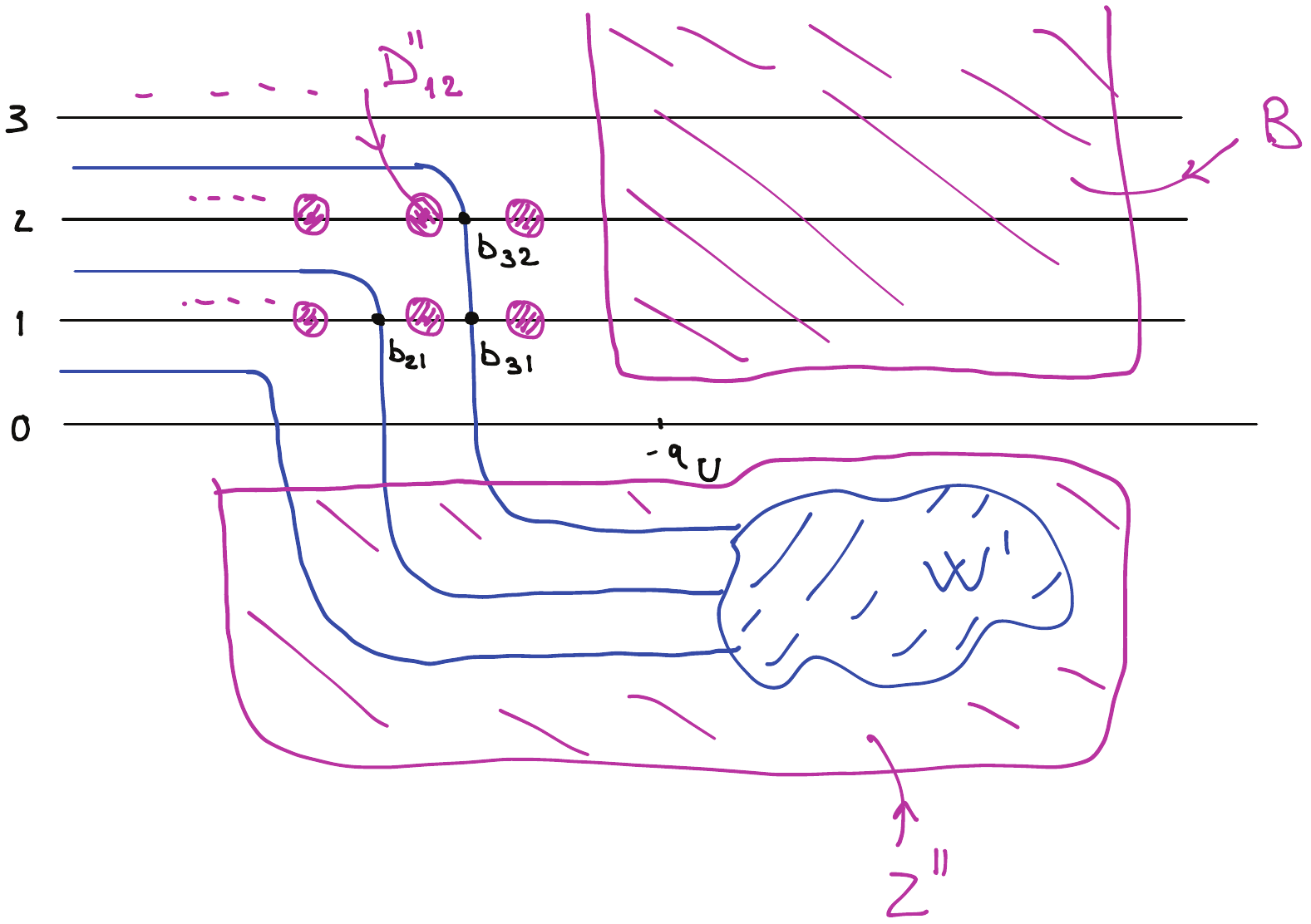}
      \end{center}
      \caption{The set $\widehat{K}$ outside which $v'$ is holomorphic
        is the union of all the regions in pink: the disks $D''_{ij}$,
        the box $$B=[-a_{U}-\frac{13}{4}-\delta, a_{U}+2
        +\delta]\times [\frac{1}{8}-\delta,+\infty)$$ and the
        neighborhood $Z''$ of the non-cylindrical part of $\pi(W')$.
        Are also pictured the points $b_{ij}$. Here $s=3$.  The
        non cylindrical part of the cobordisms $X\in
        \mathcal{L}^{\ast}(E)$ projects inside $B$.
        \label{fig:boxes}}
   \end{figure}
   In view of this transformation, compactness for the relevant moduli
   spaces follows without difficulty by the usual bottleneck argument \S 3.3 \cite{Bi-Co:lcob-fuk}.
   Thus, the only issue that requires some attention is regularity.
   Denote $$K'=\cup_{i,j} D'_{ij}\cup [-a_{U}-\frac{11}{4}, a_{U}+
   \frac{7}{4}] \times [\frac{1}{4}, s_{V_{1},\ldots, V_{k+1}}+1]\cup
   Z~.~$$ Given that $K'\subset K_{V_{1},\ldots, V_{k}}$, the
   perturbation data can be chosen freely over $K'$ and
   thus, for all moduli spaces consisting of curves whose image
   intersects $\pi^{-1}(K')$ regularity can be handled in the standard
   fashion as in \cite{Se:book-fukaya-categ}. Therefore, we are left
   to analyze the curves $u:S_{r}\to E$ so that $\pi(u)$ has an image
   disjoint from $K'$. Assume first that $u$ appears in the definition
   of the higher structures of $\fuk^{\ast}(E)$.  In this case, the
   condition $\pi^{-1}(K')\cap Image(u)=\emptyset$ implies that all
   the boundary of $u$ projects onto $\C$ along a single line
   $(-\infty, -a_{U}-2]\times\{j\}$. Given that $(o'_{i},j)\in K'$, it
   follows that the image of $\pi(u)$ can not cross any of the points
   $(o'_{i},j)$, nor can it have one of these points as asymptotic
   limit. As a consequence, the asymptotic limits of $u$ have to
   project to just one of the points $(o_{i},j)$. But by now taking a
   look to $v'$ which is holomorphic around $(o_{i},j)$ one sees
   immediately that $v'$ and thus $\pi(u)$ has to be constant (indeed,
   $(o_{i},j)$ can not be the exit point of $v'$ by an application of
   the open mapping theorem).  The second possibility to consider is
   if $u$ appears in the definition of the module structure of $W'_E$.
   It is immediate, in this case too that $\pi^{-1}(K')\cap
   Image(u)=\emptyset$ implies that all asymptotic limits of $u$
   coincide with a single point $b_{ij}$ (which is, of course, also of
   the from $(o_{i}, j)$).  It is easy to see by an application of the
   open mapping theorem that in this case $\pi(u)$ has again to be
   constant. To conclude this argument, the only moduli spaces for
   which regularity is in question consist of curves $u$ so that
   $\pi(u)$ is constant equal to one of the point $(o_{i},j)$. That
   means that these curves take values in the fiber over $(o_{i},j)$
   and, because $o_{i}$ is a local maximum of $h_{-}$, one can see, as
   in \S 4.2 \cite{Bi-Co:lcob-fuk} that by picking regular data in the
   fiber these moduli spaces are regular too.

   Thus the regularity of all the moduli spaces involved can be
   achieved by generic choices of data. We work from now on with such
   data associated to the ``snaky'' perturbations constructed at this
   step.
  
   \

   \noindent \textbf{Step 3:} {\em The proof of (\ref{eq:cone-dec}).}
   We will show now that there is a sequence of
     $\fuk^{\ast}(E)$-modules $\widetilde{L}_i$, $W'_{E,i}$, $i=1,
     \ldots, s$, with $W'_{E,i}$ being submodules of $W'_E$, so that:
   \begin{itemize}
     \item[i.] $W'_{E,1} = 0$, $W'_{E, s}=W'_E$ and for $i\geq 2$
      there exist exact sequences of $\fuk^{\ast}(E)$-modules
      $$0\to W'_{E,i-1}\to W'_{E,i}\to \widetilde{L}_{i}
      \to 0$$
     \item[ii.] there exists a quasi-isomorphism of
      $\fk^{\ast}(E)$-modules $$\widetilde{L}_{i} \simeq
      \mathcal{Y}(\gamma_i \times L_i),$$ where $\mathcal{Y}$ is the
      Yoneda embedding for $\fk^{\ast}(E)$.
   \end{itemize}
 
 These points immediately imply the statement of
 Proposition~\ref{lem:decomp-remote}. We now proceed to the
 construction of $W'_{E,i}$ and to prove the points i, ii above.

 Let $X\in \mathcal{L}^{\ast}(E)$ and let $W'$ be the remote cobordism
 as discussed at the first step.  We now assume ``snaky''
 perturbations picked as described at the second step. In particular,
 the complex $CF(X,W')$ is well defined. The generators of this
 complex are identified with the intersection $X\cap
 (\phi_{1}^{\bar{h}})^{-1}(W')$. Notice that due to the choice of
 snaky perturbations $\pi(X\cap
 (\phi_{1}^{\bar{h}})^{-1}(W'))=\pi(X\cap
 (\phi_{1}^{\bar{h}})^{-1}(W'))=\{b_{rs}\}_{r,s}$ see
 Figure~\ref{fig:perturbW}.
   \begin{figure}[htbp]
      \begin{center}
         \includegraphics[scale=0.8]{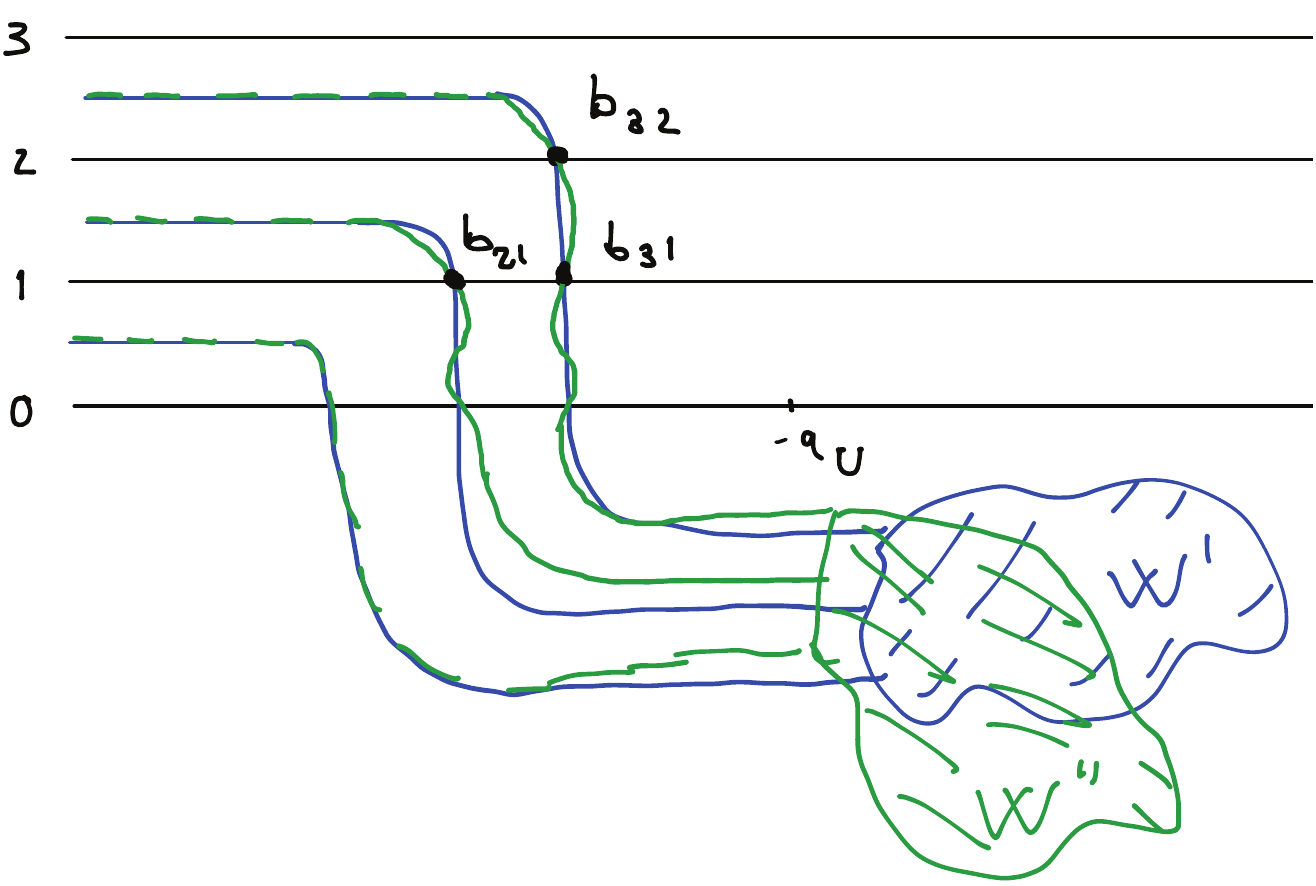}
               \end{center}
      \caption{The cobordism $W'$ and its perturbation
        $W''=(\phi^{\bar{h}}_{1})^{-1}(W')$. \label{fig:perturbW}}
   \end{figure}
   We now put
   $$P_{rs}(X)= X\cap (\phi_{1}^{\bar{h}})^{-1}(W') \cap \pi^{-1}(b_{rs})$$
   and we define $$W'_{E,i}(X)=\mathcal{A}\langle \cup_{1\leq r\leq i;
     s<r} P_{rs} \rangle \subset CF(X,W')~.~$$ In other words, the
   generators of $W'_{E,i}(X)$ are the intersection points of $X$ with
   the first $i$ branches of the $W'$. It is clear from the
   construction that $W'_{E,1}=0$ and that $W'_{E,s}=W'_E$.  We will
   show now that, for each $1\leq i \leq s$, the structural maps
   $\mu_{k}$ of $W'_E$ when restricted to $W'_{E,i}$ have values into
   $W'_{E,i}$. In other words
   \begin{equation}\label{eq:restr}
      \mu_{k}|_{W'_{E,i}} : 
      CF(V_{1}, V_{2})\otimes \ldots \otimes CF(V_{k-1}, V_{k}) 
      \otimes W'_{E,i}(V_{k}) \to W'_{E,i}(V_{1}) ~.~
   \end{equation}
   This property immediately implies that the $W'_{E,i}$ are indeed
   $A_{\infty}$-modules and moreover that the inclusions of vector
   spaces $W'_{E,i-1}(-)\subset W'_{E,i}(-)$ are actually inclusions
   of $\fuk^{\ast}(E)$-modules. The modules $\widetilde{L}_{i}$
   defined as the respective quotients.  With these definition for
   $W'_{E,i}$ and assuming~\eqref{eq:restr},  point ii follows because the quotient
   $\widetilde{L}_{i}$ is naturally identified (up to quasi-isomorphism) with $\mathcal{Y}(\gamma_{i}\times L_{i})$. In summary, to conclude the proof of the proposition
   it remains to show~\eqref{eq:restr}.

 Our argument is based on properties of the curve $v'=\pi(v)$
 where $v$ is related to a curve $u:S_{r}\to E$ by equation
 (\ref{eq:nat-v}) and $u$ is a solution of (\ref{eq:jhol1})
 contributing to the module structural map $\mu_{k}$. Here $S_{r}$ is
 the disk with $k+1$ boundary punctures, of which $k$ are the entries
 and the last one is an exit puncture. The last entry, denoted $m$, is
 the ``module'' entry and is asymptotic to a generator of $CF(V_{k-1},
 W'_{E,i})$. The exit, denoted $e$, is asymptotic to a generator of
 $CF(V_{1}, W'_{E,i})$. 
 
 \
 
 We will make the following simplifying
 assumption: we assume that the transition
 functions used in the definition of moduli spaces associated to the
 module operations are so that:
 \begin{equation}\label{eqn:1-on-last} a_{r}(z)=1 \  \ \ \forall z\in
 C_{k+1}, 
 \end{equation}
 where $C_{k+1}$ is the component of the boundary of the punctured
 disk $S_{r}$ that joins $m$ to $e$. (See Figure~\ref{fig:transition}
 for an illustration of the case $k=3$, where $C_{4}$ bounds both
 $\epsilon_{3}$ and $\epsilon_{4}$.)  In other words we use transition
 functions as in \S\ref{subsubsec:transition} except that we add
 (\ref{eqn:1-on-last}) and we modify conditions {\em i. c} and {\em
   ii. c'} in \S\ref{subsubsec:transition} such as to no longer
 require $a_{r}\circ \epsilon (s,t)=0$ for $(s,t)\in \{0\}\times
 [0,1]$ for $\epsilon$ for the strip like ends associated to $m$ and
 to $e$.  By imposing (\ref{eqn:1-on-last}) {\em just to the moduli
   spaces appearing in the definition of modules} over
 $\fuk^{\ast}(E)$ (and not to those defining the multiplication in
 $\fuk^{\ast}(E)$ itself) we easily see that, on one hand, condition
 (\ref{eqn:1-on-last}) is compatible with gluing and splitting and,
 moreover, it does not contradict the definition of the operations
 in $\fuk^{\ast}(E)$ itself.  At the same time, this means that we get
 two presumptive definitions for the Yoneda modules of objects in
 $\fuk^{\ast}(E)$: one using the conditions in
 \S\ref{subsubsec:transition} and the other making use of
 (\ref{eqn:1-on-last}).  However, it is easy to see that the two
 resulting modules are quasi-isomorphic and thus our simplifying
 condition does not affect any further arguments.
 
 \
 
 The geometric advantage of this simplifying assumption on $a_{r}$ is
 that $v$ no longer satisfies a moving boundary condition along
 $C_{k+1}$, rather $v$ maps all of $C_{k+1}$ to
 $W''=(\phi^{\bar{h}}_{1})^{-1}(W')$. We also remark that, by the
 definition of $h$, and the position of $\pi(W')$ relative to the ends
 of cobordisms $\in \mathcal{L}^{\ast}(E)$ - as in Figure
 \ref{fig:perturbW} - we have that $W''$ is just a close perturbation
 of $W'$ and $\pi(W'')$ intersects the horizontal lines of positive,
 integral imaginary coordinates transversely and in the same points as
 $\pi(W')$.

   Our claim (\ref{eq:restr}) reduces to showing that if
   $v'(m)=b_{\alpha\beta}$ and $v'(e)=b_{rs}$, then $r\leq \alpha$.
 
   We first fix some notation relative to certain regions in
   $Q_{U}^{-}$. First we denote by $F$ the region given as
   $$F=\bigcup_{0\leq t\leq 1, \, j\in\Z} \phi^{h}_{-t}( (-\infty,
   -a_{U}]\times \{j\})\cup W''~.~$$ In short, $F$ is the set swiped
   by all the potential boundary conditions of the curves $v'$.
   Further, we denote $\widehat{F}=F\cup \widehat{K}$
   (see~\eqref{eq:region-K-hat}) and we put $G=\C \setminus
   \widehat{F}$ - see Figure~\ref{fig: hol-region}.

   \begin{figure}[htbp]
      \begin{center}
  \includegraphics[scale=0.8]{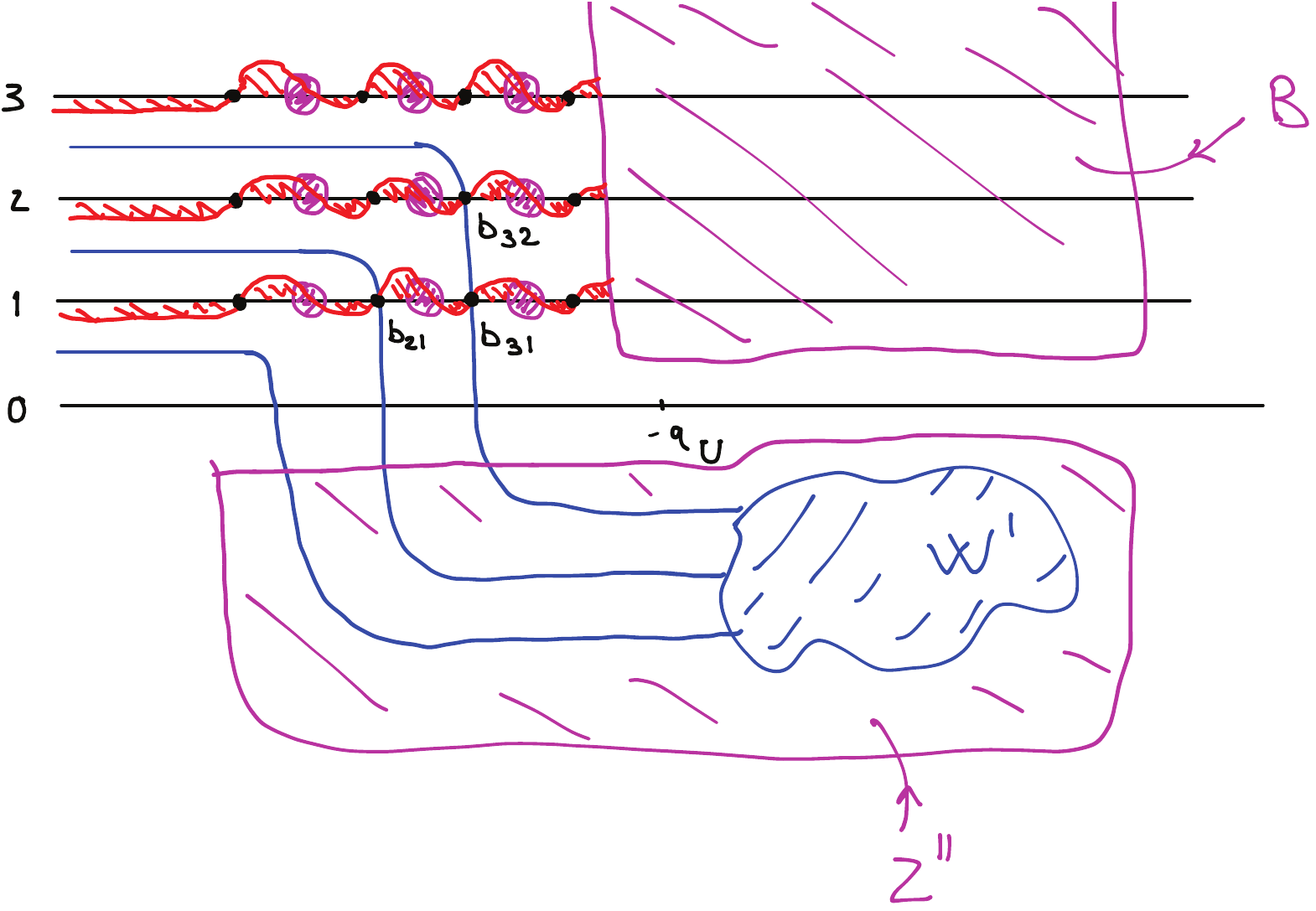}
      \end{center}
      \caption{The region $\widehat{F}$ is the union of $\widehat{K}$
        (the union of all the pink regions) and $F$ (the region in
        red). \label{fig: hol-region}}
   \end{figure}
   From step 2 we know that $v'$ is holomorphic over $G$ and clearly,
   the boundary of $S_{r}$ is so that $v'(\partial S_{r})\cap
   G=\emptyset$.  It is an elementary fact (see for instance
   Proposition 3.3.1 in \cite{Bi-Co:lcob-fuk}) that as soon as $Image(v')$
   intersects a connected component of $G$, the full component has to
   be contained in $Image(v')$. In particular, this means that
   $Image(v')$ can not intersect an unbounded component of $G$.
 
   Each point $b_{ij}$ is in the closure of four components of $G$
   that meet, basically, as four quadrants at $b_{ij}$.  Our argument
   will make use of the following:

   \begin{lem} \label{lem:quadr-ext} Suppose that $b_{ij}$ is
      different from both $v'(e)$ and $v'(m)$ and that the component
      corresponding to the fourth quadrant at $b_{ij}$ is in the image
      of $v'$, then at least one among the first or third quadrants
      are also in the image of $v'$.
   \end{lem}

   For an illustration of the statement of the Lemma take a look at
   Figure \ref{fig:lemma-arg} and the point $b_{42}$ there. The claim
   of the Lemma is that if the green region having $b_{42}$ in its
   boundary is included in $Image(v')$, then one of the yellow regions
   next to $b_{42}$ is also contained in this image.

   \begin{proof}[Proof of Lemma~\ref{lem:quadr-ext}]
      Consider a small segment $I\subset \pi(W'')$ that ends
      up at $b_{ij}$ and is included in the closure of the fourth
      quadrant (the quadrants here are defined by the vertical and horizontal lines in
      Figure \ref{fig:lemma-arg}).  We have $I \subset Image(v')$. Let $x\in I$.  If $x$
      is the image of a point $z\in Int(S_{r})$, then, by the open
      mapping theorem, the image of $v'$ also intersects the third
      quadrant which implies our claim. Thus it is sufficient to
      consider the case when all the points of $I$ are in the image of
      boundary points of $S_{r}$. The only boundary component that is
      mapped to $W''$ is $C_{k+1}$ so that $I\subset v'(C_{k+1})$.
      Moreover, as $b_{ij}$ is not the asymptotic image of the ends of
      $C_{k+1}$, it follows that $b_{ij}\in v'(C_{k+1})$. Let $z\in
      C_{k+1}$ so that $v'(z)=b_{ij}$. As shown at step 2, $v'$ is
      holomorphic outside of $\widehat{K}$ and thus, in particular,
      around $b_{ij}$. Given that (around $b_{ij}$) $v'(C_{k+1})$ is
      contained in the vertical line through $b_{ij}$ and, due to the
      bottleneck structure around $b_{ij}$, the open mapping theorem
      implies that $Image(v')$ intersects the region of $G$
      corresponding to the first quadrant and ends the proof of the
      lemma.
   \end{proof}
      
   \begin{figure}[htbp]
      \begin{center}
         \includegraphics[scale=0.8]{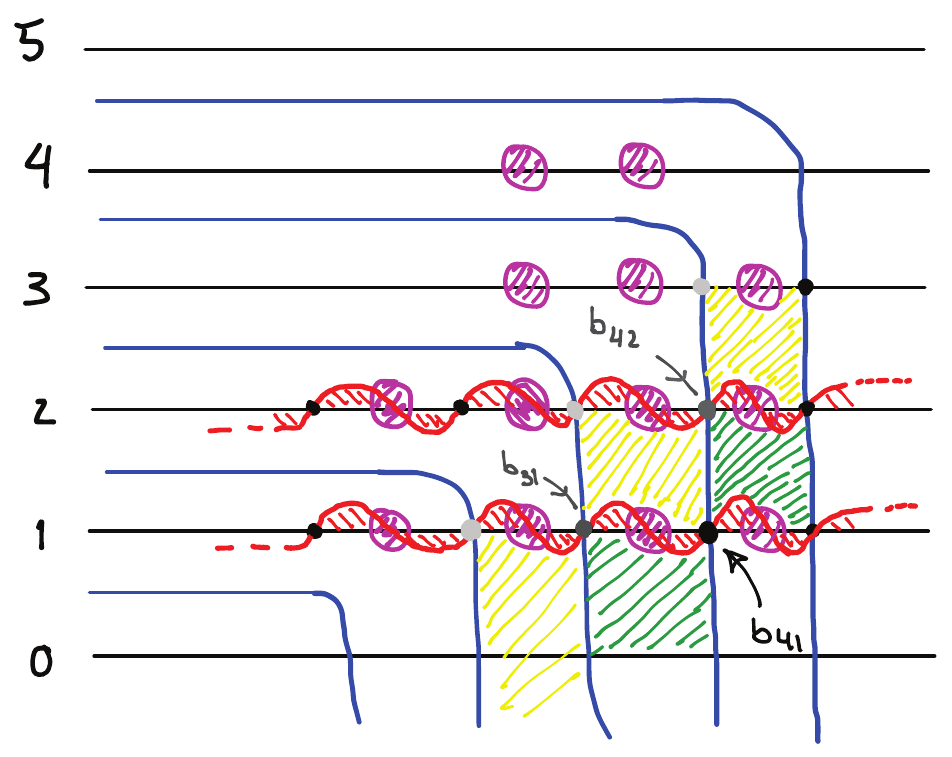}
      \end{center}
      \caption{We take here $s\geq 5$ and in blue are the projections
        of the ends of $W''$.  Assume $v'(m)=b_{41}$ and suppose
        $v'(e)=b_{rs}$ with $r\geq 4$; $v'$ exits $b_{41}$ through one
        of the green regions which is therefore included in
        $Image(v')$; Lemma \ref{lem:quadr-ext} applied to $b_{42}$ and
        $b_{41}$ shows that one of the yellow regions $\subset
        Image(v')$; by applying again Lemma \ref{lem:quadr-ext} to one
        of the upper left corners of the yellow regions - in light
        gray - we get that an unbounded region of $G$ is contained in
        $Image(v')$. Thus, we reach a contradiction in three steps.
        \label{fig:lemma-arg}}
   \end{figure}

   We return to the proof of the proposition and we recall
   $v'(m)=b_{\alpha\beta}$, $v'(e)=b_{rs}$.  Assume that $r>\alpha$.
   As $m$ is an entry point, for orientation reasons, $Image (v')$ has
   to contain at least one of the first or third quadrants at
   $b_{\alpha\beta}$. In both cases, the upper left corner of the
   respective quadrant, that we denote by $b_{i_{1}j_{1}}$, is so that
   $i_{1}\leq \alpha$.  Thus Lemma \ref{lem:quadr-ext} can be applied
   to $b_{i_{1}j_{1}}$ and it implies that the first or third quadrant
   at $b_{i_{1}j_{1}}$ is contained in $Image(v')$. Let
   $b_{i_{2}j_{2}}$ be the upper left corner of the respective
   quadrant. We have $i_{2}\leq i_{1}$.  This process can be pursued
   recursively, thus getting a sequence of points $b_{i_{1}j_{1}},
   b_{i_{2}j_{2}},\ldots $ and associated quadrants $\subset
   Image(v')$ by picking at each step the upper left corner of a
   quadrant obtained from Lemma \ref{lem:quadr-ext} applied to the
   previous point in the sequence. This process continues till one the
   quadrants in question is an unbounded region. But this contradicts
   the fact that the image of $v'$ can not intersect such a region.
   See Figure \ref{fig:lemma-arg} for an illustration of this
   argument.
\end{proof}

% !TEX root = lefcob.tex

\subsection{Disjunction via Dehn twists}
\label{subsec:Dehn-disj}

This subsection is purely geometric in nature and is of independent
interest.  Monotonicity assumptions are not required in this part.
The main purpose here is to show that certain Dehn twists of a
cobordism are Hamiltonian isotopic to remote cobordisms and therefore
can be decomposed by means of Proposition~\ref{lem:decomp-remote}.
The idea is the following. Given a cobordisms $V \subset E$, we first
add specific singularities to $E$ (with critical values in the lower
half plane) so that we can join each initial singularity $x _{i}$ of
$E$ to one of the ``new'' ones, $x'_{i}$, by a matching cycle $S_{i}$.
We then show that, with appropriate choices for the matching cycles
and the other elements of the construction, the iterated Dehn twist
$\tau_{S_{m}}\circ \ldots \circ \tau_{S_{i}}\circ \ldots\circ
\tau_{S_{1}}$ transforms $V$ into a remote cobordism $V'$.

\subsubsection{The case of a single singularity} We start with the
core of the geometric argument.  This appears in the case of a
fibration with a single singularity.

Fix $S\subset M$, a framed (or parametrized) Lagrangian sphere. We use
Seidel's terminology here~\cite{Se:long-exact, Se:book-fukaya-categ}
so that this means $S$ is Lagrangian and that we fix a parametrization
$e:S^{n}\to S$. Consider a Lefschetz fibration $\pi: E\to \C$ which is
tame outside $U\subset \R\times [\frac{1}{4},+\infty)\subset \C$ and
with a single singularity $x_1$ so that the vanishing cycle
corresponding to $x_1$ coincides with $S$. (Note that since there is
only one singularity here there is a canonical hamiltonian isotopy
class of vanishing cycles in the fibers over $\mathbb{C} \setminus
U$.) We will assume that the singularity has critical value $v_{1}=
(1,\frac{3}{2})$.  Fix also a negatively ended cobordism $V \subset E$
with ends $L_{1},L_{2},\ldots, L_{s}$.

For the construction described below it is useful to refer to
Figure~\ref{fig:complex-thimble} (which contains also details that
will be relevant only later on). We will make use of an auxiliary
Lefschetz fibration $\hat{\pi}:\hat{E}\to \C$ that coincides with $E$
over the upper half plane and that has an additional critical point
$x'_{1}$ with corresponding critical value $v'_{1}=(-1,-\frac{3}{2})$
and a matching cycle $\hat{S}_{\gamma}\subset \hat{E}$ that projects
onto $\C$ to a path joining $v'_{1}$ to $v_{1}$. More precisely,
$\hat{E}$ has the following properties.  The fibration $\hat{E}$ is
tame outside a set $\hat{U}$ (as pictured in
Figure~\ref{fig:complex-thimble}),
$\hat{U}\subset (-\infty, a_{\hat{U}}]\times [-K, +\infty)$. Moreover,
let $D$ be a disk around $v'_{1}$ that is included in the lower half
plane but is not completely included in $\hat{U}$.  Let
$v_{0}\in\partial D\setminus \hat{U}$.  Fix also a path $\gamma$ that
joins $v_{1}$ to $v_{0}$.  Denote by $T_{\gamma}$ the thimble
originating at $x_{1}$ and whose planar projection is $\gamma$.  The
boundary of $T_{\gamma}$ is identified to the vanishing cycle $S$ and,
as subset in $\pi^{-1}(v_{0})$, we denote it by $S_{0}$.  The
fibration $\hat{\pi}:\hat{E}\to \C$ is such that it admits the sphere
$S_{0}$ as vanishing cycle also relative to the singularity $x'_{1}$.
If we extend the curve $\gamma$ to a curve (that we will continue to
denote by $\gamma$) that joins $v_{1}$ to $v'_{1}$ this is covered by
a matching cycle $\hat{S}_{\gamma}\subset \hat{E}$. Given that $E$ is
trivial over the lower half-plane, the construction of $\hat{E}$
follows directly from the constructions in \S 16,
\cite{Se:book-fukaya-categ}.

For further use, we now fix another thimble $T$ originating at $x_{1}$
and whose projection is the vertical half-line $\{1\}\times
[\frac{3}{2},\infty)$.

\begin{prop}\label{lem:multiple-surgery} There exists a curve
   $\gamma$, depending on $V$, and a framed Lagrangian sphere $S'$ in
   $\hat{E}$, hamiltonian isotopic to the matching sphere
   $\hat{S}_{\gamma}$ so that the Lagrangian $V'=\tau_{S'} V$ is
   disjoint from $T$ and the intersection $V'\cap S'$ is contained in
   $D$.
\end{prop}
  
\begin{proof}
      
   We start the proof by recalling the definition of the Dehn twist \cite{Ar:monodromy}
   following the conventions in ~\cite{Se:long-exact}. We begin with the model Dehn twist.
   This construction is standard in the subject but as we need the explicit definition in the following we 
   will provide some details here.
   Let $g$ be the standard round metric on $S^n$ and for $0<\lambda$
   denote by $D_{\lambda}^*S^n \subset T^*S^n$ the disk bundle
   consisting of cotangent vectors of norm $\leq \lambda$.  We have
   identified here $T^*S^n$ with $TS^n$ via the metric $g$.  Our
   conventions are such that the symplectic form on the cotangent
   bundle $T^{\ast}S^{n}$ is $dp\wedge dq$ where $q$ is the ``base''
   coordinate and $q$ is the coordinate along the fiber.

   Denote by $\psi_t: D_{\lambda}^*S^n \setminus 0_{S^n}
   \longrightarrow D_{\lambda}^*S^n \setminus 0_{S^n}$ the {\em
     normalized} geodesic flow corresponding to $g$, defined on the
   complement of the zero-section. With our conventions this flow is
   the Hamiltonian flow of the function $H(p,q)=|p|$.

   Denote by $\sigma: S^n \longrightarrow S^n$ the antipodal map. Note
   that $\psi_{\pi}$ extends to the zero-section by $\sigma$.

   Given $0<\lambda$, pick a smooth function $\rho_{\lambda}:
   \mathbb{R} \longrightarrow \mathbb{R}$ with the following
   properties:
   \begin{enumerate}
     \item $\rho(t) + \rho(-t) = 1$ for every $|t| \leq \delta$ for
      some $0< \delta< \lambda$.
     \item $supp(\rho) \subset (-\lambda, \lambda)$; $\rho (t)\geq 0 \
      ,\ \forall \ t>0$.
   \end{enumerate}
   Note that we have $\rho(0) = \tfrac{1}{2}$.

   With the above at hand we define the model Dehn twist $\tau_{S^n}:
   D^*_{\lambda} S^n \longrightarrow D^*_{\lambda} S^n$ by the formula
   \begin{equation} \label{eq:Dehn-twist-model} \tau_{S^n}(x) =
      \begin{cases}
         \psi^g_{2 \pi \rho(||x||)}(x), & x \in T^*_{\leq
           \lambda}S^n \setminus 0_{S^n} ; \\
         \sigma(x), & x \in 0_{S^n}.
      \end{cases}
   \end{equation}
   Note that $\tau_{S^n}$ is the identity near the boundary of
   $D^*_{\lambda} S^n$.

   Now let $N$ be a symplectic manifold and $f: S^n \longrightarrow N$
   a Lagrangian embedding of the $n$-sphere.  Denote by $S = f(S^n)
   \subset N$ its image. By the Darboux-Weinstein theorem there exists
   a neighborhood $U(S)\subset N$ of $S$, a $\lambda>0$, and a
   symplectic diffeomorphism $i: D^*_{\lambda} S^n \longrightarrow
   U(S)$ that maps $0_{S^n}$ to $S$ via the map $f$. Define now the
   Dehn-twist along $S$, $\tau_S: N \longrightarrow N$, by setting
   $\tau_S = i \circ \tau_{S^n} \circ i^{-1}$ on the image of $i$ and
   extend it as the identity to the rest of $N$. By the results
   of~\cite{Se:long-exact} the diffeomorphism $\tau_S$ is symplectic
   and moreover, its symplectic isotopy class is independent of the
   choices of $\rho$ and $\lambda$, but possibly not of the class of
   parametrization of the Lagrangian sphere $f:S^n \longrightarrow S$.
The symplectomorphism $\tau_S$  is {\em the Dehn twist along
     $S$}.
  \begin{rem} \label{r:framing}
 %     \begin{enumerate}
 %       \item The symplectic isotopy class of $\tau_S$ depends only on
 %        the isotopy class of the framed Lagrangian sphere $S$. Now,
 %        isotopy classes of framings (for a fixed $n$-dimensional
 %        Lagrangian sphere $S \subset N$) are in 1-1 correspondence
 %        with $\pi_0(\textnormal{Diff}(S^n)/O(n))$. It follows that
 %        the symplectic isotopy class $\tau_S$ does not depend on the
 %        orientation of $S$. See~\cite{Se:book-fukaya-categ} for more
 %        details.
 %       \item 
 In case $S$ is a vanishing cycle in a Lefschetz
         fibration (associated to a path emanating from a critical
         value in the base of the fibration), $S$ carries a canonical
         isotopy class of parametrizations (or framings)  which we will often adopt
         implicitly. In that case $\tau_S$ is well defined
         up to symplectic isotopy without any further choices.
  %    \end{enumerate}
   \end{rem}

   In the rest of the proof the place of $N$ will be taken by
   $\hat{E}$ and the role of $S$ by the matching cycle
   $\hat{S}_{\gamma}$.

\

To start the actual proof we first assume that, after a possible
Hamiltonian isotopy of $V$, $T$ intersects $V$ transversely in the
points $p_{1},\ldots, p_{k}\in T$. All along the argument it is useful
to refer to Figure~\ref{fig:complex-thimble}.

   \begin{figure}[htbp]
      \begin{center}
         \includegraphics[scale=0.4]{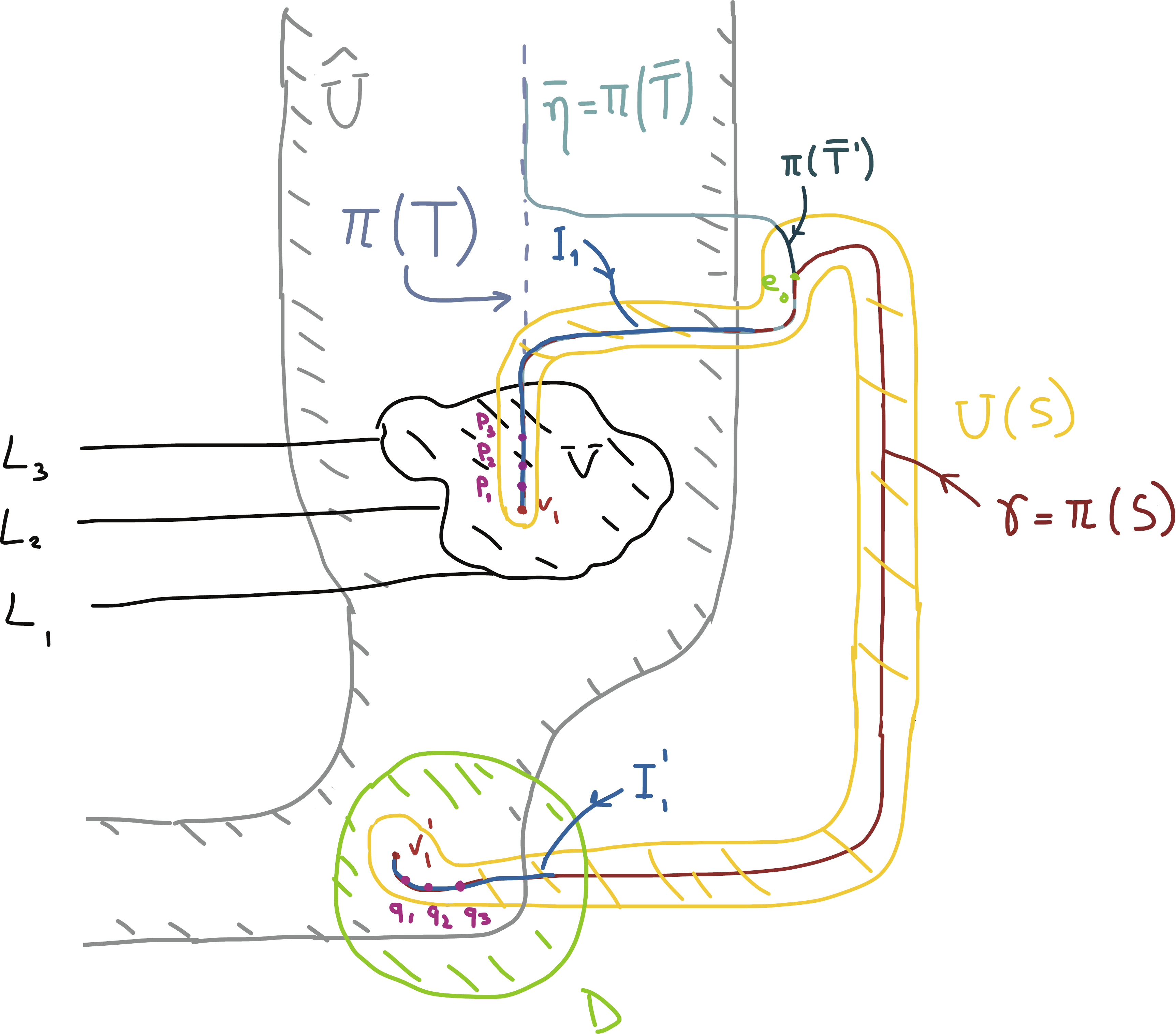}
      \end{center}
      \caption{The Lefschetz fibration $\hat{\pi}:\hat{E}\to \C$
        coincides with $E$ over the upper semi-plane; $\hat{\pi}$ has
        two singularities of critical values $v_{1}$ and $v_{1}'$ and
        is symplectically trivial outside of $\hat{U}$. Are pictured
        (in projection on $\C$): the ``straight'' vertical thimble $T$
        and its deformation $\bar{T}$; the matching cycle $S$ that
        coincides with $\bar{T}$ from $v_{1}$ to $e_{0}$; the disk
        $D$; $S\cap V=\{p_{1}, p_{2}, p_{3}\}$; $q_{i}=\sigma(p_{i})$
        (where $\sigma$ is the antipodal map); the neighborhood $U(S)$
        where is supported $\tau_{S}$; the portion $\bar{T}'$ of
        $\bar{T}$ that differs from $S$ and is included in $U(S)$; the
        projections $I_{1}$, $I_{1}'$ of two disks $K_{1}, K_{1}'$ in
        $S$ around the two singularities of $\hat{\pi}$ so that
        $S_{0}=S\setminus (K_{1}\cup K_{1}')$ lies inside a trivial
        symplectic fibration. Notice that the domain $\hat{U}$ is generally 
        unbounded along some additional
        directions compared to the domain outside which $E$ is tame. 
        This is required so  that the fibration $\hat{E}$, 
        that agrees with $E$ over the upper half plane, has additional 
        singularities compared to $E$. 
        Our choice is for this unbounded direction to be in the lower left corner, 
        as in the picture.
        \label{fig:complex-thimble}}
   \end{figure}

   \noindent \textbf{Step 1:} {\em Choice of the curve $\gamma$}.
   Recall that the fibration $\pi:E\to \C$ is tame outside the set
   $U\subset \C$ and the fibration $\hat{\pi}:\hat{E}\to \C$ is tame
   outside the larger set $\hat{U}$.  We fix two neighborhoods
   $U(V)\subset U'(V)$ of $V$.  We consider an auxiliary thimble
   $\bar{T}$ whose projection on $\C$ is as in
   Figure~\ref{fig:complex-thimble}.  In particular, $\bar{T}$
   coincides with $T$ inside $U(V)$ as well as outside of $U'(V)$ and
   $\pi^{-1}(\C\setminus \hat{U})\cap \bar{T}\not =\emptyset$ but
   $\pi^{-1}(\C\setminus \hat{U})\cap \bar{T}\cap U(V)=\emptyset$.  We
   notice that $\bar{T}$ is hamiltonian isotopic to $T$ by an isotopy
   supported away from $U(V)\cup \pi^{-1}(\R\times (-\infty, 0])$
   ($\bar{T}$ and $T$ are Lagrangian isotopic and it is easy to check
   that this isotopy is exact).
 
   Denote by $\bar{\eta}=\pi(\bar{T})$.  We assume that, as in
   Figure~\ref{fig:complex-thimble}, $\bar{\eta}$ can be written as
   the union of three closed connected sub-segments
   $\bar{\eta}=\bar{\eta}'\cup\bar{\eta}''\cup\bar{\eta}'''$ so that
   $\bar{\eta}'\cup \bar{\eta}'''$ is the closure of $\hat{U}\cap
   \bar{\eta}$.  Thus, the interior of $\bar{\eta}''$ is disjoint from
   $\hat{U}$. We also assume to fix that $\bar{\eta}''\subset
   [1,\infty)\times [1,\infty)$. Consider a point $e_{0}$ inside the
   segment $\bar{\eta}''$ so that $\bar{\eta}''=\bar{\eta}''_{1}\cup
   \bar{\eta}''_{2}$ with $\bar{\eta}''_{1}$ and $\bar{\eta}''_{2}$
   the closures of the two sub-segments given by
   $\bar{\eta}''\setminus \{e_{0}\}$ with $e_{0}$ being the end-point
   of $\bar{\eta}''_{1}$ and the starting point of $\bar{\eta}''_{2}$.
   We now pick the curve $\gamma\subset\C$ that joins $v_{1}$ to
   $v'_{1}$ so that $\gamma$ can be written as a union of two
   connected, closed parts $\gamma=\gamma_{1}\cup \gamma_{2}$ so that
   $\gamma_{1}$ originates in $v_{1}$ and coincides with
   $\bar{\eta}'\cup\bar{\eta}''_{1}$, $\gamma_{2}$ is disjoint from
   $U(V)$, it intersects $\bar{\eta}$ only in $e_{0}$, it ends in
   $v_{1}'$ and $\gamma_{2}\setminus D \subset \C\setminus \hat{U}$.
   Clearly, $e_{0}$ is a point where $\bar{\eta}$ and $\gamma$ are
   tangent and after this point $\gamma$ is to the ``right'' of
   $\bar{\eta}$ and is included in $\C\setminus \hat{U}$ till (and
   including) the moment it reaches $D$.

   Notice that if we show that:
   \begin{equation} \label{eq:reduction}
      \tau_{\hat{S}_{\gamma}}V\cap \bar{T}=\emptyset\ \mathrm{and}\ 
      \tau_{\hat{S}_\gamma}V\cap \hat{S}_{\gamma}\subset D
   \end{equation} 
   then by using the Hamiltonian isotopy $\psi$ that carries $\bar{T}$
   to $T$ and such that $\psi(V) = V$, we deduce that there is a
   Lagrangian sphere $S'=\psi (\hat{S}_{\gamma})$ so that $\tau_{S'}
   V$ is disjoint from $T$ and $\tau_{S'}V\cap S'\subset D$.  For this
   argument, $\tau_{S'}$ is defined by using the choice of framing so
   that $\tau_{S'}^{-1} = \psi \circ \tau^{-1}_{\hat{S}_{\gamma}}
   \circ \psi^{-1}$ (hence $\tau^{-1}_{S'}(V) = \psi \circ
   \tau^{-1}_{\hat{S}_{\gamma}}(V)$).  In short, it remains to
   show~\eqref{eq:reduction}.
     
   \
   
   \noindent \textbf{Step 2:} {\em Other choices involved in the
     definition of the twist.} From now on, to simplify notation, we
   put $S=\hat{S}_{\gamma}$. We first choose a small Weinstein
   neighborhood $U(S)$ of $S$.  The Dehn twist $\tau_{S}$ will be
   supported inside this neighborhood. We notice, by construction,
   that $\{p_{1},\ldots, p_{k}\}=T\cap V = \bar{T}\cap V= S\cap V$. We
   may assume that $V\cap U(S)$ is a union of small disks
   $D_{i}\subset V$ centered at $p_{i}$ which, for convenience, we may
   assume are included in the fiber of $T^{\ast}S$ through $p_{i}$
   under the identification of $U(S)$ with a disk bundle of
   $T^{\ast}S$. Further, we denote by $\bar{T}'$ the closure of
   $(\bar{T} \setminus S) \cap U(S)$. We now consider a disk
   $K_{1}\subset S$ centered at $x_{1}$ so that $U(V)\cap S\subset
   K_{1}$.  Similarly we also consider a disk $K_{1}'\subset S$
   centered at $x_{1}'$.  We assume that both $K_{1}$ and $K_{1}'$ are
   preimages of segments $I_{1}$ and $I_{1}'$ contained in $\gamma$
   and we suppose that the two disks are so that
   $\gamma_{0}=\gamma\setminus (I_{1}\cup I'_{1})\subset \C\setminus
   \hat{U}$, $e_{0}\in \gamma_{0}$ and $I'_{1}\subset D$.  We further
   pick $U(S)$, $K_{1}$ and $K_{1}'$ so that $\bar{T}'$ is disjoint
   from both $K_{1}$ and $K_{1}'$.  We consider the curve oriented so
   that it starts at $v_{1}$ and ends at $v_{1}'$.

   The boundary of $K_{1}$ is a Lagrangian sphere $A\subset
   (M,\omega)$ and the boundary of $K_{1}'$ is the same sphere
   transported to the end of $\gamma_{0}$ (parallel transport is
   trivial along $\gamma_{0}$ because $\hat{\pi}$ is symplectically
   trivial outside $\hat{U}$). We denote the sphere that appears as
   boundary of $K'_{1}$ by $A'$.  The region $S_{0}=S\setminus Int
   (K_{1}\cup K_{1}')$ is diffeomorphic to a cylinder $C=[-a,a]\times
   A$.  We think about this cylinder so that $\{-a\}\times A$
   corresponds to the boundary of $K_{1}$ and $\{a\}\times A$
   corresponds to the boundary of $K'_{1}$.

   Denote by $U(S_{0})$ the restriction of the neighborhood
   $U(S)$ (identified with a disk bundle in $T^{\ast}S$) to $S_{0}$.
   We assume $U(S)$ small enough so that $\pi(U(S_{0}))\subset
   \C\setminus \hat{U}$. As $\hat{\pi}$ is trivial over $U(S_{0})$,
   by possibly reducing $U(S)$ further, we obtain the existence of a
   symplectomorphism:
   $$k: D_{r}T^{\ast}[-a,a]\times D_{r'}T^{\ast}A \to U(S_{0})
   \approx D_{s} T^{\ast}S_{0}\subset \hat{E}~.~$$  After picking $a$
   appropriately, this symplectomorphism can be made also compatible
   with the almost complex structures involved so that
   $\pi'=\hat{\pi}\circ k$ is holomorphic with respect to the split
   standard complex structure in the domain and the standard complex
   structure in $\C$.
 
   \noindent \textbf{Step 3:} {\em The parametrization of $S$.}  This
   step consists in picking a particular framing of $S$ so that the
   associated Dehn twist $\tau_{S}$ can be tracked explicitly. To
   simplify slightly notation we assume $a=1-\delta$ with $\delta$
   very small.

   We fix a diffeomorphism $\varphi: S^{n}\to A$ in the isotopy class
   as explained at point~(2) of Remark~\ref{r:framing}.  Let $h:
   S^{n+1}\to \R$ be the height function defined on the standard round
   sphere in $\R^{n+2}$ and let $S_{\delta}= h^{-1}([-a,a])$.  We now
   pick a parametrization $\alpha: S^{n+1}\to S$ so that the
   restriction of this parametrization to $S_{\delta}$ is a
   diffeomorphism $\alpha_{0}=\alpha|_{S_{\delta}}:S_{\delta}\to C$
   with the property that for each $t\in [-a,a]$,
   $\alpha|_{h^{-1}(t))}: h^{-1}(t)\to \{ t\}\times A\subset C$ is a
   rescaling of $\varphi$, and so that $h(\alpha^{-1}(x_1)) = -1$,
   $h(\alpha^{-1}(x'_1) = 1$ (recall that $x_1, x'_1 \in \hat{E}$ are
   the critical points of $\pi$ lying over $v_1, v'_1$ respectively).
   Clearly, $\alpha_{0}$ extends to a symplectic diffeomorphism
   $\bar{\alpha}_{0}:T^{\ast}S_{\delta}\to T^{\ast} C$ so that
   $T^{\ast} h^{-1}(t)$ is mapped by a symplectomorphism to
   $\{t\}\times T^{\ast}A$.  Basically, we are parametrizing here the
   ``flat'' cylinder $C$ (which is identified with $S_{0}$) by the
   ``round'' cylinder $S_{\delta}$ and we then extend this
   parametrization as symplectomorphisms at the level of the cotangent
   bundles. All the parametrizations involved identify level sets of
   the height function on $S_{\delta}$ to slices of the cylinder $C$.

   We denote by $\sigma: S\to S$ the antipodal map defined using this
   parametrization. This means, in particular, that the points
   $q_{i}=\sigma (p_{i}) $ are contained in $D$ (the disk appearing in
   the statement of the proposition).  It is easy to see, as for instance 
   in~\S 1.2 \cite{Se:long-exact}, with an appropriate choice of function
   $\rho$ in the definition of the Dehn twist (which we have assumed
   here) the intersection $\tau_{S} V\cap S$ is transverse and
   consists precisely of the antipodal of the intersection $S\cap V$.
   Thus, $\tau_{S} V\cap S=\{q_{1},\ldots, q_{k}\}\subset D$ as
   claimed in the second part of (\ref{eq:reduction}).  It remains to
   show the main part of the claim: $\tau_{S}V\cap \bar{T}=\emptyset$.
   As $\tau_{S}V\cap S=\{q_1, \ldots, q_k\}$, the Dehn twist
   $\tau_{S}$ is supported inside $U(S)$ and given that $\bar{T}$ and
   $S$ coincide along the segment of $\gamma$ that starts at $v_{1}$
   and ends at $e_{0}$ it follows that
   \begin{equation} \label{eq:inters}
      \tau_{S}V\cap \bar{T}=\tau_{S}V\cap \bar{T}'=
      \tau_{S}(V\cap \tau_{S}^{-1}(\bar{T}'))
   \end{equation}
   Thus, to conclude the proof, it is enough to show
   $\tau_{S}^{-1}(\bar{T}')\cap V=\emptyset$.
  
   \

   \noindent \textbf{Step 4:} {\em Showing $\tau_{S}^{-1}(\bar{T}')\cap
     V=\emptyset$.} By possibly adjusting the neighborhood $U(S)$ we
   may assume that $U$ can be written as $U(S)=(k\circ
   \bar{\alpha}_{0})(U(S^{n+1}))$ for some neighborhood $U(S^{n+1})$
   of the zero section inside $T^{\ast}S_{\delta}$. Let $\tilde{T}'=
   (k\circ \bar{\alpha}_{0})^{-1}(\bar{T}')$.  We denote by
   $U(S_{\delta})$ the corresponding neighborhood of $S_{\delta}$ (so
   that $U(S_{\delta})$ is the preimage of $U(S_{0})$) and we let
   $\tilde{K}_{1}$ be the cap $\tilde{K}_{1}=h^{-1}(-1,-1+\delta]=
   (k\circ \bar{\alpha}_{0})^{-1}(K_{1})$. Further, we let
   $U(\tilde{K}_{1})$ be the restriction of $U(S^{n+1})$ over
   $\tilde{K}_{1}$.  Clearly $\tilde{T}'\subset U(S_{\delta})$, and to
   show the claim it is enough to notice that
   $\tau_{S^{n+1}}^{-1}\tilde{T}'\cap U(\tilde{K}_{1})=\emptyset$ where now
   $\tau_{S^{n+1}}$ is the standard model for the Dehn twist.
 
   Let $(x,v)\in \tilde{T}'\subset T^{\ast}S_{\delta}$ with $ v\in
   T^{\ast}_{x}S^{n+1}$, $v\not=0$.  We now notice that the condition
   that $\bar{T}'$ is to the ``left'' of $S$ in
   Figure~\ref{fig:complex-thimble} 
    translates to the fact that
   \begin{equation}\label{eq:inequality}
      \langle   v,  J \nabla h (x) \rangle > 0 ~.~
   \end{equation} 
   Here $J$ is an almost complex structure on $T^{\ast}S_{\delta}$
   with respect to which, as at Step 2, the map
   $\pi'=\hat{\pi}\circ k$ is holomorphic.  This follows from the same
   inequality that is valid for the planar projection of $\bar{T}'$
   relative to $\gamma_{0}$.  Equation (\ref{eq:inequality}) implies
   that the geodesic flow with origin $(x,v)$ has its vertical
   component pointing in the direction of $-\nabla h$ (because if
   $\langle v, w\rangle >0$, then the geodesic associated to $v$
   points in the direction of $Jw$). Thus, the inverse of the geodesic
   flow points in the direction of $\nabla h$ and therefore away from
   $\tilde{K}_{1}$. As a consequence, it is easy to see that the orbit
   $\phi_{t}^{g}(x,v)$ for $-\pi\leq t\leq 0$ does not intersect
   $U(\tilde{K}_{1})$ and, as a consequence,
   $\tau_{S}^{-1}(\bar{T}')\cap V=\emptyset$ - see also
   Figure~\ref{fig:geodesic}.
\end{proof}

\begin{figure}[htbp]
  \begin{center}
    \includegraphics[scale=0.35]{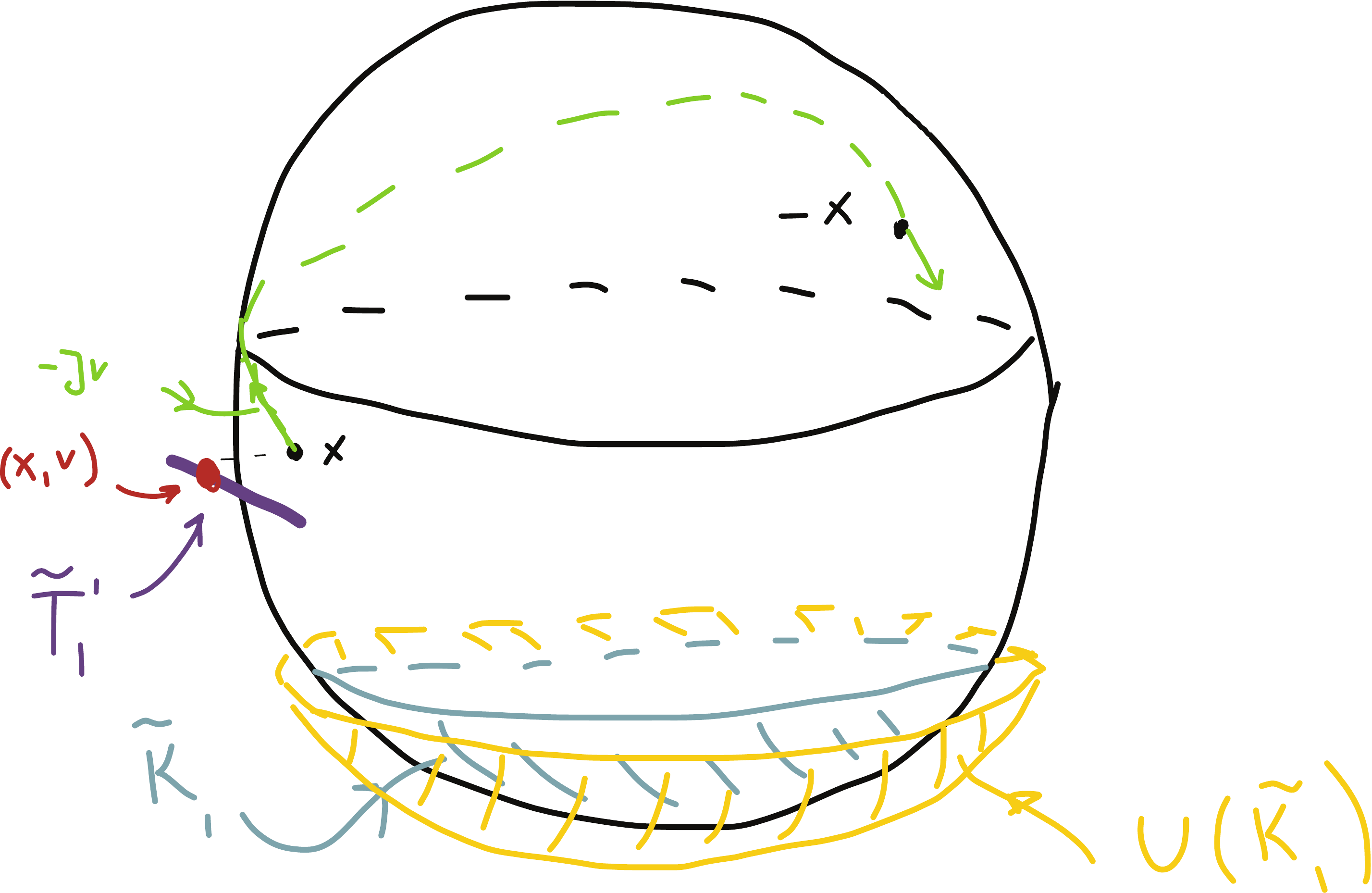}
  \end{center}
  \caption{The cap $\tilde{K}_{1}\subset S^{n+1}$ the set
    $\tilde{T}'_{1}$ containing the point $(x,v)$ together with the geodesic 
starting from $x$ in the direction of $-Jv$ and ending at
    $-x$. \label{fig:geodesic}}
\end{figure}

\begin{cor}\label{cor:twist-remote}
   With the notation in Proposition~\ref{lem:multiple-surgery} the
   cobordism $\tau_{S'}V$ is hamiltonian isotopic - via an isotopy
   with compact support - to a cobordism that is remote relative to
   $E$. \end{cor}
\begin{proof}
   We already know from Proposition \ref{lem:multiple-surgery} that
   $V'=\tau_{S'}V$ is disjoint from $T$. Consider an
   $\Omega$-compatible almost complex structure $J$ on $E$ with the
   additional property that $\pi: E \longrightarrow \mathbb{C}$ is
   $J$-holomorphic. It is well known that the function $Im(\pi):E\to
   \R$ defines a Morse function on $E$ whose negative gradient flow
   $\xi$ (with respect to the metric induced by $(\Omega, J)$) is also
   Hamiltonian. Moreover $\xi$ has the thimble $T$ as a stable
   manifold. Write $\xi =X^{H}$ with $H:E\to \R$. Now consider a
   smooth function $\eta:\C\to \R$ so that $\eta(z)=1$ if $z\in
   [-a_{U}-1, a_{U}+1]\times [-\frac{1}{4}, +\infty)$ and $\eta(z)=0$
   if $z\in ((-\infty, -a_{U}-2] \times \R) \cup ([-a_{U}-2,
   a_{U}+2]\times (-\infty, -\frac{1}{2}]) \cup ([a_{U}+2,
   \infty)\times \R)$.  Let $\xi'$ be the Hamiltonian flow of the
   function $(\eta \circ \pi) H$ defined on $\hat{E}$.  It is easy to
   see that, after sufficient time, the flow $\xi'$ isotopes $V'$ to a
   new cobordism $V''$ that is included in $\hat{\pi}^{-1}(\R\times
   (-\infty,0]\times \R \cup Q_{U}^{-})$.  Therefore, $V''$ is remote
   relative to $E$. Moreover, as the ends of $V'$ are not moved by
   this isotopy, it is easy to see that, by a further truncation of
   $\xi'$, $V''$ is hamiltonian isotopic to $V'$ through a compactly
   supported isotopy.
\end{proof}
 
 \subsubsection{Multiple singularities} 
\label{subsubsec:null-cob-remote}

Consider a Lefschetz fibration $\pi:E\to \C$ as
in~\S\ref{subsec:main-decomp}, thus possibly with more than one
singularity.

We fix $V\in \mathcal{O}b (\fk^{\ast}(E))$, $V:\emptyset\cobto
(L_{1},\ldots, L_{s})$.  The purpose of this subsection is to describe
an extension of Proposition~\ref{lem:multiple-surgery} and
Corollary~\ref{cor:twist-remote} to the case of multiple
singularities.

We will consider a fibration $\hat{\pi}:\hat{E}\to \C$ that extends
$E$ and has one more singularity $x'_{i}$ for each singular point
$x_{i}$, $1\leq i\leq m$, of $\pi$ so that the vanishing cycles of
$x_{i}$ and $x'_{i}$ can be related by matching cycles $\hat{S}_{i}$
that are the analogues of the matching cycle $\hat{S}_{\gamma}$ from
Proposition \ref{lem:multiple-surgery}. The specific positioning of
the corresponding critical values $v'_{i}$ in the plane $\C$ is
important as is as in Figure~\ref{fig:matching-cycles}.  We then
obtain Lagrangian spheres, $S'_{i}$ that are hamiltonian isotopic to
$\hat{S}_{i}$ (as in Figure~\ref{fig:matching-cycles}) and we then
consider the image of $V$ under the iterated Dehn twist
$$V'=\tau_{\hat{S}_{m}} \circ \tau_{\hat{S}_{m-1}} 
\circ \cdots \circ \tau_{\hat{S}_{1}} (V)$$ inside $\hat{E}$ as
well as the following Hamiltonian isotopic copy of it
$V''=\tau_{{S}_{m}'}\circ \tau_{{S}_{m-1}'} \circ \cdots
\circ \tau_{{S}_{1}'}(V)$ obtained by applying an iterated Dehn
twist along the Lagrangian spheres $S'_j$ which are Hamiltonian
isotopic to the $\hat{S}_j$'s.

Let $\thmb_i$ be the vertical thimble with origin the critical point
$x_{i}$ and projecting to the vertical half-line $\{i\}\times
[\frac{3}{2},\infty)$. The thimbles $\thmb_i$ generalize the thimble
$T$ considered earlier (just before
Proposition~\ref{lem:multiple-surgery}) in the context of one
singularity to the case of multiple singularities. We denote them by
$\thmb_i$ (this avoids confusion with the thimbles $T_{i}$ that are
horizontal at infinity and are associated to the curves $t_i$, see
Figure~\ref{fig:spec-curves}).
 
\begin{figure}[htbp]
  \begin{center}
    \includegraphics[scale=0.4]{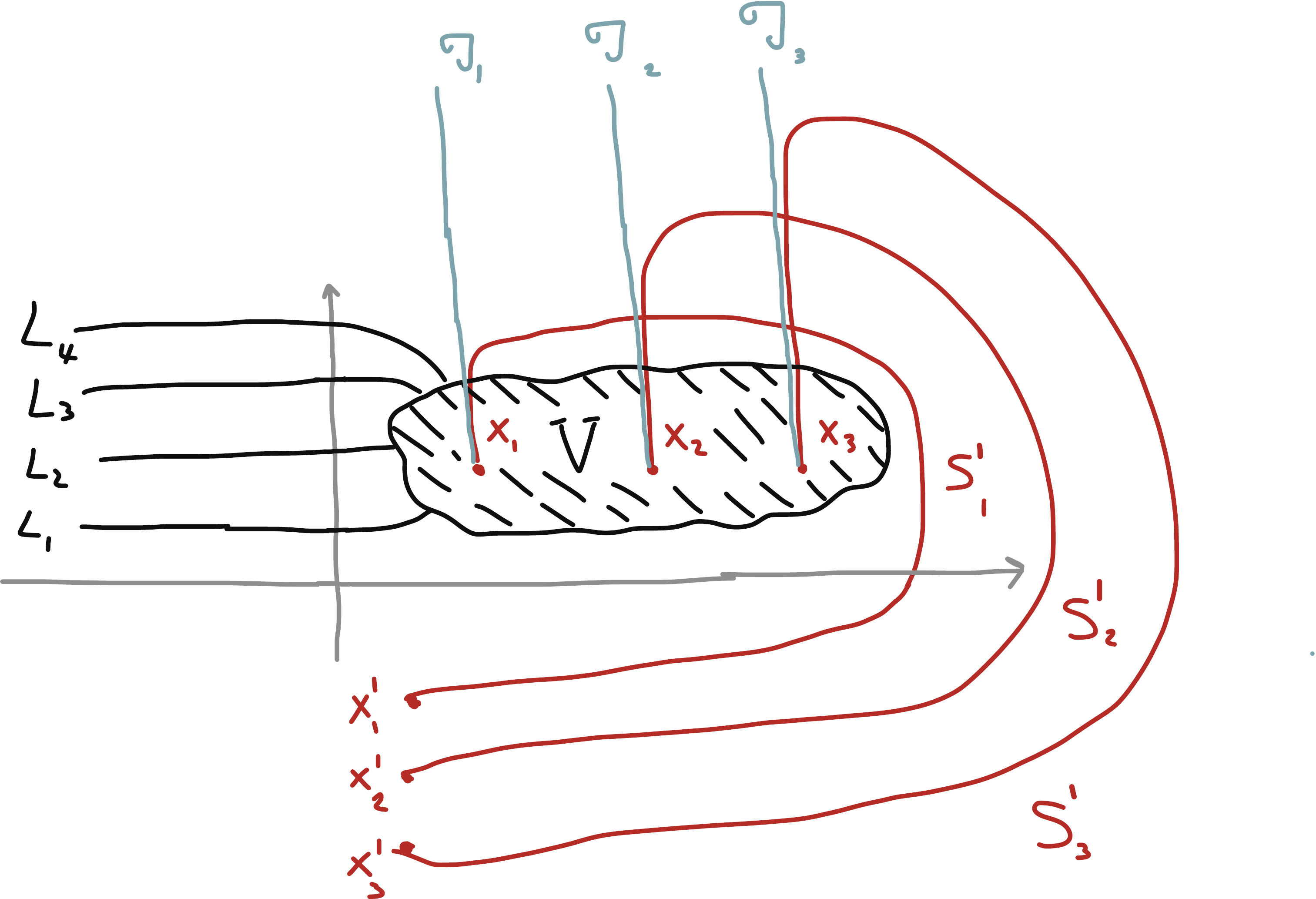}
  \end{center}
  \caption{The cobordism $V:\emptyset\cobto (L_{1},L_{2}, L_{3},
    L_{4})$, the Lagrangian spheres $S_{1}', S_{2}', S_{3}'$ together
    with the vertical thimbles $\thmb_{1}, \thmb_{2}, \thmb_{3}$ so
    that $V''=\tau_{S_{m}'}\circ \tau_{S_{m-1}'} \circ
    \cdots \circ \tau_{S_{1}'}(V)$ is disjoint from the
    $\thmb_{i}$'s.
   \label{fig:matching-cycles}
 }
\end{figure}

\begin{cor}\label{cor:moving-cob} 
  It is possible to construct $\hat{E}$ and the Lagrangian spheres
  $S'_{i}$ so that the cobordism $V''$ is disjoint from all the
  thimbles $\thmb_{i}$.  As a consequence, there exists a horizontal
  Hamiltonian isotopy $\phi$ so that the cobordism $\phi(V'')\subset
  \hat{E}$ is remote relative to $E$. In particular, in
  $D\fk^{\ast}(E)$, there exists a cone decomposition:
  $$V_{E}' \cong (\gamma_{s}\times L_{s}\to \gamma_{s-1}\times L_{s-1}
  \to\ldots \to\gamma_{2}\times L_{2})~.~$$
\end{cor}
\begin{proof} The first part of the proof is to construct iteratively
   fibrations $\hat{\pi}_{i}:\hat{E}_{i}\to \C$ with $\hat{E}_{0}=E$
   and with the final fibration $\hat{E}=\hat{E}_{m}$ so that
   $\hat{E}_{i+1}$ extends $\hat{E}_{i}$ and has one more singularity,
   $x'_{i+1}$, compared to $\hat{E}_{i}$. At each step we also
   construct the matching cycles $\hat{S}_{i}$ joining $x_{i}$ to
   $x_{i}'$ and their Hamiltonian isotopic images $S'_{i}$ so that the
   relevant properties are satisfied. Here are more details on the
   induction step. Assume that $\hat{E}_{k}$ has already been
   constructed together with the matching cycles $\hat{S}_{i}$ and
   their hamiltonian isotopic copies $S'_{i}$, $1\leq i\leq k$ so that
   $V''_{k}=\tau_{S_{k}'}\circ \tau_{S_{k-1}'}\circ \cdots \circ
   \tau_{S_{1}'} (V)$ is disjoint from $\thmb_{i}$, $1\leq i\leq k$.
   We now consider the cobordism $V''_{k}$ and the vertical thimble
   $\thmb_{k+1}$ and we apply to them the construction described in
   the proof of Proposition~\ref{lem:multiple-surgery}.  This produces
   first a new fibration $\hat{E}_{k+1}$ that has an additional
   singularity denoted now by $x'_{k+1}$.  Here, the only difference
   with respect to the construction of $\hat{E}$ in
   Proposition~\ref{lem:multiple-surgery} is that the coordinates of
   the critical value $v'_{k+1}$ associated to $x'_{k+1}$ is now $(-1,
   -k-\frac{3}{2})$ and the set $\hat{U}$, outside which
   $\hat{E}_{k+1}$ is tame, is extended appropriately inside the
   third-quadrant.  Further, just as in the proof of
   Proposition~\ref{lem:multiple-surgery} we can construct the
   deformed thimble $\barthmb_{k+1}$ as well as the matching cycle
   $\hat{S}_{\gamma}$ so that $\hat{S}_{\gamma}$ coincides with
   $\barthmb_{k+1}$ over a certain sub-segment of $\gamma$. Two
   important points should be made here: first, the place of $V$ in
   the proof of Proposition~\ref{lem:multiple-surgery} is taken here
   by $V''_{k+1}$; second $\thmb_{k+1}$ as well as $\barthmb_{k+1}$
   and $\hat{S}_{\gamma}$ are all disjoint from $\thmb_{i}$ for $i\leq
   k$.  Now, again as in the proof of
   Proposition~\ref{lem:multiple-surgery}, we obtain that there exists
   a hamiltonian isotopy $\psi_{k+1}$ supported outside a neighborhood
   of $V''_{k+1}$ so that $S'_{k+1}=\psi_{k+1}(\hat{S}_{\gamma})$ has
   the property that $V''_{k+1}=\tau_{S'_{k+1}}V''_{k}$ is disjoint
   from $\thmb_{k+1}$.  One additional point appears here: it is easy
   to see that the isotopy $\psi_{k+1}$ can be assumed to leave fixed
   $\thmb_{i}$ for $i\leq k$. By defining $V''_{k+1}$ by using a
   sufficiently small neighborhood $U(S'_{k+1})$ of $S'_{k+1}$ so that
   $U(S'_{k+1})\cap \thmb_{i}= \emptyset$ for all $i\leq k$, we also
   deduce $V''_{k+1}\cap \thmb_{i}=\emptyset\ 1\leq i\leq k$ and the
   induction step is completed.

   We now put $V''=V''_{m}$ and we know that $V''$ is disjoint from
   all the thimbles $\thmb_{i}$. Constructing the horizontal isotopy
   that transforms $V''$ into a cobordism $V'''$ remote relative to
   $E$ is a simple exercise by, possibly, iterating the construction
   in Corollary~\ref{cor:twist-remote}.

   Finally, the cone-decomposition in the statement follows by
   applying to $V'''$ Proposition~\ref{lem:decomp-remote}.
\end{proof}

\pbhl{The following proposition establishes monotonicity properties
  for $\hat{E}$ that will be used later on}
  in~\S\ref{subsec:prof-main-t} when proving
  Theorems~\ref{thm:main-dec} and~\ref{thm:main-dec-gen0}.

\begin{prop} \label{p:strong-mon-Ehat} If the Lefschetz fibration
  $E \longrightarrow \mathbb{C}$ is strongly monotone (see
  Definition~\ref{df:monlef}) then the extended fibration
  $\hat{E} \longrightarrow \mathbb{C}$ is strongly monotone too and
  has the same monotonicity class $*$. The matching spheres
  $\hat{S_j} \subset \hat{E}$ are monotone of class $*$ and if the
  cobordism $V \subset E$ is monotone of class $*$ then it continues
  to be monotone of the same class when viewed as a cobordism in
  $\hat{E}$.
\end{prop}
\begin{proof}
  Denote by $M$ the generic fiber of $E$. Assume first that
  $\dim_{\mathbb{R}}M \geq 4$. By Remark~\ref{r:S_k-monot} $M$ is
  monotone. Denote for every $1\leq j\leq m$ by $\lambda_j$ the path
  connecting $x_j$ to $x'_j$ over which the matching cycle $\hat{S}_j$
  was constructed, as in Figure~\ref{fig:matching-cycles}. Pick a
  point $p_j$ on $\lambda_j$ in such a way that all the points
  $p_1, \ldots, p_m$ are in the upper half-plane and all of them lie
  in one of the domain where $\hat{E}$ is tame. Divide each of the
  $\lambda_j$ into two parts: $\lambda_j^+$ going from $x_j$ to $p_j$
  and $\lambda_j^{-}$ that goes (in the opposite orientation to
  $\lambda_j$) from $x'_j$ to $p_j$. Since $\hat{S}_j$ is a matching
  cycle, the two vanishing spheres in $E_{p_j} = \pi^{-1}(p_j)$
  associated to the paths $\lambda^{+}_j$ and $\lambda^{-}_j$
  coincide. It follows that if Case~(ii) in Definition~\ref{df:monlef}
  is applicable then it is satisfied also for the fibration
  $\hat{E}$. This proves that $\hat{E}$ is strongly monotone under the
  assumption that $\dim_{\mathbb{R}}M \geq 4$. It is not hard to see
  that its monotonicity class $*$ is the same as the one of $*$. That
  $V$ remains monotone when viewed in $\hat{E}$ follows easily from
  the fact that when $\dim_{\mathbb{R}}M \geq 4$ the map induced by
  the inclusion $\pi_2(E, V) \longrightarrow \pi_2(\hat{E}, V)$ is
  surjective.

  The statement about the matching spheres will be proved below, at
  the present proof, as it does not require any assumptions on the
  dimension of $M$.

  We now turn to the case $\dim_{\mathbb{R}}M = 2$. Recall that in
  this case strong monotonicity assumes that $E$ itself is a monotone
  manifold. We will first determine the homotopy type of $E$ and that
  of $\hat{E}$.  Consider the complement (in $\mathbb{C}$) of the
  union of curves $\cup_{j=1}^m \lambda_j$. This has several unbounded
  connected components and several bounded ones (unless $m=1,2$, when
  there are only unbounded ones). Denote by
  $\mathcal{B} \subset \mathbb{C}$ the closure of the union of the
  bounded components. If $m=1$ take $\mathcal{B}$ to be just a point
  on $\lambda_1$ in the upper half-plane which is not $x_1$ or $x'_1$
  and if $m=2$ take $\mathcal{B}= \lambda_1 \cap \lambda_2$. Put
  $\hat{E}_{\lambda, \mathcal{B}} = \hat{\pi}^{-1}(\cup_{j=1}^m
  \lambda_j \cup \mathcal{B})$. Then the inclusion
  $\hat{E}_{\lambda, \mathcal{B}} \longrightarrow \hat{E}$ is a
  homotopy equivalence.

  Denote by $l^+_j \subset \lambda_j$ the part of $\lambda_j$ that
  starts from $x_j$ till the point where it enters $\mathcal{B}$, and
  by $l^{-}_j$ the path starting at $x'_j$ and goes along $\lambda_j$,
  with the reverse orientation, till the point it hits the domain
  $\mathcal{B}$. Put
  $E_{l^+, \mathcal{B}} = \pi^{-1}(\cup_{j=1}^m l^+_j \cup
  \mathcal{B})$. The inclusion
  $E_{l^+, \mathcal{B}} \longrightarrow E$ is a homotopy equivalence
  too.

  Consider now the following subspaces:
  $$\hat{E}^0 = 
  E|_{\mathcal{B}} \cup (\cup_{j=1}^m T_{l^{+}_j}) \cup (\cup_{j=1}^m
  T_{l^{-}_j}) \subset \hat{E}_{\lambda, \mathcal{B}}, \quad E^0 =
  E|_{\mathcal{B}} \cup (\cup_{j=1}^m T_{l^{+}_j}) \subset E_{l^+,
    \mathcal{B}},$$ where $T_{l^+_j}$ is the thimble associated to
  $l^+_j$ and similarly for $T_{l^{-}_j}$. Thus $\hat{E}^0$ is
  obtained from $E_{\mathcal{B}}$ by attaching to it $m$ pairs of
  $(n+1)$-dimensional balls by identifying their boundaries with
  vanishing spheres of some fibers of $E$ over $\mathcal{B}$. The
  space $E^0$ has an analogous description, by using only the
  $T_{l^+_j}$'s. Note also that
  $$\hat{E}^0 = E|_{\mathcal{B}} \cup (\cup_{j=1}^m \hat{S}_j).$$

  By standard arguments from Morse theory the inclusions
  $\hat{E}^0 \longrightarrow \hat{E}_{\lambda, \mathcal{B}}$ and
  $E_0 \longrightarrow E_{l^+, \mathcal{B}}$ are homotopy
  equivalences.
  
  We are now ready to show that $\hat{E}$ is a monotone symplectic
  manifold. For a space $X$ we denote by
  $H_2^S(X) = \textnormal{image\,} (\pi_2(X) \longrightarrow H_2(X))$
  the image of the Hurewicz homomorphism.  Denote by
  $j: E^0 \longrightarrow \hat{E}^0$ the inclusion and by $j_*$ its
  induced map on $H_2^S$. Since $\mathcal{B}$ is contractible, it is
  easy to see that $H^S_2(\hat{E^0})$ is generated by
  $\textnormal{image\,}(j_*)$ together with the classes
  $[\hat{S}_1], \ldots, [\hat{S}_m]$. As $E$ is assumed to be monotone
  and $\hat{S_j}$ are Lagrangian it readily follows that $\hat{E}$ is
  monotone too.

  The monotonicity of $V \subset \hat{E}$ can be proved by similar
  methods.  For a pair of spaces $Y \subset X$ put
  $H_2^D(X,Y) = \textnormal{image\,} (\pi_2(X,Y) \longrightarrow
  H_2(X,Y))$. A similar argument to the preceding one combined with
  the homotopy long exact sequence of the triple $(\hat{E}, E , V)$
  shows that $H_2^D(\hat{E}, V)$ is generated by
  $\textnormal{image\,}(i_*)$ together with
  $[\hat{S}_1], \ldots, [\hat{S}_m]$, where $i_*$ is the map induced
  by the inclusion $(E, V) \to (\hat{E}, V)$ and the $[S_j]$'s are
  viewed as elements of $H_2^D(\hat{E}, V)$ via the map
  $H_2^S \to H_2^D$. As before, since $[S_j]$ are Lagrangian it
  follows that $V \subset \hat{E}$ remains monotone. Moreover, it is
  easy to see that its monotonicity class $*$ remains unchanged.

  Finally, we prove the statement about the matching spheres. The
  argument below works for $M$ of arbitrary positive dimension.  Let
  $\hat{S}$ be one of the matching spheres $\hat{S}_j$. Since
  $\hat{S}$ is simply connected (recall that $\dim M > 0$) and
  $\hat{E}$ is monotone it follows that $\hat{S}$ is monotone
  too. Moreover, if the monotonicity constant of $E$ satisfies
  $\rho>0$, then $\hat{S}$ will have the same constant. 

  It remains to show that $d_{\hat{S}} = d_E$. Recall that
  $d_{\hat{S}}$ counts the number of pseudo-holomorphic disks in
  $\hat{E}$ with boundary on $\hat{S}$ that go through a given point
  in $\hat{S}$. Pick a point $p \in \hat{S}$ such that its projection
  $z=\pi(p)$ belongs to the upper half-plane and is in a region where
  both $E$ and $\hat{E}$ are tame. Denote by
  $\mathcal{U}\subset \mathbb{C}$ the domain over which $\hat{E}$ is
  tame. Let $\hat{J}$ be an almost complex structure on $\hat{E}$,
  compatible with the symplectic structure and such that
  $\pi: \hat{E} \longrightarrow \mathbb{C}$ is $\hat{J}$-holomorphic
  above $\mathcal{U}$. Standard arguments show that the class of such
  almost complex structures contain regular ones and therefore one can
  calculate $d_{\hat{S}}$ using such a $\hat{J}$.

  Let $u:(D, \partial D) \longrightarrow (\hat{E}, \hat{S})$ be a
  $\hat{J}$-holomorphic disk with $u(\partial D) \ni p$. Let
  $v = \pi \circ u: D \longrightarrow \mathbb{C}$ be the projection of
  $u$ to $\mathbb{C}$. We claim that $v$ is constant, hence the image
  of $u$ is in the fiber $E_z \cong M$. To prove this, suppose by
  contradiction that $v$ is not constant.  We have
  $v(\partial D) \subset \lambda$, where $\lambda = \pi(\hat{S})$ is a
  curve (connecting two critical values $x_j$ and $x'_j$ of
  $\pi$). Note that $v$ is holomorphic on
  $\mathcal{H} := v^{-1}(\mathcal{U})$. Let $\xi \in \partial D$ be a
  point such that $v(\xi) = z$. Clearly $\xi \in \mathcal{H}$. Without
  loss of generality we may assume that $z$ is not a critical value of
  $v$ (otherwise, move $z$ slightly to a nearby point on $\lambda$
  which is still in the image of $v$ and which is a regular value of
  $v$). By the open mapping theorem it is impossible for $v(D)$ to
  intersect the part of $\mathcal{U}$ that is on the right-hand side
  of $\lambda$. Thus in a neighborhood of $z$, the image $v(D)$ must
  be on the left-hand side of $\lambda$. Since $v$ is holomorphic near
  $\xi$ it follows that when we go along $\partial D$ counterclockwise
  through $\xi$, the image of $v$ goes along $\lambda$ in the upper
  direction. This holds for all points $\xi \in v^{-1}(z)$. But this
  is impossible since $\lambda$ is not a closed curve, so there must
  be another point $\xi' \in \partial D$ with $v(\xi') = z$ and such
  that when we go counterclockwise along $\partial D$ around $\xi'$
  the image of $v$ goes in the lower direction of $\lambda$. A
  contradiction. This proves that $v$ is constant, hence
  $\textnormal{image\,}u \subset E_z$. We thus conclude that
  $d_{\hat{S}} = d_{S}$, where $S \subset E_z$ is the vanishing sphere
  corresponding to the matching sphere $\hat{S}$. It now easily follows
  that $d_{\hat{S}} = d_E$.
\end{proof}

\subsubsection{Dehn twist as multiple surgery}
\label{subsubsec:multipl-surg} 

Here we give an interpretation of the action of a Dehn twist on
Lagrangian submanifolds in terms of surgery. Fix $S^n \longrightarrow
S\subset M$, a parametrized Lagrangian sphere and let $L$ be another
Lagrangian submanifold of $(M,\omega)$. It is know that if $L$ and $S$
intersect transversely and in a single point, then Lagrangian surgery
at this point produces a Lagrangian $S\# L$ that is Hamiltonian
isotopic to the Dehn twist $\tau_{S}L$ of $L$ along $S$ (see
e.g.~\cite{Se:knotted, Th-RP:moment}). (See~\cite{Po:surgery} as well
as~\cite{La-Si:Lag} for the definition of Lagrangian surgery, and see
below for our conventions regarding the choice of handles in the
surgery).  Assume now that $L$ is still transverse to $S$ but that the
number of intersection points $L\cap S$ is more than one. In this case
too, one can express the Dehn twist $\tau_{S}(L)$ as a certain type of
surgery.  The construction goes as follows. Assume that $L\cap
S=\{p_{1},\ldots, p_{r}\}$. Fix an additional point $p_{0}\in S$ and a
small neighborhood of it $V\subset S$.

\begin{itemize}
  \item[i.] Consider $r$ hamiltonian diffeomorphisms $\phi^{j}$,
   $1\leq j\leq r$ supported in a small Weinstein neighborhood of $S$,
   so that $S_{j}=\phi^{j}(S)$ is transverse to $S$ and $S_{j}\cap S=
   \{p_{j}, p'_{j}\}$ for some additional point $p'_{j}\in V$.
  \item[ii.]  Pick small disks $D^L_{j}\subset L$ centered at $p_{j}$
   and disks $D^{S_j}_{j}\subset S_{j}$ also centered at $p_{j}$ as
   well as Lagrangian handles $H_{j}\subset M$ defined in a small
   neighborhood of $p_{j}$ that join $S_{j}$ to $L$ so that $
   (L\setminus D^L_{j}) \cup (S_{j}\setminus D^{S_j}_{j})\cup H_{j}$
   is the usual Lagrangian surgery $L\# S_{j}$ between $L$ and $S_{j}$
   at the point $p_{j}$ (this is, in general, an immersed Lagrangian).
   Notice that there are two choices for Lagrangian surgery at each
   intersection point. The choice used here is the same at each point
   and is the one defined as follows (this is the same convention as
   in \cite{Bi-Co:cob1}).  The sphere $S$ is oriented hence so are the
   $S_j$'s.  This induces a local orientation on $L$ (even if $L$ is
   not orientable) near each intersection point $p_j$ in such a way
   that $T_{p_j}S_{j}\oplus T_{p_j}L$ gives the orientation of
   $T_{p_j}M$.  We then symplectically identify a neighborhood of $p_j
   \in M$ with a neighborhood of $0$ in $\mathbb{R}^{2n}$ in such a
   way that $D^{S_j}_j$ is identified with a small disk around $0$ in
   $\mathbb{R}^n\times \{0\}$ and $D^{L}_j$ with a small disk around
   $0$ in $\{0\}\times \mathbb{R}^n $, with the last two
   identifications being orientation preserving. The model Lagrangian
   handle is then defined to be $H_j = \cup_{t \in [-1,1]} \gamma(t)
   S^{n-1} \in \mathbb{C}^n \cong \mathbb{R}^{2n}$, where
   $\gamma(t):[-1,1] \longrightarrow \mathbb{C}$ is an appropriately
   chosen curve whose image is in the 2'nd quadrant and such that
   $\gamma(t) \in \mathbb{R}_{<0}$ for $t$ close to $-1$ and
   $\gamma(t) \in i\mathbb{R}_{>0}$ for $t$ close $1$.
          
\item[iii.] Define $ S\#_{r} L$ by
  \begin{equation}\label{eq:multipl-surg}
    S \#_{r} L = (\cup_{j} S_{j}\setminus D^{S_j}_{j}) \cup
      (L\setminus \cup_{j} D^L_{j}) \cup (\cup_{j} H_{j})~.~
  \end{equation}
  In other words $S\#_{r} L$ is obtained by performing simultaneously,
  for all $1\leq j\leq r$, the one point surgery at $p_{j}$ between
$S_{j}$ and $L$.
\end{itemize}

\begin{figure}[htbp]
   \begin{center}
      \includegraphics[scale=0.9]{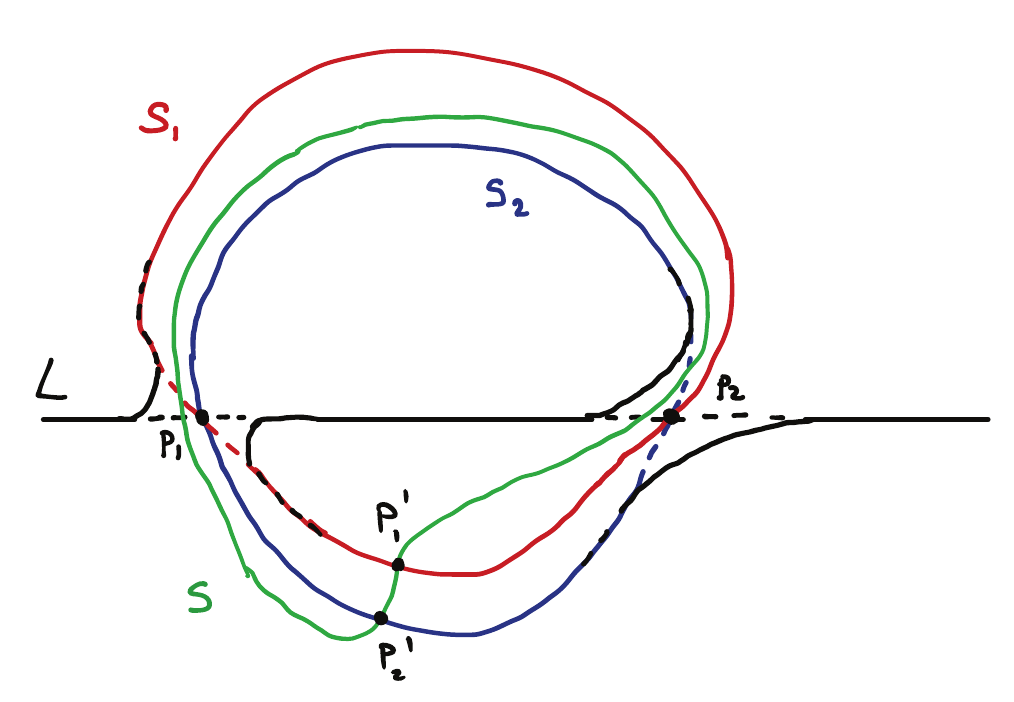}
   \end{center}
   \caption{Dehn twist as multiple surgery for $n=1$ 
     assuming two intersection points $p_{1},p_{2}$
    between $L$ and $S$.
    \label{fig:multiple-surgery}}
\end{figure}

Either by a direct argument - this is instructive to draw in dimension
two as in Figure~\ref{fig:multiple-surgery} - or by comparing this
multiple surgery construction with the definition of $\tau_{S} L$, we
see that there exist choices of $\phi^{j}, D^L_{j}, D^{S_j}_{j},
H_{j}$ so that:
\begin{itemize}
\item[i.] $S\#_{r} L$ is embedded and is Hamiltonian isotopic to
  $\tau_{S}L$.
\item[ii.] $S\#_{r} L$ is transverse to $S$ and it intersects $S$ in
  the $r$ points $p'_{j}\in V$, $1\leq j \leq r$.
 \item[iii.] If both $L$ and $S$ are monotone of monotonicity constant
  $\rho$, then so is $S\#_{r} L$.
\end{itemize}
As explained above, the local model for surgery at a point requires an
order among the two Lagrangians involved. By reversing the order for
all the one-point surgeries, we obtain again a Lagrangian denoted now
$L\#_{r} S$. This has properties similar to i,ii,iii above except that
it is hamiltonian isotopic to $\tau^{-1}_{S}L$. From this perspective,
Proposition~\ref{lem:multiple-surgery} claims that, with appropriate
choices of handles, we have $(S'\#_{r} V) \cap T=\emptyset$.

  \begin{rem} a. The ``doubling'' of singularities used in Proposition
     \ref{lem:multiple-surgery} first appeared in a somewhat different
     form and with a different purpose in the work of Seidel
     \cite{Se:book-fukaya-categ}.  From the perspective of our paper,
     the initial approach to the setting of
     Proposition~\ref{lem:multiple-surgery} was to consider a thimble
     $T'$ (inside $E$) that projects over the curve $\gamma$ in
     Figure~\ref{fig:complex-thimble} and continues horizontally to
     $-\infty$. The idea was to disjoin $V$ from $T$ by a process of
     multiple surgery with multiple copies of $T'$, in other words to
     define $V'=T'\#_{r} V$ so that $V'\cap T=\emptyset$.  Purely
     geometrically, this operation is possible. However, the problem
     in drawing algebraic conclusions from it is that the condition
     $V'\cap T=\emptyset$ turns out to force that the copies of $T'$
     used in the surgery are not cylindrical at infinity
     (alternatively, one can achieve cylindricity at infinity
       at the expense that the resulting manifold $V'$ would no longer
       be embedded but only immersed, see
       also~\S\ref{subsubsec:simple-cob}). As a consequence the
     machinery involving $J$-holomorphic curves can not be applied
     directly to $V'$. On the other hand, by compactifying $T'$ to the
     sphere $S'$ - as described in the paper - this issue is no longer
     present.  The price to pay is that we need to add singularities
     to the initial fibration $E$.
     
     b. It is likely that  Proposition \ref{lem:multiple-surgery} can be proven 
     also along an approach closer to Seidel's constructions involving
      bifibrations. The basic idea along this line would be to construct the 
      fibration $\hat{E}$ by symmetrizing the restriction of the  fibration $E$ to the upper
      semi-plane by a rotation $\sigma$ by $180^{\circ}$
       around the origin in $\C$. This gives rise to a specific matching cycle
      that projects to a segment joining the singular value $v_{1}$ to its ``mirror''  $v'_{1}$. By restricting to
       a suitable disk $D$ containing this segment,  we see that the the Dehn twist around this vanishing cycle is identified to the rotation $\sigma$  (Lemma 18.2 in \cite{Se:book-fukaya-categ}). At the same time 
if  $V$ is assumed to be a Lagrangian without ends and included in $D$, then $\sigma (V)$ is remote. 
However, as  $V$ is in general more complicated this argument does not work directly and thus we 
gave a direct geometric proof.
     \end{rem}

% !TEX root = lefcob.tex

\subsection{A cobordism viewpoint on Seidel's exact triangle}
\label{s:cob-vpt}

In this section we present a new proof of Seidel's exact
triangle~\cite{Se:long-exact, Se:book-fukaya-categ}.  This is the last
essential ingredient for the proof of Theorem~\ref{thm:main-dec}.  Our
proof is based on cobordism considerations and is valid in the
monotone setting.  We give full details not only for the sake of
self-containedness but also in order to emphasize the reason why the
Novikov ring $\mathcal{A}$ is required in the proof of
Theorem~\ref{thm:main-dec}: this is precisely in establishing Seidel's
exact triangle.  Additionally, in the proof of Theorem
\ref{thm:main-dec} we need a variant of the exact triangle that
applies to the case when the Lagrangian to which the Dehn
twist is applied is itself a cobordism in the total space of a Lefschetz
fibration and the proof is robust enough to cover this case with
minimal adjustment.

Seidel's proof~\cite{Se:book-fukaya-categ} assumes an
exact setting but his argument adapts to the monotone case
too and also admits further generalizations as in \cite{We-Wo-Dehn}.

\subsubsection{The exact triangle} \label{sbsb:ex-tr}

We work, as in the rest of the paper, with coefficients in the
universal Novikov ring $\mathcal{A}$ over $\mathbb{Z}_2$ and with
monotone Lagrangians assumed to be of class $\ast$. Floer complexes
and Fukaya categories are ungraded.

Below we will have two versions of the Seidel's exact
  triangle. The first is for symplectic manifolds $X$ (which are
  either closed or symplectically convex at infinity) and their
  compact Fukaya categories (i.e. the Fukaya categories whose objects
  are {\em closed} Lagrangian submanifolds). The second version is
  specially tailored to the situation when $X$ is itself the total
  space of a Lefschetz fibration and the Fukaya category considered in
  $X$ is that of negatively ended cobordisms in $X$. It is the second
  version that will be used in the proof of
  Theorem~\ref{thm:main-dec}.  We will later exhibit $X$ as a
  fiber in a Lefschetz fibration denoted by $\mathcal{E}$.  The choice
  of notation ($\mathcal{E}$ and $X$) is intentional, in order to
  avoid confusion with the Lefschetz fibrations $E \longrightarrow
  \mathbb{C}$ and their fibers $M$ that appear in the rest of the
  paper.

  Let $(X^{2n+2}, \omega)$ be a symplectic manifold which is either
  closed or symplectically convex at infinity. \pbhl{Throughout this
    section we add the assumption that $\dim_{\mathbb{R}}X \geq
    4$. (The reason for this restriction will be explained in}
  Remark~\ref{r:dimX4} below.) Let $S$ a parametrized Lagrangian
  sphere in $X$, i.e. a Lagrangian submanifold $S \subset X$ together
  with a diffeomorphism $i_S: S^{n+1} \longrightarrow S$. Recall that
  we denote by $\tau_S: X \longrightarrow X$ the Dehn twist associated
  to $S$. Assume further that $S \subset X$ is monotone and denote by
  $*$ its monotonicity class. Following the conventions of the paper,
  we write $\fuk^{*}(X)$ for the Fukaya category of monotone closed
  Lagrangian submanifolds of $X$ of monotonicity class $*$.

The following important result was proved by
Seidel~\cite{Se:long-exact} in the exact case.  As mentioned above, we extend
the result to the monotone case and provide an independent proof.

\begin{prop} \label{t:ex-tr-compact} Let $X$, $S$
   be as above and let $Q \subset X$ be another monotone
     closed Lagrangian submanifold of monotonicity class $*$. In
   $D\fuk^{\ast}(X)$ there is an exact triangle of the form:
   \begin{equation} \label{eq:ex-tr-1}
      \xymatrix{
        \tau_S(Q) \ar[r] & Q \ar[d]\\
        & S \otimes HF(S,Q) \ar[ul]
      }
   \end{equation}
\end{prop}
The proof of this result will occupy most of~\S\ref{sb:prf-ex-tr}
below. We note that the maps appearing in this exact triangle will be identified 
along the proof, they coincide with the corresponding maps in Seidel's exact triangle.

\begin{rem}
   If one restricts the objects in the Fukaya category of $X$ to
   orientable Lagrangians, our proof should hold also with a
   $\mathbb{Z}_2$-grading. Similarly, under more assumptions on the
   Lagrangians (and additional structures) the proof is expected to carry over
 with a $\mathbb{Z}$-grading as well as, if one assumes all
   Lagrangians to be endowed with spin structures, with coefficients in $\mathbb{Z}$. 
   \end{rem}

\subsubsection{Second version of the exact triangle: the case when $X$
 is a Lefschetz fibration} \label{sbsb:ex-tr-cob}

Here we assume that $X$ is the total space of a tame Lefschetz
fibration $\pi_X^{2n+2}: X \longrightarrow \mathbb{C}$, $n\geq 1$, as
defined in~\S\ref{subsec:lef-fibr}. (The assumption that $X$ is
symplectically convex at infinity is now dropped.) We denote by
$\fuk^*(X)$ the Fukaya category of $X$ whose objects are negatively
ended {\em Lagrangian cobordisms} in $X$ of monotonicity class $*$ as
defined in \S\ref{subsubsec:Fuk-cob}.

  \begin{prop} \label{t:ex-tr-cob} For $X$ as above, let $S \subset
     X$ be a monotone Lagrangian sphere of class $*$ and let $Q \subset X$ be
     a monotone Lagrangian cobordism (possibly without ends) of the
     same monotonicity class. Then in $D\fuk^*(X)$ there is an exact
     triangle as in~\eqref{eq:ex-tr-1}.
\end{prop}

The proof is very similar to the proof of
  Proposition~\ref{t:ex-tr-compact} (which is given
  in~\S\ref{sb:prf-ex-tr} below), the only difference being that now
  $Q$ is allowed to be a cobordism rather than just a closed
  Lagrangian (and similarly for the objects of $\fuk^*(X)$). We
  explain the necessary modifications in~\S\ref{sbsb:prf-ex-tr-cob} below.

\subsubsection{Outline of the proof of Proposition~\ref{t:ex-tr-compact}} \label{sb:prf-ex-tr}

The idea of the proof is simple and we summarize it here (the precise
details are given in~\S\ref{sbsb:prf-ex-tr-compact} below). By the
geometric interpretation of the monodromy around an isolated Lefschetz
singularity - \cite{Ar:monodromy}, see also \cite{Se:long-exact} -
there exists a Lefschetz fibration $\pi:\mathcal{E}\to \C$ with a
single singularity (chosen at the origin) and with general fiber
$X$. Moreover, there is a cobordism $V\subset \mathcal{E}$ as in
Figure~\ref{f:lags-bf-stretching}, that projects to the curve
$\gamma''$ there, and has ends $Q$ and $\tau_{S}Q$. Consider a second
cobordism $W$, as in the same picture, obtained as the trail of $N$
along the curve $\gamma'$, where $N$ is any Lagrangian in
$\mathcal{L}^{\ast}(X)$. The cobordism techniques in \cite{Bi-Co:cob1}
produce an associated chain morphism $CF(N,\tau_{S}Q)\to CF(N, Q)$
given by counting the Floer strips going from the intersections of $W$
and $V$ that project to $w_1$ to the intersections that project to
$w_0$ and the cone - in the sense of chain complexes - over this
morphism is $CF(W,V)$.  The proof reduces to finding a
quasi-isomorphism $CF(N,S)\otimes CF(S,Q)\to CF(W,V)$. The next step
is again geometric and is based on the well-known fact that the
function $\textnormal{Re}(\pi)$ is Morse with a single singularity at
the origin and that its gradient with respect to the standard metric
is Hamiltonian. Moreover, the positive horizontal thimble originating
at $0$ is the stable manifold of $\textnormal{Re}(\pi)$ and the
negative horizontal thimble is the unstable manifold of
$\textnormal{Re}(\pi)$.  To start this stage in the proof, we use the
flow of $\nabla \textnormal{Re}(\pi)$ to push $W$ to the right in
picture Figure~\ref{f:lags-bf-stretching} thus getting
$\widetilde{W}$; similarly, we use the gradient of
$-\textnormal{Re}(\pi)$ to push $V$ to the left in the same picture
thus getting $\widetilde{V}$ - see Figure~\ref{f:stretching}. We
notice that $CF(\widetilde{W},\widetilde{V})\cong CF(W,V)$ and analyze
the complex $CF(\widetilde{W},\widetilde{V})$.  Assuming all relevant
intersections are generic, by standard Morse theory, if $W$ is pushed
enough to the right, $\widetilde{W}$ intersects a neighborhood around
the singularity in a number $n_{1}$ of copies of the stable manifold
of $\textnormal{Re}(\pi)$. Moreover, $n_{1}$ is equal to the number of
intersections of $W$ with the unstable manifold of
$\textnormal{Re}(\pi)$. Similarly, $\widetilde{V}$ intersects the same
neighborhood in $n_{2}$ copies of the unstable manifold of
$\textnormal{Re}(\pi)$ and $n_{2}$ is equal to the number of
intersections of $V$ with the stable manifold of
$\textnormal{Re}(\pi)$. The interpretation of the stable and unstable
manifolds as thimbles (and our transversality assumptions) immediately
imply that $n_{1}$ equals the number of intersection points $N\cap S$
and $n_{2}$ is the number of intersections $S\cap Q$. Moreover, each
copy of the stable manifold that is associated to $W'$ intersects
precisely once each copy of the unstable manifold that is contained in
$V'$. In short, it follows that there is a bijection $\Xi$ between the
following two sets
$(N\cap S) \times (S\cap Q)\equiv (\widetilde{W}\cap
\widetilde{V})$. The last step of the proof is more technical and
shows that $\Xi$ extends to a quasi-isomorphism of chain
complexes. The basic idea here is to compare the bijection $\Xi$ with
the product
$\mu_{2}:CF(\widetilde{W}, T_{\Delta})\otimes CF(T_{\Delta},
\widetilde{V})\to CF(\widetilde{W},\widetilde{V})$ where $T_{\Delta}$
is a thimble as in Figure \ref{f:lags-bf-stretching}. The key part of
the argument is to notice that the $J$-holomorphic triangles giving
this product decompose in two classes: ``short'' ones, of small area,
and ``long'' ones, of big area, and that the short component of
$\mu_{2}$ is a bijection identified to $\Xi$.  Because we work over
$\mathcal{A}$ this means that the product $\mu_{2}$ is a quasi
isomorphism and the wanted statement easily follows.

\begin{rem} \label{r:dimX4} \pbhl{The reason we restrict ourselves to
    $\dim_{\mathbb{R}} X \geq 4$ is the following. As mentioned above,
    the proof uses an auxiliary Lefschetz fibration $\mathcal{E}$ with
    a single singularity and with general fiber $X$. Moreover, we will
    use a version of the Fukaya category of of cobordisms in
    $\mathcal{E}$. For this to work we need $\mathcal{E}$ to be
    strongly monotone} (see Definition~\ref{df:monlef}). \pbhl{This
    easy follows from the monotonicity of $X$ when
    $\dim_{\mathbb{R}} X \geq 4$. However, when
    $\dim_{\mathbb{R}} X = 2$ this might not be the case anymore.} It
  seems plausible that this difficulty can be overcome (since in
  dimension $4$ (i.e. the dimension of $\mathcal{E}$) for a generic
  almost complex structure there are no holomorphic disks with
  non-positive Maslov numbers.)
\end{rem}

\subsubsection{Proof of Proposition~\ref{t:ex-tr-compact}}
\label{sbsb:prf-ex-tr-compact} 

The actual proof consists of seven steps that follow below. Two
auxiliary Lemmas that are used along the way are proved
in~\S\ref{sb:prf-lem-small-big} and~\S\ref{sb:prf-prop-loc-sol}.

To fix ideas, we first carry out the proof under the assumption that $X$ is closed.
We discuss the non-compact case at the end.

\

\noindent \textbf{Step 1:} {\em Constructing an appropriate Lefschetz
  fibration.}
\label{sb:cons-lf}

We first claim that there exists a Lefschetz fibration $\pi:
\mathcal{E} \longrightarrow \mathbb{C}$ with symplectic structure
$\Omega$ so that $\mathcal{E}$ is tame over a subset $\mathcal{W}
\subset \mathbb{C}$ as in Figure~\ref{f:fibration-E}, and there is a
symplectic trivialization $\psi$ over $\mathcal{W}$ (see
Definition~\ref{df:tame-lef-fib}), such that $\mathcal{E}$, $\Omega$
and $\psi$ have the following properties:
\begin{enumerate}
  \item The fibration has only one critical point $p \in \mathcal{E}$
   lying over $0 \in \mathbb{C}$.
  \item The fiber $(\mathcal{E}_{z_0}, \Omega|_{\mathcal{E}_{z_0}})$
   over $z_0 = -10 \in \mathbb{C}$ is symplectomorphic via the
   trivialization $\psi$ to $(X, \omega)$. (Henceforth we make this
   identification.)
  \item The vanishing cycle in $\mathcal{E}_{z_0}$ associated to the
   path going from $z_0$ to $0$ along the $x$-axis is $S$.
  \item The monodromy associated to a loop $\lambda$ based at $z_0$
   that goes around $0$ counterclockwise is Hamiltonian isotopic to
   $\tau_S$.
\end{enumerate}

\begin{figure}[htbp]
   \begin{center}
      \includegraphics[width=0.6\linewidth]{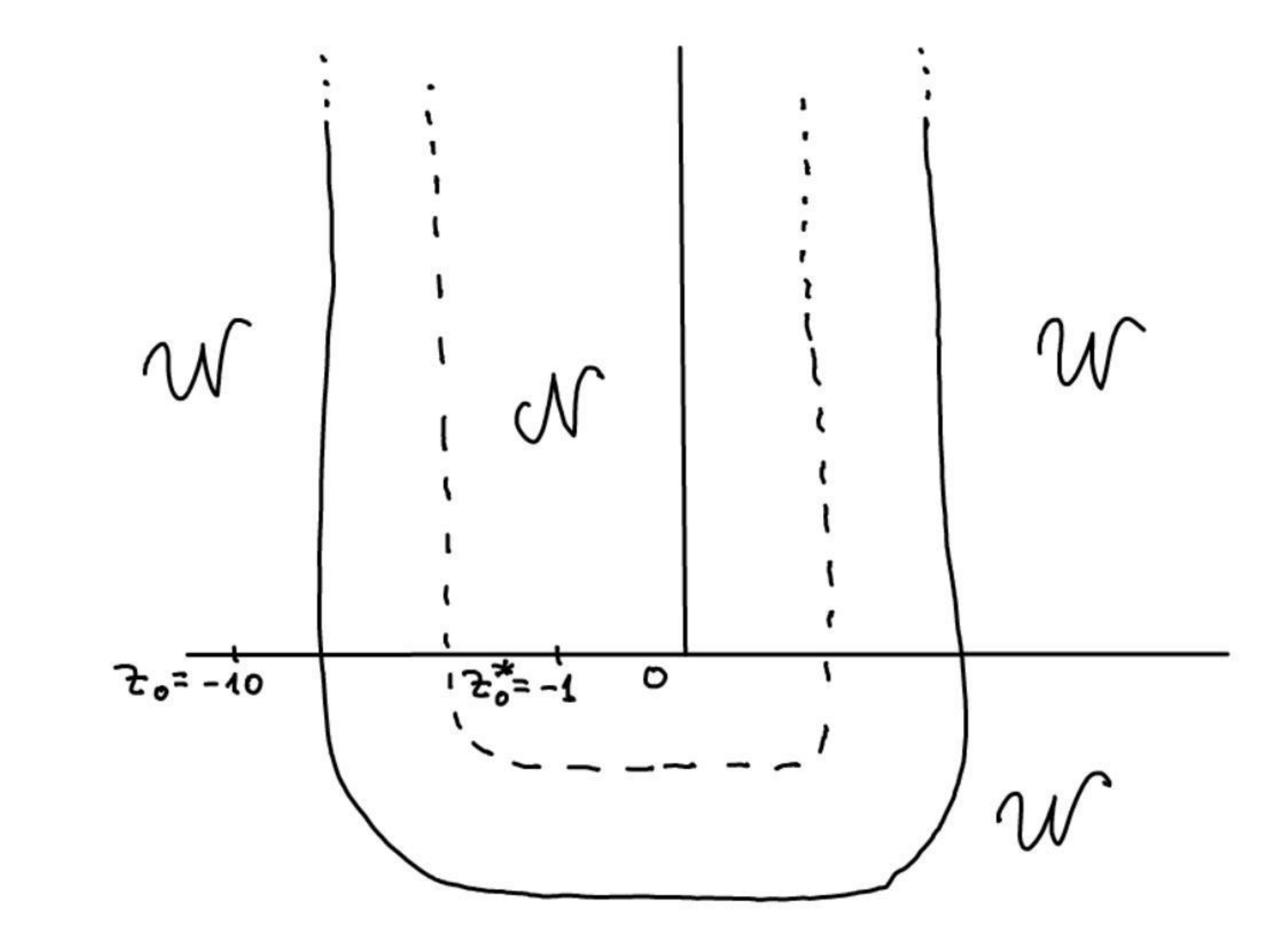}
   \end{center}
   \caption{Constructing the fibration $\mathcal{E}$.
     \label{f:fibration-E}}
\end{figure}

To prove this we first construct a Lefschetz fibration $\mathcal{E}
\longrightarrow \mathbb{C}$ (not necessarily tame) whose total space
is endowed with a symplectic structure $\Omega^*$ and with the
following properties:
\begin{enumerate}
  \item The fibration has only one critical point $p \in \mathcal{E}$
   lying over $0 \in \mathbb{C}$.
  \item The fiber over $z^*_0 = -1 \in \mathbb{C}$ is
   $(\mathcal{E}_{z^*_0}, \Omega|_{\mathcal{E}_{z^*_0}}) = (X,
   \omega)$.
  \item The vanishing cycle in $\mathcal{E}_{z^*_0}$ associated to the
   path going from $z^*_0$ to $0$ along the $x$-axis is Hamiltonian
   isotopic to $S$.
  \item The monodromy around a loop $\lambda^*$ based at $z^*_0$ which
   goes counterclockwise around the critical value $0$ is Hamiltonian
   isotopic to the Dehn twist $\tau_S$.
\end{enumerate}
The proof that such a Lefschetz fibration exists follows
from~\cite{Se:long-exact} (see also Chapter~16e
in~\cite{Se:book-fukaya-categ}), where it is proved for exact
Lagrangian spheres. This is a local argument and therefore that proof
extends to the case when $X$ is possibly not exact.

Given the fibration $\mathcal{E} \longrightarrow \mathbb{C}$ and
$\Omega^*$ we apply Proposition~\ref{p:from-gnrl-to-tame} with
appropriate subsets $\mathcal{N}$ and $\mathcal{W}$ as in
Figure~\ref{f:fibration-E} and base point $z_0 = -10$. We obtain a new
symplectic structure $\Omega'$ on $\mathcal{E}$ with respect to which
the fibration is tame over $\mathcal{W}$ and such that $\Omega'$
coincides with $\Omega^*$ over $\mathcal{N}$. We thus obtain a
trivialization
$\psi': (\mathcal{W} \times X', c\omega_{\mathbb{C}} \oplus \omega')
\longrightarrow (\mathcal{E}|_{\mathcal{W}}, \Omega')$, where
$(X', \omega') = (\mathcal{E}|_{z_0}, \Omega^*|_{\mathcal{E}|_{z_0}})$
and $c>0$.

Consider the loop $\lambda$ which starts at $z_0$, goes to $z^*_0$
along the $x$-axis, then goes along $\lambda^*$ and finally comes back
to $z_0$ along the $x$-axis. Parallel transport along the straight
segment connecting $z_0$ to $z^*_0$ and with respect to the connection
$\Gamma' = \Gamma(\Omega')$ gives a symplectomorphism $\varphi: (X',
\omega') \to (X, \omega)$. Put $S' = \varphi^{-1}(S)$.  Clearly the
monodromy (with respect to $\Gamma'$) along $\lambda$ is $\varphi^{-1}
\circ \tau_S \varphi = \tau_{S'}$.

Finally, the desired symplectic structure on $\mathcal{E}$ and the
trivialization are obtained by taking $\Omega = \Omega'$ and $\psi =
\psi' \circ (\id \times \varphi^{-1})$.

From now on the trivialization $\psi$ will be implicitly assumed and
we make the following identification
$$(\mathcal{E}|_{\mathcal{W}}, \Omega|_{\pi^{-1}(\mathcal{W})}) =
(\mathcal{W} \times X, c\omega_{\mathbb{C}} \oplus \omega).$$

\noindent \textbf{Step 2:} {\em Translating the problem to cobordisms.}
\label{sb:cobs-in-E} 

\pbhl{First note that $\mathcal{E}$ is strongly monotone of class
  $*$. This follows immediately from the} Definition~\ref{df:monlef}
(recall that we have assumed that $\dim_{\mathbb{R}} X \geq 4$) and
Remark~\ref{r:tvsg-mntlf}.

Let $\gamma' \subset \mathbb{C}$ be the curve depicted in
Figure~\ref{f:lags-bf-stretching}. In a similar way
to~\cite{Bi-Co:lcob-fuk} $\gamma'$ gives rise to an inclusion functor
$$\mathcal{I}_{\gamma'}: \fuk^*(X) \longrightarrow
\fuk^*(\mathcal{E})$$ whose action on objects is
$\mathcal{I}_{\gamma'}(N) = \gamma'N$, where $\gamma' N\subset
\mathcal{E}$ stands for the trail of $N$ along the curve $\gamma'$
(see~\S\ref{sbsb:con-ptrans-trail}). Here, by $\fuk^*(\mathcal{E})$ we
mean the Fukaya category of cobordisms in $\mathcal{E}$ of
monotonicity class $*$ but we do not require the cobordisms to be only
negatively ended. This category is defined, following the recipe in
\cite{Bi-Co:lcob-fuk} as described in \S\ref{subsubsec:Fuk-cob}, but
by also using perturbations and bottlenecks associated to the positive
ends.  For the purpose of the proof below, it is actually enough to
restrict to a subcategory whose objects are 
cobordisms in $\mathcal{E}$ that project to curves in $\C$.

\begin{figure}[htbp]
   \begin{center}
      \includegraphics[width=0.8\linewidth]{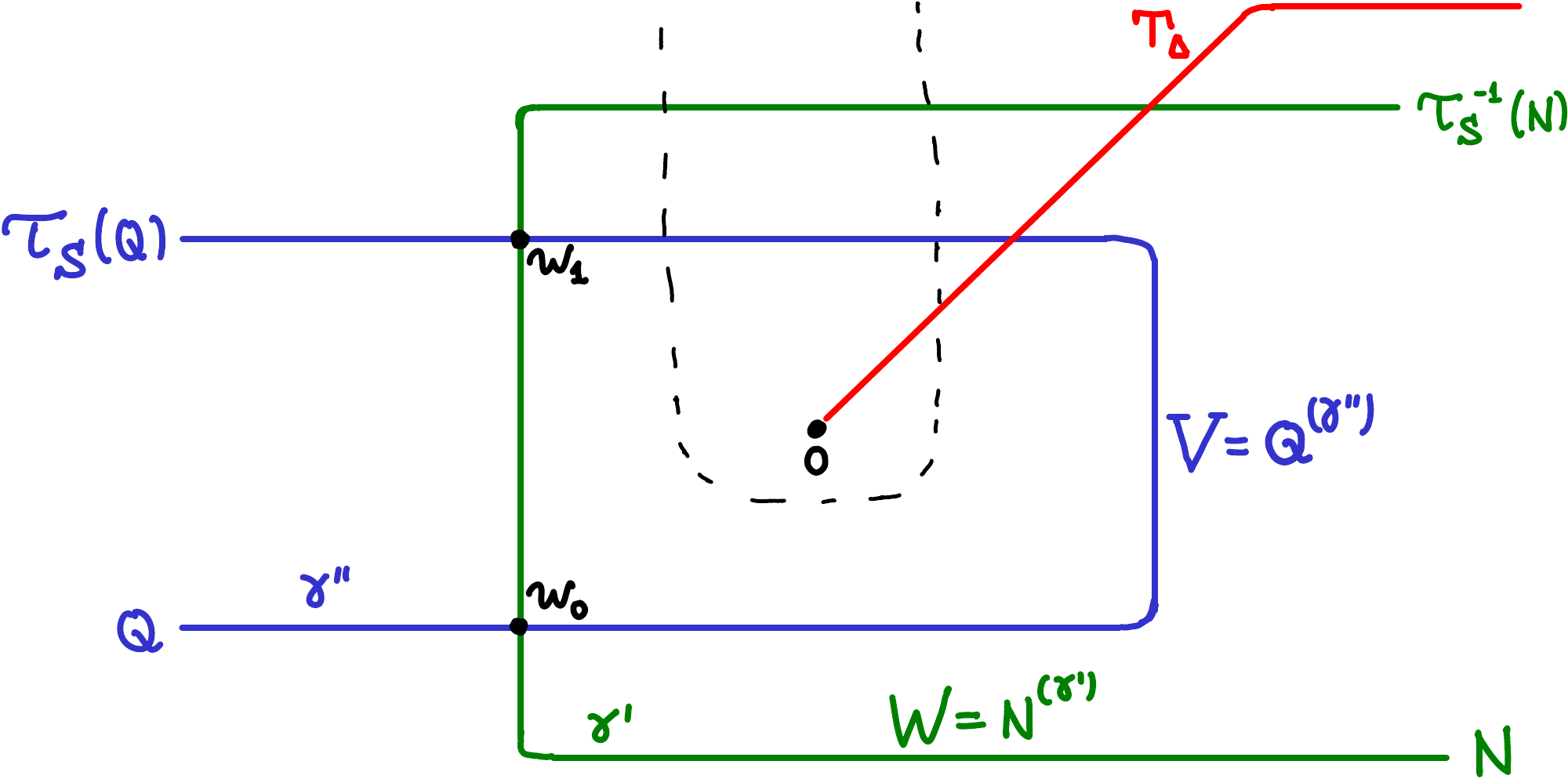}
   \end{center}
   \caption{The cobordisms $V$, $W$ and $T_{\Delta}$.
     \label{f:lags-bf-stretching}}
\end{figure}

Denote $W = \mathcal{I}_{\gamma'} N = \gamma' N$ and view it as a
cobordism in $\mathcal{E}$. Next, consider the curve
$\gamma'' \subset \mathbb{C}$ as depicted in
Figure~\ref{f:lags-bf-stretching} and fix a base point
$w_0 \in \gamma'' \cap \mathcal{W}$. Define
$V \subset (\mathcal{E}, \Omega)$ to be the Lagrangian submanifold
obtained as the \trl\ of $Q \subset \mathcal{E}_{w_0} = X$ along
$\gamma''$. \pbhl{Clearly both $V$ and $W$ are monotone and by
  standard arguments} (see~\cite{Chek:cob} and also~\cite[Remark
2.2.4]{Bi-Co:cob1}) \pbhl{we have $d_{V} = d_Q$ and $d_{W}=d_N$. It
  follows that both $V$ and $W$ are monotone of class $*$ hence are
  legitimate objects of the Fukaya category $\fuk^*(\mathcal{E})$ as
  considered in this section.}

Note that since the fibration $(\mathcal{E}, \Omega)$ is
symplectically trivial over $\mathcal{W}$ the lower end of $V$ is
identified with $Q$ and due to the homotopy class of $\gamma''$ (in
$(\mathbb{C} \setminus \{0\}, \; \textnormal{rel} \; \infty)$) the
upper end of $V$ is a Lagrangian submanifold of $X$ which is
Hamiltonian isotopic to $\tau_S(Q)$. Similarly, the lower end of $W$
is cylindrical over $N$ and the upper end is cylindrical over
$\tau_S^{-1}(N)$.

Below we will work with the Fukaya categories of both $X$ and
$\mathcal{E}$. Our choices of auxiliary parameters (Floer and
perturbation data, etc.) for these categories will be as described
in~\S\ref{subsec:fuk-cat}. We therefore omit them from the notation in
Floer complexes and the other $A_{\infty}$-structures.  There are a
few modifications compared to the conventions used in
\S\ref{subsec:fuk-cat}: we assume that the ends of the curve $\gamma'$
are at height $-2$ and $2$ and the ends of $\gamma''$ are at $-1$ and
$1$. In other words, to fit precisely the setting in
\S\ref{subsec:fuk-cat} we need to translate the whole picture by
$+3i$.  Clearly, this adjustment is formal and it has no impact on the
construction of the relevant Fukaya categories (it is required because
we prefer to keep the critical value of $\pi$ to be at $0$).

Denote by $\mathcal{Y}_X: \fuk^*(X) \longrightarrow mod(\fuk^*(X))$
and
$\mathcal{Y}: \fuk^*(\mathcal{E}) \longrightarrow
mod(\fuk^*(\mathcal{E}))$ the Yoneda embeddings associated to the
Fukaya categories of $X$ and $\mathcal{E}$ respectively. When no
confusion may arise we will simplify the notation and denote the
module $\mathcal{Y}_X(L)$ associated to a Lagrangian $L \subset X$
simply by $L$ and similarly for $\mathcal{E}$.

We now analyze the pullback module
$\mathcal{I}_{\gamma'}^*V \in mod(\fuk^*(X))$. Similar arguments to~
\S 4.4 \cite{Bi-Co:lcob-fuk} (see also~\S\ref{subsec:dec-Yo} in this
paper, in particular the exact sequence at Step 3 i.) show that we
have a quasi-isomorphism:
\begin{equation} \label{eq:cone-tau-S-1} \mathcal{I}_{\gamma'}^*V
   \simeq \textnormal{cone} \bigl( \tau_S(Q) \xrightarrow{\quad
     \varphi \quad } Q \bigr),
\end{equation}
for some homomorphism of $A_{\infty}$-modules $\varphi$ that is induced by
counting holomorphic strips (and polygons) going from the intersection
of $V$ with $W$ at the $\tau_{S}(Q)$ end to the intersection of $V$
and $W$ at the $Q$ end - see Figure \ref{f:lags-bf-stretching}. 

Let $T_{\Delta} \subset \mathcal{E}$ be the thimble corresponding to
the ``diagonal'' curve $\Delta$ depicted in
Figure~\ref{f:lags-bf-stretching}. \pbhl{By}
Proposition~\ref{p:monot-E} \pbhl{$T_{\Delta}$ is monontone of class
  $(*)$ and we view it as an object of $\fuk^*(\mathcal{E})$.}

Consider now the $\fuk^*(\mathcal{E})$-module
\begin{equation}
   \mathcal{M} = T_{\Delta} \otimes CF(T_{\Delta}, V),
\end{equation}
where the second factor in the tensor product is regarded as a chain
complex (see Chapter~3c in~\cite{Se:book-fukaya-categ} for the
definition of the tensor product of an $A_{\infty}$-module and a chain
complex).

The $A_{\infty}$-operations $\mu_k$, $k \geq 2$, induce a homomorphism
of modules $\mathcal{M} \longrightarrow V$. Pulling back by
$\mathcal{I}_{\gamma'}$, this homomorphism induces a homomorphism of
$\fuk^*(X)$-modules:
\begin{equation} \label{eq:M-equiv-TTV} \nu:
   \mathcal{I}^*_{\gamma'}\mathcal{M} \longrightarrow
   \mathcal{I}^*_{\gamma'}V.
\end{equation}

We claim that Proposition~\ref{t:ex-tr-compact} reduces to the
next statement:
\begin{prop} \label{p:nu-iso} The homomorphism $\nu$ is a
   quasi-isomorphism.
\end{prop}
This is due to the following quasi-isomorphisms:
\begin{equation} \label{eq:q-iso-for-nu}
   \mathcal{I}^*_{\gamma'} \mathcal{M} =
   \mathcal{I}^*_{\gamma'} T_{\Delta} \otimes CF(T_{\Delta}, V) \simeq S \otimes CF(S,Q).
\end{equation}

Here we identify $S$ and its image under the Yoneda embedding. 

In turn, by the general theory of $A_{\infty}$-categories, in order to
prove Proposition~\ref{p:nu-iso} it is enough to show that for every
Lagrangian $N \in \mathcal{O}b(\fuk^*(X))$ the map

\begin{equation} \label{eq:mu-2-qi} \mu_2: CF(\gamma'N,
   T_{\Delta}) \otimes CF(T_{\Delta}, V) \longrightarrow
   CF(\gamma'N, V)
\end{equation}
is a quasi-isomorphism. (Recall that $\gamma'N$ stands for the
trail of $N$ along $\gamma'$.)

\begin{rem} \label{r:mu-2-qi} We have not indicated at this moment the
   choices of Floer and perturbation data in~\eqref{eq:mu-2-qi} for
   two reasons. This is because, whether or not the map
   in~\eqref{eq:mu-2-qi} is a quasi-isomorphism does not depend on
   these specific choices (the induced product in homology is
   canonical). Moreover, later on in the proof we will actually make
   use of a very specific choice of parameters (which is different
   than the one used in~\S\ref{subsec:fuk-cat} when setting up the
   entire Fukaya category of $\mathcal{E}$ !) for which it will be
   convenient to prove that the map in~\eqref{eq:mu-2-qi} is a
   quasi-isomorphism.
\end{rem}
The rest of this section will be devoted to proving
  that~\eqref{eq:mu-2-qi} is a quasi-isomorphism. For brevity we
  denote from now on $W = \gamma'N\subset \mathcal{E}$ (see
  Figure~\ref{f:lags-bf-stretching}). 

\

\noindent \textbf{Step 3:} {\em Stretching the cobordisms.}
\label{sb:stretching-cob}

Write the projection $\pi: \mathcal{E} \longrightarrow \mathbb{C}$ as
$$\pi = \textnormal{Re}(\pi) + \textnormal{Im}(\pi) i.$$ Denote by
$Z = -\nabla \textnormal{Re}(\pi)$ the negative gradient vector field
of the real part of $\pi$ with respect to the Riemannian metric
induced on $\mathcal{E}$ by $(\Omega, J_{\mathcal{E}})$. Since the
functions $\textnormal{Re}(\pi)$ and $\textnormal{Im}(\pi)$ are
harmonic conjugate (recall that $\pi$ is holomorphic), it follows that
$Z$ is also the Hamiltonian vector field associated to the function
$\textnormal{Im}(\pi)$.

\label{pg:grad-vf-complete}
The flow of the vector field $Z$ will be used extensively
  throughout the rest of the proof. However, due to the
  non-compactness of $\mathcal{E}$, it might lack to be defined for
  all times. To overcome this difficulty we proceed as follows. Write
  $y_1 + iy_2 \in \mathbb{C}$ for the standard coordinates on
  $\mathbb{C}$. Denote by $R^{\Omega}$ the curvature of the connection
  $\Gamma(\Omega)$. (Recall that this is a $2$-form on $\mathbb{C}$
  with values in the space of Hamiltonian functions of the fibers of
  $\mathcal{E}$.) A straightforward calculation shows that
  for every $z \in \mathbb{C}$, $p \in \mathcal{E}_z$ we have:
\begin{equation} \label{eq:Z-vf-1}
   Z_{(z,p)} = \frac{-1}{C(z) -
     R^{\Gamma}_{z}(\partial_{y_1}, \partial_{y_2})(p)}
   (\partial_{y_1})^{\textnormal{hor}},
\end{equation}
where $C: \mathbb{C} \longrightarrow \mathbb{R}$ is a function and
$(\partial_{y_1})^{\textnormal{hor}}$ stands for the horizontal lift
of $\partial_{y_1}$. Since $Z = -\nabla \textnormal{Re}(\pi)$ it
follows that the denominator on the right-hand side
of~\eqref{eq:Z-vf-1} is always positive. Fix a positive real number $a
> 0$ and define
$$\Omega' = \Omega + a \pi^* dy_1 \wedge dy_2.$$ Note that
$J_{\mathcal{E}}$ continues to be compatible with $\Omega'$.  Denote
by $Z'$ the negative gradient of the same function,
$\textnormal{Re}(\pi)$, but now defined via the metric associated to
$(\Omega', J_{\mathcal{E}})$. A simple calculation shows that:
\begin{equation} \label{eq:Z-vf-2} Z'_{(z,p)} = \frac{-1}{a + C(z) -
     R^{\Gamma}_{z}(\partial_{y_1}, \partial_{y_2})(p)}
   (\partial_{y_1})^{\textnormal{hor}}.
\end{equation}
Clearly the coefficient standing before
$(\partial_{y_1})^{\textnormal{hor}}$ on the right-hand side
of~\eqref{eq:Z-vf-2} is bounded from above by $1/a$. It now easily
follows that the flow of $Z'$ exists for all times (recall that we are
under the assumption that the fiber $X$ is compact).  Finally, note
that the connections of $\Omega$ and $\Omega'$ are the same and
moreover, $V$ and $W$ continue both to be Lagrangian cobordisms with
respect to the new form $\Omega'$.

Summarizing the preceding procedure, by replacing $\Omega$ by
$\Omega'$ we may assume that the negative gradient flow of
$\textnormal{Re}(\pi)$ exists for all times. For simplicity we
continue to denote the symplectic structure of $\mathcal{E}$ by
$\Omega$.

Denote by $\phi_t$, $t \in \mathbb{R}$, the flow of $Z$. Note that the
function $\textnormal{Re}(\pi)$ is a Morse function with exactly one
critical point $p \in \mathcal{E}$ lying over $0 \in \mathbb{C}$.  The
Morse index of $\textnormal{Re}(\pi)$ at $p$ is precisely $n+1 =
\dim_{\mathbb{C}}\mathcal{E}$. Denote by $\phi_t$, $t \in \mathbb{R}$,
the flow of $Z$. The stable submanifold of $Z$ is the thimble $T'$
lying over the positive $x$-axis and the unstable submanifold of $Z$
is the thimble $T''$ lying over the negative $x$-axis. Note that we
have $J_{\mathcal{E}} T_p(T') = T_p(T'')$.

\begin{figure}[htbp]
   \begin{center}
      \includegraphics[width=0.8\linewidth]{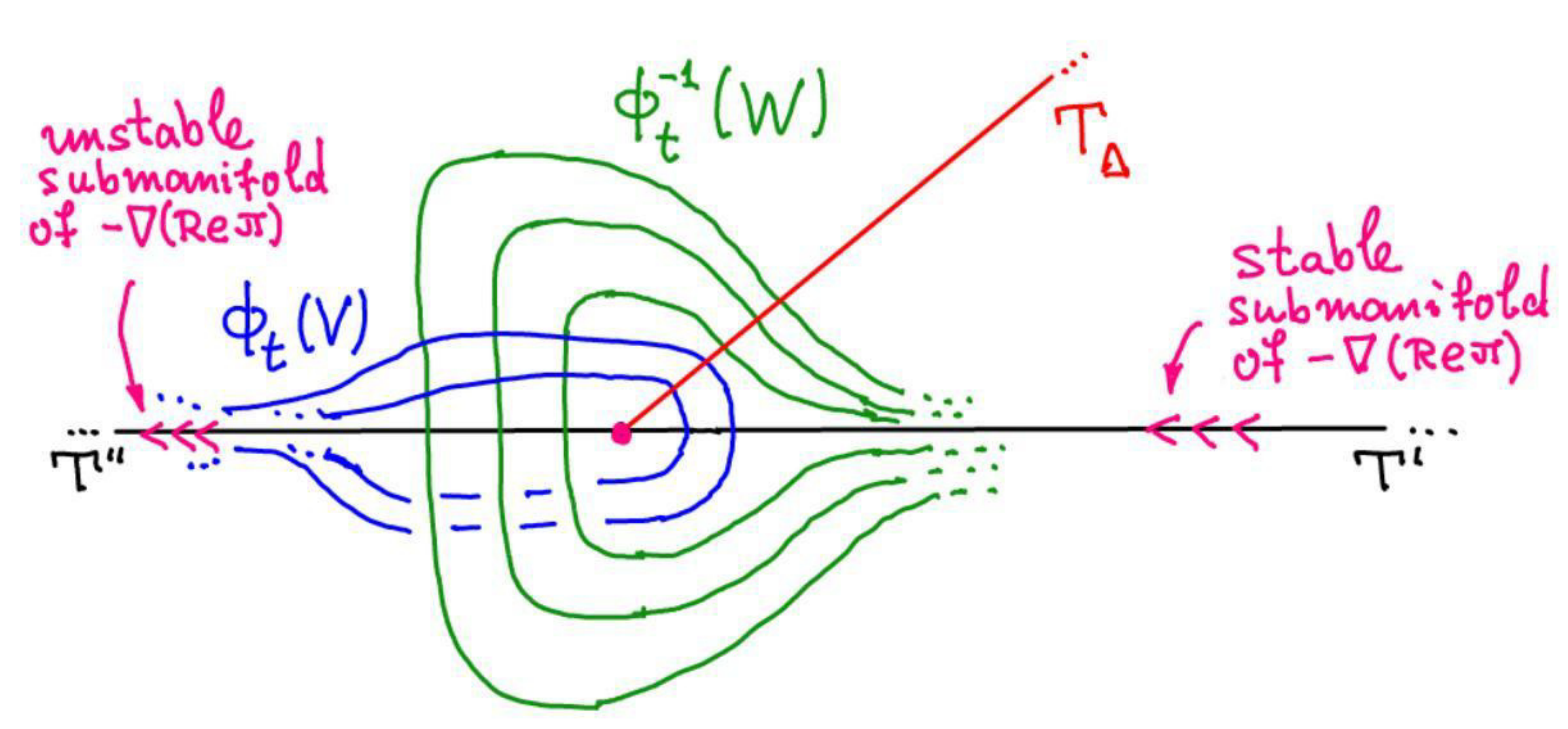}
   \end{center}
   \caption{The cobordisms $V$, $W$ after the flows $\phi_t$ and
     $\phi^{-1}_{t}$ are applied to them for large time $t$.
     \label{f:stretching}}
\end{figure}

%From now on we set:
%$$m = \tfrac{1}{2} \dim \mathcal{E} = n+1, \quad 
%\textnormal{where} \; n = \tfrac{1}{2} \dim X.$$ 
Denote by $B'(r) =
B''(r) = B^{n+2}(r) \subset \mathbb{R}^{n+2}$ two copies of the
$n+1$-dimensional closed Euclidean ball of radius $r$ around $0 \in
\mathbb{R}^{n+2}$. (Since each of these two balls corresponds to a
different factor of $\mathbb{R}^{n+2} \times \mathbb{R}^{n+2}$ we have chosen
to denote them by different symbols.)

Fix a little neighborhood $Q_p \subset \mathcal{E}$ of $p$ which is
symplectomorphic to a product $B'(r_0) \times B''(r_0) \subset
(\mathbb{R}^{n+2} \times \mathbb{R}^{n+2}, \omega_{\textnormal{can}} = dp_1
\wedge dq_1 + \cdots dp_m \wedge dq_m)$ for some small $r_0$. Below we
will abbreviate $B' = B'(r_0)$, $B'' = B''(r_0)$.

We may assume that the symplectic identification $Q_p \approx B'
\times B''$ sends $T' \cap Q_p$ to $B' \times\{0\}$ and $T'' \cap Q_p$
to $\{0\}\times B''$ and $T_{\Delta}$ to the diagonal $\{(x,y) \in B'
\times B'' \mid x = y\}$. From now on we assume the identification
$Q_p \approx B' \times B''$ explicit and when convenient view $Q_p$ as
a subset of $\mathbb{R}^{2m}$.

We now apply the flow $\phi_t$ to $V$ and $\phi^{-1}_t$ to $W$ (see
Figures~\ref{f:stretching},~\ref{f:nbhd-Qp}). Standard arguments in
Morse theory imply that for $t_0 \gg 1$ we have
$$\phi_{t_0}^{-1}(W) \cap Q_p = \coprod_{i=1}^{s''} D'_i, \quad 
\phi_{t_0}(V) \cap Q_p = \coprod_{j=1}^{s'} D''_j,$$ where $D'_i
\subset Q_p$ are graphs of exact $1$-forms on $B'$ and $D''_j \subset
Q_p$ are graphs of exact $1$-forms on $B''$. Here $s'' = \# (W \cap
T'')$ and $s' = \# (V \cap T')$ are the number of intersection points
(counted without signs) of $W \cap T''$ and $V \cap T'$ respectively.
Note also that by our construction of $\mathcal{E}$ we have $s'' =
\#(N \cap S)$ and $s' = \#(Q \cap S')$, where $S'$ is the vanishing
sphere $T' \cap \mathcal{E}_{x}$ with $0<x$ large enough so that $x
\in \mathcal{W}$. Note that $S'$, when viewed as a Lagrangian in $(X,
\omega)$ is Hamiltonian isotopic to $S$.

Fix $0 < \delta_0 \ll 1/3$. By taking $t_0$ large enough we may assume
that
\begin{equation} \label{eq:delta_0} \phi^{-1}_{t_0}(W) \cap Q_p
   \subset B' \times B''(\delta_0 r_0), \quad \phi_{t_0}(V) \cap Q_p
   \subset B'(\delta_0 r_0) \times B''
\end{equation}
and moreover that each of the $D'_i$ (resp. $D''_j$'s) is $C^1$-close
to a constant section of $B' \times B'' \to B'$ (resp. $B' \times B''
\to B'$). See Figure~\ref{f:nbhd-Qp}.

\begin{figure}[htbp]
   \begin{center}
      \includegraphics[width=0.5\linewidth]{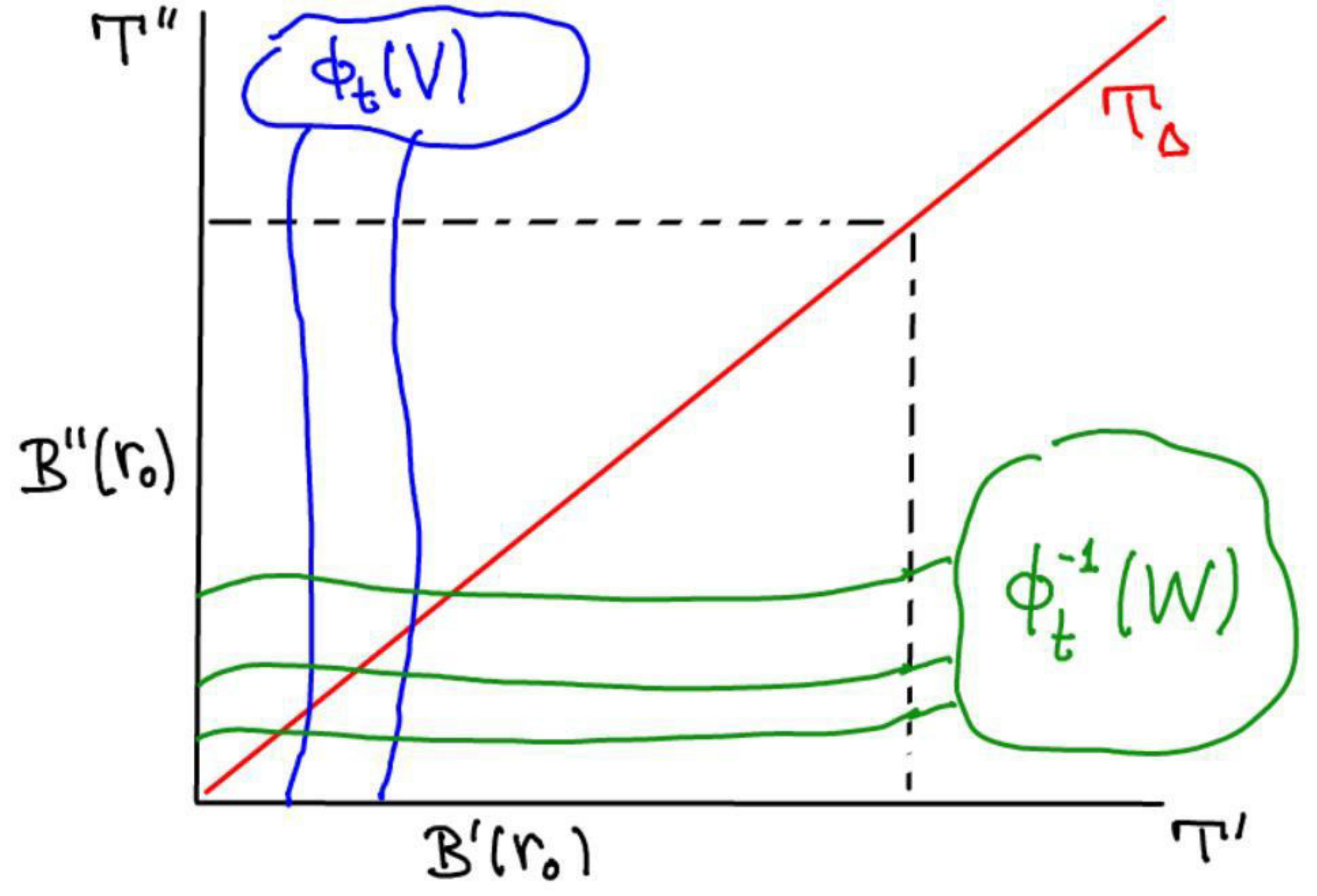}
   \end{center}
   \caption{The parts of $\phi_t(V)$ and $\phi^{-1}_t(W)$ that lie in $Q_p$.
     \label{f:nbhd-Qp}}
\end{figure}

Thus by applying a suitable Hamiltonian diffeomorphism of $Q_p$ (which
extends to the rest of $\mathcal{E}$) we may assume that
$$\phi^{-1}_{t_0}(W) \cap Q_p = \coprod_{i=1}^{s''} B' \times \{a''_i(t_0)\},
\quad \phi_{t_0}(V) \cap Q_p = \coprod_{j=1}^{s'} \{a'_j(t_0)\} \times
B'',$$ where $|a'_i(t_0)|, |a''_j(t_0)| < \delta_0 r_0$. See
Figure~\ref{f:const-sect}.

\begin{figure}[htbp]
   \begin{center}
      \includegraphics[width=0.5\linewidth]{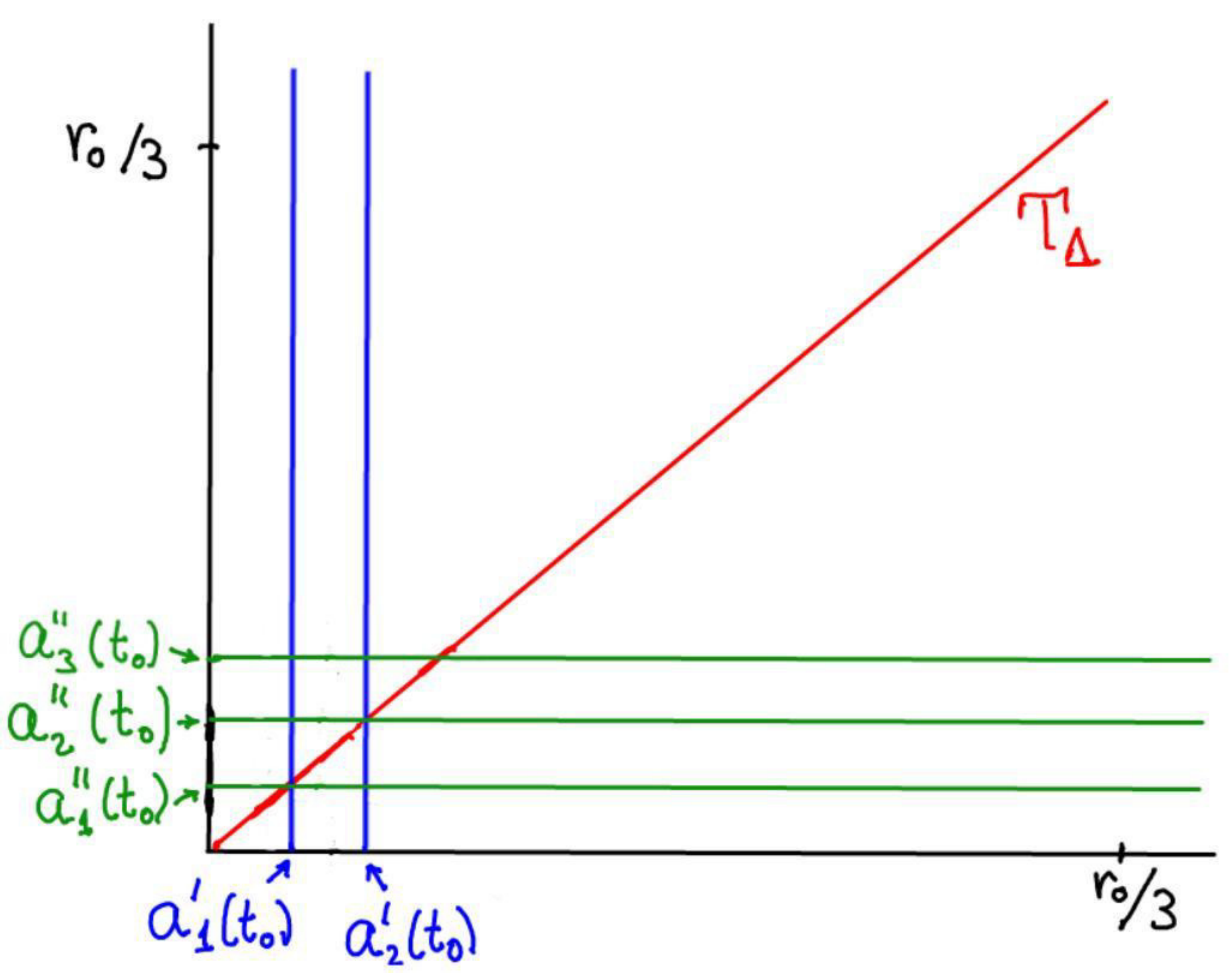}
   \end{center}
   \caption{Isotoping $\phi_{t_0}(V) \cap Q_p$ and $\phi^{-1}_{t_0}(W)
     \cap Q_p$ to be constant sections.
     \label{f:const-sect}}
\end{figure}

Fix now $t_0$ large enough as above and set $$\widetilde{V} :=
\phi_{t_0}(V), \quad \widetilde{W} = \phi^{-1}(t_0)(W).$$ For $r', r''
< r_0$ we abbreviate $B(r',r'') := B'(r') \times B''(r'')$ and also $B
= B(r_0, r_0) = B' \times B''$.

\

\noindent \textbf{Step 4:} {\em A further isotopy of $\widetilde{V}$ and
  $\widetilde{W}$.} \label{sb:further-isotop}

We claim there exist two Hamiltonian isotopies $\psi'_t, \psi''_t$, $0
\leq t < 1$, with $\psi'_0 = \psi''_0 = \id$ and with the following
properties for every $0\leq t< 1$:
\begin{enumerate}
  \item $\psi'_t$, $\psi''_t$ are both supported in
   $\textnormal{Int\,}(B)$.
  \item $\psi'_t(\widetilde{W}) \cap B(r_0/3, r_0/3) =
   \coprod_{i=1}^{s''} B'(r_0/3) \times \{b''_i(t)\}$ with $|b''_i(t)|
   \leq (1-t)\delta_0 r_0$ for every $i$.
  \item $\psi''_t(\widetilde{V}) \cap B(r_0/3, r_0/3) =
   \coprod_{j=1}^{s'} \{b'_j(t)\} \times B''(r_0/3)$ with $|b'_j(t)|
   \leq (1-t)\delta_0 r_0$ for every $j$.
  \item $\psi'_t(\widetilde{W}) \cap \Bigl( \bigl( (B'(r_0) \setminus
   B'(2r_0/3) \bigr) \times B''(r_0) \Bigr) = \widetilde{W} \cap
   \Bigl( \bigl( B'(r_0) \setminus B'(2r_0/3) \bigr) \times B''(r_0)
   \Bigr)$.
  \item $\psi''_t(\widetilde{V}) \cap \Bigl( B'(r_0) \times \bigl(
   B''(r_0) \setminus B''(2r_0/3) \bigr) \Bigr) = \widetilde{V} \cap
   \Bigl( B'(r_0) \times \bigl( B''(r_0) \setminus B''(2r_0/3) \bigr)
   \Bigr)$.
  \item $\psi'_t(\widetilde{W})$ and $\psi''_t(\widetilde{V})$
   intersect only inside $B(\delta_0 r_0, \delta_0 r_0)$.  Moreover,
   their intersection is: $\psi'_t(\widetilde{W}) \cap
   \psi''_t(\widetilde{V}) = \{(b'_j(t), b''_i(t)) \mid 1 \leq i \leq
   s'', \; 1 \leq j \leq s' \}$.
  \item $T_{\Delta} \cap \psi''_t(\widetilde{V}) \subset B(r_0/3,
   r_0/3)$ and $T_{\Delta} \cap \psi'_t(\widetilde{W}) \subset
   B(r_0/3, r_0/3)$. \label{i:T-D-cap-WV}.
\end{enumerate}
See Figure~\ref{f:isotopies-psi}. The construction of the isotopies
$\psi'_t, \psi''_t$ is elementary and can be done quite explicitly.
For point~(\ref{i:T-D-cap-WV}) one might need to reduce the size of
the parameter $\delta_0$ from~\eqref{eq:delta_0}, which can be done in
advance. 

\begin{figure}[htbp]
   \begin{center}
      \includegraphics[width=0.8\linewidth]{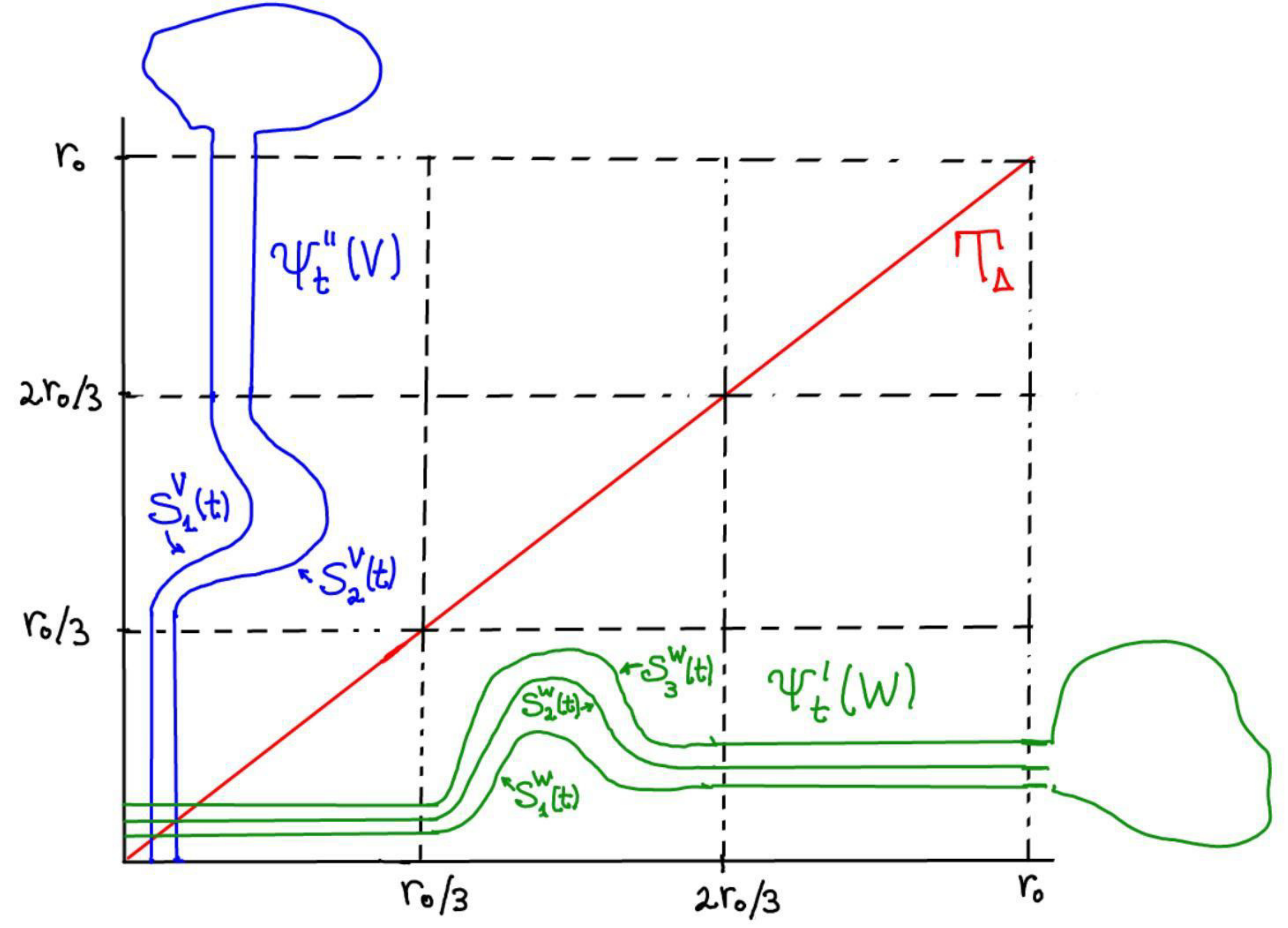}
   \end{center}
   \caption{The isotopies $\psi'_t(\widetilde{V})$,
     $\psi''_t(\widetilde{W})$ \label{f:isotopies-psi}} \end{figure}
To keep the notation short we now set: $$\widetilde{V}_t =
\psi''_t(\widetilde{V}), \quad \widetilde{W}_t =
\psi'_t(\widetilde{W}).$$ Note that $\widetilde{W}_t \cap B(r_0, r_0)$
is {\em disconnected} and has precisely $s''$ connected components,
each of which looks like a copy of $B' \times \{0\}$ which is
(non-linearly) translated along the $B''$-axis. These components lie
in ``parallel'' position one with respect to the other (see
Figure~\ref{f:isotopies-psi}). We will refer to these components as
the {\em sheets} of $\widetilde{W}_t$ inside $B(r_0, r_0)$ and denote
them by $\mathcal{S}^{W}_i(t)$, $i = 1, \ldots, s''$.  The indexing
here is so that $\mathcal{S}^{W}_i(t)$ coincides with $B'(r_0/3)
\times \{b''_i(t)\}$ inside $B(r_0/3, r_0/3)$. Similarly,
$\widetilde{V}_t \cap B(r_0, r_0)$ is {\em disconnected} and consists
of $s'$ ``parallel'' sheets which are all ``translates'' of
$\{0\}\times B''$. We denote them by $\mathcal{S}^{V}_j(t)$, $j=1,
\ldots, s'$, where the indexing is done so that $\mathcal{S}^{V}_j(t)$
coincides with $\{b_j(t)\} \times B''(r_0/3)$ inside $B'(r_0/3,
r_0/3)$. See Figure~\ref{f:isotopies-psi}. Clearly we have
\begin{equation} \label{eq:sheets} \begin{aligned} &
      \mathcal{S}^W_i(t) \cap \mathcal{S}^V_j(t) =
      \{ (b'_j(t), b''_i(t))\}, \\
      & \mathcal{S}^W_i(t) \cap T_{\Delta} = \{(b''_i(t), b''_i(t))
      \}, \quad T_{\Delta} \cap \mathcal{S}^V_j(t) = \{ (b'_j(t),
      b'_j(t))\}.  \end{aligned} \end{equation} 
  
  \
      
\noindent \textbf{Step 5:} {\em
  Area estimates for large holomorphic triangles.}
\label{sb:area-estim-large} Let $D' = D \setminus \{z_1, z_2, z_3\}$
be the unit disk punctured at three boundary points $z_1, z_2, z_3$
ordered clock-wise along $\partial D$. Fix strip-like ends around the
punctures (see~\S\ref{subsec:fuk-cat}), and denote by $\partial_{i,j}
D'$, the arc on $\partial D'$ connecting $z_i$ with $z_j$.

We will now consider a special almost complex structure $J^0_{B}$ on
$B = B' \times B''$. We identify $\mathbb{R}^{n+2} \times \mathbb{R}^{n+2}$
with $\mathbb{C}^{n+2}$ in the obvious way via $(x_1, \ldots, x_m, y_1,
\ldots, y_{n+2}) \longmapsto (x_1+iy_1, \ldots, x_{n+2}+iy_{n+2})$. This induces a
complex structure $J_{\textnormal{std}}$ on $\mathbb{R}^{n+2} \times
\mathbb{R}^{n+2}$. We define $J^0_B$ to be the restriction of
$J_{\textnormal{std}}$ to $B \subset \mathbb{R}^{n+2} \times
\mathbb{R}^{n+2}$. Define now $\mathcal{J}_0$ to be the space of
$\Omega$-compatible domain-dependent almost complex structures $J =
\{J_z\}_{z \in D'}$ which coincide with $J^0_B$ on $B$. For elements
$J \in \mathcal{J}_0$, and $z \in D'$, $p \in \mathcal{E}$ we will
also write $J(z,p)$ for the restriction of $J_z$ to $T_p\mathcal{E}$.

Consider now finite energy solutions to the Floer equation with
boundary conditions on the Lagrangians $\widetilde{W}_t$,
$T_{\Delta}$, $\widetilde{V}_t$:
\begin{equation} \label{eq:Floer-eq-D'}
   \begin{aligned}
      & u:D' \longrightarrow \mathcal{E}, \quad E(u) < \infty, \\
      & Du + J(z,u)\circ Du\circ j = 0, \\
      & u(\partial_{3,1}D') \subset \widetilde{W}_t, \quad
      u(\partial_{1,2}D') \subset T_{\Delta}, \quad
      u(\partial_{2,3}D') \subset \widetilde{V}_t
   \end{aligned}
\end{equation}
together with the requirement that $u$ converges along each strip-like
end of $D'$ to an intersection point between the corresponding pair of
Lagrangians (associated to the two arcs of $\partial D'$ that neighbor
a given puncture). Thus $u$ extends continuously to a map $u:D
\longrightarrow \mathcal{E}$ with
$$u(z_1) \in \widetilde{W}_t \cap T_{\Delta}, \quad u(z_2) \in 
T_{\Delta} \cap \widetilde{V}_t, \quad u(z_3) \in \widetilde{W}_t
\cap \widetilde{V}_t.$$ In what follows we denote for a (finite
energy) map $u: D \longrightarrow \mathcal{E}$ by $A_{\Omega}(u) =
\int_{D'} u^*\Omega$ its symplectic area.

We now fix once and for all $r_1$ with $2r_0/3 < r_1 < r_0$.
\begin{lem} \label{l:small-big} There exists a constant $C = C(r_1,
   \widetilde{W}, \widetilde{V}) > 0$ (that depends only on $r_1$ and
   $\widetilde{W}=\widetilde{W}_0$, $\widetilde{V} = \widetilde{V}_0$)
   with the following property. Let $0 \leq t < 1$ and $J \in
   \mathcal{J}_0$. Then every solution $u: D' \longrightarrow
   \mathcal{E}$ of~\eqref{eq:Floer-eq-D'} with $u(D') \not\subset
   B(r_1, r_1)$ must satisfy $A_{\Omega}(u) \geq C$.
\end{lem}

The proof of the lemma is given in~\S\ref{sb:prf-lem-small-big} below.

Next consider the intersections between any of $\widetilde{W}_t$,
$\widetilde{V}_t$ and $T_{\Delta}$. Recall from~\eqref{eq:sheets} the
intersections between $\mathcal{S}^W_i(t)$, $\mathcal{S}^{V}_j(t)$ and
$T_{\Delta}$. For simplicity we set $$w_i(t) = (b''_i(t),b''_i(t)),
\quad v_j(t) = (b'_j(t),b'_j(t)), \quad x_{i,j}(t) = (b'_j(t),
b''_i(t)).$$ With this notation we have:
\begin{equation}
   \begin{aligned}
      & \widetilde{W}_t \cap T_{\Delta} = \{w_i(t) \mid 1 \leq i \leq
      s'' \}, \quad T_{\Delta} \cap \widetilde{V}_t = \{v_j(t) \mid 1
      \leq j \leq s'\}, \\
      & \widetilde{W}_t \cap \widetilde{V}_t = \{x_{i,j}(t) \mid 1
      \leq i \leq s'', \; 1\leq j \leq s' \}.
   \end{aligned}
\end{equation}

As a consequence from Lemma~\ref{l:small-big} we have:
\begin{cor} \label{c:small-big-2} Let $0 \leq t < 1$, $J \in
   \mathcal{J}_0$ and $u:D' \longrightarrow \mathcal{E}$ a solution
   of~\eqref{eq:Floer-eq-D'}. If $$u(z_1) = w_i(t), \quad u(z_2) =
   v_j(t), \quad u(z_3) \neq x_{i,j}(t),$$ then $A_{\Omega}(u) \geq
   C$, where $C$ is the constant from Lemma~\ref{l:small-big}.
\end{cor}

\begin{proof}[Proof of Corollary~\ref{c:small-big-2}]
   Let $u:D' \longrightarrow \mathcal{E}$ be as in the statement of
   the corollary. We claim that $u(\partial D') \not \subset B(r_1,
   r_1)$.

   To prove this, assume the contrary were the case, i.e. that
   $u(\partial D') \subset B(r_1, r_1)$. Since $u(z_1) = w_i(t)$ it
   follows that $u(\partial_{1,3} D') \subset \mathcal{S}^{W}_i(t)$.
   Similarly, from $u(z_2) = v_j(t)$ we conclude that
   $u(\partial_{1,2} D') \subset \mathcal{S}^V_j(t)$. It now follows
   that $u(z_3) \in \mathcal{S}^{W}_i(t) \cap \mathcal{S}^{V}_j(t) =
   \{x_{i,j}(t) \}$, which is a contradiction.  This proves that
   $u(\partial D') \not \subset B(r_1, r_1)$. By
   Lemma~\ref{l:small-big} we have $A_{\Omega}(u) \geq C$.
\end{proof}

\noindent \textbf{Step 6:} {\em
Estimating the small holomorphic triangles.}
\label{sb:area-estim-small}

\begin{lem} \label{p:loc-solutions-1} There exists $\epsilon >0$ and
   a constant $C'>0$ such that the following holds. Let $1-\epsilon
   \leq t < 1$ and $1 \leq i \leq s''$, $1 \leq j \leq s'$ and $J \in
   \mathcal{J}_0$. Then among the solutions of
   equation~\eqref{eq:Floer-eq-D'} there exists a unique one $u$ with
   the following two properties:
   \begin{enumerate}
     \item $u(z_1) = w_i(t)$, $u(z_2) = v_j(t)$, $u(z_3) =
      x_{i,j}(t)$.
     \item $A_{\Omega}(u) < C'$.
   \end{enumerate}
   Moreover, this solution $u$ satisfies $u(D') \subset B(r_0/3,
   r_0/3)$ and $A_{\Omega}(u) \leq \sigma(t)$, where $\sigma(t)
   \xrightarrow[\; t\to 1^{-} \; ]{} 0$. Furthermore $J$ is regular
   for the solution $u$ in the sense that the linearized
   $\overline{\partial}$ operator is surjective at $u$.
\end{lem}
The proof is given in~\S\ref{sb:prf-prop-loc-sol} below.

\noindent \textbf{Step 7:} {\em End of the proof.} 
\label{sb:sum-up} We are now ready to
prove that the map in~\eqref{eq:mu-2-qi} is a quasi-isomorphism, thus
proving Proposition~\ref{p:nu-iso}.

Following Steps 1-6 above it is enough to show that the map
\begin{equation} \label{eq:mu-2-qi-2} \mu_2: CF(\widetilde{W}_t,
   T_{\Delta}) \otimes CF(T_{\Delta}, \widetilde{V}_t)
   \longrightarrow CF(\widetilde{W}_t, \widetilde{V}_t)
\end{equation}
is a quasi-isomorphism for some $0\leq t < 1$.

Next, note that the whether or not~\eqref{eq:mu-2-qi-2}
(or~\eqref{eq:mu-2-qi}) is a quasi-isomorphism is independent of the
Floer and perturbation data used for the respective Floer complexes
and for the operation $\mu_2$. Therefore for the sake of our proof any
choice of such data would do as long as it is regular and amenable to
the situation of cobordisms. (In contrast, consistency with respect to
the perturbation data used for the higher $\mu_k$'s is irrelevant for
our present purposes.) We therefore choose for~\eqref{eq:mu-2-qi-2}
Floer data for which the Hamiltonian perturbations are $0$ and $J \in
\mathcal{J}_0$.

By construction, $CF(\widetilde{W}_t, T_{\Delta})$ has the elements
$w_1(t), \ldots, w_{s''}(t)$ as a basis. Similarly $CF(T_{\Delta},
\widetilde{V}_t)$ has a basis consisting of $v_1(t), \ldots,
v_{s'}(t)$ and $CF(\widetilde{W}_t, \widetilde{V}_t)$ can be endowed
with the basis $\{x_{i,j}(t)\}_{1\leq i \leq s'', \, 1\leq j \leq
  s'}$. Thus we have a 1-1 correspondence between the associated basis
of $CF(\widetilde{W}_t, T_{\Delta}) \otimes CF(T_{\Delta},
\widetilde{V}_t)$ and the basis of $CF(\widetilde{W}_t,
\widetilde{V}_t)$, given by $$w_i(t) \otimes v_j(t) \longmapsto
x_{ij}(t).$$ 

We will now show that for $t<1$ close enough to $1$ and appropriate
$J$, the matrix of $\mu_2$ with respect to these bases is invertible.
This will prove that for such a choice of $t$ and $J$, $\mu_2$ is in
fact a chain isomorphism (hence a quasi-isomorphism for any other
choice). Below we will denote the matrix of $\mu_2$ with respect to
these bases by $M$.

Fix a generic $J \in \mathcal{J}_0$ and $t_0$ with $1\leq t_0 <
1-\epsilon$ such that $\sigma(t_0) \ll C'$, where $\epsilon$, $C'$ and
$\sigma$ are as in Proposition~\ref{p:loc-solutions-1}. By
Proposition~\ref{p:loc-solutions-1} the entries in the diagonal of $M$
have the form $$M_{k,k}(T) = T^{\alpha_k} + O(T^{C'}),$$ with $0 \leq
\alpha_{k} \leq \sigma(t_0)$. Here $o(T^{C'})$ stands for an element
of the Novikov ring in which every monomial is of the form $c_l
T^{\lambda_l}$ with $c_l \in \mathbb{Z}_2$ and $\lambda_l\geq C'$.

Similarly, by Corollary~\ref{c:small-big-2}, the elements of $M$ that
are off the diagonal are all of the form $$M_{k,l} = O(T^C), \quad
\forall \, k \neq l,$$ where $C$ is the constant from
Corollary~\ref{c:small-big-2} and Lemma~\ref{l:small-big}.

By choosing $t_0$ close enough to $1$ we obtain $\alpha_k$ as close as
we want to $0$. It easily follows that for such a choice of
  $t_0$ the matrix $M$ can be transformed via elementary row
  operations to an upper triangular matrix with non-zero elements in
  the diagonal. It follows that $M$ is invertible. \Qed

\begin{rem}
   It is not difficult to see that the map $\varphi$
   from~\eqref{eq:cone-tau-S-1} is chain-homotopic to the
   corresponding map constructed by Seidel (in the exact case) in his
   construction of the exact triangle associated to a Dehn twist.  As
   a consequence, the exact triangle constructed above coincides with
   his.
\end{rem}

\subsubsection{Proof of Lemma~\ref{l:small-big}}
\label{sb:prf-lem-small-big}

Throughout the proof we will denote by $\textnormal{Ball}_x(r) \subset
\mathbb{R}^{n+2} \times \mathbb{R}^{n+2}$ the open Euclidean ball of radius
$r$ centered at $x$.

Fix $r_2$ with $2r_0/3 < r_2 < r_1$ and let $\rho_2>0$ small enough so
that:
\begin{enumerate}
  \item For $i$ and every $x \in \mathcal{S}^{W}_i(t) \cap (\partial
   B'(r_2) \times B'')$ the closed ball
   $\overline{\textnormal{Ball}_x(\rho_1)}$ is disjoint from all
   $\mathcal{S}^{W}_k(t)$ for every $k \neq i$ as well as from
   $\widetilde{W}_t$ and from $T_{\Delta}$.
  \item For $j$ and every $x \in \mathcal{S}^{V}_j(t) \cap (B' \times
   \partial B''(r_2))$ the closed ball
   $\overline{\textnormal{Ball}_x(\rho_1)}$ is disjoint from all
   $\mathcal{S}^{V}_k(t)$ for every $k \neq j$ as well as from
   $\widetilde{V}_t$ and from $T_{\Delta}$.
  \item For every $x \in T_{\Delta} \cap (\partial B'(r_2) \times
   \partial B''(r_2))$ the closed ball
   $\overline{\textnormal{Ball}_x(\rho_1)}$ is disjoint from
   $\widetilde{W}_t$ and $\widetilde{V}_t$.
\end{enumerate}
By construction, such a $\rho_1$ exists and can be chosen to be {\em
  independent} of $0 \leq t < 1$. (Recall that
$\widetilde{W}_t \cap (B(r_0, r_0) \setminus B(2r_0/3, r_0))$ is
independent of $t$.)  See Figure~\ref{f:area-estim-Lelong}.

\begin{figure}[htbp]
   \begin{center}
      \includegraphics[width=0.6\linewidth]{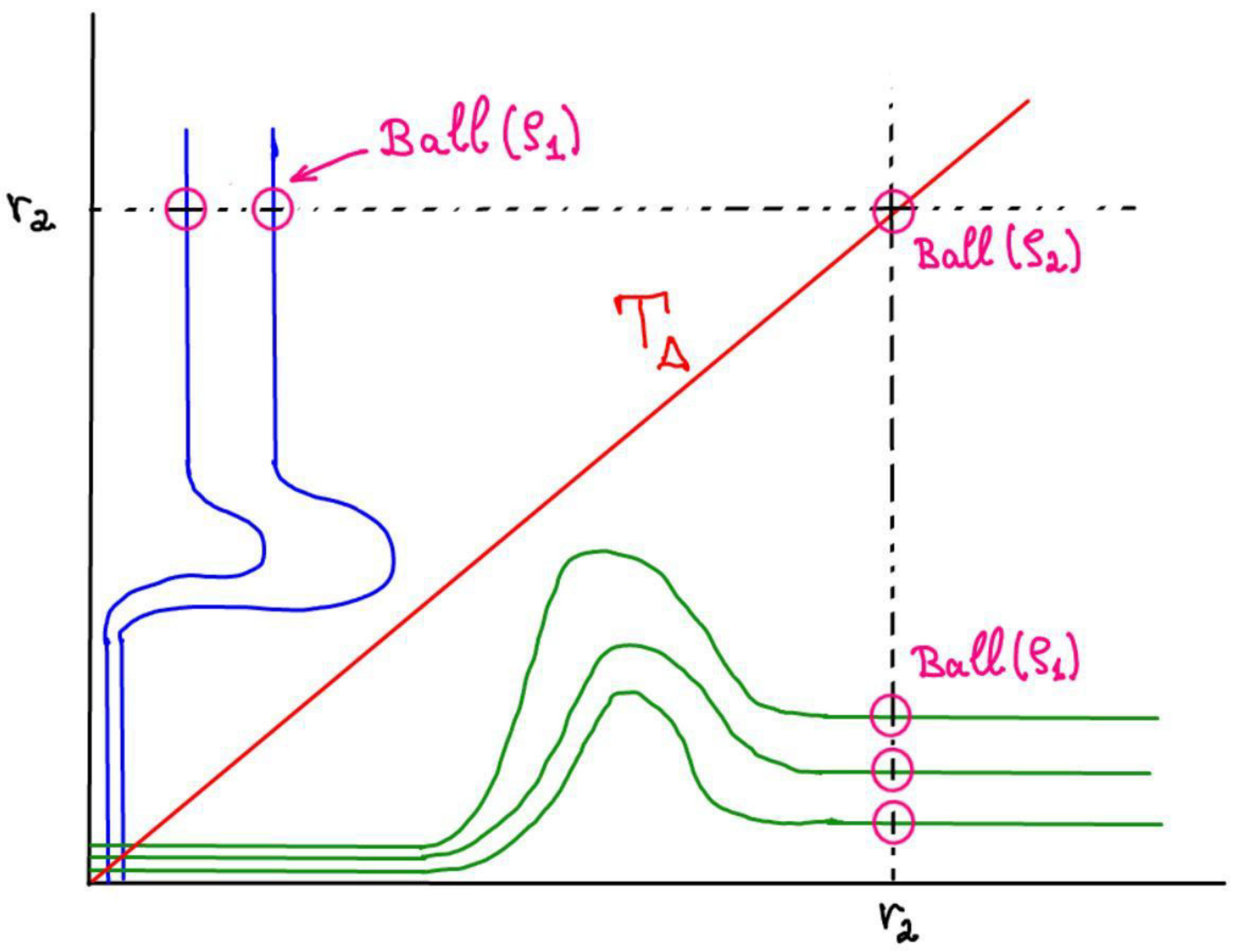}
   \end{center}
   \caption{Estimating the area of holomorphic curves that go out of 
     $B(2r_0/3, 2r_0/3)$.
     \label{f:area-estim-Lelong}}
\end{figure}

Similarly, choose $\rho_2>0$ such that for every $x \in \partial
B(r_1,r_1)$ the closed ball $\overline{\textnormal{Ball}_x(\rho_2)}$
is disjoint from $B(r_2, r_2)$ and is also contained inside $B =
B(r_0, r_0)$.

Set $C: = \min \{\tfrac{\pi}{2}\rho_2^2, \pi \rho_1^2 \}$.

Now let $u: D' \longrightarrow \mathcal{E}$ be a solution
of~\eqref{eq:Floer-eq-D'} and assume first that $u$ satisfies the
following special assumption: $u(\partial D') \not \subset B(r_2,
r_2)$. We will prove that $A_{\Omega}(u) \geq C$.

Since $u(z_i) \in B(2r_0/3, 2r_0/3)$ (recall $z_i$ are the punctures
of $D'$) it follows that there exists $z_* \in \partial D'$ such that
$u(z_*)$ lies in one of the following three: 
\begin{enumerate}
  \item $\mathcal{S}^{W}_i(t) \cap (\partial B'(r_2) \times B'')$ for
   some $i$; or
  \item $\mathcal{S}^{V}_j(t) \cap (B' \times \partial B''(r_2))$ for
   some $j$; or
  \item $T_{\Delta} \cap (\partial B'(r_2) \times
   \partial B''(r_2))$.
\end{enumerate}
Consider now the intersection $u(D') \cap
\overline{\textnormal{Ball}_{u(z_*)}(\rho_2)}$. By the Lelong
inequality (applied after a reflection in the ball with respect to the
corresponding Lagrangian) it follows that
$$A_{\Omega}(u) \geq \tfrac{\pi}{2} \rho_2^2 \geq C.$$ 
(Alternatively one can use an appropriate version of the monotonicity
lemma for minimal surfaces to obtain the same inequality.) We have
thus proved the lemma under the assumption that $u(\partial D') \not
\subset B(r_2, r_2)$.

We are now ready to prove the general case. Assume that $u(D') \not
\subset B(r_1, r_1)$. There are two cases (mutually not exclusive):
either $u(\partial D') \not \subset B(r_1, r_1)$, or
$u(\textnormal{Int\,}D') \not \subset B(r_1, r_1)$.

If the first case occurs then clearly $u(\partial D') \not \subset
B(r_2, r_2)$ and we are done. Therefore we may assume that $u(\partial
D') \subset B(r_2, r_2)$ {\em and} that the second case occurs, namely
$u(\textnormal{Int\,}D') \not \subset B(r_1, r_1)$. It follows that
there is $z_* \in \textnormal{Int\,}D'$ with $u(z_*) \in \partial
B(r_1, r_1)$.  Applying the Lelong inequality for $u(D') \cap
\overline{\textnormal{Ball}_{u(z_*)}(\rho_1)}$ we obtain
$$A_{\Omega}(u) \geq \pi \rho_1^2 \geq C.$$
\Qed

\subsubsection{Proof of Lemma~\ref{p:loc-solutions-1}}
\label{sb:prf-prop-loc-sol}
Before defining the constant $C'$, we first consider solutions $u$
of~\eqref{eq:Floer-eq-D'} that satisfy property~(1) of our proposition
as well as property~(2) with the constant $C'$ replaced by the
constant $C$ from Lemma~\ref{l:small-big}. (The constant $C'$, defined
below, will have the property that $0 < C' \leq C$.) By
Lemma~\ref{l:small-big} we have $u(D') \subset B(r_1, r_1)$. Since
$$\widetilde{W}_t \cap B(r_1, r_2) = \coprod_{k=1}^{s''}
\mathcal{S}^W_k(t), \quad \widetilde{V}_t \cap B(r_1, r_2) =
\coprod_{k=1}^{s'} \mathcal{S}^V_k(t)$$ it follows that
\begin{equation} \label{eq:new-bndry-cond-1}
   u(\partial_{3,1}D') \subset \mathcal{S}^{W}_i(t), \quad 
   u(\partial_{1,2}D') \subset T_{\Delta}, \quad u(\partial_{2,3}D')
   \subset \mathcal{S}^{V}_j(t).
\end{equation}
Thus we are considering here finite energy solutions $u:D'
\longrightarrow B' \times B''$ of~\eqref{eq:Floer-eq-D'} subject to
the boundary condition~\eqref{eq:new-bndry-cond-1} and the asymptotics
(see Figure~\ref{f:triangle-w-v-x}
\begin{equation} \label{eq:asympt-in-B}
   u(z_1) = w_i(t), \quad u(z_2) = v_j(t), \quad u(z_3) = x_{i,j}(t).
\end{equation}
Recall also that our almost complex structure $J$ is in
$\mathcal{J}_0$, hence by definition $J \equiv J^0_B$ on $B = B'
\times B''$.

\begin{figure}[htbp]
   \begin{center}
      \includegraphics[width=0.5\linewidth]{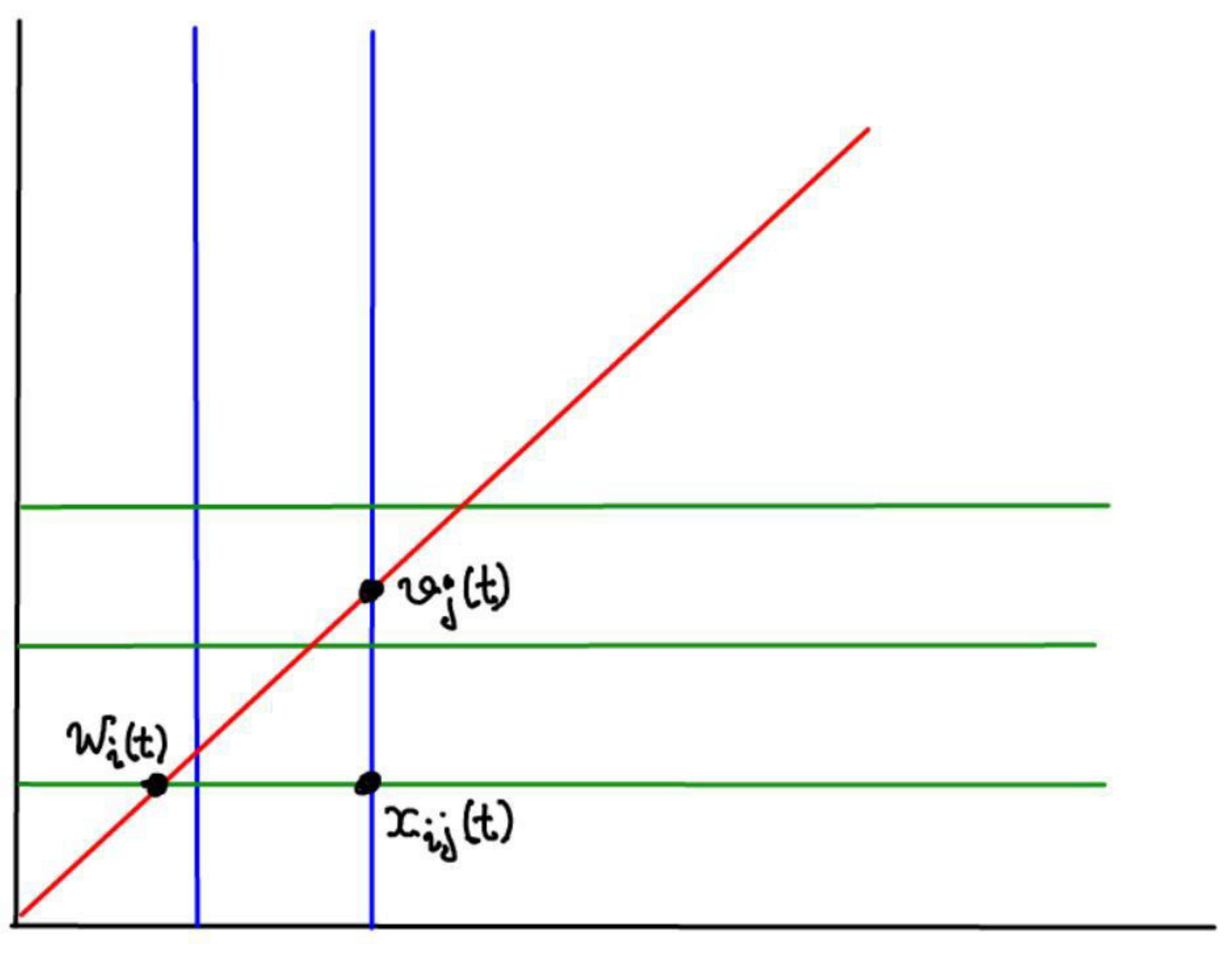}
   \end{center}
   \caption{Holomorphic triangles going from 
     $w_i(t), v_j(t)$ to $x_{i,j}(t)$.
     \label{f:triangle-w-v-x}}
\end{figure}

We now claim that there is a constant $0< C' \leq C$ such that all
solutions $u$ of~\eqref{eq:new-bndry-cond-1} with
asymptotics~\eqref{eq:asympt-in-B} and with $A_{\Omega}(u)\leq C'$
must satisfy $u(D') \subset B(r)/3, r_0/3)$. The proof of this claim
is very similar to that of Lemma~\ref{l:small-big} and in fact even
simpler since we are considering here boundary conditions only on one
pair of sheets ($\mathcal{S}^{W}_i(t)$, $\mathcal{S}^{V}_j(t)$) and
$T_{\Delta}$, and the distance between each pair of these three
Lagrangians outside of $B(r_0/3, r_0/3)$ is uniformly bounded below.

This proves that all solutions $u:D' \longrightarrow \mathcal{E}$ that
satisfy assumptions~(1) and~(2) of our proposition have their images
inside $B(r_0/3, r_0/3)$. 

It remains to show the existence and uniqueness of such solutions, the
area estimate and the regularity. To this end, set:
$$\mathcal{S}^W(t) := \mathbb{R}^m \times \{b''_i(t)\}, 
\quad \mathcal{S}^V(t) := \{b'_j(t) \} \times \mathbb{R}^m, \quad
\mathcal{T}_{\Delta} = \{(x,x) \mid x \in \mathbb{R}^m\}.$$ Clearly
$\mathcal{S}^{W}_i(t)$ coincides with $\mathcal{S}^W(t)$ inside
$B(r_0/3, r_0/3)$ and similarly for $\mathcal{S}^{V}_j(t)$ and
$\mathcal{S}^V(t)$ as well as for $T_{\Delta}$ and
$\mathcal{T}_{\Delta}$. Thus for our purposes we can consider now the
equation~\eqref{eq:Floer-eq-D'} for maps $u:D' \longrightarrow
\mathbb{R}^m \times \mathbb{R}^m$ with $J = J_{\textnormal{std}}$ and
with the following boundary condition and asymptotics:
\begin{equation} \label{eq:new-bndry-cond-2}
   \begin{aligned}
      & u(\partial_{3,1}D') \subset \mathcal{S}^{W}(t), \quad 
      u(\partial_{1,2}D') \subset \mathcal{T}_{\Delta}(t), \quad
      u(\partial_{2,3}D') \subset \mathcal{S}^{V}(t), \\
      & u(z_1) = w_i(t), \quad  u(z_2) = v_j(t), \quad u(z_3) =
      x_{i,j}(t).
   \end{aligned}
\end{equation}

Note that this problem splits. If we rearrange the coordinates by
identifying of $\mathbb{R}^m \times \mathbb{R}^m \cong
(\mathbb{R}^2)^{\times m}$ via the symplectic isomorphism $(p_1,
\ldots, p_m, q_1, \ldots, q_m) \longmapsto (p_1, q_1, \ldots, p_m,
q_m)$ then $J_{\textnormal{std}}$ is sent to the standard split
complex structure (which we continue to denote
$J_{\textnormal{std}}$), and $\mathcal{S}^W(t)$ becomes $(\mathbb{R}
\times q_1(t)) \times \cdots \times (\mathbb{R} \times q_m(t))$, where
$b''_i(t) = (q_1(t), \ldots, q_m(t))$. Similarly $\mathcal{S}^{V}(t)$
becomes $(p_1(t) \times \mathbb{R}) \times \cdots \times (p_m(t)
\times \mathbb{R})$, where $b'_j(t) = (p_1(t), \ldots, p_m(t))$.
Finally, $\mathcal{T}_{\Delta}$ becomes $\Delta_1 \times \cdots
\Delta_m$ where $\Delta_i$ is the diagonal in each of the
$\mathbb{R}^2$ factors. We continue to denote the corresponding three
Lagrangians by $\mathcal{S}^{W}(t)$, $\mathcal{S}^{V}(t)$ and
$\mathcal{T}_{\Delta}$.

We will now write maps $u: D' \longrightarrow (\mathbb{R}^2)^{\times
  m}$ as: $u(z) = (u_1(z), \ldots, u_m(z))$ with $u_k(z) \in
\mathbb{R}^2$. Clearly each of the maps $u_k: D' \longrightarrow
\mathbb{R}^2 \cong \mathbb{C}$ is holomorphic (in the usual sense) and
satisfies the boundary conditions and asymptotics (see
Figure~\ref{f:triangle-in-plane}):
\begin{equation} \label{eq:new-bndry-cond-3} 
   \begin{aligned}
      & u(\partial_{3,1}D') \subset \mathbb{R} \times q_k(t), \quad
      u(\partial_{1,2}D') \subset \Delta_k, \quad
      u(\partial_{2,3}D') \subset p_k(t) \times \mathbb{R}, \\
      & u(z_1) = (q_k(t), q_k(t)), \quad u(z_2) = (p_k(t), p_k(t)),
      \quad u(z_3) = (p_k(t), q_k(t)).
   \end{aligned}
\end{equation}

\begin{figure}[htbp]
   \begin{center}
      \includegraphics[width=0.5\linewidth]{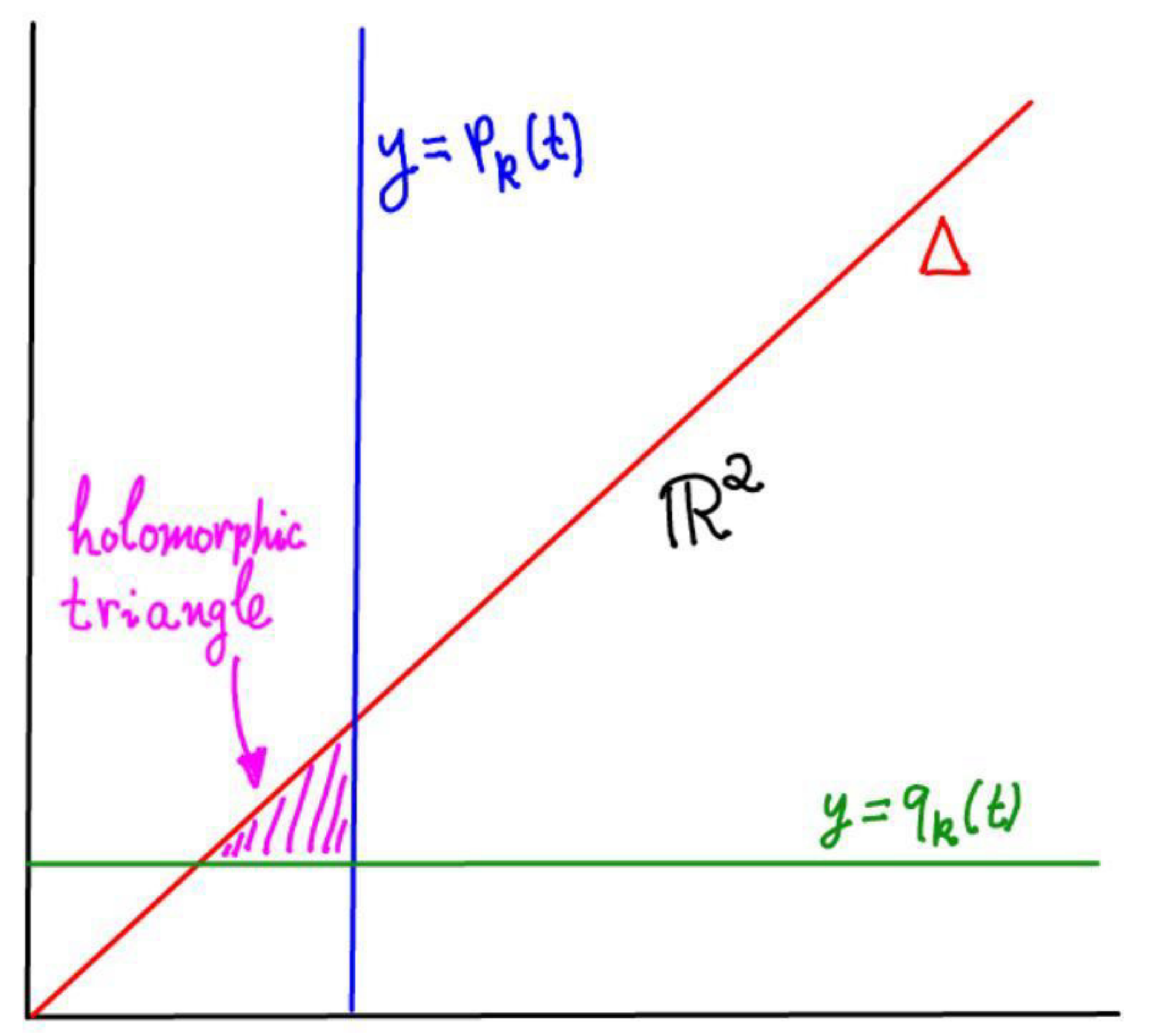}
   \end{center}
   \caption{Holomorphic triangles in $\mathbb{R}^2$ 
     corresponding to the projection on the $k$'th factor of $u$.
     \label{f:triangle-in-plane}}
\end{figure}

Standard 1-dimensional complex analysis show that there is a unique
holomorphic map $u^0_k: D' \longrightarrow \mathbb{C} \cong
\mathbb{R}^2$ with the boundary
conditions~\eqref{eq:new-bndry-cond-3}, the image of which is
precisely the triangle consisting of the convex hull of the three
points $(q_k(t), q_k(t))$, $(p_k(t) p_k(t))$, $(p_k(t), q_k(t))$.
Moreover, a straightforward calculation (using e.g. the methods from
Chapter 13 of~\cite{Se:book-fukaya-categ}) shows that the Maslov index
of $u^0_k$ is $0$ and that the standard complex structure of
$\mathbb{C}$ is regular for this solution.

Note that the mutual position of the three Lagrangians
from~\eqref{eq:new-bndry-cond-3} plays a crucial role here. If for
example, one would replace $\Delta_k$ by the anti-diagonal line
$\{(x,-x): x \in \mathbb{R}\}$ then there would be no solutions with
the boundary conditions~\eqref{eq:new-bndry-cond-3}, the reason being
that the order of the punctures $z_1, z_2, z_3$ on $\partial D$ is
``wrong''.)

It follows that $u^0(z)= (u^0_1(z), \ldots, u^0_m(z))$ is the unique
holomorphic map $u: D' \longrightarrow (\mathbb{R}^2)^{\times m}$
satisfying~\eqref{eq:new-bndry-cond-2}. Since the
$\overline{\partial}$-operator splits in a compatible way with the
splitting $(\mathbb{R}^2)^{\times m}$ it follows that the index of
$u^0$ is $0$ and that $J_{\textnormal{std}}$ is regular.

Finally, it is clear that the symplectic area $A_{\Omega}(u^0)$ of
$u^0$ is the sum of the areas of the triangles $u^0_k$, $k=1, \ldots,
m$. Since $p_k(t), q_k(t) \xrightarrow[t \longrightarrow 1^{-}]{} 0$
it follows that $A_{\Omega}(u^0) \xrightarrow[t \longrightarrow
1^{-}]{} 0$. 

This concludes the proof of the proposition. 

\begin{remnonum}
   An alternative calculation of the index and regularity can be done
   by degenerating the problem to $t=1$. Then the three Lagrangians
   forming the boundary conditions in~\eqref{eq:new-bndry-cond-3}
   become $\mathbb{R} \times \{0\}$, $\Delta_k$ and $\{0\} \times
   \mathbb{R}$.  The asymptotics at the punctures become $u_k(z_1) =
   u_k(z_2) = u_k(z_3) = (0,0)$.  It is easy to see that the only
   solution now is the constant solution at $(0,0)$. The fact that its
   index is $0$ and that $J$ is regular follow e.g.
   from~\cite{Bi-Co:lcob-fuk}~(section~4.3). By a standard implicit
   function theorem it follows that the same holds for $t$'s close
   enough to $1$. Note that also here, if one would replace $\Delta_k$
   by a line going through the 2'nd and 4'th quadrants, e.g.
   $\{(x,-x): x \in \mathbb{R}\}$, things would go wrong. The constant
   map at $0$ would still be a solution but its index would be
   negative and $J$ would not be regular with respect to it.
\end{remnonum}

It remains to discuss the case when $X$ is non-compact but
symplectically convex at $\infty$.  The proof is very similar to the
one for the case when $X$ is closed. Recall that although now $X$ is
not compact the objects of $\fuk^*(X)$ (i.e.  the Lagrangians in $X$)
are still assumed to be compact.

The results of Seidel (see Chapter 16e of~\cite{Se:book-fukaya-categ}
and~\cite{Se:long-exact}) can be used to produce a fibration
$\mathcal{E}$ of generic fibre $X$, in the sense of the definitions in
\S\ref{sb:defs-lef-fibr}, in particular this fibration satisfies
assumption $T_{\infty}$.  As in the compact fibre case, we then use
the Proposition \ref{p:from-gnrl-to-tame} to transform the fibration
into a tame one that continues to satisfy $T_{\infty}$.  The proof
then pursues just as in the compact case.  Indeed, notice that
Assumption~$T_{\infty}$ implies that the monodromy is well defined
over any path in $\mathbb{C} \setminus \textnormal{Critv}(\pi)$ (and
in fact over any path in $\mathbb{C}$ if we restrict the monodromy to
``infinity in the fibers'').  Similarly, the procedure from
page~\pageref{pg:grad-vf-complete} that ensures that the negative
gradient flow of $\textnormal{Re}(\pi)$ is defined for all times
continues to work in the present setting. Indeed, the fact that the
fibers of $\mathcal{E}$ are not compact does not pose any problems
because (in the notation of Assumption~$T_{\infty}$) on
$\mathcal{E}^{\infty} \approx \mathcal{E}^{\infty}_{w_0} \times
\mathbb{C}$ this flow is just a translation in the
$\mathbb{C}$-direction.  Finally, in what concerns the Floer and
perturbation data we use as in \S\ref{subsubsec:non-comp-J} almost
complex structures that are split at $\infty$ as $i\oplus J_{0}$ with
$J_{0}$ compatible with the symplectic convexity of (the end) of $X$.
\Qed

\subsubsection{Proof of Proposition~\ref{t:ex-tr-cob}}
\label{sbsb:prf-ex-tr-cob}

We now explain how to modify the proof of
Proposition~\ref{t:ex-tr-compact} under the assumptions of
Proposition~\ref{t:ex-tr-cob}, namely that $X$ is itself the total
space of a tame Lefschetz fibration $\pi_X : X \longrightarrow
\mathbb{C}$ as described in~\S\ref{sbsb:ex-tr-cob}. Denote by
$(N,\omega)$ the generic fibre of $\pi_{X}$ which is compact or
symplectically convex at infinity.

As in the the proof of Proposition~\ref{t:ex-tr-compact}, we again
construct a Lefschetz fibration $\pi_{\mathcal{E}}: \mathcal{E}
\longrightarrow \mathbb{C}$ with fiber over $w_0$ being $X$. As
before, the fibration $\mathcal{E}$ can be assumed to satisfy
Assumption~$T_{\infty}$ as well as the other assumptions in
\S\ref{sb:defs-lef-fibr}. By applying to this fibration the same
procedure as in the proof of Proposition \ref{p:from-gnrl-to-tame} we
may further assume that this fibration is also tame.
 
In what concerns the Fukaya category $\fuk^*(\mathcal{E})$ of
$\mathcal{E}$, by inspecting the proof of Proposition
\ref{t:ex-tr-compact}, we see that we can actually use here only a
smaller category whose objects are cylindrical cobordisms $V \subset
\mathcal{E}$ (not necessarily negatively ended) obtained by taking the
trail of a given cobordism $Q \subset \mathcal{E}_{w_0} = X$ along a
curve $\gamma \subset \mathbb{C} \setminus
\textnormal{Critv}(\pi_{\mathcal{E}})$.  To avoid confusion denote the
Fukaya category involved here by $\fuk^{\ast}_{r}(\mathcal{E})$ (where
$r$ indicates that our objects are restricted as above).  Notice that
later in the proof we apply certain isotopies (e.g. the negative
gradient flow of $\textnormal{Re}(\pi_{\mathcal{E}})$) to these
cobordisms that might not keep them everywhere cylindrical. However,
as we shall see below, this is not a problem since for that stage of
the proof we do not need the entire Fukaya category anymore but only
Floer homology calculations.

Using the notation from Assumption~$T_{\infty}$, put $X^{\infty} =
\mathcal{E}^{\infty}_{w_0}$ and fix a symplectic identification
\begin{equation} \label{eq:E-infty-CX} \mathcal{E}^{\infty} \approx
   \mathbb{C}_{\mathcal{E}} \times X^{\infty}.
\end{equation}
  
Here $\mathbb{C}_{\mathcal{E}}$ stands for the base of the fibration
$\mathcal{E}$, which is just a copy of $\mathbb{C}$.  The subscript
$\mathcal{E}$ is there only in order to emphasize the relation to
$\mathcal{E}$.  Denote by $\widehat{\pi}_X: \mathcal{E}^{\infty}
\longrightarrow \mathbb{C}_X$ the projection (on the other copy of
$\mathbb{C}$) induced via~\eqref{eq:E-infty-CX} by $\pi_X: X^{\infty}
\longrightarrow \mathbb{C}_X$.

Notice that due to the $T_{\infty}$ assumption a cobordism $V\in
\mathcal{O}b(\fuk^{\ast}_{r}(\mathcal{E}))$ has the property that
$V\cap \mathcal{E}^{\infty}$ is a union of finitely many components of
the form $\gamma \times l_{i}\times L_{i}$ where $\gamma\in
\C_{\mathcal{E}}$ is the projection of $V$ onto $\C_{\mathcal{E}}$,
$l_{i}$ is a negative ray in $\C_{X}$ (of imaginary coordinate $i$ in)
and $L_{i}\subset N$ is a Lagrangian in $N$. To fix ideas we will call
these Lagrangians $L_{i}$, {\em the ends of $V$ in the direction of
  $\C_{X}$}.  The important fact to keep in mind is that these ends
remain constant along $\gamma$.  Obviously, there are also the
``usual'' ends of $V$ that are of the form $l_{i}\times C_{i}$ where
$l_{i}$ is a ray (negative or positive) in $\C_{\mathcal{E}}$ and
$C_{i}\subset X$ is a negative-ended cobordism in $X$.  We will refer
to these cobordisms $C_{i}$ as the {\em ends of $V$ in the direction
  of $\C_{\mathcal{E}}$}.  For each $V$ there are obviously at most
two such ends.  Notice also that the ends of $C_{i}$, itself viewed as
cobordism, are Lagrangians in $N$ that coincide with the ends of $V$
in the direction of $\C_{X}$.
 
We now pass to explaining the choices of Floer and perturbation data
required to define the category $\fuk^{\ast}_{r}(\mathcal{E})$.  We
first pick a profile function $h_{X}:\C_{x}\to \R$ such as in
\S\ref{subsubsec:profile} but with the property that the bottlenecks
are inside $\pi_{X}(X^{\infty})$.

Consider $V_{1},\ldots,
V_{k+1}\in\mathcal{O}b(\fuk^{\ast}_{r}(\mathcal{E}))$. Let
$C^{1},\ldots, C^{s}\in \mathcal{O}b(\fuk^{\ast}(X))$ be the
collection of all the ends in the direction of $\C_{X}$ of the objects
$V_{1},\ldots, V_{k+1}$. We use the function $h_{X}$ and the method in
\S\ref{subsubsec:perturb} to construct the Floer and perturbation
data, associated to $C^{1},\ldots, C^{s}$ as objects of the category
$\fuk^{\ast}(X)$ associated to the tame Lefschetz fibration
$\pi_{X}:X\to \C$.  We denote all this data by
$\mathcal{D}^{X}_{V_{1},\ldots, V_{k+1}}$.  As described in
\S\ref{subsubsec:Fuk-cob}, this data consists of particular choices of
Hamiltonians on $X$, that are grouped here in
$\mathcal{H}^{X}_{V_{1},\ldots, V_{k+1}}$, and almost complex
structures on $X$, grouped in $\mathcal{J}^{X}_{V_{1},\ldots,
  V_{k+1}}$ so that $\mathcal{D}^{X}_{V_{1},\ldots,
  V_{k+1}}=(\mathcal{H}^{X}_{V_{1},\ldots, V_{k+1}},
\mathcal{J}^{X}_{V_{1},\ldots, V_{k+1}})$.

Pick a profile function $h_{\mathcal{E}}:\C_{\mathcal{E}}\to \R$ again
as described in \S\ref{subsubsec:profile}.  Let $\gamma_{i}$ be the
projection of $V_{i}$ onto $\C_{\mathcal{E}}$.  Now modify
$h_{\mathcal{E}}$, away from the region of the bottlenecks, in such a
way that the new function $h_{V_{1},\ldots, V_{k+1}}$ conserves the
same bottlenecks as $h_{\mathcal{E}}$ and, additionally,
$(\phi^{h_{V_{1},\ldots, V_{k+1}}}_{1})^{-1}(\gamma_{i})$ is
transverse to $\gamma_{j}$ for all $i,j$.  Now define a new set of
Hamiltonians, this time defined on $\C_{\mathcal{E}}\times X$ as
follows: $\mathcal{H}'_{V_{1},\ldots, V_{k+1}}=\{h_{V_{1},\ldots,
  V_{k+1}}+H \ : \ H\in \mathcal{H}^{X}_{V_{1},\ldots, V_{k+1}}\}$.
 
With these choices, we can describe the constraints on the class of
Hamiltonians $\mathcal{H}^{\mathcal{E}}_{V_{1},\ldots, V_{k+1}}$
defined on $\mathcal{E}$ that are part of the perturbation data
$\mathcal{D}^{\mathcal{E}}_{V_{1},\ldots,
  V_{k+1}}=(\mathcal{H}^{\mathcal{E}}_{V_{1},\ldots, V_{k+1}},
\mathcal{J}^{\mathcal{E}}_{V_{1},\ldots, V_{k+1}})$ that we associate
to the family $V_{1},\ldots, V_{k+1}$, as required to define
$\fuk^{\ast}_{r}(\mathcal{E})$. There is a compact set
$K_{V_{1},\ldots, V_{k+1}}\subset \C_{\mathcal{E}}$ away from the
bottlenecks of $h_{\mathcal{E}}$ and a compact set $K'_{V_{1},\ldots,
  V_{k+1}}\subset \mathcal{E}^{\infty}$ away from the bottlenecks of
$h_{X}$ so that the hamiltonians in
$\mathcal{H}^{\mathcal{E}}_{V_{1},\ldots, V_{k+1}}$ coincide with
corresponding Hamiltonians in $\mathcal{H}'_{V_{1},\ldots, V_{k+1}}$
on the set
$$\mathcal{S}_{V_{1},\ldots V_{k+1}}=(\mathcal{E}^{\infty}\setminus 
K'_{V_{1},\ldots, V_{k+1}})\cup
\pi_{\mathcal{E}}^{-1}(\C_{\mathcal{E}}\setminus K_{V_{1},\ldots,
  V_{k+1}})~.~$$ It is useful to notice at this point that, because
the ends of $V_{i}$ in the direction of $\C_{X}$ do not change along
$\gamma_{i}$ this choice of Hamiltonian perturbations ensures the
required transversality at $\infty$ both in the $\C_{\mathcal{E}}$
direction as well as in the $\C_{X}$ direction. As the Hamiltonians in
$\mathcal{H}^{\mathcal{E}}_{V_{1},\ldots, V_{k+1}}$ are basically
arbitrary perturbations of the Hamiltonians in
$\mathcal{H}'_{V_{1},\ldots, V_{k+1}}$ outside of
$\mathcal{S}_{V_{1},\ldots V_{k+1}}$ this (together with the choice of
almost complex structures as detailed below) is also sufficient to
achieve the regularity of the relevant moduli spaces.

The family of almost complex structures
$\mathcal{J}^{\mathcal{E}}_{V_{1},\ldots, V_{k+1}}$ associated to
$V_{1},\ldots, V_{k+1}$ satisfies similar constraints. Namely, over
$\mathcal{S}_{V_{1},\ldots V_{k+1}}$ they are of the form
$i_{\mathcal{E}}\oplus J$ with $J\in \mathcal{J}^{X}_{V_{1},\ldots,
  V_{k+1}}$ but can be perturbed freely, so as to insure regularity,
outside of $\mathcal{S}_{V_{1},\ldots V_{k+1}}$.

With these choices the compactness results required to define the
category $\fuk^{\ast}_{r}(\mathcal{E})$ are valid. More specifically,
all solutions $u$ of the relevant perturbed Cauchy-Riemann equation
lie in a prescribed compact subset. The argument is very similar to
the one in \cite{Bi-Co:lcob-fuk}. We consider a hamiltonian
$\bar{h}:\mathcal{E}\to \R$ so that away from
$\mathcal{S}_{V_{1},\ldots V_{k+1}}$, $\bar{h}$ coincides with
$h_{\mathcal{E}}\oplus h_{X}$. We then use the naturality
transformation involving $\bar{h}$, as summarized
in~\S\ref{sbsb:nat-transf}, to turn the solutions $u$ into curves $v$
that are (non-perturbed) $J$-holomorphic away from
$\mathcal{S}_{V_{1},\ldots V_{k+1}}$. We then apply the open mapping
theorem to the projections $\widehat{\pi}_X \circ v$ and
$\pi_{\mathcal{E}} \circ v$. To summarize, the arguments for both
regularity and compactness of the relevant moduli spaces follow
closely the corresponding arguments in \cite{Bi-Co:lcob-fuk} that are
used to set up the Fukaya category of cobordisms in $\C\times M$.
  
Beyond the definition of $\fuk^{\ast}_{r}(\mathcal{E})$ an additional
remark is in order. A key part of the proof
in~\S\ref{sbsb:prf-ex-tr-compact} uses the Floer homology for the
pairs $(W, V)$, $(W, T_{\Delta})$ and $(T_{\Delta}, V)$. In the course
of the proof we apply to $W$ and $V$ the negative and positive
gradient flows of $\textnormal{Re}(\pi_{\mathcal{E}})$.  While $V$ and
$W$ are cylindrical, these flows do not preserve cylindricity.
Nevertheless, cylindricity is preserved at infinity in the
fiber-direction due to Assumption~$T_{\infty}$ on $\mathcal{E}$.
Therefore the Floer data can easily be adjusted in this case too by
using possibly another compactly supported perturbation to ensure
transversality.
  
With this remark taken into account and with the definition of
$\fuk^{\ast}_{r}(E)$ as above the remainder of the proof proceeds just
as in the proof of Proposition~\ref{t:ex-tr-compact}.

% !TEX root = lefcob.tex

\subsection{The decomposition in Theorem~\ref{thm:main-dec-gen0}}
\label{subsec:prof-main-t}

To construct this decomposition we start with the proof of
Theorem~\ref{thm:main-dec}.

\subsubsection{Proof of Theorem \ref{thm:main-dec}}
\label{subsubsec:proof-main-dec}
We assume for the moment that we are in the setting of
\S\ref{subsec:dec-tame}.  In particular, $\pi: E\to \C$ is a tame
Lefschetz fibration with the properties listed there.

Let $V:\emptyset \cobto (L_{1},\ldots, L_{s})$ and consider the
Lefschetz fibration $\hat{\pi}:\hat{E}\to \C$ obtained from $E$ by
adding singularities as described
in~\S\ref{subsubsec:null-cob-remote}. \pbhl{By}
Proposition~\ref{p:strong-mon-Ehat} \pbhl{$\hat{E}$ is strongly
  monotone. The cobordism $V$ continues to be monotone in $\hat{E}$
  and the matching spheres $\hat{S}_j$ are monotone too. Moreover, all
  these Lagrangians are of monotonicity class $*$. Recall also that by
  assumption $\dim_{\mathbb{R}}E \geq 4$.} Consider now the cobordism
$$V'=\tau_{\hat{S}_{m}} \circ \tau_{\hat{S}_{m-1}} \circ
\cdots \circ \tau_{\hat{S}_{1}}(V) \subset \hat{E}.$$ 

Given $W \in \mathcal{L}^{\ast}(E)$ we rewrite the exact sequence in
Proposition~\ref{t:ex-tr-cob} as $$ W= ( S\otimes
HF(S,W)\to \tau_{S} W)$$ and deduce that in $D\fuk^*(\hat{E})$ we have
the following decomposition of $V$:
$$V \cong (\hat{S}_{1}\otimes E_{1}\to \hat{S}_{2}\otimes E_{2}
\to \ldots \to \hat{S}_{m}\otimes E_{m}\to V'),$$ where 
\begin{equation}\label{eq:E-i-s}
E_{i}=HF(\hat{S}_{i},\tau_{\hat{S}_{i-1}} \circ \cdots \circ 
\tau_{\hat{S}_{1}}(V))~.~
\end{equation}

Notice that in $D\fuk^{\ast}(E)$ we have $T_{i}\cong
(J^{E,\hat{E}})^{\ast}(\hat{S}_{i})$ where $J^{E,\hat{E}}$ is the
inclusion~\eqref{eq:Lef-inclusion} and $T_{i}$ are the thimbles in the
statement of Theorem \ref{thm:main-dec}.  Thus, in
$D\fuk^{\ast}(E)$ we have the decomposition:
 
\begin{equation}\label{eq:thimbles-dec}
   V\cong (T_{1}\otimes E_{1}\to T_{2}\otimes E_{2}\to \ldots 
   \to T_{m}\otimes E_{m}\to V')~.~
\end{equation}
 
By Corollary \ref{cor:moving-cob} we know that inside
$D\fuk^{\ast}(E)$ we have:
\begin{equation} \label{eq:ends-dec} V'\cong (\gamma_{s}\times L_{s}\to
   \gamma_{s-1}\times L_{s-1}\to\ldots \to \gamma_{2}\times L_{2})
\end{equation}
 
Splicing together (\ref{eq:thimbles-dec}) and (\ref{eq:ends-dec}) we
obtain:
$$V\cong (
T_{1}\otimes E_{1}\to \ldots \to T_{m}\otimes E_{m}\to
\gamma_{s}\times L_{s}\to \ldots \to \gamma_{2}\times L_{2})$$ which
concludes the proof of Theorem \ref{thm:main-dec}. \Qed

\medskip 

\subsubsection{The decomposition in Theorem~\ref{thm:main-dec-gen0}}
\label{sbsb:prf-main-dec-gen0} 
We assume the setting from Theorem \ref{thm:A-rfm} (which we recall is
just a more precise reformulation of Theorem \ref{thm:main-dec-gen0})
and recall a bit of the necessary background.  
The fibration $\pi:E\to
\C$ is no longer assumed to be tame. All the singularities of $\pi$
are included in $\pi^{-1}(S_{x,y})$, $x<0<y$ and there is a tame fibration
$\pi:E_{\tau}\to \C$ that coincides with $E$ over $[x-\frac{7}{2},y+\frac{7}{2}]\times
[-\frac{1}{2}, \infty)$ and is tame outside of a set $U$ that contains
$(x-4,y+4)\times (-1,\infty)$.  Recall also the category $\fuk^{\ast}(E_{\tau})$
whose objects are cobordisms (with only negative ends) as in Definition \ref{def:Lcobordism}.
In particular, these cobordisms have ends that project to the axes $(-\infty, -a_{U}]\times \{i\}\subset \C$.
The constant $a_{U}$ verifies $-a_{U}< x-4$.  Recall from
\S\ref{subsubsec:fuk-cob-gen} that the objects of the category
$\fuk^{\ast}(E;\tau)$ are uniformly monotone cobordisms $V\subset E$
that are cylindrical outside $S_{x-3,y-3}$ and the operations $\mu_{k}$ of
$\fuk^{\ast}(E;\tau)$ are defined by means of the corresponding
operations in the category $\fuk^{\ast}(E_{\tau})$ associated to the
tame fibration $E_{\tau}$.

The decomposition in Theorem~\ref{thm:A-rfm} (and thus that in
Theorem~\ref{thm:main-dec-gen0}) follows rapidly from that in Theorem
\ref{thm:main-dec}.  Indeed, recall from \S\ref{subsubsec:fuk-cob-gen}
that we have an inclusion:
\begin{equation}\label{eq:inclu-fuk-cat}
\fuk^{\ast}(E;\tau) \to \fuk^{\ast}(E_{\tau})
\end{equation}
that is a quasi-equivalence and which, on objects, is defined by $V\to
\overline{V}$ where $\overline{V}$ is obtained by cutting off the the
ends of $V$ along the line $\{x-\frac{7}{2}\}\times \R$ and extending them
horizontally by parallel transport in the fibration $E_{\tau}$.  As
$E_{\tau}$ is a tame fibration, Theorem \ref{thm:main-dec} can be
applied to it.  We deduce decompositions involving two types of curves
in the plane, the $t_{k}$'s and $\gamma_{i}$'s as in Figure
\ref{fig:spec-curves}.  The curves $\gamma_{i}$ appearing here are
included in the negative quadrant $Q_{U}^{-}=(-\infty, -a_{U}]\times
[0,\infty)$ and they are away from $U$.  For reasons that will become
clear in a moment, it is convenient to refine the notation for these
curves such as to explicitly indicate their dependence on $U$.  Thus
we will further denote them by $\gamma^{U}_{i}$.
 
The decomposition result that we want to show here - for the statement
of Theorem \ref{thm:A-rfm} - applies to $\fuk^{\ast}(E;\tau)$.  It
again involves the same thimbles $T_{k}$ associated to the curves
$t_{k}$ as before as well certain ``trails'' denoted in Theorem
\ref{thm:A-rfm} by $\gamma_{i}L_{i}$.  It is important to notice at
this point that the curves $\gamma_{i}$ appearing in
the statement of Theorem \ref{thm:A-rfm} 
do not coincide with the $\gamma^{U}_{i}$'s above - see also Figure \ref{fig:u-and-x}.
Indeed, following the definition in \S\ref{subsubsec:atoms}, these curves 
have image inside $(-\infty, x)\times [\frac{1}{2},\infty)$ and they ``bend'' inside $[x-2,x-1]\times [1,\infty]$, while $\gamma^{U}_{i}$ is away from $U$ and thus away from $(x-4,y+4)\times \R$. 

\begin{figure}[htbp]
      \begin{center}
        \includegraphics[scale=0.7]{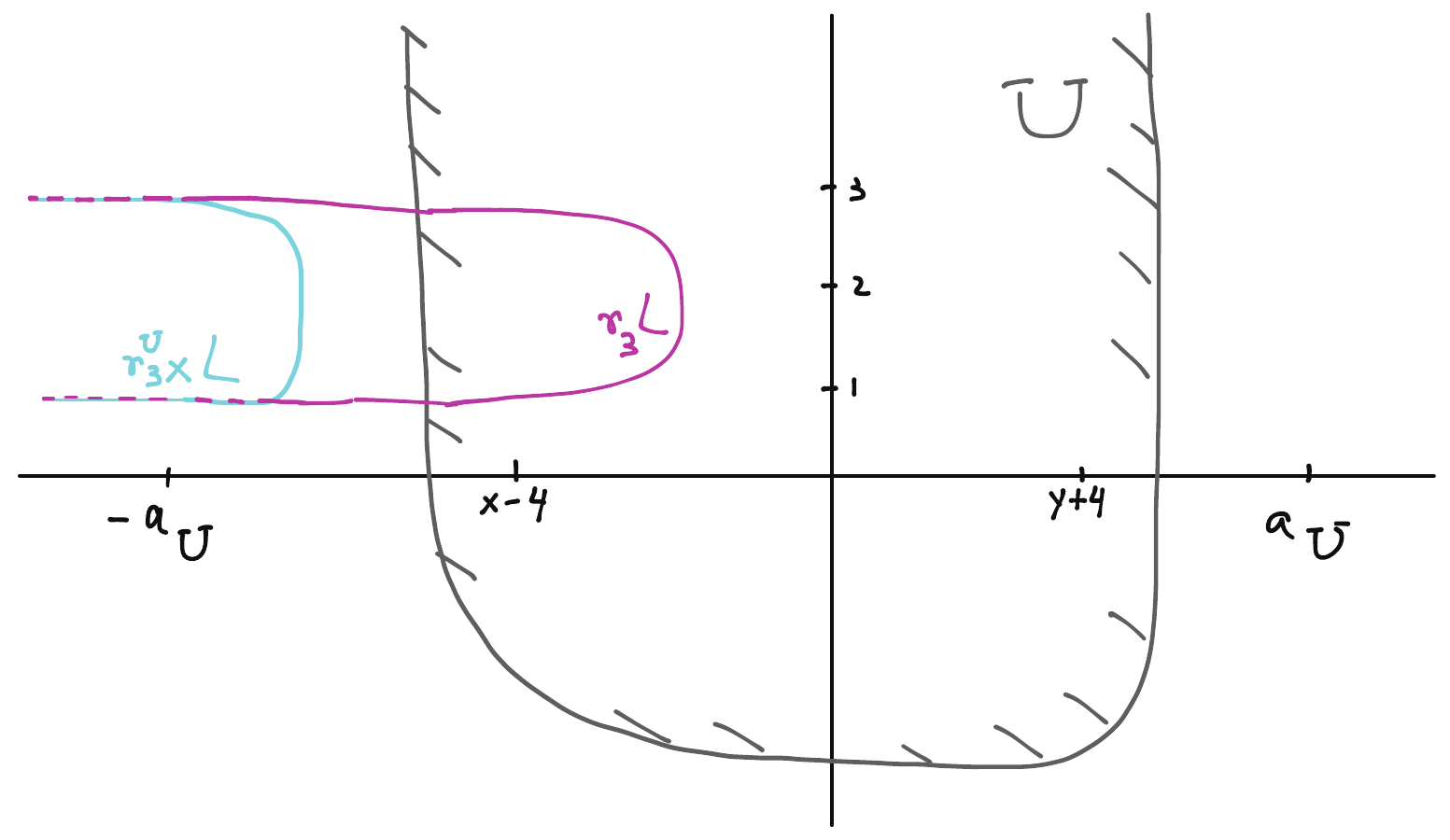}
      \end{center}
      \caption{The Lagrangian $\gamma^{U}_{3}\times L$ is an object in
        $\fuk^{\ast}(E_{\tau})$ but is not cylindrical outside
        of $[x-3,y+3]\times \R$ and thus it not an object in
        $\fuk^{\ast}(E,\tau)$.
        \label{fig:u-and-x}}
   \end{figure}
   Nonetheless, for $L\in\mathcal{L}^{\ast}(M)$ and any curve
   $\gamma_{i}$ consider the cobordism $\overline{\gamma_{i}L}$ as an
   object of $\fuk^{\ast}(E_{\tau})$. This object is quasi-isomorphic
   to $\gamma^{U}_{i}\times L$ (this can proved directly, but it also
   follows immediately from Theorem~\ref{thm:main-dec} itself).  As a
   consequence, we may replace in the decomposition given by Theorem
   \ref{thm:main-dec} the objects $\gamma^{U}_{i}\times L_{i}$ by the
   objects $\overline{\gamma_{i}L_{i}}$ and by pulling back the
   resulting decomposition from $\fuk^{\ast}(E_{\tau})$ to
   $\fuk^{\ast}(E;\tau)$ via the inclusion (\ref{eq:inclu-fuk-cat}) we
   obtain the decomposition claimed in Theorem \ref{thm:A-rfm}.  \Qed

% !TEX root = lefcob.tex

\section{Main consequences} \label{sec:conseq}

\subsection{From the total space to the fiber and back}
\label{subsec:fibr-tot}

We will work in this subsection only with tame Lefschetz fibrations -
see Definition~\ref{df:tame-lef-fib}.  In view
of~\S\ref{sb:tame-vs-gnrl} this is not restrictive.  Thus we assume
that $\pi: E\to \C$ is a Lefschetz fibration which is tame outside of
$U\subset \C$ and $(M,\omega)$ is the generic fibre. The fibration $E$
has singularities $x_{1},\ldots, x_{m}$ of respective critical values
$v_{1}, \ldots, v_{m}$ (assumed to be, for simplicity, $v_{k}=(k,\frac{3}{2})$).   
Denote by $O \in \mathbb{C}$ the origin and
  recall that the fibration $E$ is assumed to be tame over a region
  that contains $O$. Connect each critical value $v_k$ to $O$ by a
  straight segment, and denote by $S_k \in \pi^{-1}(O) = M$ the
  vanishing cycle associated to that path.

We use the rest of the set-up and notation
from~\S\ref{subsec:dec-tame}.  The results described below are all
consequences of Theorem~\ref{thm:main-dec}.

\subsubsection{Descent: from decompositions in $D\fuk^{\ast}(E)$ to
  decompositions in $D\fuk^{\ast}(M)$} \label{subsubsec:desc}

\begin{cor} \label{cor:dec-M} As in Theorem~\ref{thm:main-dec}, let
   $V\in\mathcal{L}^{\ast}(E)$, $V:\emptyset \to (L_{1},\ldots,
   L_{s})$. Then there exists an iterated cone decomposition that
   depends on $V$ and takes place in $D\fuk^{\ast}(M)$:
   \begin{equation} \label{eq:cone-dec-M}
      \begin{aligned}
         L_{1} \cong \bigl( 
         \widetilde{\tau}_{2, \ldots,
           m}^{-1} S_{1} \otimes E_{1} & \to \widetilde{\tau}_{3, \ldots,
           m}^{-1}S_{2} \otimes E_{2} \to \cdots \\
         & \to \widetilde{\tau}^{-1}_{i+1, \ldots, m} S_{i} \otimes
         E_i \to \cdots \to  S_{m}\otimes E_{m}\to L_{s}\to L_{s-1} \to \cdots \to L_{2} \bigr),
      \end{aligned}
   \end{equation}
   where $\widetilde{\tau}_{i, \ldots, m}$ stands for the composition:
   $$\widetilde{\tau}_{i, \ldots, m} = \tau_{S_{i}} 
   \circ \tau_{S_{i+1}} \circ \cdots \circ \tau_{S_{m}}~.~$$
 
\end{cor}

\begin{proof}
   In this proof it is convenient to consider again the category
   $D\fuk^{\ast}_{\frac{1}{2}}(E)$ from \S\ref{subsec:dec-Yo}.  Recall
   that the difference between this category and $D\fuk^{\ast}(E)$ is
   that the objects $V$ of the underlying category
   $\fuk^{\ast}_{\frac{1}{2}}(E)$ are more general cobordisms than
   those given in Definition \ref{def:Lcobordism} in that the
   imaginary coordinates of the ends of $V$ are allowed to also be
   positive half-integers. In other words, $V$ has only negative ends
   and
   $$V\cap \pi^{-1}(Q_{U}^{-})= \coprod_{i}((-\infty, -a_{U}]\times 
   \frac{i}{2})\times L_{i}~.~$$ We now consider curves $\eta_{i}$ as
   in Figure \ref{fig:curves-eta}.

   \begin{figure}[htbp]
      \begin{center}
      \includegraphics[scale=0.5]{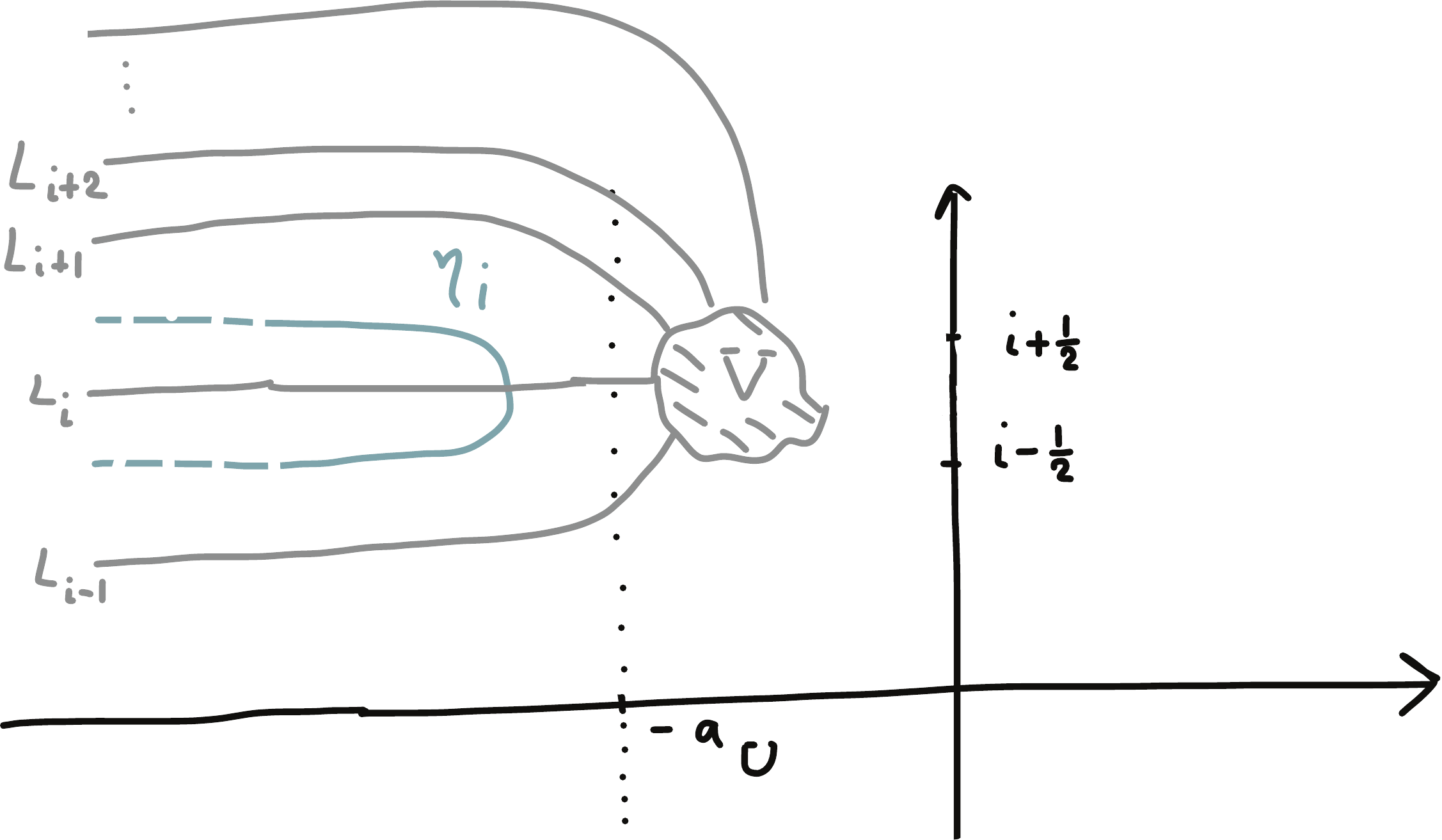}
      \end{center}
      \caption{The auxiliary curves $\eta_{i}$ together with the cobordism 
        $V\in \mathcal{L}^{\ast}(E)$.
        \label{fig:curves-eta}}
   \end{figure}

   These curves satisfy
   $$\eta_{i}((-\infty,-1])=(-\infty,-a_{U}-2]\times \frac{2i-1}{2}\ ,
   \ \eta_{i}([1,+\infty))=(-\infty, -a_{U}-2]\times \frac{2i+1}{2}$$
   and $\eta_{i}(\R)\subset Q_{U}^{-}$.
 
   As shown in ~\cite{Bi-Co:lcob-fuk} \S 4  there exists an
   $A_{\infty}$-functor:
   $$i^{\eta_{j}}:\fuk^{\ast}(M)\to \fuk^{\ast}_{\frac{1}{2}}(E)$$
   which acts on objects by $L \longmapsto \eta_{j}\times L$.
   Consider now the pull-back functor:
   $$(i^{\eta_{j}})^{\ast}: mod(\fuk^{\ast}_{\frac{1}{2}}(E))\to 
   mod(\fuk^{\ast}(M))~.~$$

   Notice that there is a full and faithful embedding
   $e:\fuk^{\ast}(E)\to \fuk^{\ast}_{\frac{1}{2}}(E)$.  Consider the
   Yoneda embeddings $\mathcal{Y}:\fuk^{\ast}(E)\to mod
   (\fuk^{\ast}(E))$ and
   $\mathcal{Y}_{\frac{1}{2}}:\fuk^{\ast}_{\frac{1}{2}}(E)\to mod
   (\fuk^{\ast}_{\frac{1}{2}}(E))$.  Let $\mathcal{Y}':
   \fuk^{\ast}(E)\to mod (\fuk^{\ast}_{\frac{1}{2}}(E))$ be
   $\mathcal{Y}'=\mathcal{Y}_{\frac{1}{2}}\circ e$. The homology
   category associated to the triangular completion
   $(Image(\mathcal{Y}'))^{\wedge}$ of the image of $\mathcal{Y}'$
   inside $ mod(\fuk^{\ast}_{\frac{1}{2}}(E))$ is easily seen to be
   quasi-equivalent to $D\fuk^{\ast}(E)$ (see
   also~\S\ref{subsec:Fuk-fibr}).

   For an object $V\in \fuk^{\ast}(E)$ let
   $\mathcal{M}'_{V}=\mathcal{Y}'(V)$.  Notice that
   $(i^{\eta_{j}})^{\ast}(\mathcal{M}'_{V})$ is precisely the Yoneda
   module associated to the $j$-end of $V$.  Thus $i^{\eta_{j}}$ takes
   Yoneda modules to Yoneda modules and given that
   $H(Image(\mathcal{Y}')^{\wedge})=D\fuk^{\ast}(E)$ we deduce
   that the functor $(i^{\eta_{j}})^{\ast}$ induces a functor of
   triangulated categories
   \begin{equation} \label{eq:functor-R-j}
      \mathcal{R}_{j} :D\fuk^{\ast}(E)\to D\fuk^{\ast}(M)
   \end{equation}
   that we will refer to as the restriction to the $j$-th end.

   The decomposition in the statement is obtained by applying
   $\mathcal{R}_{1}$ to the decomposition in
   Theorem~\ref{thm:main-dec}.   Symplectic Picard-Lefschetz theory shows that the  end
    of the thimble $T_{k}$ is Hamiltonian isotopic
   to $(\tau_{S_{m}}^{-1}\circ \tau_{S_{m-1}}^{-1}\circ
   \tau_{S_{k+1}}^{-1}) (S_{k})=\widetilde{\tau}^{-1}_{k+1, \ldots, m}
   S_{k}$ and its projection to $\mathbb{C}$ has $y$-coordinate $1$.
   Clearly, the end of $\gamma_{k}\times L_{k}$ over $y=1$ is
     $L_{k}$ for $k\geq 2$ and, similarly, the end of $V$ over $y=1$
     is $L_{1}$.
\end{proof}

\begin{rem}\label{rem:restr-funct1}
   The functor $\mathcal{R}_{j}$ from~\eqref{eq:functor-R-j} can also
   be interpreted in a different fashion. We can view it as the
   triangulated functor induced by an $A_{\infty}$-functor
   $\widetilde{\mathcal{R}}_j: \fuk^{\ast}(E)\to \fuk^{\ast}(M)$ that,
   on objects, associates to each cobordism $V:\emptyset \cobto
   (L_{1},\ldots, L_{s})$ its $j$-th end, $L_{j}$.  It is not
   difficult to see that, with appropriate choices of auxiliary
   structures, such a functor is indeed defined and that it induces at
   the derived level precisely $\mathcal{R}_{j}$.  At the derived level we also have
     $\mathcal{R}_j \circ i^{\eta_j} = \id$.  Notice also that the pull-back
     functor
   $$\widetilde{\mathcal{R}}^{\ast}_{j}: 
   mod(\fuk^{\ast}(M))\to mod(\fuk^{\ast}(E))$$ takes the Yoneda
 module $\mathcal{Y}(L)$ to the Yoneda module $\mathcal{Y}(\eta_{j}
 \times L)=i^{\eta_{j}}(L)$.
\end{rem}

\subsubsection{Ascent: from $D\fuk^{\ast}(M)$ to the category
  $D\fuk^{\ast}(E)$} \label{subsubsec:ascent}

We assume the same setting as fixed at the beginning of
\S\ref{subsec:fibr-tot} and start with some algebraic notation. Let
$\mathcal{B}$ be an $A_{\infty}$-category (over a given ring
$\mathcal{A}$, e.g. the Novikov ring) and $R_{1},\ldots R_{m}$ a
collection of $m$ objects of $\mathcal{B}$.  The following
construction is a straightforward extension of the notion of directed
$A_{\infty}$-category as it appears in \cite{Se:book-fukaya-categ}
(see, in particular, (5m) there).

Consider the ordered set $I_m = \{1, \ldots, m\}$ and let $\N_{+m}$ be
the disjoint union $\N \cup I_m$ ordered strictly in a way that
respects the order of $\N$ and $I_m$ and so that each element in $I_m$
is strictly bigger than any element of $\N$. We still denote the
resulting order relation by $\geq$.  For any two $i,j\in \N_{+m}$ we
put $\xi^{i,j}=1$ if $i\geq j$ and $\xi^{i,j}=0$ if $i< j$ and we let
$\xi^{i_{1},i_{2},\ldots, i_{k+1}}=\xi^{i_{1},i_{2}}\xi^{i_{2},i_{3}}
\ldots \xi^{i_{k},i_{k+1}}$.

   We denote by $\N_{+m}\otimes \mathcal{B}$ the unique
   $A_{\infty}$-category with the properties:
\begin{itemize}
  \item[i.] The objects of $\N_{+m}\otimes\mathcal{B}$ are couples
   $(i, L)$ with $i\in \N_{+m}$ and $L$ an object of $\mathcal{B}$
   with the constraint that if $i\in I_{m}$, then $L=R_{i}$. We will
   write the couples $(i,L)$ as $i\times L$.
  \item[ii.] The morphisms of $\N_{+m}\otimes\mathcal{B}$ are defined
   by:
   $$\mor (i\times L, j\times L')= \xi^{i,j}\mor_{\mathcal{B}}(L,L')$$ 
   except if $i=j\in I_m$. In this case $\mor (i\times R_{i}, i\times
   R_{i})= \mathcal{A}e_{R_{i}}$. Here $e_{R_{i}}$ is, by definition,
   a strict unit in the category $\N_{+m}\otimes B$.

  \item[iii.] We denote by $$\mu_{k}: \mor ( L_{1}, L_{2})\otimes
   \mor(L_{2}, L_{3})\otimes \ldots \otimes \mor( L_{k}, L_{k+1})\to
   \mor ( L_{1}, L_{k+1})$$ the multiplications in $\mathcal{B}$.
   Consider successive indices $(i_{1}, i_{2},\ldots, i_{k+1})$ so
   that no two successive indexes $i_{r},i_{r+1}$ satisfy
   $i_{r}=i_{r+1}\in I_{m}$. Then the multiplications in
   $\N_{+m}\otimes \mathcal{B}$ are given by:
   \begin{eqnarray}\label{eq:multipl}
      \nonumber \mu'_{k}: \mor (i_{1}\times L_{1}, i_{2}\times L_{2})
      \otimes \mor(i_{2}\times L_{2}, i_{3}\times L_{3})\otimes \ldots
      \otimes \mor(i_{k}\times L_{k},i_{k+1}\times L_{k+1}) &\to  \\
      \nonumber\to \ \mor (i_{1}\times L_{1}, i_{k+1}\times L_{k+1}) 
      \hspace{1in}&\\
      \mu'_{k}=\xi^{i_{1},\ldots, i_{k+1}}\mu_{k}\hspace{2in}& ~.~
   \end{eqnarray}
   In case for some index $r$ we have $i_{r}=i_{r+1}\in I_{m}$, then
   $\mu'_{k}$ is completely described by the requirement that
   $e_{R_{i}}$ be a strict unit: $\mu'_{k}$ vanishes if $k\not=2$ and
   $\mu'_{2}(a,e_{R_{i}})=a$, $\mu'_{2}(e_{R_{i}},b)=b$.
\end{itemize}
The notation $\N_{+m}\otimes \mathcal{B}$ is slightly imprecise as
this category actually depends on the choice of objects $R_{1},\ldots,
R_{m}$.  Moreover, there is obviously an abuse of notation here as 
$\N_{+m}\otimes \mathcal{B}$ is not a tensor product (there is no addition among the objects etc). 

In case the $A_{\infty}$-category $\mathcal{B}$ is such that
the objects $R_{i}$ have strict units $e'_{R_{i}}\in
\mor_{\mathcal{B}}(R_{i},R_{i})$, then by taking
$e_{R_{i}}=e'_{R_{i}}$, equation (\ref{eq:multipl}) applies without
treating separately the case $i_{r}=i_{r+1}\in I_{m}$. In general,
when the $R_{i}$'s do not have strict units, we treat the
$e_{R_{i}}$'s as formal elements, part of the construction of
$\N_{+m}\otimes \mathcal{B}$.

\begin{cor} \label{cor:cat-eq} There exists a choice of 
Lagrangians spheres $R_{1},\ldots, R_{m}\in
   \mathcal{L}^{\ast}(M)$ and an equivalence of categories:
   $$\mathcal{I}:D(\N_{+m}\otimes \fuk^{\ast}(M))\to D\fuk^{\ast}(E)~.~$$
\end{cor}

\begin{proof}
   Consider the full and faithful subcategory $\mathcal{F}(E)$ of
   $\fuk^{\ast}(E)$ whose objects consist of the following two
   collections:
   \begin{itemize}
     \item[i.] $\gamma_{i+2}\times
        L$ with $i\in \N$ and $L\in\mathcal{L}^{\ast}(M)$. Here
      $\gamma_{k}$, $k \geq 2$, are the plane curves defined
      in~\S\ref{subsubsec:atoms} (see also
      Figure~\ref{fig:spec-curves}).
     \item[ii.] the thimbles $T_{j}$, $j \in I_m$.
   \end{itemize}
   The generation Theorem~\ref{thm:main-dec} combined with the
   algebraic Lemma~3.34 in~\cite{Se:book-fukaya-categ} implies that
   there is an equivalence of categories
   $$D\mathcal{F}(E)\to D\fuk^{\ast}(E)$$
   induced by the inclusion
   $$\mathcal{F}(E)\to \fuk^{\ast}(E)~.~$$  

   We now intend to show the existence of a quasi-equivalence of
   $A_{\infty}$-categories:
   $$\Xi :\N_{+m}\otimes \fuk^{\ast}(M)\to \mathcal{F}(E)~.~$$
   To this end we first pick a specific family of objects
   $R_{1},\ldots, R_{m}$ in $\fuk^{\ast}(M)$. By definition, these
   objects are the following Lagrangian spheres:
   $$R_{m+1-i} :=  \widetilde{\tau}_{i+1, \ldots, m}^{-1} (S_{i}) 
   \ , \ i=1,\ldots, m$$ - see Corollary \ref{cor:dec-M} for the
   notation.  For $i\in \N$, and $L\in \mathcal{L}^{\ast}(M)$, we
   define 
   $\Xi'(i\times L) = \gamma_{i+2}\times L$.  For $i\in I_m$ we
   define $\Xi'(i\times R_{i})=T_{m+1-i}$.

   It is not difficult to see - as in the construction of the
   inclusion functor $\mathcal{I}_{\gamma,h}$
   in~\cite{Bi-Co:lcob-fuk}, in particular Proposition 4.2.3 there -
   that by using appropriate choices for the curves $\gamma_{i}$ as
   well as almost complex structures and perturbation data, we can
   describe the morphisms and higher products in $\mathcal{F}(E)$ by
   the formulas corresponding to $\N_{+m}\otimes \fuk^{\ast}(M)$.
   There is however one exception concerning this correspondence and
   due to it the map $\Xi'$ can not be assumed directly to be a
   morphism of $A_{\infty}$ categories: the difficulty comes from the
   fact that the objects $T_{j}$ of $\mathcal{F}(E)$ do not, in
   general, have strict units.  However, there is an algebraic
   argument - Lemma 5.20 in \S (5n) in \cite{Se:book-fukaya-categ} -
   that applies also to our case with minor modifications and implies
   that we can replace $\Xi'$ by a true $A_{\infty}$ functor: $\Xi:
   \N_{+m}\otimes \fuk^{\ast}(M)\to \mathcal{F}(E)$ that acts on
   objects in the same way as $\Xi'$ and so that $\Xi$ is a
   quasi-equivalence.  Clearly, this implies the equivalence of the
   associated derived categories and the existence of $\mathcal{I}$.
 \end{proof}

\begin{rem}\label{rem:Seidel-reference-Lef}

   a. Corollary \ref{cor:cat-eq} extends a result of Seidel in \S 18
   of \cite{Se:book-fukaya-categ} (see also
   \cite{Se:Lefschetz-Fukaya}) which provides a similar description
   for the subcategory of $D\fuk^{\ast}(E)$ that is generated by the
   thimbles $T_{i}$.

   b. It is easy to see by direct calculation that there are
   inclusions $\mathcal{J}_{s}:D\fuk^{\ast}(M)\to D(\N_{+m}\otimes
   \fuk^{\ast}(M))$ induced by $L\to (s,L)$ for all $s\in \N$. The
   compositions $\mathcal{J}'_{s}=\mathcal{I}\circ \mathcal{J}_{s}$
   have a simple geometric interpretation. Consider the inclusion
   $i^{\gamma_{s+2}}:\fuk^{\ast}(M)\to \fuk^{\ast}(E)$ which acts on
   objects as $L \to \gamma_{s+2}\times L$. This induces a functor
   $i^{\gamma_{s+2}}:D\fuk^{\ast}(M)\to D\fuk^{\ast}(E)$ that
   coincides with $\mathcal{J}'_{s}$.
   
   c. An obvious by-product of this Corollary is that the derived
   categories $D\fuk^{\ast}(E;\tau)$ from the statement of Theorem
   \ref{thm:A-rfm} are independent of the choice of tame fibration
   $E_{\tau}$ up to equivalence. Together with
   \S\ref{sbsb:prf-main-dec-gen0} this concludes the proof of Theorem
   \ref{thm:A-rfm}.
\end{rem}

\subsection{The Grothendieck group} \label{subsec:groth}
 
The purpose of this section is to discuss a variety of consequences of
Theorem \ref{thm:main-dec} in what concerns the morphism $\Theta$
from~\eqref{eq:theta} as well as the Grothendieck group itself.

\subsubsection{Cobordism groups and the Grothendieck group.}

We start by defining the appropriate cobordism groups that will be of
interest to us here.  We will restrict here too the discussion to tame
Lefschetz fibrations.  Fix such a fibration $\pi :E \to \C$ that is
tame outside $U\subset \C$.  Let $(M,\omega)$ be the fibre of $\pi$ at
a point $z_{0}\in \C\setminus U$.  Let $\Omega^{\ast}_{Lag}(M; E)$ be
the abelian group defined as the quotient of the free abelian group
generated by the Lagrangians $L\in \mathcal{L}^{\ast}(M)$-modulo the
relations $\mathcal{R}_{cob}^{E}$ generated by the cobordisms $V:
\emptyset \cobto (L_{1},\ldots, L_{s})$, $V\in\mathcal{L}^{\ast}(E)$
in the sense that to each such $V$ we associate the relation
$L_{1}+\ldots + L_{s}\in \mathcal{R}_{cob}^{E}$. Basically, the point
of view here is that cobordisms are relators among their ends. As we
do not take into account orientations this group is obviously
$2$-torsion. Notice that all vanishing spheres $S \subset M$
(associated to any path between a critical value of $\pi$ and $z_0$)
belong to $\mathcal{R}_{cob}^{E}$, hence their cobordism class is $0
\in \Omega^{\ast}_{Lag}(M; E)$. This follows from the fact that a
vanishing sphere is the single end of a cobordism which is a thimble
of some path going from one critical value of $\pi$ to $z_0$.

In case $\pi: E \longrightarrow \mathbb{C}$ is the trivial fibration
(i.e. $E$ splits symplectically as $E = \mathbb{C} \times M$ and $\pi
= \textnormal{pr}_{\mathbb{C}}$) we will abbreviate
$\Omega^{\ast}_{Lag}(M; E)$ by $\Omega^{\ast}_{Lag}(M)$.

\begin{rem}\label{rem:non-comm}
   a. While we will not explore this issue here, notice that the group
   $\Omega^{\ast}_{Lag}(M;E)$ is the abelianization of a group
   $\mathcal{G}^{\ast}_{Lag}(M;E)$ that is defined as the free {\em
     non-abelian} group generated by the $L\in \mathcal{L}^{\ast}(M)$
   modulo relations $L_{1}\cdot L_{2}\cdot\ldots \cdot L_{s}$
   associated as before to cobordisms $V: \emptyset \cobto
   (L_{1},\ldots, L_{s})$. In other words, in this case we take into
   account the geometric order of the ends of $V$.

   b. It is easy to adjust the definition of the groups
   $\Omega^{\ast}_{Lag}(-)$ to the case of non-tame fibrations.
   However, in view of \S\ref{sb:tame-vs-gnrl}, all interesting
   phenomena concerning these cobordism groups are already present in
   the case of tame fibrations.
\end{rem}

Recall the Grothendieck group $K_{0}(D\fuk^{\ast}(M))$ that is
associated to the triangulated category $D\fuk^{\ast}(M)$ as in
\S\ref{subsec:Fuk-fibr}. Notice that this group too is $2$-torsion
because we work in an ungraded setting. We are interested in a
quotient of this Grothendieck group that is associated to our tame
fibration $\pi: E\to \C$.  To construct it assume $x_{1},\ldots,
x_{m}$ are the critical points of $\pi$ and let the corresponding
critical values be $v_{1},\ldots, v_{m}$. Then for each $i$ pick a
path in $\C$ from $v_{i}$ to $z_{0}$ that does not encounter any other critical value
(such as, for instance, the paths
$t_{i}$ in Figure \ref{fig:spec-curves}). There is an associated
thimble to each such path and let $\Sigma_{i}$ be the vanishing sphere
in $M=\pi^{-1}(z_{0})$ that is the end of the thimble from $x_{i}$ to
$M$. Denote by $\mathcal{S}_{E}$ the subgroup in
$K_{0}(D\fuk^{\ast}(M))$ that is generated by the spheres
$\Sigma_{i}$. Finally, define the quotient:
$$ K_{0}(D\fuk^{\ast}(M); E)=K_{0}(D\fuk^{\ast}(M))/\mathcal{S}_{E}~.~$$

\begin{cor} \label{cor:basic-morphism} The group
   $K_{0}(D\fuk^{\ast}(M); E)$ does not depend on the choices made in
   its construction and there exists a morphism of groups:
   $$\Theta^{E}:\Omega^{\ast}_{Lag}(M; E)\to K_{0}(D\fuk^{\ast}(M);E)$$
   that is induced by $L\to L$.
\end{cor}

This morphism extends the Lagrangian Thom morphism initially
constructed in~\cite{Bi-Co:lcob-fuk} and already mentioned at
(\ref{eq:theta})
$$\Theta : \Omega^{\ast}_{Lag}(M)\to K_{0}(D\fuk^{\ast}(M))$$

\begin{proof} We first discuss the independence of
   $K_{0}(D\fuk^{\ast}(M);E)$ of the choices of the vanishing spheres
   $\Sigma_{i}$. Assume for instance that one of these spheres, say
   $\Sigma_{1}$ - that is the end of a thimble $K_{1}$ that projects
   to a path $k_{1}$ from $v_{1}$ to $z_{0}$ - is replaced with a
   sphere $\Sigma'_{1}$ which is the end of a thimble $K'_{1}$,
   associated to a different path, $k'_{1}$. By the results of
   Seidel~\cite{Se:book-fukaya-categ}, the difference between
   $\Sigma_{1}$ and $\Sigma'_{1}$ (up to hamiltonian isotopy) can be
   described as follows: one sphere is obtained from the other by 
   applying a symplectic diffeomorphism $\phi$ which can be written as
   word in the elements $\tau_{\Sigma_2}, \ldots, \tau_{\Sigma_m}$
   (i.e. $\phi$ is a composition of Dehn twists and their inverses
   along spheres from the collection $\Sigma_2, \ldots, \Sigma_m$).
   From Seidel's exact triangle as given in
   Proposition~\ref{t:ex-tr-compact} we see that the subgroups
   generated, respectively, by $\Sigma_{1}, \Sigma_{2},\ldots,
   \Sigma_{m}$ and $\Sigma'_{1}, \Sigma_{2},\ldots, \Sigma_{m}$ are
   the same.

   The existence of the morphism $\Theta^{E}$ is now an immediate
   consequence of the decomposition in Corollary \ref{cor:dec-M}.
\end{proof}

\subsubsection{The Grothendieck group as an algebraic cobordism group.}
We now focus our attention on the category $\fuk^{\ast}(E)$.

For each module $\mathcal{M}\in \mathcal{O}b(D\fuk^{\ast}(E))$, define
$[\mathcal{M}]_{j}\in\mathcal{O}b(D\fuk^{\ast}(M))$ by
$$[\mathcal{M}]_{j}=\mathcal{R}_{j}(\mathcal{M})$$
where $\mathcal{R}_{j}$ are the restriction functors defined in the
proof of Corollary \ref{cor:dec-M} (see also
Remark~\ref{rem:restr-funct1}). Basically, this extends to all objects
in $D\fuk^{\ast}(E)$ the operation that associates to a cobordism $V$
its $j$-th end. It is easy to see that for all objects $\mathcal{M}$
of $D\fuk^{\ast}(E)$ there are only finitely many non-vanishing
$[\mathcal{M}]_{j}$'s.

We now define another group $\Omega_{Alg}^{\ast}(M; E)$, which we call
the {\em algebraic cobordism group}, as the free abelian group
generated by all the {\em isomorphisms types of} objects $\in
\mathcal{O}b(D\fuk^{\ast}(M))$ modulo the relations
$$[\mathcal{M}]_{1}+[\mathcal{M}]_{2}+[\mathcal{M}]_{3}+\ldots = 0$$
for each $\mathcal{M}\in\mathcal{O}b(D\fuk^{\ast}(E))$.

The group $\Omega_{Alg}^{\ast}(M;E)$ can be viewed as an algebraic
cobordism group in the following sense. The generators of this group
are the (isomorphism type of) objects of $D\fuk^{\ast}(M)$, thus they
are obtained by completing algebraically the objects of
$\fuk^{\ast}(M)$ as in the construction of the derived Fukaya
category.  Similarly, the relations defining the group are again an
algebraic completion - in a similar sense but now involving the
categories $\fuk^{\ast}(E)$ and $D\fuk^{\ast}(E)$ - of the relations
providing $\Omega_{Lag}^{\ast}(M; E)$.  By definition, there is an
obvious group morphism:
$$q:\Omega_{Lag}^{\ast}(M; E)\to \Omega^{\ast}_{Alg}(M; E)~.~$$

\begin{cor}\label{cor:alg-cob}
   There is a group isomorphism
   $$\Theta^{E}_{Alg}:\Omega_{Alg}^{\ast}(M;E)\to 
   K_{0}(D\fuk^{\ast}(M); E)$$ so that
   $\Theta^{E}=\Theta^{E}_{Alg}\circ q$.
\end{cor}

\begin{proof} Throughout the proof we abbreviate $K_0 =
   K_{0}(D\fuk^{\ast}(M); E)$.

   At the level of generators we define $\Theta^{E}_{Alg}$ to be the
   identity. The surjectivity of $\Theta^{E}_{Alg}$ is clear as well
   as the relation $\Theta^{E}=\Theta^{E}_{Alg}\circ q$. The only two
   things to check are that this map is well-defined and injective.

   To show that $\Theta^{E}_{Alg}$ is well-defined we need to prove
   that if $\mathcal{M}$ is an object of $D\fuk^{\ast}(E)$, then
   $\sum_{i} [\mathcal{M}]_{i}=0$ in $K_{0}(D\fuk^{\ast}(M);E)$.  To
   see this recall that, by the definition of $D\fuk^{\ast}(E)$, there
   are $V_{j}\in \mathcal{L}^{\ast}(E)$ so that:
   $$\mathcal{M}\cong (V_{m}\to V_{m-1}\to \ldots \to V_{2}\to V_{1})~.~$$
   By Theorem \ref{thm:main-dec}, in $K_{0}$ we have:
   $$\sum_{i}[V_{j}]_{i}=0 \ , \ \forall j ~.~$$
   Moreover, $\forall i$, we have the following cone decomposition of
   $[\mathcal{M}]_i$ in $D\fuk^*(M)$:
   $$[\mathcal{M}]_{i} \cong
   ([V_{m}]_{i}\to [V_{m-1}]_{i}\to \ldots \to [V_{2}]_{i}\to
   [V_{1}]_{i})$$ because the functor $\mathcal{R}_{i}$ is
   triangulated. This means that in $K_{0}$:
   $$\sum_{i}[\mathcal{M}]_{i}=\sum_{i,j}[V_{j}]_{i}=0~.~$$
   This concludes the proof of the well-definedness of the map
   $\Theta^{E}_{Alg}$.

   It remains to show that $\Theta^{E}_{Alg}$ is injective. We start
   by proving the injectivity in the case when $\pi$ is trivial and so
   $E=\C\times M$.  We omit $E$ from the notation of $\Theta_{Alg}$ in
   this case and, similarly, we put $\Omega_{Alg}(M)=\Omega_{Alg}(M;
   \C\times M)$.
   % Let $\Omega^{\ast}_{s-alg}(M)$ be the group generated by the
   % objects of $D\fuk^{\ast}(M)$ modulo the relations
   % $\sum_{i}\hat{\mathcal{M}}_{i}=0$ for each module $\mathcal{M}\in
   % \mathcal{O}b(D\mathcal{F}(\C\times M))$.  The functor
   % $\mathcal{I}$ from Corollary \ref{cor:cat-eq} induces an
   % isomorphism:
   % $\Omega^{\ast}_{s-%alg}(M)\to \Omega^{\ast}_{alg}(M)$ (this is because each
   %   object of $D\fuk^{\ast}_{0}(\C\times M)$ is
   %   isomorphic to an object in $D \mathcal{F}(\C\times M)$).
   %   The morphism $\Theta^{alg}$ also induces a morphism:
   %   $\Theta^{alg}_{s}:\Omega^{\ast}_{s-alg}(M)\to
   %   K_{0}(D\fuk^{\ast}(M))$.  To finish the proof in the case of
   %   $E=\C\times M$ it suffices to show that this morphism is
   %   injective.
   Assume that
   $$\mathcal{M}\to \mathcal{M}'\to \mathcal{M}''$$
   is an exact triangle of $\fuk^{\ast}(M)$-modules. The injectivity
   of $\Theta_{Alg}$ follows by constructing for each such triangle an
   object $T$ in $D\fuk^{\ast}(\C\times M)$ so that
   $[T]_{1}=\mathcal{M}''$, $[T]_{2}=\mathcal{M}'$ and
   $[T]_{3}=\mathcal{M}$. Indeed, this implies that all the relations
   that are used in the definition of $K_{0}$ also appear among the
   relations that define $\Omega^{\ast}_{Alg}(M)$ which means that
   $\Theta_{Alg}$ is invertible.

   To construct this object $T$ we proceed as follows.  We first
   recall that, by definition, $\mathcal{M}''$ is - up to isomorphism
   - the cone over a module map $f: \mathcal{M}\to \mathcal{M}'$.

   Now recall the $A_{\infty}$-category $\N\otimes \fuk^{\ast}(M)$ as
   in~\S\ref{subsubsec:ascent} (notice that now $m=0$). We first
   construct an object $\tilde{T}$ of $\N \otimes \fuk^{\ast}(M)$.
   This consists of two steps. First, for each $\fuk^{\ast}(M)$-module
   $\mathcal{N}$ and each curve $\gamma_{i}$ we define a
   $\N\otimes
     \fuk^{\ast}(M)$-module denoted by $\gamma_{i}\times
   \mathcal{N}$. On objects $\gamma_{j}\times L$ we put
   $(\gamma_{i}\times \mathcal{N}) (\gamma_{j}\times
   L)=\xi^{j,i}\mathcal{N}(L)$. The $A_{\infty}$-module operations are
   defined by a direct adaptation of the formulas giving the
   operations in $\N\otimes \fuk^{\ast}(M)$.  The second step is to
   define a morphism
   $$\bar{f}:\gamma_{3}\times \mathcal{M}\to \gamma_{2}\times 
   \mathcal{M}'~.~$$ We then define $\tilde{T}$ by
   $\tilde{T}=cone(\bar{f})$.  The morphism $\bar{f}$ is induced by
   $f$ and is given by a formula again perfectly similar to the
   formula of the multiplication in $\N\otimes \fuk^{\ast}(M)$, but
   using $f$ instead of $\mu_{k}$ and replacing $\mor (i_{k}\times
   L_{k}, i_{k+1}\times L_{k+1})$ by $(\gamma_{3}\times
   \mathcal{M})(\gamma_{i_{k}-2}\times L_{k+1})$ and $\mor(i_{1}\times
   L_{1}, i_{k+1}\times L_{k+1})$ by $(\gamma_{2}\times
   \mathcal{M}')(\gamma_{i_{1}-2}\times L_{1})$.  We now consider the
   sequence of functors, the first two being equivalences and the last
   a full and faithful embedding:
   \begin{equation}\label{eq:seq-eq}
      D(\N\otimes \fuk^{\ast}(M))\to D\mathcal{F}(\C\times M)\to 
      D\fuk^{\ast}(\C\times M)\to 
      D\fuk^{\ast}_{\frac{1}{2}}(\C\times M).
   \end{equation}
   Here, the $A_{\infty}$-category $D\mathcal{F} (\mathbb{C} \times
   M)$ is defined as in the proof of Corollary~\ref{cor:cat-eq}.  We
   now use the composition of the functors in~\eqref{eq:seq-eq} to
   define $[\mathcal{H}]_{j}=(i^{\eta_{j}})^{\ast}(\mathcal{H})$ for
   each module $\mathcal{H}$ in $D(\N\otimes \fuk^{\ast}(M))$ - see
   the proof of Corollary~\ref{cor:dec-M} for the definition of
   $i^{\eta_j}$.  We take $T$ to be the image of $\tilde{T}$ by the
   first two equivalences in~\eqref{eq:seq-eq} and we claim that:
   \begin{itemize}
     \item[a.] for each object $\mathcal{N}$ in $D\fuk^{\ast}(M)$ we
      have that $[(\gamma_{i}\times \mathcal{N})]_{j}\cong\mathcal{N}$
      if $i=j$ or $j=1$ and is $0$ otherwise. Moreover,
      $(i^{\eta_{1}})^{\ast}(\bar{f})\cong f$.
     \item[b.] $[T]_{1}=\mathcal{M}''$, $[T]_{2}=\mathcal{M}'$,
      $[T]_{3}=\mathcal{M}$ and $[T]_{i}=0$ whenever $i\geq 4$.
   \end{itemize}
   Notice that point b concludes the proof for $E=\C\times M$. Given
   that the equivalences in~\eqref{eq:seq-eq} are triangulated, point
   b follows directly from a. Thus, it remains to check a. For this we
   notice that pull-back respects triangles and as each object
   $\mathcal{N}$ is isomorphic to an iterated cone of objects $L\in
   \fuk^{\ast}(M)$ it is enough to verify the statement for the Yoneda
   modules $\gamma_{i}\times L$, $L\in\mathcal{L}^{\ast}(M)$. But for
   these modules the statement is obvious. The statement for $\bar{f}$
   follows in a similar fashion.

   We are left to show the more general statement for a Lefschetz
   fibration $\pi: E\to \C$ that is not trivial. For this we recall
   that, for each thimble $T_{i}$ we have
   $(i^{\eta_{1}})^{\ast}(T_{i})=\widetilde{\tau}^{-1}_{i+1,
       \ldots, m}S_{i}$. (The definition of the spheres $S_i$ appears
     in~\S\ref{sec:conseq}.) Thus, by the definition of the groups
   involved, we have a quotient map
   \begin{equation} \label{eq:Om-*-K_0}
      \Omega^{\ast}_{Alg}(M)/\mathcal{S}'_{E}\to
      \Omega^{\ast}_{Alg}(M; E)\xrightarrow{\; \Theta^{E}_{Alg} \;}
      K_{0}(D\fuk^{\ast}(M);E),
   \end{equation}
   where $\mathcal{S}'_{E}$ is the subgroup generated by the vanishing
   spheres of $\pi$.  To conclude the proof of the theorem it is
   enough to show that the composition of maps in~\eqref{eq:Om-*-K_0}
   is an isomorphism. Recall that
   $$K_{0}(D\fuk^{\ast}(M);E)=K_{0}(D\fuk^{\ast}(M))/\mathcal{S}_{E}$$
   and notice that the isomorphism $\Theta_{Alg}$ - associated to the
   trivial fibration $\C\times M$ - has the property that
   $\Theta_{Alg}(\mathcal{S}'_{E})=\mathcal{S}_{E}$.  Therefore the
   composition of maps in~\eqref{eq:Om-*-K_0} is an isomorphism and
   this concludes the proof.
\end{proof}

\subsubsection{Comparison with ambient quantum homology.}
There is an obvious morphism:
$$i:\Omega_{Lag}^{\ast}(M)\to QH(M)$$
that associates to each Lagrangian $L$ its homology class $[L] \in
H_n(M;\mathbb{Z}_2) \subset QH(M)$.  From the point of view of
Corollary~\ref{cor:alg-cob} it is natural to expect that $i$ factors
through a morphism:
$$i':\Omega_{Alg}^{\ast}(M)\to QH(M)~.~$$
This is indeed true as we will see below.

\begin{cor}\label{cor:quantum-incl}
Consider a module
   $\mathcal{M}\in \mathcal{O}b(D\fuk^{\ast}(M))$.  Such a module
   admits a cone-decomposition (up to quasi-isomorphism)
   $$\mathcal{M}\cong (L_{s}\to L_{s-1}\to \ldots \to L_{1})~.~$$ 
With this notation, the equation
   \begin{equation}
   \label{eq:i-prime}
   i'(\mathcal{M})=\sum_{j} [L_{j}]\in QH(M)
   \end{equation}
 provides a well-defined group morphism
 $$i':\Omega_{Alg}^{\ast}(M)\to QH(M)$$
 so that $i=i'\circ q$.
 \end{cor}

 \begin{proof}
  While this definition of $i'$ seems very simple the fact that $i'$
  is a well-defined morphism of groups is somewhat surprising. 
  We only know a proof of this fact which follows
  from the indirect construction that we give below. 
   
  We will write $i'$ as a composition of two morphisms
  $i'=\tilde{i}'\circ \Theta_{Alg}$ where
  $\Theta_{Alg}:\Omega_{alg}^{\ast}(M)\to K_{0}(D\fuk^{\ast}(M))$ is
  the isomorphism in Corollary \ref{cor:alg-cob} and
$$\tilde{i}':  K_{0}(D\fuk^{\ast}(M))\to QH(M)$$ is a morphism 
that is known to experts, see for instance~\S~5
in~\cite{Se:homology-spheres}.  The definition of $\tilde{i}'$ is
somewhat subtle so we review it here.

 The morphism $\tilde{i}'$ is a composition of morphisms:
\begin{equation*}
   \begin{aligned}
      & K_{0}(D\fuk^{\ast}(M)) \stackrel{f_{1}}{\longrightarrow}
      K_{0}(\mathcal{Y}(\fuk^{\ast}(M))^{\wedge})
      \stackrel{f_{2}}{\longrightarrow} \\
      & \stackrel{f_{2}}{\longrightarrow} HH_{\ast}(\mathcal{Y}
      (\fuk^{\ast}(M))^{\wedge}) \stackrel{f_{3}}{\longrightarrow}
      HH_{\ast}(\fuk^{\ast}(M))\stackrel{f_{4}}{\longrightarrow}
      QH(M)~.~
   \end{aligned}
\end{equation*}
Here, the category $\mathcal{Y}(\fuk^{\ast}(M))$ is the Yoneda image of
$\fuk^{\ast}(M)$; $(\mathcal{Y}(\fuk^{\ast}(M))^{\wedge}$ is its
triangular completion (as $A_{\infty}$-category);
$HH_{\ast}(\mathcal{B})$ is the Hochschild homology of the
$A_{\infty}$-category $\mathcal{B}$ with values in itself (generally
denoted by $HH_{\ast}(\mathcal{B},\mathcal{B})$). The morphisms
involved are as follows: $f_{1}$ is an obvious isomorphism that
reflects the definition of the triangular structure of
$D\fuk^{\ast}(M)$, the morphism $f_{2}$ sends each module in
$\mathcal{M}\in\mathcal{Y}(\fuk^{\ast})^{\wedge}$ to the
  Hochschild homology class of its unit endomorphism $e_{\mathcal{M}}
  \in \hom(\mathcal{M}, \mathcal{M})$. The latter descends to $K_{0}$
because, as it follows from Proposition~3.8
in~\cite{Se:book-fukaya-categ}, if $\mathcal{M}'\to \mathcal{M}\to
\mathcal{M}''$ is an exact triangle in a triangulated
$A_{\infty}$-category $\mathcal{A}$, then $e_{\mathcal{M}} =
e_{\mathcal{M'}} + e_{\mathcal{M}''}$ in $HH_{\ast}(\mathcal{A})$.
The morphism $f_{3}$ comes from the fact that the natural inclusion
$$\fuk^{\ast}(M)\to \mathcal{Y}(\fuk^{\ast}(M))^{\wedge}$$ induces an 
isomorphism in Hochschild homology (this is sometimes referred to as a
form of Morita invariance. See~\cite{To:dg-cat} for the analogous
though different context of dg-categories); $f_{3}$ is the inverse of
this isomorphism.  Finally, $f_{4}$ is the open-closed map (see for
instance~\cite{Se:homology-spheres} where it is defined for in the exact case, the adaptation
to the monotone setting is immediate).

\end{proof}

\begin{rem}
   Assume that $\mathcal{M}'$ is another module in $D\fuk^{\ast}(M)$
   as in the statement of the corollary such that
   $\mathcal{M}'\cong\mathcal{M}$ and
   $$\mathcal{M}'=(L'_{r}\to L'_{r-1}\to  \ldots \to L'_{1})~.~$$
   The existence of $i'$ then implies that
   $\sum_{j}[L'_{j}]=\sum_{k}[L_{k}]$. It is interesting to note that
   the only way we know to show this fact is through the indirect
   method contained in the proof of the Corollary.
\end{rem}

\subsubsection{The periodicity isomorphism~\eqref{eq:K0}}

In view of Corollary \ref{cor:cat-eq} it is natural to expect that
$K_{0}(D\fuk^{\ast}(E))$ can be calculated in terms of
$K_{0}(D\fuk^{\ast}(M))$.  We will give here such a calculation but
only in the case when $E$ is the trivial fibration $E=\C\times M$. An
analogous statement for non-trivial fibrations is expected to also
hold, but would require further algebraic elaboration.

\begin{cor} \label{cor:periodicity} There exists a canonical
   isomorphism
   $$K_{0}(D\fuk^{\ast}(\C\times M))\cong \Z_{2}[t]\otimes 
   K_{0}(D\fuk^{\ast}(M))$$ induced by the map that sends $\mathcal{M}
   \in \mathcal{O}b(D\fuk^*(\mathbb{C} \times M))$ to $\sum_{i\geq 2}
   t^{i-2} \otimes \mathcal{R}_i(\mathcal{M})$, where $\mathcal{R}_i$
   is the restriction functor from~\eqref{eq:functor-R-j}.
\end{cor}

\begin{proof}
   From Corollary \ref{cor:cat-eq} it is enough to show that
   $$K_{0}(D(\N \otimes \fuk^{\ast}(M)))\cong \Z_{2}[t]\otimes 
   K_{0}(D\fuk^{\ast}(M))~.~$$ To simplify notation we denote
   $G_{1}=K_{0}(D(\N\otimes \fuk^{\ast}(M)))$ and $G_{2}=\Z_{2}[t]
   \otimes K_{0}(D\fuk^{\ast}(M))$.  Given a module $\mathcal{M}$
   which is an object of $D(\N\otimes \fuk^{\ast}(M))$ we use the
   composition in (\ref{eq:seq-eq}) to define the restriction modules
   $[\mathcal{M}]_{i}$ that are objects of $D\fuk^{\ast}(M)$ and
   define the sum $\phi(\mathcal{M})=\sum_{i\geq 2} t^{i-2} \otimes
   [\mathcal{M}]_{i}\in G_{2}$. Because the restriction functors
   $\mathcal{R}_{j}$ are triangulated it is easy to see that this map
   descends to a morphism $\phi:G_{1}\to G_{2}$.  The construction of
   the modules $\gamma_{i}\times \mathcal{N}$ in the proof of
   Corollary~\ref{cor:alg-cob}, in particular point~(a) in the course
   of that proof, shows that $\phi$ is surjective. To show that $\phi$
   is injective we construct an inverse $\psi :G_{2}\to G_{1}$. We
   define $\psi (t^{i}\otimes \mathcal{N})=\gamma_{i+2}\times
   \mathcal{N}$ for each object in $\mathcal{N}\in D\fuk^{\ast}(M)$,
   where we have used here the notation from the proof of
   Corollary~\ref{cor:cat-eq}. Once we show that $\psi$ is well
   defined (in other words, that it respects the relations giving
   $K_{0}$) it immediately follows that it is an inverse of $\phi$ by
   the point~(a) in the proof of Corollary \ref{cor:alg-cob}.  But
   again as in the proof of Corollary \ref{cor:alg-cob}, namely the
   construction of $\tilde{T}$, it is easy to see that the map
   $\mathcal{N}\mapsto \gamma_{i}\times \mathcal{N}$ respects
   triangles.  As a consequence, $\psi$ is well defined and this
   concludes the proof.
\end{proof}

% !TEX root = lefcob.tex

\section{Examples} \label{S:examples}

The purpose of this section is to exemplify various aspects of the
machinery in the paper.  We start by making more explicit the
structure contained in the writing of the cone-decompositions in
Theorem \ref{thm:main-dec-gen0} and exemplify this in the simplest
possible setting consisting of cobordisms in $\C$.  We then indicate
how the cone-decompositions associated to cobordisms in our previous
paper \cite{Bi-Co:lcob-fuk} are a consequence of the results here.  We
pursue with some cobordism examples in non-trivial Lefschetz
fibrations.  We first consider a simple horse-shoe like curve in a
Lefschetz fibration with just one critical value and make explicit how
Seidel's exact sequence follows by applying our machinery to this
case.  Finally, and this is the novel and longest part of the section,
we discuss real Lefschetz fibrations and their relation to Lagrangian
cobordism.

\subsection{Unwrapping cone-decompositions.} The decompositions
provided by Theorem \ref{thm:main-dec-gen0} contain more structure
than it appears superficially in the writing:
$$V \cong (T_{1}\otimes E_{1}\to T_{2}\otimes E_{2}\to
\ldots \to T_{m}\otimes E_{m}\to \gamma_s L_s \to \gamma_{s-1} L_{s-1}
\to \ldots \to \gamma_2 L_2 )~.~$$ Namely, see also
\S\ref{sbsb:iter-cone}, writing
$$V\cong (C_{3}\to C_{2} \to C_{1})$$
actually means
$$V\cong \textnormal{cone}(C_{3}\stackrel{f_{2}}{\to} 
\textnormal{cone}(C_{2}\stackrel{f_{1}}{\to} C_{1}))$$ and the
attaching maps $f_{i}$ as well as the intermediate cones are, of
course, crucial in determining the result of the iterated cone.
   
This point is already in evidence in the simplest setting to which can
be applied the machinery of the paper: cobordisms in $\C$ without any
positive ends (and with the negative ends having integral imaginary
coordinates).  Obviously, these cobordisms are simply disjoint unions
of circles and arcs diffeomorphic to $\R$ with horizontal ends
pointing in the negative direction. Notice that due to the uniform
monotonicity condition all circles have to enclose the same area. At
the same time, circles do not play a significant role here since they
have vanishing quantum homology and thus they are not seen by Floer
and Fukaya category machinery.

Consider two Lagrangians $V$ and $V'$ as in
Figure~\ref{fig:planar-cob} below.

Namely, $V$ consists of two connected components: $V_{0}$ and $V_{1}$
with $V_{0}$ an arc with ends at height $2$ and $6$ and $V_{1}$ an arc
with ends at height $3$ and $5$; $V'$ has also two components $V'_{0}$
an arc with ends at height $2$ and $3$ and $V'_{1}$ again an arc with
ends at height $5$ and $6$.  It is easy to see that $V$ and $V'$ are
the results of the two types of surgery on the Lagrangians $W$ and
$W'$ in the middle part of Figure \ref{fig:planar-cob}.
\begin{figure}[htbp]
   \begin{center}
     \includegraphics[scale=0.65]{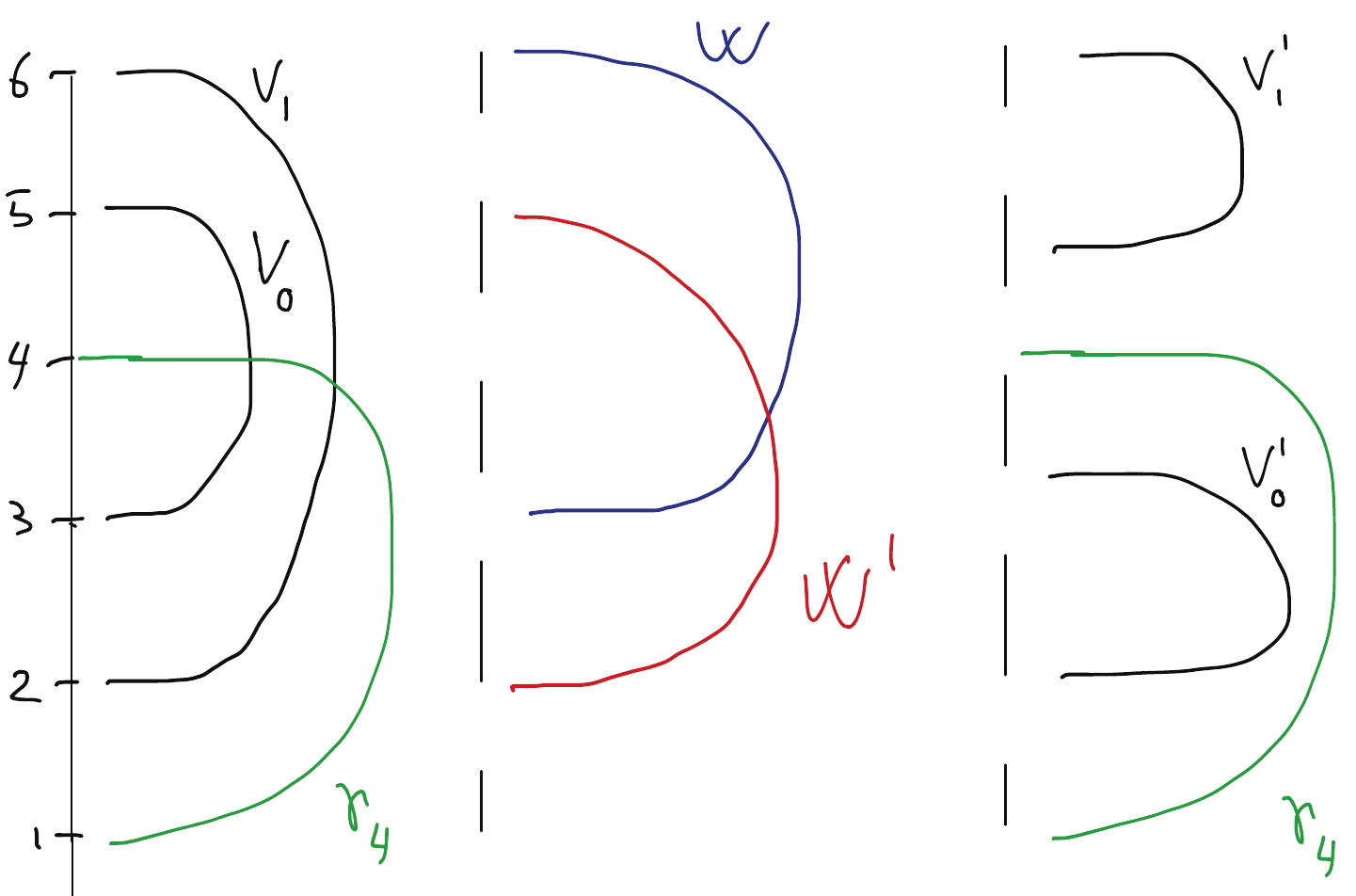}
   \end{center}
   \caption{The planar cobordisms $V=V_{0}\cup V_{1}$ and $V'=V'_{0}\cup V'_{1}$.  
     They are obtained through the two types of surgery on $W$ and $W'$. 
     We have $HF(\gamma_{4},V)\not=HF(\gamma_{4},V')$.
     \label{fig:planar-cob}}
\end{figure}
This means, in particular, as seen in \cite{Bi-Co:cob1} that $V$ and
$V'$ are themselves Lagrangian cobordant.

Theorem \ref{thm:main-dec-gen0} applied to $V$ and $V'$ produces
decompositions that, formally, in the writing of the statement of that
Theorem both look as:
$$(\gamma_{6}\to \gamma_{5}\to \gamma_{3}\to \gamma_{2})~.~$$
However, it is easy to see that $V$ and $V'$ are not isomorphic
objects in $D\fuk^{\ast}(\C)$. Indeed, $HF(\gamma_{4}, V)\not=0$ but
$HF(\gamma_{4}, V')=0$ and it is an easy exercise to see that the
actual two cone decompositions associated to $V$ and $V'$ by Theorem
\ref{thm:main-dec-gen0} are different: the intermediate cones and the
relevant attaching maps are not the same.

\

Other examples relevant in this context are associated to elementary
Lagrangian cobordisms $W:Q\cobto Q$, $W\subset \C\times M$ (here
$(M,\omega)$ is our fixed symplectic manifold). Examples of such
cobordisms are provided by Lagrangian suspension. To such a
$W$ we can easily associate a cobordism $V:\emptyset \cobto
(\emptyset, Q, Q)$. This can be done by first translating $W$ by using
$(z,x)\to (z+i,x)$ and then bending the positive end to the right and
extending it to $-\infty$ so that it has height $3$. The ends of $V$
have heights $2$ and $3$ - as in Figure \ref{fig:bending}.
\begin{figure}[htbp]
   \begin{center}
      \includegraphics[scale=0.65]{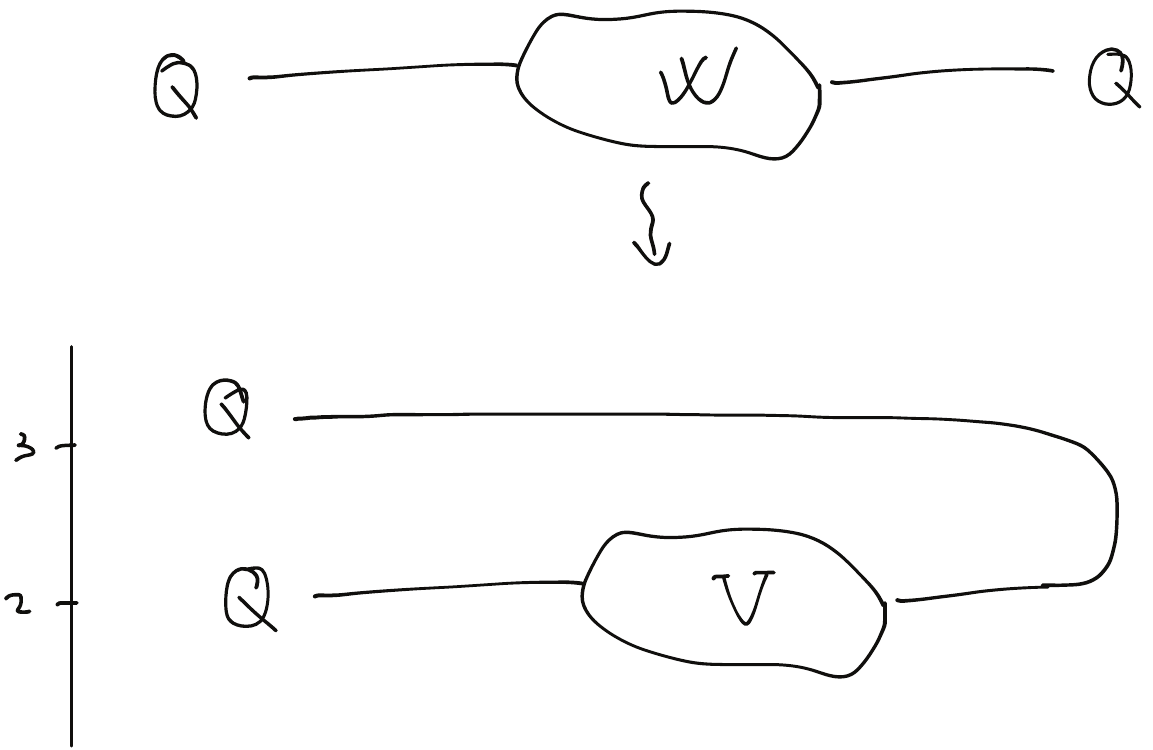}
   \end{center}
   \caption{The cobordism $V$ is obtained by bending the positive end
     of the elementary cobordism $W:Q\cobto Q$.
     \label{fig:bending}}
\end{figure}
Of course, the simplest such example, $V_{0}$, is associated to the
trivial cobordism $W_{0}=\R\times\{0\}\times Q$.

The first remark for this class of examples is that all such $V$'s are
isomorphic in $D\fuk^{\ast}(\C\times M)$ to $V_{0}$.  The reason is
that from Theorem \ref{thm:main-dec-gen0} we have a decomposition:
$$V\cong \textnormal{cone}(\gamma_{3}\times 
Q\stackrel{\bar{\varphi}_{V}}{\longrightarrow} \gamma_{2}\times
Q)~.~$$ The morphism $\bar{\varphi}_{V}$ can be identified with a
class $\varphi_{V}\in HF(Q,Q)$ which is given by the image of the
fundamental class $[Q]\in HF(Q,Q)$ under the morphism $\varphi$
defined as in Equation (\ref{eq:cone-tau-S-1}) - see also Figure
\ref{f:lags-bf-stretching} (of course, in our discussion here the
fibration is trivial so that both ends of $V$ in Figure
\ref{f:lags-bf-stretching} are equal to $Q$). Moreover, $\varphi_{V}$
is an invertible element (see also \cite{Bi-Co:cob1}). As a
consequence, the cone over $\bar{\varphi}_{V}$ is easily identified
with the cone over $\bar{\varphi}_{V_{0}}$, where
$\varphi_{V_{0}}=[Q]$. In short, the two decompositions are isomorphic
as in the diagram below
\begin{eqnarray}\label{eq:diag1}
   \begin{aligned}
      \xymatrix{ \gamma_{3}\times Q
        \ar[r]^{\bar{\varphi}_{V}}\ar[d]_{id} &
        \gamma_{2}\times Q \ar[r]\ar[d]^{\varphi_{V}^{-1}}& V\ar[d]\\
        \gamma_{3}\times Q \ar[r]^{\bar{\varphi}_{V_{0}}} &
        \gamma_{2}\times Q \ar[r] & V_{0} }
   \end{aligned}
\end{eqnarray}
but they are not identical.

\subsection{Decompositions in $D\fuk^{\ast}(M)$ induced from
  cobordisms in $\C\times M$}
Let $V'$ be a cobordism $V':\emptyset \cobto (L_{1},\ldots, L_{k})$,
$V'\subset (\C\times M, \omega_{0}\oplus \omega)$.  Theorem
\ref{thm:main-dec-gen0} and its Corollary \ref{cor:dec-M} associate to
$V'$ a cone decomposition
\begin{equation}\label{eq:new-dec} L_{1}\cong (L_{k}\to L_{k-1}\to
   \ldots \to L_{2})
\end{equation}

At the same time, the machinery in \cite{Bi-Co:lcob-fuk} applies to
cobordisms $V'': L\cobto (L_{1},\ldots, L_{k})$ and associates to such
a $V''$ another cone decomposition:
\begin{equation}\label{eq:old-dec}
   L\cong (L_{k}\to L_{k-1}\to\ldots L_{1})
\end{equation}

We want to briefly remark here that the decomposition
(\ref{eq:old-dec}) is a consequence of (\ref{eq:new-dec}). By
elementary manipulations, to see this it is sufficient to consider a
cobordism $V:\emptyset \cobto (L_{2}, L_{3},\ldots, L_{k})$ without
positive ends and with the first negative end, $L_{1}$, also empty and
show that the cone decompositions (\ref{eq:old-dec}) and
(\ref{eq:new-dec}), both associated to $V$, coincide.

For this, notice that, by following the proofs of Theorem
\ref{thm:main-dec-gen0} and Corollary \ref{cor:dec-M}, the cone
decomposition (\ref{eq:new-dec}) is deduced from the following exact
sequences of $\fuk^{\ast}(M)$ modules:
\begin{equation}\label{eq:ex-mod0}W'_{E,i-1}(r\times -)\to
   W'_{E,i}(r\times -)\to \mathcal{Y}(L_{i})~.~
\end{equation}
Here $W'_{E,i}$ are the $\fuk^{\ast}(\C\times M)$ modules that are
introduced at the Step 3 of the proof of Proposition
\ref{lem:decomp-remote}, $r$ is the horizontal line $r=\R\times \{1\}$
and $-$ stands for a variable $Y\in \mathcal{O}b(\fuk^{\ast}(M))$. The
first map in (\ref{eq:ex-mod0}) is an inclusion and the second a
quotient.  There is a slight abuse here as cobordisms of type $r\times
Y$ have obviously a positive end by contrast to the objects
considered in most of this paper, still the modules $W'_{E,i}(r\times
-)$ are well defined.  Indeed, as explained at the Step 3 of the proof
of Proposition \ref{lem:decomp-remote}, $W'_{E,i-1}(r \times Y)$ is
generated by the intersection points of $r\times Y$ with the first $i$
branches of $W'$ where $W'$ is, in our case, obtained from $V$ by a
Hamiltonian isotopy that keeps its ends fixed and moves the
non-cylindrical part of $V$ in the lower half-plane - see, for
instance, Figure \ref{fig:boxes}.  By inspecting
\cite{Bi-Co:lcob-fuk}, we see that the cone
decomposition~\eqref{eq:old-dec} follows from exact sequences of
$\fuk^{\ast}(M)$ modules:
$$\mathcal{M}_{V,i-1}\to \mathcal{M}_{V,i}\to \mathcal{Y}(L_{i})~.~$$
For the description of these modules see Figure 4 and Equation (4) in
\cite{Bi-Co:lcob-fuk}. It immediately, follows that
$\mathcal{M}_{V,i}=W'_{E,i}(r\times -)$ and thus (\ref{eq:new-dec})
and (\ref{eq:old-dec}) are identified.

\subsection{A simple cobordism in a Lefschetz fibration with a single
  critical point}\label{subsubsec:simple-cob}
Consider a Lefschetz fibration $\pi :E\to \C$ of fibre $(M,\omega)$
and with a single singularity $x_{1}$ of critical value $v_{1}$. We
assume that the fibration is tame outside a set $U\subset \C$ as in
Figure \ref{fig:example} and we consider a cobordism $V\subset E$ that
projects to the curve $\gamma\in \C$. As in the picture this curve
turns once around $v_{1}$.  Are also pictured there the curves
$\gamma_{2}$ and $t_1$ that appear in the statement of Theorem
\ref{thm:main-dec-gen0} as well as the ``mirror'' singularity $x'_{1}$
and the matching sphere $\hat{S}_{1}$ that appear in the proof of this
theorem (see \S\ref{subsec:prof-main-t}).

By the relation between the Dehn twist and the monodromy of Lefschetz
fibrations, the ends of $V$ are so that if the first end of $V$ is the
Lagrangian $L\subset M$, then the second end is $\tau_{S}L$ for $S$ an
appropriate vanishing sphere associated to $x_{1}$, this can be taken
to be the sphere over the end of the curve $t_{1}$.

\begin{figure}[htbp]
   \begin{center}
      \includegraphics[scale=0.85]{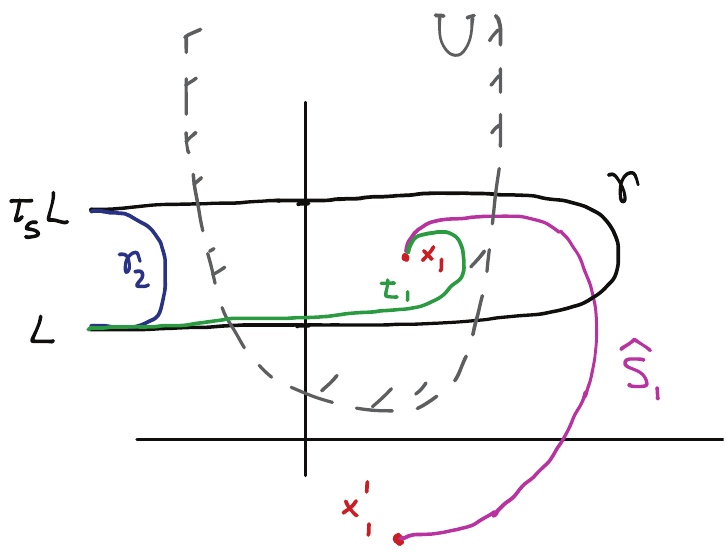}
   \end{center}
   \caption{The curves $\gamma$, $\gamma_{2}$, $t_{1}$, the region $U$ 
     outside which the fibration $\pi:E\to \C$ is tame and the matching 
     sphere $\hat{S}_{1}$ that is included in the extended
     fibration $\hat{\pi}: \hat{E}\to\C$.
     \label{fig:example}}
\end{figure}
Theorem \ref{thm:main-dec-gen0} applied to $V$ shows that:
\begin{equation}\label{eq:V-cone}
   V\cong \textnormal{cone} \/(T_{1}\otimes E_{1}\to \gamma_{2}\times \tau_{S}L )
\end{equation}
where, as in (\ref{eq:E-i-s}), $E_{1}=HF(\hat{S}_{1}, V)$. In this
case we easily see that $HF(\hat{S}_{1}, V)\cong HF(S,L)$. By applying
the restriction functor $\mathcal{R}_{1}$ to the equation
(\ref{eq:V-cone}) we obtain
$$L\cong \textnormal{cone}\/(S\otimes HF (S,L)\to \tau_{S}L)$$
which is just another way to express Seidel's exact triangle from
Proposition~\ref{t:ex-tr-compact}.

It is instructive to briefly discuss the case when the intersection
between $S$ and $L$ is a single point. In this case consider a
thimble $\hat{T}_{1}$ that is included in the initial fibration
$\pi:E\to \mathbb{C}$ and covers the curve that is given by the
projection of $\hat{S}_{1}$ in Figure \ref{fig:example} but extended
horizontally to $-\infty$.  (there is no added singularity $x'_{1}$ in
this case).  This thimble intersects $V$ in a single point and one can
surger $V$ and $\hat{T}_{1}$ at this point. The resulting manifold
$\hat{V} = \hat{T}_1 \# V$ is monotone and has cylindrical ends $S$,
$L$ and $\tau_{S}L$. Moreover, by the same arguments as
in~\S\ref{subsec:Dehn-disj}, $\hat{V}$ can be Hamiltonian isotoped
(with compact support) away from $U$.  That means that $\hat{V}$ can
be actually regarded as a cobordisms embedded in $\C\times M$ and thus
the decomposition result from~\cite{Bi-Co:lcob-fuk} (that applies to
cobordisms in $\C\times M$) implies already the existence of the exact
triangle $L\cong (S\to \tau_{S} L)$. This argument applies as well
when the initial cobordism $V$ is more general than the one discussed
till now but again under the restriction that $\hat{T}_{1}$ intersects
$V$ (transversely) in a single point.
 
Coming back to our $V$, pictured in Figure~\ref{fig:example}, there is
yet another equivalent approach to produce a cobordism $\hat{V}$ with
the properties mentioned above that is possibly even more direct.
This is pictured in Figure \ref{fig:example2}.
\begin{figure}[htbp]
   \begin{center}
\includegraphics[scale=0.65]{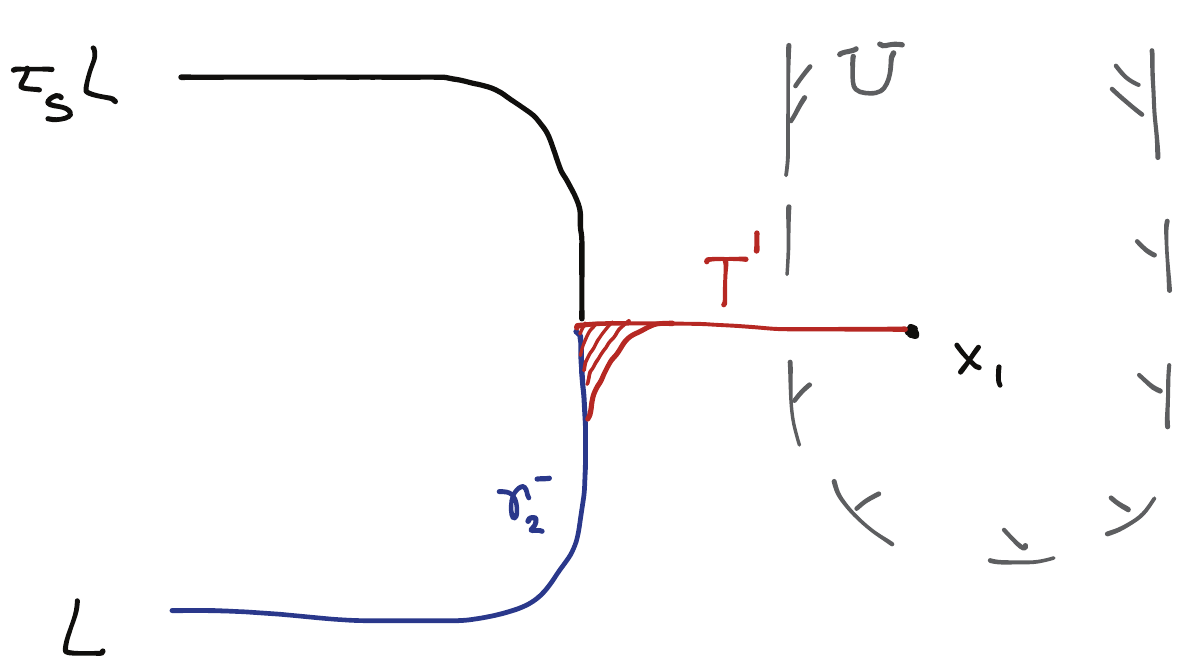}
   \end{center}
   \caption{$Y$-surgery between $\gamma_{2}^{-}\times L$ and the
     thimble $T'$ in case $L$ and $S$ have a single intersection
     point.
     \label{fig:example2}}
\end{figure}
In this case, we consider a thimble $T'$ that goes horizontally
towards $-\infty$ starting from $x_{1}$ and we do a $Y$-surgery in a
single point between $T'$ and $\gamma_{2}^{-}\times L$. Here
$\gamma_{2}^{-}$ is the first half of the curve $\gamma_{2}$ and
$Y$-surgery is the construction of the trace of the surgery as
Lagrangian cobordism as described in \cite{Bi-Co:cob1} \S 6.1. We can
then cut $T'$ outside of $U$ and thus obtain another cobordism which
can be regarded as embedded in $\C\times M$. Moreover the latter
cobordism will have $S, L$ and $\tau_{S}L$ as its ends.
  \begin{figure}[htbp]
   \begin{center}
\includegraphics[scale=0.65]{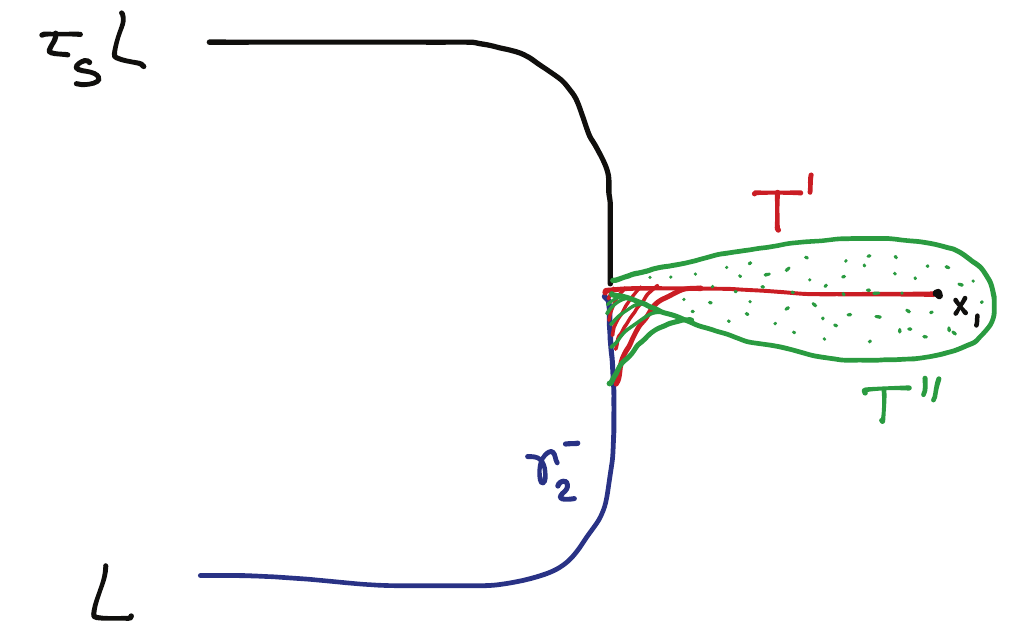}
   \end{center}
   \caption{Iterated $Y$-surgery with two thimbles $T'$ and a copy 
     $T''$ of $T'$. The projection of $T''$ fills the 
   green dotted area. 
   \label{fig:example3}}
\end{figure} 
Finally, it is useful to note that in case the number of intersection
points of $L$ and $S$ is at least two, both constructions above fail.
In both cases, it is still possible to do an iterated surgery with a
number of thimbles equal to the number of intersection points between
$L$ and $S$, basically by the same method as described in
\S\ref{subsubsec:multipl-surg}.  However when using these copies
either cylindricity at infinity is lost or the resulting manifold,
after surgery, is no longer embedded but only immersed.  As an
example, if we perform the $Y$-surgery in the case when there are two
intersection points and project the resulting manifold $\hat{V}$ onto
$\C$ the image of $\hat{V}$ is as in Figure \ref{fig:example3}: the
thimble $T'$ can be conserved as before - its projection is in red -
but the additional copy of it, $T''$, will project as the green dotted
region there, and it is not clear how to obtain a cobordism (which is
cylindrical at $\infty$) when passing to $\mathbb{C} \setminus U$.  As
a last remark, this $\hat{V}$, or a small perturbation thereof, can
also be viewed as obtained by stretching $V$ in the direction of
$-\nabla \textnormal{Re}(\pi))$.

\subsection{Changes of generators}
The generators appearing in Theorem~\ref{thm:main-dec-gen0}, in
particular, the $T_{i}$'s are not always the most convenient for
calculations even if they appear naturally in our proof. It is however
easy to change generators in case a different choice is preferable. We
exemplify this in the case of one Lefschetz fibration which we assume
to fit the setting of Theorem \ref{thm:main-dec} and with only three
critical points, of critical values $v_{1}, v_{2}, v_{3}$. In particular, $m=3$. 

We consider two families of thimbles $T_i$, $T'_i$, $i=1,2,3$, that
are like in the statement of Theorem \ref{thm:main-dec-gen0} and such
that the $T_i's$ cover curves $t_{i}$ and the $T'_i$'s cover curves
$t'_{i}$ as in in Figure~\ref{fig:var-thimbles}.

\begin{figure}[htbp]
   \begin{center}
\includegraphics[scale=0.7]{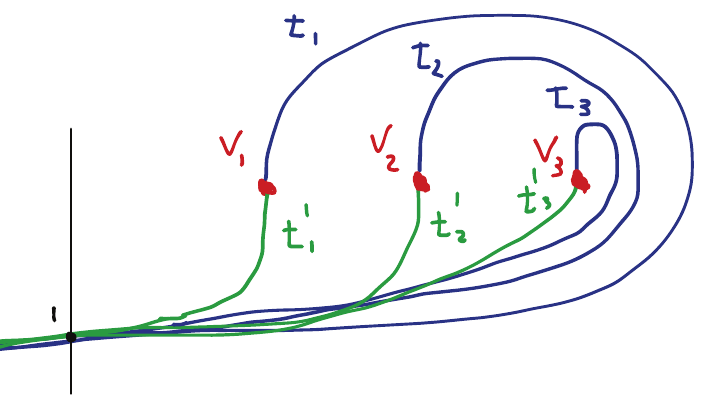}
   \end{center}
   \caption{The projections $t_{i}$ and respectively $t'_{i}$, 
     of the thimbles $T_{i}$, respectively $T'_{i}$, $i=1,2,3$ associated to the critical points
     $x_{1}, x_{2}, x_{3}$ of critical values $v_{1}, v_{2}, v_{3}$.
     \label{fig:var-thimbles}}
\end{figure}
It is easy to see that by applying Theorem \ref{thm:main-dec-gen0} to
the thimbles $T'_{i}$ we obtain first $T'_{3}\cong T_{3}$. Further,
$T'_{2}\cong \textnormal{cone}(T_{2}\to T_{3}\otimes E^{2}_{3})$, with
$E^{2}_{3}=HF(\hat{S}_{3}, \tau_{\hat{S}_{2}}\tau_{\hat{S}_{1}}
T'_{2})$. Notice also $\tau_{\hat{S}_{1}}T'_{2}= T'_{2}$ and
$\tau_{\hat{S}_{2}} T'_{2}$ is just the one point surgery between
$\hat{S}_{2}$ and $T'_{2}$. It follows $E^{2}_{3}\cong
HF(S_{3},S_{2})$ where $S_{i}$ are vanishing spheres associated to the
singularity $x_{i}$ (inside a fixed fibre
$(M,\omega)=\pi^{-1}(z_{0})$).  Thus
$$T_{2}\cong (T'_{3}\otimes HF(S_{3},S_{2})\to T'_{2})~.~$$ Similarly,
$T_{1}'\cong \textnormal{cone}(T_{1}\to T_{2}\otimes E^{1}_{2}\to
T_{3}\otimes E^{1}_{3})$ and we can again estimate:
$E^{1}_{2}=HF(\hat{S}_{2}, \tau_{\hat{S}_{1}}T'_{1})\cong HF(S_{2},
S_{1})$,
$E^{1}_{3}=HF(\hat{S}_{3},\tau_{\hat{S}_{2}}\tau_{\hat{S}_{1}}T'_{1})$.
Thus we get:
$$T_{1}\cong (T'_{3}\otimes HF(S_{3},S_{2})\otimes HF(S_{2},S_{1})
\to T'_{2}\otimes HF(S_{2},S_{1})\to T'_{3} \otimes E^{1}_{3}\to
T'_{1})~.~$$ This expression can be further simplified. For instance,
the second and third terms can be switched because
$\hom(T'_{2},T'_{3})$ is acyclic (i.e. $HF(T'_{2},T'_{3}) = 0$). In
conclusion, we can write
$$T_{1}\cong (T'_{3}\otimes E'_{3}\to T'_{2}\otimes E'_{2}\to T'_{1})$$
for appropriate $\mathcal{A}$-modules $E'_{3}$, $E'_{2}$.  Using these
arguments the decompositions given by Theorem \ref{thm:main-dec-gen0}
can be re-written in the generators $T'_{i}$: the sequence
$(T_{1}\otimes E_{1}\to \ldots T_{3}\otimes E_{3})$ inside the
cone-decomposition provided by that theorem will be replaced by
$(T'_{3}\otimes G_{3}\to T'_{2}\otimes G_{2}\to T'_{1}\otimes G_{1})$
for appropriate modules $G_{i}$.
 
The manipulations above can be extended to fibrations with more than
three singularities in a straightforward way. The main difficulty in
making these changes of generators explicit is in determining the
modules $G_{i}$. In this respect, it is useful to note that there
exists an alternative proof of the decompositions in Theorem
\ref{thm:main-dec-gen0} that avoids the geometric disjunction step
contained in \S\ref{subsec:Dehn-disj} and implements iteratively the
stretching argument in \S\ref{s:cob-vpt} to the case of more
singularities. While this method becomes quite involved for more than
a few singularities, it offers sometimes a more direct way to estimate
the relevant modules for specific generating families of thimbles.
 
\subsection{Real Lefschetz fibrations} \label{sb:real-lef} Real
Lefschetz fibrations have recently been studied from the topological
and real algebraic geometry viewpoints (see e.g.~\cite{De-Sa:products,
  Sal:real-elements, Sal:classif-tr, Sal:invariants-tr}).  Lagrangian
cobordism is naturally related to this notion and we describe this
relationship in the first subsection below. We then pursue with a
construction of such fibrations and, in the last subsection, with a
concrete example.
  
\subsubsection{Lagrangian cobordism and real Lefschetz fibrations} 
  
Let $\pi: E \longrightarrow \mathbb{C}$ be a Lefschetz fibration
endowed with a symplectic structure $\Omega$, as in
Definition~\ref{df:lef-fib}. Denote by $(M, \omega)$ the general fiber
of $(E, \Omega)$. Let $c_E: E \longrightarrow E$ be an anti-symplectic
involution, i.e. $c_E^* \Omega = -\Omega$ and $c_E \circ c_E = \id$.
Assume further that $c_E$ covers the standard complex conjugation
$c_{\mathbb{C}}: \mathbb{C} \longrightarrow \mathbb{C}$, namely $\pi
\circ c_E = c_{\mathbb{C}} \circ \pi$. Denote by $V =
\textnormal{Fix}(c_E)$ the fixed point locus of $c_E$.  Note that the
projection $\pi(V)$ of $V$ to $\mathbb{C}$ is a subset of
$\mathbb{R}$. The following proposition shows that $V$ is a Lagrangian
cobordism and also gives a criterion for its monotonicity.

\begin{prop} \label{p:real-lef} Under the above assumptions $V$ is a
   Lagrangian cobordism with at most one positive end and at most one
   negative one (but possibly without any ends at all). Its projection
   $\pi(V) \subset \mathbb{R}$ is of the form $\cup_{j \in
     \mathcal{S}} \overline{I}_j$, where $\mathcal{S}$ is a subset of
   the set of connected components of $\mathbb{R} \setminus
   \textnormal{Critv}(\pi)$, $I_j$ stands for the path connected
   component corresponding to $j$ and $\overline{I}_j$ is the closure
   of $I_j$.  Thus $\partial \, \pi(V)$ is a subset of
   $\textnormal{Critv}(\pi) \cap \mathbb{R}$.

   Moreover, for every $z \in \mathbb{R} \setminus
   \textnormal{Critv}(\pi)$ the part of $V$ lying over $z$, $V_z :=
   E_z \cap V$, coincides with the fixed point locus of the
   anti-symplectic involution $c_E|_{E_z}$ hence is either empty or a
   smooth Lagrangian submanifold of $E_z$ (possibly disconnected). In
   particular, the Lagrangians corresponding to the ends of $V$ (if
   they exist) are real with respect to restriction of $c_E$ to the
   regular fibers over the real axis at $\pm \infty$.

   \pbhl{
   If $(E, \Omega)$ is a monotone symplectic manifold then $V$ is a
   monotone Lagrangian submanifold of $E$. Further, denote by
   $c_1^{\min}(E)$ the minimal Chern number on spherical classes in
   $E$ and by $N_V$ the minimal Maslov number of $V$. If
   $c_1^{\min}(E)$ is odd then $c_1^{\min}(E) | N_V$, and if
   $c_1^{\min}(E)$ is even then $\tfrac{1}{2}c_1^{\min}(E) | N_V$.}

 \pbhl{If $\dim_{\mathbb{C}} M \geq 2$ and $(M, \omega)$ is monotone
   then $(E, \Omega)$ is monotone too and
   $c_1^{\min}(E) = c_1^{\min}(M)$, hence $V$ is a monotone Lagrangian
   cobordism.}
\end{prop}

\begin{proof}
   That $V$ is a (smooth) Lagrangian submanifold follows from it being
   the fixed point locus of an anti-symplectic involution.

   We now show that $V$ is a cobordism and prove the other statements
   about the projection $\pi(V)$.  Since $V$ is Lagrangian,
   $D\pi_{x}|_{T_x V} \longrightarrow \mathbb{R}$ vanishes iff $x \in
   \textnormal{Crit}(\pi)$ (see e.g.  Chapter~16
   of~\cite{Se:book-fukaya-categ}). It follows that $\pi(V) \setminus
   \textnormal{Critv}(\pi)$ is an open subset of $\mathbb{R}$ and all
   the points in this subset are regular values of the projection
   $\pi|_V : V \longrightarrow \mathbb{R}$. By construction $V \subset
   E$ is a closed subset. Therefore if $I \subset \mathbb{R} \setminus
   \textnormal{Critv}(\pi)$ is a connected component and $\pi(V) \cap
   I \neq \emptyset$ then $I \subset \pi(V)$. Next, notice that since
   $V$ is Lagrangian it is invariant with respect to parallel
   transport along any intervals $I \subset \pi(V) \setminus
   \textnormal{Critv}(\pi)$.

   The statements about $V_z = \textnormal{Fix}(c_E|_{E_z})$ follow
   directly from the definitions.

   We now address the monotonicity of $V$. This follows from spherical
   monotonicity of $(E, \Omega)$, by a standard reflection argument
   based on the existence of the anti-symplectic involution $c_E$ and
   the fact that $V = \textnormal{Fix}(c_E)$.

   Finally, it remains to prove the statement relating the spherical
   monotonicity of $(M, \omega)$ with that of $E$. Let
   $E_{z_0} \subset E$ be a smooth fiber endowed with the symplectic
   structure induced by $\Omega$ (so that $(M, \omega)$ is
   symplectomorphic to $E_{z_0}$). Assume that
   $\dim_{\mathbb{C}} E_{z_0} \geq 2$ and that $E_{z_0}$ is
   monotone. It is easy to see that the inclusion,
   $\pi_2(E_{z_0}) \longrightarrow \pi_2(W)$ is surjective and this
   implies the monotonicity statement.
   \end{proof}

In the next subsection we will show how to construct real Lefschetz
fibrations out of Lefschetz pencils arising in real algebraic
geometry.

\subsubsection{Constructing real Lefschetz fibrations}
\label{sb:cnstr-real-lef}

Let $X$ be a smooth complex projective variety endowed with a real
structure, namely an anti-holomorphic involution $c_X: X
\longrightarrow X$. Let $\mathscr{L}$ be a very ample line bundle on
$X$ and assume further that it is endowed with a real structure
compatible with $c_X$. By this we mean an anti-holomorphic involution
$c_{\mathscr{L}} : \mathscr{L} \longrightarrow \mathscr{L}$ covering
$c_X$, i.e. $\textnormal{pr} \circ c_{\mathscr{L}} = c_X \circ
\textnormal{pr}$, where $\textnormal{pr}: \mathscr{L} \longrightarrow
X$ is the bundle projection.

Denote by $H^0(\mathscr{L})$ the space of holomorphic sections of
$\mathscr{L}$ and by $\mathbf{P} := \mathbb{P}\bigl( H^0(\mathscr{L})
\bigr)^*$ the projectivization of its dual (which can also be thought
of as the space of hyperplanes in $H^0(\mathscr{L})$). We denote by
$\mathbf{P}^* := \mathbb{P} H^0(\mathscr{L})$ the projectivization of
the space of sections itself. Note that $\mathbf{P}^*$ is the dual
projective space of $\mathbf{P}$, hence the notation.

The real structure of $\mathscr{L}$ induces a real structure $c_H$ on
$H^0(\mathscr{L})$ defined by $c_H(s) = c_{\mathscr{L}} \circ s \circ
c_X$. Denote by $H^0_{\mathbb{R}}(\mathscr{L}) \subset
H^0(\mathscr{L})$ the space of real sections of $\mathscr{L}$ (i.e.
sections $s$ with $c_H(s) = s$). The real structure $c_H$ descends to
real structures on $\mathbf{P}^*$ and $\mathbf{P}$ which, by abuse of
notation, we continue to denote both by $c_H$.  The fixed point locus
of $c_H$ on $\mathbf{P}$ will be denoted by $\mathbf{P}_{\mathbb{R}}$
and that on $\mathbf{P}^*$ by $\mathbf{P}^*_{\mathbb{R}}$.

Consider now the projective embedding defined using the sections of
$\mathscr{L}$, $X \hooklongrightarrow \mathbf{P}$.  This embedding is
real in the sense that it commutes with $(c_X, c_H)$. Furthermore,
there is an isomorphism between $\mathbf{P}$ and ${\mathbb{C}}P^N$
which sends $c_H$ to the standard real structure $c_{{\mathbb{C}}P^N}$
of ${\mathbb{C}}P^N$ (hence $\mathbf{P}_{\mathbb{R}}$ is sent under
this isomorphism to ${\mathbb{R}}P^N$). We fix once and for all such
an isomorphism.  Denote by $\omega_{{\mathbb{C}}P^N}$ the standard
symplectic structure of ${\mathbb{C}}P^N$ normalized so that the area
of $\mathbb{C}P^1$ is $1$. Since $c_{{\mathbb{C}}P^N}$ is
anti-symplectic with respect to $\omega_{{\mathbb{C}}P^N}$ the
previously mentioned isomorphism yields a K\"{a}hler form
$\omega_{\mathbf{P}}$ on $\mathbf{P}$ and therefore also a K\"{a}hler
form $\omega_X$ on $X$ so that $c_X$ is anti-symplectic with respect
to $\omega_X$.

%\CCPB{Alternatively, and perhaps simpler, would be to do the following
%  instead of the part of the last paragraph that starts with
 % ``Furthermore...'': We now endow $\mathbf{P}$ with a K\"{a}hler form
 % $\omega_{\mathbf{P}}$ as follows. Pick any K\"{a}hler form
 % $\omega_0$ on $\mathbf{P}$ so that the area of a projective line is
 % $1$. Define $\omega_{\mathbf{P}} = \tfrac{1}{2}(\omega_0 - c_H^*
 % \omega_0)$. Clearly $\omega_{\mathbf{P}}$ is a K\"{a}hler form and
 % $c_H$ is anti-symplectic with respect to it. Restricting to $X$ we
 % obtain a K\"{a}hler form $\omega_X$ on $X$ with respect to which
 % $c_X$ is anti-symplectic.}

Let $\Delta(\mathscr{L}) \subset \mathbf{P}^*$ be the discriminant
locus (a.k.a. the dual variety of $X$), which by definition is the
variety consisting of all section $[s] \in \mathbf{P}^*$ (up to a
constant factor) which are somewhere {\em non-transverse} to the
zero-section.  Denote by $\Delta_{\mathbb{R}}(\mathscr{L}) =
\Delta(\mathscr{L}) \cap \mathbf{P}^*_{\mathbb{R}}$ its real part.

Let $\ell \subset \mathbf{P}^*$ be a line which is invariant under
$c_H$ and intersects $\Delta(\mathscr{L})$ only along its smooth
strata and transversely. Fix an isomorphism $\ell \approx
\mathbb{C}P^1$ and endow $\ell$ with a standard K\"{a}hler structure
$\omega_{\ell}$ normalized so that its total area is $1$. Consider the
symplectic manifold $\ell \times X$ endowed with the symplectic
structure $\omega_{\ell} \oplus \omega_X$. For every $\lambda \in
\mathbf{P}^*$ denote by $\Sigma^{(\lambda)} = s^{-1}(0) \subset X$ the
zero locus corresponding to a section $s$ representing $\lambda$.
(The varieties $\Sigma^{(\lambda)}$ are sometimes called hyperplane
sections since they can also be viewed as the intersection of the
image of $X$ in $\mathbf{P}$ with linear hyperplanes.) Note that for
all $\lambda \not \in \Delta(\mathscr{L})$, the variety
$\Sigma^{(\lambda)}$ is smooth. We endow these varieties with the
symplectic structure induced from $\omega_X$. The complement of the
discriminant, $\mathbf{P}^* \setminus \Delta(\mathscr{L})$, is path
connected (since $\Delta(\mathscr{L})$, being a proper complex
subvariety of $\mathbf{P}^*$, has real codimension $\geq 2$).
Therefore all the symplectic manifolds $\Sigma^{(\lambda)}$, $\lambda
\in \mathbf{P}^* \setminus \Delta(\mathscr{L})$, are mutually
symplectomorphic.

For every $\lambda \in \mathbf{P}^*_{\mathbb{R}} \setminus
\Delta_{\mathbb{R}}(\mathscr{L})$ the manifold $\Sigma^{(\lambda)}$
has a real structure induced by $c_X$. Denote its real part by
$\Sigma_{\mathbb{R}}^{(\lambda)}$. We stress that {\em in contrast to
  $\mathbf{P}^* \setminus \Delta(\mathscr{L})$, its real part
  $\mathbf{P}^*_{\mathbb{R}} \setminus
  \Delta_{\mathbb{R}}(\mathscr{L})$ is in general disconnected and the
  topology of $\Sigma^{(\lambda)}_{\mathbb{R}}$ depends on the
  connected component $\lambda$ belongs to.} Define now
$$\widehat{E} = \{ (\lambda, x) \mid \lambda \in \ell, \; x \in
\Sigma^{(\lambda)} \} \subset \ell \times X.$$ Due to the
transversality assumptions between $\ell$ and $\Delta(\mathscr{L})$
the variety $\widehat{E}$ is smooth. We endow it with the symplectic
structure $\widehat{\Omega}$ induced by $\omega_{\ell} \oplus
\omega_X$.

The space $\widehat{E}$ comes with two ``projections'', $\pi:
\widehat{E} \longrightarrow \ell$ and $p_X: \widehat{E}
\longrightarrow X$, induced by the two projections from $\ell \times
X$ to its factors. The first one is a Lefschetz fibration (whose base
is $\ell \approx \mathbb{C}P^1$). The fact that the critical points of
$\pi$ are non-degenerate follows from the transversality assumptions
on the intersection of $\ell$ and $\Delta(\mathscr{L})$. The second
projection (which will not be used here) realizes $\widehat{E}$ as the
blow-up $\textnormal{Bl}_B(X) \longrightarrow X$ of $X$ along the base
locus $B$ of the pencil $\ell$ (i.e. $B = \{ x \in X \, |\, x \in
\Sigma^{(\lambda)} \; \forall \, \lambda \in \ell\}$). The involutions
$c_H$ and $c_X$ induce an anti-holomorphic involution on $\widehat{E}$
which is also anti-symplectic with respect to $\widehat{\Omega}$.

Let $D \subset \ell$ be a closed disk which is invariant under $c_H$.
Identify $\ell \setminus D$ with $\mathbb{C}$ via an orientation
preserving diffeomorphism which commutes with $(c_H, c_\mathbb{C})$,
where $c_{\mathbb{C}}$ is the standard conjugation on $\mathbb{C}$.
The real part $\ell_{\mathbb{R}} \setminus D$ of $\ell \setminus D$ is
sent by this diffeomorphism to $\mathbb{R}$.

By restricting $\pi$ to the complement of $D$ we obtain a Lefschetz
fibration $E = \pi^{-1}(\ell \setminus D)$ over $\ell \setminus D
\cong \mathbb{C}$.  We endow $E$ with the symplectic structure
$\Omega$ coming from $\widehat{\Omega}$ and by a slight abuse of
notation denote its projection by $\pi: E \longrightarrow \mathbb{C}$.
Restricting the preceding anti-symplectic involution of $\widehat{E}$
to $E$ we obtain an anti-symplectic involution $c_E$ on $E$ which
covers the standard conjugation $c_{\mathbb{C}}$ as
in~\S\ref{sb:real-lef}. The critical values of $\pi$ are precisely
$(\ell \setminus D) \cap \Delta(\mathscr{L})$. Some of them lie on
$\ell_{\mathbb{R}}$ (i.e. the real axis) and the others come in pairs
of conjugate points.

Note that $\ell_{\mathbb{R}} \setminus \Delta(\mathscr{L})$ might have
several connected components. If $\lambda', \lambda'' \in
\ell_{\mathbb{R}} \setminus \Delta(\mathscr{L})$ are in the same
component then $\Sigma_{\mathbb{R}}^{(\lambda')}$ and
$\Sigma_{\mathbb{R}}^{(\lambda'')}$ are diffeomorphic, but otherwise
not necessarily.  

%\CCPB{There should be even a symplectomorphism
%  $(\Sigma^{(\lambda')}, \Sigma_{\mathbb{R}}^{(\lambda')}) \approx
 % (\Sigma^{(\lambda'')}, \Sigma_{\mathbb{R}}^{(\lambda'')})$ induced
 % by parallel transport or something like that, but I haven't checked
 % that.}

Consider now the fixed point locus $V = \textnormal{Fix}(c_E) \subset
E$. By Proposition~\ref{p:real-lef}, $V$ is a Lagrangian cobordism.
Its ends correspond to $\Sigma_{\mathbb{R}}^{(\lambda_-)}$ and
$\Sigma_{\mathbb{R}}^{(\lambda_+)}$, where $\lambda_-, \lambda_+ \in
\ell_{\mathbb{R}} \setminus D$ are close enough to the two boundary
points of $\ell_{\mathbb{R}} \cap D$. As hinted above, any of the
$\Sigma^{(\lambda_{\pm})}$ might be disconnected. At the other
extremity any of these ends might also be void.

Finally we address the issue of monotonicity. Assume that
$\dim_{\mathbb{C}} X \geq 3$ and that the symplectic manifold
$(\Sigma^{(\lambda)}, \omega_X|_{\Sigma^{(\lambda)}})$,
$\lambda \not \in \Delta(\mathscr{L})$, is monotone. By
Proposition~\ref{p:real-lef} the Lagrangian cobordism $V$ is monotone.

Turning to more algebraic-geometric terms, here is a criterion that
assures monotonicity of the $\Sigma^{(\lambda)}$'s. For an algebraic
variety we denote by $-K_X$ its canonical class. The following follows
easily from adjunction.
\begin{prop} \label{p:Fano} Let $X$ be a Fano manifold with
   $\dim_{\mathbb{C}}X \geq 3$ and write $-K_X = r D$, with $r \in
   \mathbb{N}$ and $D$ a divisor class. Further, suppose that
   $\mathscr{L} = q D$ with $0<q \in \mathbb{Q}$ and $q<r$.  Then the
   symplectic manifolds $(\Sigma^{(\lambda)},
   \omega_X|_{\Sigma^{(\lambda)}})$, $\lambda \not \in
   \Delta(\mathscr{L})$, are monotone. In particular $V$ is a monotone
   Lagrangian cobordism.
\end{prop}

\subsubsection{A concrete example - real quadric surfaces}
\label{sb:exp-quad}

We present here a concrete example of a real Lefschetz fibration
associated to a pencil of complex quadric surfaces in
${\mathbb{C}}P^3$.  The example can be easily generalized to higher
dimensions.

Let $X = {\mathbb{C}}P^3$ and $\mathscr{L} =
\mathcal{O}_{{\mathbb{C}}P^3}(2)$, both endowed with their standard
real structures (induced by complex conjugation). Clearly
$\mathscr{L}$ is very ample and gives rise to the so called degree-$2$
Veronese embedding which we describe shortly.

Using coordinates $[X_0 : X_1 : X_2 : X_3]$ on ${\mathbb{C}}P^3$ we
identify the space $H^0(\mathscr{L})$ of sections of $\mathscr{L}$
with the space of quadratic homogeneous polynomials
$\lambda(\underline{X})$ in the variables $\underline{X} = (X_0,X_1,
X_2, X_3)$:
\begin{equation} \label{eq:quad-polyn} \lambda(\underline{X}) =
   \sum_{0 \leq i \leq j \leq 3} a_{i,j} X_i X_j.
\end{equation}
Taking $X_i X_j$, $0 \leq i \leq j \leq 3$, as a basis for this space
we obtain an identifications $\mathbf{P} \cong {\mathbb{C}}P^9$ under
which the projective embedding $X \hooklongrightarrow {\mathbb{C}}P^9$
is given by:
$$[z_0:z_1:z_2] \longmapsto [z_0^2: z_0z_1:\cdots : z_i z_j: \cdots : 
z_2z_3: z_3^2],$$ where the coordinates on the right-hand side go over
all $(i,j)$ with $0 \leq i \leq j \leq 3$.

The hyperplane section corresponding to the polynomial $\lambda$ is a
quadric surface $$\Sigma^{(\lambda)} = \bigl\{[z_0:z_1:z_2:z_3] \mid
\lambda(z_0, z_1, z_3, z_3) = 0\bigr\} \subset {\mathbb{C}}P^3.$$ A
straightforward calculation shows that $\lambda \in
\Delta(\mathscr{L})$ if and only if
\begin{equation} \label{eq:det=0} \det \left(
   \begin{matrix}
      2 a_{00} & \phantom{2}a_{01} & \phantom{2}a_{02} & \phantom{2}a_{03} \\
      \phantom{2}a_{10} & 2a_{11} & \phantom{2}a_{12} & \phantom{2}a_{13} \\
      \phantom{2}a_{20} & \phantom{2}a_{21} & 2a_{22} & \phantom{2}a_{23} \\
      \phantom{2}a_{30} & \phantom{2}a_{31} & \phantom{2}a_{32} &
      2a_{33}
   \end{matrix}
   \right) = 0.
\end{equation}
This shows that the discriminant $\Delta(\mathscr{L})$ is a variety of
degree $4$ in $\mathbf{P}^* \cong {\mathbb{C}}P^9$. The smooth stratum
of $\Delta(\mathscr{L})$ consists of those $\lambda$'s where the
matrix in~\eqref{eq:det=0} has rank $3$.

The real part $\Delta_{\mathbb{R}}(\mathscr{L})$ of the discriminant
consists of those polynomials $\lambda$ which in addition
to~\eqref{eq:det=0} have real coefficients (i.e. $a_{i,j} \in
\mathbb{R}$ for every $i,j$).

It is well known that for $\lambda \not \in \Delta(\mathscr{L})$ the
variety $\Sigma^{(\lambda)}$ is isomorphic to $\mathbb{C}P^1 \times
\mathbb{C}P^1$, and moreover when viewed as a symplectic manifold
(endowed with the structure induced from the projective embedding) it
is symplectomorphic to $(\mathbb{C}P^1 \times \mathbb{C}P^1,
2\omega_{\mathbb{C}P^1} \oplus 2\omega_{\mathbb{C}P^1})$, where
$\omega_{\mathbb{C}P^1}$ is normalized so that the area of
$\mathbb{C}P^1$ is $1$.

Consider now the following two sections
$$\lambda_0(\underline{X}) = X_0^2 + X_1^2 + X_2^2 - X_3^2, \qquad 
\lambda_1(\underline{X}) = X_0X_3 - X_1 X_2.$$

A simple calculation shows that $\lambda_0, \lambda_1 \not \in
\Delta(\mathscr{L})$. Denote the real part of $\Sigma^{(\lambda_i)}$
by $L^{(\lambda_i)}$, $i=0,1$. It is easy to see that
$L^{(\Lambda_1)}$ is a Lagrangian tours and moreover we can find a
symplectomorphism $\phi^{(\lambda_1)}: \Sigma^{(\lambda_1)}
\longrightarrow \mathbb{C}P^1 \times \mathbb{C}P^1$ so that
$\phi^{(\lambda_1)}(L^{(\lambda_1)})$ is the split torus $T =
{\mathbb{R}}P^n \times {\mathbb{R}}P^1$. We fix such a diffeomorphism
$\phi^{(\lambda_1)}$. Similarly, there is a symplectomorphisms
$\phi^{(\lambda_0)}: \Sigma^{(\lambda_0)} \longrightarrow
\mathbb{C}P^1 \times \mathbb{C}P^1$ that sends $L^{(\lambda_0)}$ to
the Lagrangian sphere $S = \{(z, \bar{z}) \mid z \in \mathbb{C}P^1\}
\subset \mathbb{C}P^1 \times \mathbb{C}P^1$ which is so-called the
anti-diagonal.

We now consider the pencil $\ell \subset \mathbf{P}^*$ that passes
through the two points $\lambda_0$ and $\lambda_1$. Clearly $\ell$ is
invariant under the anti-holomorphic involution $c_H$. We can
parametrize $\ell$ by $$\mathbb{C}P^1 \ni [t_0:t_1] \longmapsto
\lambda_{[t_0:t_1]} := t_0 \lambda_0 + t_1 \lambda_1.$$ A simple
calculation shows that the intersection points of $\ell$ with
$\Delta(\mathscr{L})$ occur for the following values of $[t_0:t_1]$:
\begin{equation} \label{eq:ell-Delta} [t_0:t_1] \in \bigl\{[1:2],
   [1:-2], [1:2i],[1:-2i] \bigr\},
\end{equation}
and that $\ell$ intersects $\Delta(\mathscr{L})$ only along the
regular stratum. Moreover this intersection is transverse. See the
left part of Figure~\ref{f:l-pencil}.

\begin{figure}[htbp]
   \begin{center}
\includegraphics[scale=0.75]{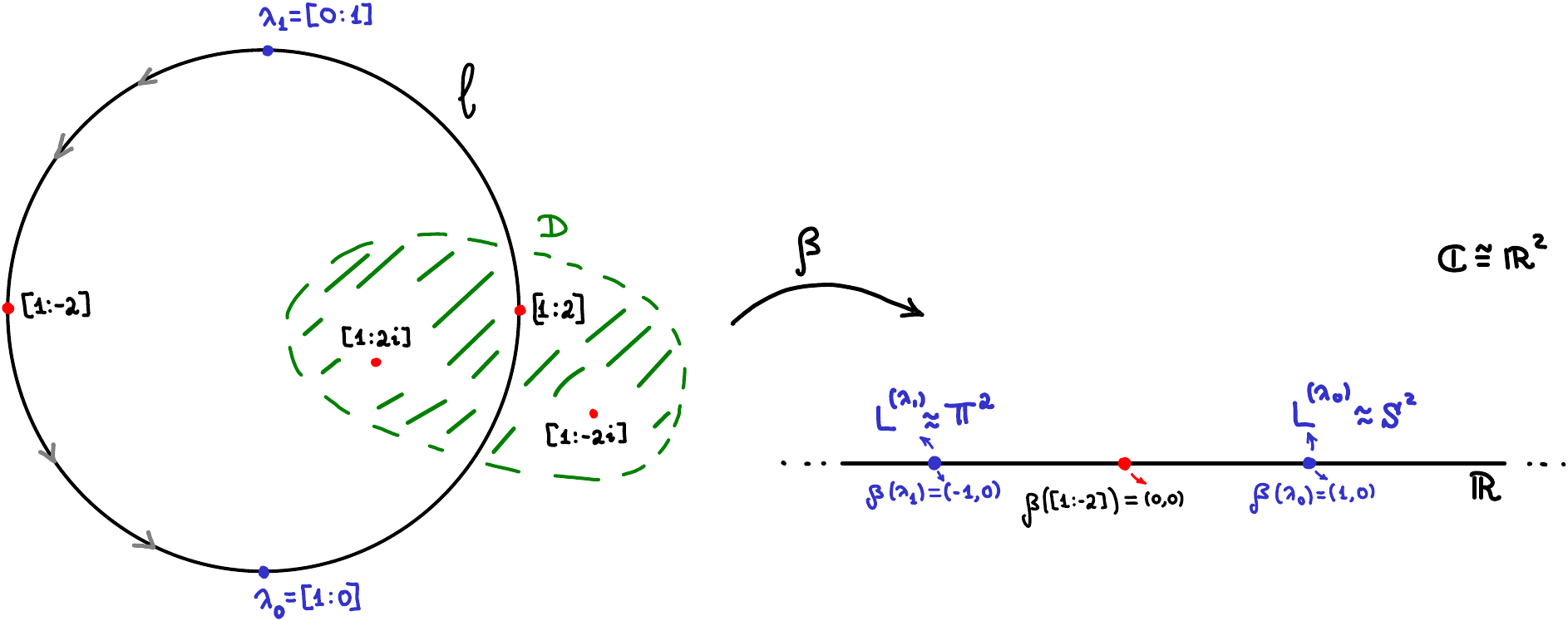}
   \end{center}
   \caption{The real pencil $\ell$ on the left, and the image of $\ell
     \setminus D$ under $\beta$ in $\mathbb{C}$.}
     \label{f:l-pencil}
\end{figure}

We now appeal to the construction in~\S\ref{sb:cnstr-real-lef}. Below
we will often identify $\mathbb{C} \cong \mathbb{R}^2$ in the obvious
way. Choose a disk $D \subset \ell$ which is invariant under $c_H$ and
contains the point $[1:2], [1:2i], [1:-2i]$ but not the point
$[1:-2]$. Fix an orientation preserving diffeomorphism $\beta: \ell
\setminus D \longrightarrow \mathbb{C} \cong \mathbb{R}^2$ such that:
$$\beta(\lambda_1) = (-1,0), \quad \beta(\lambda_0) = (1,0), 
\quad \beta([1:-2]) = (0,0).$$ See the right part of
Figure~\ref{f:l-pencil}. From now on we use the identification $\beta$
implicitly and simply write $\lambda_1 = (-1,0)$, $\lambda_0 = (1,0)$.

Restricting $\widehat{E}$ to $\ell \setminus D$ and applying a base
change via $\beta$ we obtain a Lefschetz fibration
$\pi: E \longrightarrow \mathbb{C}$ with general fiber
$\mathbb{C}P^1 \times \mathbb{C}P^1$ and with a real
structure. \pbhl{Since the minimal Chern number of the general fiber
  is $c_1^{\min} = 2$, $E$ is a strongly monotone Lefschetz fibration
  in the sense of} Definition~\ref{df:monlef}. \pbhl{Its monotonicity
  class is $*=(0)$.}

The projection $\pi$ has exactly one critical value at
$0 \in \mathbb{C}$ (corresponding to $[1:-2] \in \ell$). The real part
$V$ of $E$ is a cobordism with one negative end associated to
$L^{-} = L^{(\lambda_1)}$ which is a Lagrangian torus, and one
positive end associated to $L^+ = L^{(\lambda_0)}$ which is a
Lagrangian sphere. By Proposition~\ref{p:Fano} $V$ is monotone and a
simple calculation shows that it has minimal Maslov number $N_V =
2$. Interestingly we have $N_{L^{-}} = 2$ while $N_{L^{+}} = 4$. Note
also that $d_{L^-} = d_{L^{+}} = 0$, \pbhl{hence $V$ is of the right
  monotonicity class $*=(0)$.}

\subsubsection*{Transforming $V$ to a negative ended cobordism}
In order to obtain a cobordism with only negative ends (as considered
in the rest of the paper) we proceed as follows. Take the Lefschetz
fibration $\pi: E \longrightarrow \mathbb{C}$ and $V \subset E$ as
constructed above. Recall that $0 \in \mathbb{C}$ was the (single)
critical value of $\pi$. Consider a smooth embedding
$\alpha': [0, \infty) \longrightarrow \mathbb{R}^2$ so that:
\begin{enumerate}
  \item $\alpha'(t) = (t, 0)$ for every $0 \leq t \leq 1$.
  \item For $1 < t$, $\alpha'$ lies in the lower half plane and
   $\alpha'(2) = (0, -1)$.
  \item For every $2\leq t$, $\alpha'(t) = (2-t, -1)$. 
\end{enumerate}
Now take the part of the cobordism $V$ that lies over $(-\infty, 1]
\times \mathbb{R} \subset \mathbb{R}^2$ and glue to its right hand
side the trail of the Lagrangian sphere $L^{(\lambda_0)} = V|_{(1,
  0)}$ along the curve $\alpha'|_{[1,\infty)}$. Denote the result by
$W$. It is easy to see that $W$ is a smooth Lagrangian cobordism with
two negative ends. The lower end is a Lagrangian sphere and the upper
end is a Lagrangian torus, both living inside symplectic manifolds
that are symplectomorphic to $\mathbb{C}P^1 \times \mathbb{C}P^1$. See
Figure~\ref{f:S2T2-cob}.

\begin{figure}[htbp]
   \begin{center}
      \includegraphics[scale=0.75]{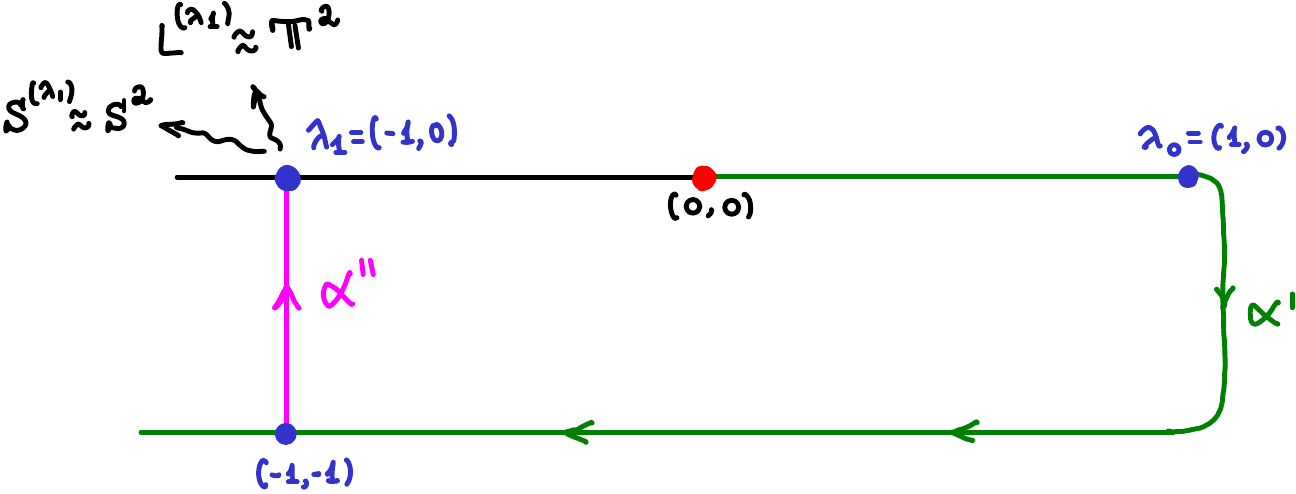}
         \end{center}
   \caption{The cobordism $W$ with two negative ends, and the parallel
     transport of the sphere $L^{(\lambda_0)}$ to the fiber over $\lambda_1$.}
     \label{f:S2T2-cob}
\end{figure}

Note that the Lefschetz fibration $E$ is not tame. Therefore In order
to apply the cone decomposition from Corollary~\ref{cor:dec-M} we need
to identify fibers over different ends. To this end, denote by
$\alpha''$ the straight segment connecting $\alpha'(3) = (-1, -1)$ to
$\lambda_1 = (-1,0)$. Denote by $\alpha = \alpha'|_{[1,3]} * \alpha''$
the concatenation of $\alpha'|_{[1,3]}$ with $\alpha''$. Denote by
$\Pi_{\alpha}: E_{\lambda_0} \longrightarrow E_{\lambda_1}$ the
parallel transport along $\alpha$. Let $S^{(\lambda_1)} =
\Pi_{\alpha}(L^{(\lambda_0)})$ be the parallel transport of the
Lagrangian sphere $L^{(\lambda_0)}$ to the fiber $\Sigma^{(\lambda_1)}
= E_{\lambda_1}$ of $E$ over $\lambda_1$.  See
Figure~\ref{f:S2T2-cob}. By Corollary~\ref{cor:dec-M} we have in
$D\fuk^*(\Sigma^{(\lambda_1)})$ an isomorphism:
\begin{equation} \label{eq:cone-CP1xCP1-1} S^{(\lambda_1)} \cong
   \textnormal{cone} \bigl(S_1 \otimes E \longrightarrow
   L^{(\lambda_1)} \bigr),
\end{equation}
where $S_1 \subset \Sigma^{(\lambda_1)}$ is the vanishing cycle
associated to the critical point of $\pi$ over $0$ and the path
$\alpha'|_{[0,3]} * \alpha''$. According to~\eqref{eq:E-i-s}, the
space $E$ is $HF(\hat{S}_1, W)$, where $\hat{S}_1$ is the matching
cycle emanating from $z_1$, which lies in a suitable extension of the
fibration $E$ (see~\S\ref{subsubsec:null-cob-remote}).

In our case, it is not hard to see that $\hat{S}_1$ intersects $W$ at
a single point and the intersection is transverse. Therefore $E$ is a
$1$-dimensional space. Applying $\phi^{(\lambda_1)}$
to~\eqref{eq:cone-CP1xCP1-1} we now obtain the following isomorphism
in $D\fuk^*(\mathbb{C}P^1 \times \mathbb{C}P^1)$:
\begin{equation*} \label{eq:cone-CP1xCP1-2}
   \phi^{(\lambda_1)}(S^{(\lambda_1)}) \cong \textnormal{cone}
   \bigl(\phi^{(\lambda_1)}(S_1) \longrightarrow T \bigr).
\end{equation*}

By a result of Hind~\cite{Hind:lag-S2xS2} all Lagrangian spheres in
$\mathbb{C}P^1 \times \mathbb{C}P^1$ are Hamiltonian isotopic. In
particular $\phi^{(\lambda_1)}(S^{\lambda_1})$ and
$\phi^{(\lambda_1)}(S_1)$ are both Hamiltonian isotopic to the
anti-diagonal $S$. It follows that:
\begin{equation} \label{eq:cone-STS-1} S \cong \textnormal{cone}
   \bigl(S \longrightarrow T \bigr).
\end{equation}
By rotating the exact triangle corresponding to~\eqref{eq:cone-STS-1}
we obtain the following result:

\begin{cor} Let $M = \mathbb{C}P^1 \times \mathbb{C}P^1$, endowed with
   the symplectic structure $\omega_{\mathbb{C}P^1} \oplus
   \omega_{\mathbb{C}P^1}$. Denote by $S = \{ (z, \bar{z}) \mid z \in
   \mathbb{C}P^1\} \subset M$ the anti-diagonal and by $T =
   \mathbb{R}P^1 \times \mathbb{R}P^1 \subset M$ the split torus. Then
   in $D\fuk^*(M)$ there is an isomorphism
   \begin{equation} \label{eq:cone-S-T-2} T \cong \textnormal{cone}
      \bigl(S \longrightarrow S \bigr).
   \end{equation}
\end{cor}

\begin{remsnonum}
   \begin{enumerate}
     \item[a.] The existence of an isomorphism of the
      type~\eqref{eq:cone-S-T-2} could probably be derived also by the
      following construction whose details need to be precisely worked
      out. Consider a Hamiltonian isotopic copy $S'$ of $S$ so that
      $S'$ intersects $S$ transversely at exactly two points. By
      performing Lagrangian surgery of $S'$ and $S$
      at the intersection points (with appropriate choices of handles)
      one obtains a Lagrangian torus $T' \subset M$. Moreover, for a
      suitable choice of $S'$ and choices of handles the torus $T'$
      should be Hamiltonian isotopic to the split torus $T$. Applying
      the ``figure-Y'' surgery construction from~\cite{Bi-Co:cob1} we
      obtain a cobordism $V$ in $\mathbb{R}^2 \times M$ with two
      negative ends $S$, $S'$ and one positive end $T'$. The cobordism
      $V$ should also be monotone for suitable choices of handles in
      the figure-Y surgery. The cone decomposition
      in~\eqref{eq:cone-S-T-2} would now follow from the main results
      of~\cite{Bi-Co:lcob-fuk}.
     \item[b.] Our work does not provide much information about the
      precise morphism $S \longrightarrow S$
      from~\eqref{eq:cone-S-T-2}. It would be interesting to determine
      the precise map and also to figure out how~\eqref{eq:cone-S-T-2}
      behaves with respect to grading (in this case a
      $\mathbb{Z}_2$-grading).
   \end{enumerate}
\end{remsnonum}

\subsubsection*{A few variations on the same example}
One can alter the construction of $E$ and $V$ to obtain a Lefschetz
fibrations $\pi: E' \longrightarrow \mathbb{C}$ with more critical
values. This can be done for example by choosing the disk $D$ to
contain the point $[1:-2]$ and none of the other points
from~\eqref{eq:ell-Delta}. The result will then be a fibration with
three critical values - one lying on the $x$-axis and another pair of
critical points conjugate one to the other. The cobordism $V$ in this
case would still be between a Lagrangian sphere and a torus.

If one chooses the disk $D$ not to contain any of the points
in~\eqref{eq:ell-Delta} and its center to lie somewhere along the
interval $[1:x]$, $x \in [-2,2]$, then the fibration will have four
critical values, two real ones and to conjugate ones. The cobordism
$V$ will have a Lagrangian $S^2$ on its both ends, and the topology of
$V$ will still be non-trivial (i.e. $V$ will not be diffeomorphic to
$\mathbb{R} \times S^2$). A similar example with Lagrangian
$\mathbb{T}^2$'s on both ends can be constructed by taking the disk to
have its center somewhere along $[1:x]$, $x > 2$.

%\bibliography{/home/biran/latex/general/bibliography}

\bibliography{bibliography}

\def\cprime{$'$} \def\cprime{$'$}
\begin{thebibliography}{FOOO}

\bibitem[Arn]{Ar:monodromy}
V.~Arnold.
\newblock Some remarks on symplectic monodromy of {M}ilnor fibrations.
\newblock In {\em The Floer memorial volume}, volume 133 of {\em Progr. Math.},
  pages 99--103. Birkh\"{a}user Verlag, Basel, 1995.

\bibitem[AS]{Ab-Sm:Kho}
M.~Abouzaid and I.~Smith.
\newblock Khovanov homology from floer cohomology.
\newblock Preprint, 2015.

\bibitem[BC1]{Bi-Co:rigidity}
P.~Biran and O.~Cornea.
\newblock Rigidity and uniruling for {L}agrangian submanifolds.
\newblock {\em Geom. Topol.}, 13(5):2881--2989, 2009.

\bibitem[BC2]{Bi-Co:cob1}
P.~Biran and O.~Cornea.
\newblock Lagrangian cobordism. {I}.
\newblock {\em J. Amer. Math. Soc.}, 26(2):295--340, 2013.

\bibitem[BC3]{Bi-Co:lcob-fuk}
P.~Biran and O.~Cornea.
\newblock Lagrangian cobordism and {F}ukaya categories.
\newblock {\em Geom. Funct. Anal.}, 24(6):1731--1830, 2014.

\bibitem[Che]{Chek:cob}
Y.~Chekanov.
\newblock Lagrangian embeddings and {L}agrangian cobordism.
\newblock In {\em Topics in singularity theory}, volume 180 of {\em Amer. Math.
  Soc. Transl. Ser. 2}, pages 13--23. Amer. Math. Soc., Providence, RI, 1997.

\bibitem[DS]{De-Sa:products}
A.~Degtyarev and N.~Salepci.
\newblock Products of pairs of {D}ehn twists and maximal real {L}efschetz
  fibrations.
\newblock {\em Nagoya Math. J.}, 210:83--132, 2013.

\bibitem[Flo]{Fl:Morse-theory}
A.~Floer.
\newblock Morse theory for {L}agrangian intersections.
\newblock {\em J. Differential Geom.}, 28(3):513--547, 1988.

\bibitem[FOOO]{FO3:book-chap-10}
K.~Fukaya, Y.-G. Oh, H.~Ohta, and K.~Ono.
\newblock Lagrangian intersection {F}loer theory - anomaly and obstruction,
  chapter 10.
\newblock Preprint, can be found at
  \url{http://www.math.kyoto-u.ac.jp/\~fukaya/Chapter10071117.pdf}.

\bibitem[HAV]{Ho-Iq-Va:ms}
K.~Hori, A.Iqbal, and C.~Vafa.
\newblock D-branes and mirror symmetry.
\newblock Preprint (2000), can be found at
  \url{http://arxiv.org/abs/hep-th/0005247}.

\bibitem[Hin]{Hind:lag-S2xS2}
R.~Hind.
\newblock Lagrangian spheres in {$S\sp 2\times S\sp 2$}.
\newblock {\em Geom. Funct. Anal.}, 14(2):303--318, 2004.

\bibitem[LS]{La-Si:Lag}
F.~Lalonde and J.-C. Sikorav.
\newblock Sous-vari\'{e}t\'{e}s {L}agrangiennes et {L}agrangiennes exactes des
  fibr\'{e}s cotangents.
\newblock {\em Comment. Math. Helv.}, 66(1):18--33, 1991.

\bibitem[MS1]{McD-Sa:Intro}
D.~McDuff and D.~Salamon.
\newblock {\em Introduction to symplectic topology}.
\newblock Oxford Mathematical Monographs. The Clarendon Press, Oxford
  University Press, New York, second edition, 1998.

\bibitem[MS2]{McD-Sa:jhol}
D.~McDuff and D.~Salamon.
\newblock {\em {$J$}-holomorphic curves and symplectic topology}, volume~52 of
  {\em American Mathematical Society Colloquium Publications}.
\newblock American Mathematical Society, Providence, RI, second edition, 2012.

\bibitem[Oh1]{Oh:HF1}
Y.-G. Oh.
\newblock Floer cohomology of {L}agrangian intersections and pseudo-holomorphic
  disks. {I}.
\newblock {\em Comm. Pure Appl. Math.}, 46(7):949--993, 1993.

\bibitem[Oh2]{Oh:HF1-add}
Y.-G. Oh.
\newblock Addendum to: "{F}loer cohomology of {L}agrangian intersections and
  pseudo-holomorphic disks. {I}.".
\newblock {\em Comm. Pure Appl. Math.}, 48(11):1299--1302, 1995.

\bibitem[Oh3]{Oh:Seidel-triangle}
Y.-G. Oh.
\newblock Seidel's long exact sequence on {C}alabi-{Y}au manifolds.
\newblock {\em Kyoto J. Math.}, 51(3):687--765, 2011.

\bibitem[Pol]{Po:surgery}
L.~Polterovich.
\newblock The surgery of {L}agrange submanifolds.
\newblock {\em Geom. Funct. Anal.}, 1(2):198--210, 1991.

\bibitem[Sal1]{Sal:real-elements}
N.~Salepci.
\newblock Real elements in the mapping class group of {$T\sp 2$}.
\newblock {\em Topology Appl.}, 157(16):2580--2590, 2010.

\bibitem[Sal2]{Sal:classif-tr}
N.~Salepci.
\newblock Classification of totally real elliptic {L}efschetz fibrations via
  necklace diagrams.
\newblock {\em J. Knot Theory Ramifications}, 21(9):1250089, 28, 2012.

\bibitem[Sal3]{Sal:invariants-tr}
N.~Salepci.
\newblock Invariants of totally real {L}efschetz fibrations.
\newblock {\em Pacific J. Math.}, 256(2):407--434, 2012.

\bibitem[Sei1]{Se:knotted}
P.~Seidel.
\newblock Lagrangian two-spheres can be symplectically knotted.
\newblock {\em J. Differential Geom.}, 52(1):145--171, 1999.

\bibitem[Sei2]{Se:long-exact}
P.~Seidel.
\newblock A long exact sequence for symplectic floer cohomology.
\newblock {\em Topology}, 42(5):1003--1063, 2003.

\bibitem[Sei3]{Se:book-fukaya-categ}
P.~Seidel.
\newblock {\em Fukaya categories and {P}icard-{L}efschetz theory}.
\newblock Zurich Lectures in Advanced Mathematics. European Mathematical
  Society (EMS), Z\"urich, 2008.

\bibitem[Sei4]{Se:Lefschetz-Fukaya}
P.~Seidel.
\newblock Fukaya {$A\sb \infty$}-structures associated to {L}efschetz
  fibrations. {I}.
\newblock {\em J. Symplectic Geom.}, 10(3):325--388, 2012.

\bibitem[Sei5]{Se:homology-spheres}
P.~Seidel.
\newblock Lagrangian homology spheres in ${A}_m$ {M}ilnor fibres via
  ${C}^*$-equivariant ${A}_{\infty}$ modules.
\newblock {\em Geometry and Topology}, (16):2343--2389, 2012.

\bibitem[Tho]{Th-RP:moment}
R.~P. Thomas.
\newblock Moment maps, monodromy and mirror manifolds.
\newblock In {\em Symplectic geometry and mirror symmetry ({S}eoul, 2000)},
  pages 467--498. World Sci. Publ., River Edge, NJ, 2001.

\bibitem[Toe]{To:dg-cat}
B.~Toen.
\newblock Lectures on dg-categories.
\newblock In {\em Topics in algebraic and topological K-theory}, Lecture Notes
  in Math., pages 243--302. Springer, Berlin, 2011.

\bibitem[WW]{We-Wo-Dehn}
K.~Wehrheim and C.~Woodward.
\newblock Exact triangle for fibered {D}ehn twists.
\newblock Preprint 2015, can be found at
  \url{http://arxiv.org/abs/1503.07614v1}.

\end{thebibliography}

%
%\begin{thebibliography}{10}
%
%\end{thebibliography}
%

\end{document}